\documentclass{CUP-JNL-FMP}%

\usepackage[full]{textcomp}
\usepackage[osf]{newtxtext}
\usepackage[bigdelims,vvarbb]{newtxmath}
\usepackage[cal=boondoxo]{mathalfa}

\usepackage{yfonts}
\usepackage{enumitem}

\usepackage{tikz}
\usepackage{mathrsfs}
\usepackage{mhequ}
\usepackage{mhenvs}
\usepackage{shortcuts}

\usepackage{comment}

\usepackage{centernot}
\usepackage{wasysym}

\usepackage{booktabs}
\usepackage{longtable}

\usepackage{makeidx}
\usepackage{array}   

\usepackage{hyperref} 
\usepackage{breakurl}
\usepackage{microtype}

\makeindex

\usetikzlibrary{external}

\colorlet{mygrey}{black!50!white}

\colorlet{symbols}{black}
\colorlet{treesDarkblue}{blue!80!black}
\colorlet{treesDarkred}{red!70!black}
\colorlet{treesDarkgreen}{green!50!black}

\colorlet{darkblue}{blue!90!black}
\colorlet{darkred}{red!90!black}
\colorlet{darkgreen}{green!75!black}

\colorlet{lightblue}{blue!25!white}
\colorlet{lightred}{red!60!white}
\colorlet{lightgreen}{green!60!white}

\def\gCT{{\color{darkgreen}{\CT}}}
\def\gCM{{\color{darkgreen}{\CM_0}}}
\def\gCTm{{\color{darkgreen}{\CT_-}}}
\def\gR{{\color{darkgreen}{R}}}
\def\gXi{{\color{darkgreen}{\Xi}}}

\def\gFLm{{\color{darkgreen}{\FL_-}}}
\def\gFLp{{\color{darkgreen}{\FL_+}}}
\def\gFL{{\color{darkgreen}{\FL}}}

\usetikzlibrary{shapes}
\usetikzlibrary{shapes.misc}
\usetikzlibrary{shapes.symbols}

\usetikzlibrary{positioning,arrows,patterns}
\usetikzlibrary{decorations}
\usetikzlibrary{decorations.markings}
\usetikzlibrary{decorations.shapes}
\usetikzlibrary{decorations.pathmorphing}
\usetikzlibrary{calc}

\makeatletter
\def\DeclareSymbol#1#2#3{%
	\expandafter\gdef\csname MH@symb@#1\endcsname{\tikzsetnextfilename{symbol#1}%
	\tikz[baseline=#2,scale=0.15,draw=symbols,line join=round]{#3}}%
	\expandafter\gdef\csname MH@symb@#1s\endcsname{\scalebox{0.75}{\tikzsetnextfilename{symbol#1}%
	\tikz[baseline=#2,scale=0.15,draw=symbols,line join=round]{#3}}}%
	\expandafter\gdef\csname MH@symb@#1ss\endcsname{\scalebox{0.65}{\tikzsetnextfilename{symbol#1}%
	\tikz[baseline=#2,scale=0.15,draw=symbols,line join=round]{#3}}}%
	}
\def\<#1>{\ifthenelse{\boolean{mmode}}{\mathchoice{\csname MH@symb@#1\endcsname}{\csname MH@symb@#1\endcsname}{\csname MH@symb@#1s\endcsname}{\csname MH@symb@#1ss\endcsname}}{\csname MH@symb@#1\endcsname}}
\makeatother

\definecolor{btube}{rgb}{0.82, 0.62, 0.91}

\tikzset{
	root/.style={circle, fill=black,inner sep=0pt,minimum size=2.5pt},
	node/.style={circle, fill=black,inner sep=0pt,minimum size=2.5pt},
	nodec/.style={circle, fill=lightred,inner sep=0pt,minimum size=2.5pt},
	noded/.style={circle, fill=lightgreen,inner sep=0pt,minimum size=2.5pt},
	rootSmall/.style={circle,draw=black,fill=white,inner sep=0pt,minimum size=1.5 pt},
	nodeSmall/.style={circle, fill=black,inner sep=0pt,minimum size=1pt},
	nodeS/.style={circle, fill=black,inner sep=0pt,minimum size=2pt},
	leaf/.style={circle, draw=black,thin,inner sep=0pt,minimum size=5pt},
	leafb/.style={circle, draw=black, fill = lightblue,thin,inner sep=0pt,minimum size=5pt},
	leafc/.style={circle, draw=black, fill = lightred,thin,inner sep=0pt,minimum size=5pt},
	leafd/.style={circle, draw=black, fill = lightgreen,thin,inner sep=0pt,minimum size=5pt},
	leafs/.style={circle, draw=black,thin, fill = lightblue, inner sep=0pt, minimum size=3pt},
	leafsa/.style={circle, draw=black,thin, fill = treesDarkblue, inner sep=0pt, minimum size=3pt},
	leafsE/.style={rectangle, draw=black,thin, fill = lightblue, inner sep=0pt, minimum size=3pt},
	leafsaE/.style={rectangle, draw=black,thin, fill = treesDarkblue, inner sep=0pt, minimum size=3pt},
	leafsEE/.style={rectangle, draw=black,thin, fill = lightred, inner sep=0pt, minimum size=3pt},
	leafsaEE/.style={rectangle, draw=black,thin, fill = treesDarkred, inner sep=0pt, minimum size=3pt},
	leafsb/.style={rectangle, draw=black,thin, fill=white,inner sep=0pt,minimum size=3pt},
	leafsc/.style={circle, draw=black,thin, fill = lightred, inner sep=0pt, minimum size=3pt},
	leafsd/.style={circle, draw=black,thin, fill = lightgreen, inner sep=0pt, minimum size=3pt},
	leafsx/.style={cross out, draw=black,inner sep=0pt, minimum size=2pt},
	leafsxw/.style={cross out, draw=white,inner sep=0pt, minimum size=2pt},
	leafsy/.style={diamond, draw=black,inner sep=0pt,minimum size=2pt},
	ctrcn/.style={circle, draw=black, fill=black, inner sep=0pt, minimum size = 2pt},
	leafsyellow/.style={rectangle, fill=black,inner sep=0pt,minimum size=3pt},
	leafsgreen/.style={rectangle, draw=black, fill = lightblue, inner sep=0pt, minimum size=3pt},
	leafsred/.style={circle, draw=darkred, fill = lightred, inner sep=0pt, minimum size=3pt},
	psi/.style={star,star points=7,star point ratio=0.6, draw=black,inner sep=0pt,minimum size=4pt},
	edgeA/.style={},
	edgeAbold/.style={line width=2pt, color=mygrey},
	edgeAbolds/.style={line width=1.5pt, color=mygrey},
	edgeB/.style={decorate, decoration={snake,amplitude=0.5,segment length = 1.7}},	
	edgeF/.style={decorate,decoration={snake,amplitude=0.6,segment length = 1}},
	edgeC/.style={decorate,decoration={coil,amplitude=1.2,aspect=-0.8,segment length=2.5pt},color=blue},
	edgeD/.style={decorate,decoration={coil,amplitude=1.2,aspect=-0.8,segment length=1.5pt},color=blue},
	edgeE/.style={color=darkblue},	
	edgeG/.style={dotted},
	edgeNshift/.style={decorate,decoration={coil,amplitude=0.6,aspect=-1.0,segment length=1.1pt},color=darkgreen},
	edgeNshiftB/.style={decorate,decoration={zigzag,amplitude=0.4,segment length=3pt}},
	edgeNshiftC/.style={decorate,decoration={expanding waves,amplitude=1.8,segment length=1pt,angle=17}},
	kernelF/.style={decorate, decoration={snake,amplitude=1.5,segment length = 4}},
	kernelT/.style={},
	kernelEff/.style={double},
	kernelG/.style={dotted},
	sroot/.style={circle,draw=black,fill=white,inner sep=0pt,minimum size=2 pt},
	sleaf/.style={circle, fill=black,inner sep=0pt,minimum size=1pt},
	skernelF/.style={decorate, decoration={snake,amplitude=0.8,segment length = 2}},
	leg/.style={line width=1pt, color=mygrey}
}

\def\x{0.2}
\def\sPL{1.2}

\newcommand{\nodec}{
	\begin{tikzpicture}[baseline=(basepoint)]
	\node(basepoint) at (0,-0.1) {};

	\coordinate[nodec] (root) at (0,0);
	\end{tikzpicture}
}
\newcommand{\noded}{
	\begin{tikzpicture}[baseline=(basepoint)]
	\node(basepoint) at (0,-0.1) {};

	\coordinate[noded] (root) at (0,0);
	\end{tikzpicture}
}
\newcommand{\leafs}{
	\begin{tikzpicture}[baseline=(basepoint)]
	\node(basepoint) at (0,-0.1) {};

	\coordinate[leafs] (root) at (0,0);
	\end{tikzpicture}
}
\newcommand{\leafsx}{
	\begin{tikzpicture}[baseline=(basepoint)]
	\node(basepoint) at (0,-0.1) {};

	\coordinate[leafs] (root) at (0,0);
	\coordinate[leafsx] (root) at (0,0);
	\end{tikzpicture}
}
\newcommand{\leafslow}{
\hspace{-0.05cm}
	\begin{tikzpicture}[baseline=(basepoint)]
	\node(basepoint) at (0,-0.06) {};

	\coordinate[leafs] (root) at (0,0);
	\end{tikzpicture}
}

\newcommand{\leafsa}{
	\begin{tikzpicture}[baseline=(basepoint)]
	\node(basepoint) at (0,-0.1) {};

	\coordinate[leafsa](root) at (0,0);
	\end{tikzpicture}
}
\newcommand{\leafsE}{
	\begin{tikzpicture}[baseline=(basepoint)]
	\node(basepoint) at (0,-0.1) {};
	
	\coordinate[leafsE](root) at (0,0);
	\end{tikzpicture}
}
\newcommand{\leafsaE}{
	\begin{tikzpicture}[baseline=(basepoint)]
	\node(basepoint) at (0,-0.1) {};
	
	\coordinate[leafsaE](root) at (0,0);
	\end{tikzpicture}
}
\newcommand{\leafsEE}{
	\begin{tikzpicture}[baseline=(basepoint)]
	\node(basepoint) at (0,-0.1) {};
	
	\coordinate[leafsEE](root) at (0,0);
	\end{tikzpicture}
}
\newcommand{\leafsaEE}{
	\begin{tikzpicture}[baseline=(basepoint)]
	\node(basepoint) at (0,-0.1) {};
	
	\coordinate[leafsaEE](root) at (0,0);
	\end{tikzpicture}
}
\newcommand{\leafsalow}{
\hspace{-0.05cm}
	\begin{tikzpicture}[baseline=(basepoint)]
	\node(basepoint) at (0,-0.06) {};

	\coordinate[leafsa] (root) at (0,0);
	\end{tikzpicture}
}
\newcommand{\leafsb}{
	\begin{tikzpicture}[baseline=(basepoint)]
	\node(basepoint) at (0,-0.1) {};

	\coordinate[leafsb] (root) at (0,0);
	\coordinate[leafsx] (root) at (0,0);
	\end{tikzpicture}
}

\newcommand{\leafsc}{
	\begin{tikzpicture}[baseline=(basepoint)]
	\node(basepoint) at (0,-0.1) {};

	\coordinate[leafsc] (root) at (0,0);
	\end{tikzpicture}
}

\newcommand{\leafsd}{
	\begin{tikzpicture}[baseline=(basepoint)]
	\node(basepoint) at (0,-0.1) {};

	\coordinate[leafsd] (root) at (0,0);
	\end{tikzpicture}
}

\newcommand{\legRed}{
	\begin{tikzpicture}[baseline=(basepoint)]
	\node(basepoint) at (0,0.05) {};

	\coordinate[] (root) at (0,0);
	\coordinate[] (node1) at (0,0.3);

	\draw[edgeB,color=red] (root)--(node1);
	\end{tikzpicture}
}
\newcommand{\legBlue}{
	\begin{tikzpicture}[baseline=(basepoint)]
	\node(basepoint) at (0,0.05) {};

	\coordinate[] (root) at (0,0);
	\coordinate[] (node1) at (0,0.3);

	\draw[edgeB,color=blue] (root)--(node1);
	\end{tikzpicture}
}

\newcommand{\treeExampleProperlyLegged}{
	\begin{tikzpicture}[baseline=(basepoint)]
	\node(basepoint) at (0,0.2) {};
	
	\coordinate[root] (root) at (0,0);
	\coordinate[leafs] (leaf1) at (\x*\sPL,0.25*\sPL);
	\coordinate[leafs]  (leaf2) at (-\x*\sPL,0.25*\sPL);
	\coordinate[leafs] (leaf3) at (0,0.5*\sPL);

	\coordinate[] (leg21) at (-2*\x*\sPL, 0.5*\sPL);
	\coordinate[] (leg23) at (-1*\x*\sPL, 0.6*\sPL);
	\coordinate[] (leg12) at (2*\x*\sPL, 0.5 *\sPL);
	\coordinate[] (leg13) at (1*\x*\sPL,0.6*\sPL);
	\coordinate[] (leg31) at (-\x*\sPL,0.75*\sPL);
	\coordinate[] (leg32) at (\x*\sPL,0.75*\sPL);

	\coordinate[] (leg34) at (0,0.85*\sPL);
	\coordinate[] (leg43) at (0,0.35*\sPL);

	\draw[edgeAbolds] (root)--(leaf1);
	\draw[edgeAbolds] (root)--(leaf2);
	\draw[edgeA] 	  (leaf2)--(leaf3);

	\draw[edgeB,color=blue] (leaf1)--(leg12);
	\draw[edgeB,color=red] (leaf1)--(leg13);
	\draw[edgeF,color=blue] (leaf2)--(leg21);
	\draw[edgeB,color=darkgreen] (leaf2)--(leg23);
	\draw[edgeF,color=red] (leaf3)--(leg31);
	\draw[edgeF,color=darkgreen] (leaf3)--(leg32);

	\draw[edgeB,color=gray] (leaf3)--(leg34);
	\draw[edgeF,color=gray] (root)--(leg43);
	\end{tikzpicture}
}
\newcommand{\treeExampleProperlyNOLEGS}{
	\begin{tikzpicture}[baseline=(basepoint)]
	\node(basepoint) at (0,0.2) {};
	
	\coordinate[root] (root) at (0,0);
	\coordinate[leafs] (leaf1) at (\x,0.25);
	\coordinate[leafs]  (leaf2) at (-\x,0.25);
	\coordinate[leafs] (leaf3) at (0,0.5);

	\draw[edgeAbolds] (root)--(leaf1);
	\draw[edgeAbolds] (root)--(leaf2);
	\draw[edgeA] 	  (leaf2)--(leaf3);
	\end{tikzpicture}
}
\newcommand{\treeExampleProperlyAd}{
	\begin{tikzpicture}[baseline=(basepoint)]
	\node(basepoint) at (0,0.2) {};
	
	\coordinate[root] (root) at (0,0);
	\coordinate[leafs] (leaf1) at (\x*\sPL,0.25*\sPL);
	\coordinate[leafs]  (leaf2) at (-\x*\sPL,0.25*\sPL);
	\coordinate[leafs] (leaf3) at (0,0.5*\sPL);

	\coordinate[] (leg21) at (-2*\x*\sPL, 0.5*\sPL);
	\coordinate[] (leg23) at (-1*\x*\sPL, 0.6*\sPL);
	\coordinate[] (leg12) at (2*\x*\sPL, 0.5 *\sPL);
	\coordinate[] (leg13) at (1*\x*\sPL,0.6*\sPL);
	\coordinate[] (leg31) at (-\x*\sPL,0.75*\sPL);
	\coordinate[] (leg32) at (\x*\sPL,0.75*\sPL);

	\draw[edgeAbolds] (root)--(leaf1);
	\draw[edgeAbolds] (root)--(leaf2);
	\draw[edgeA] 	  (leaf2)--(leaf3);

	\draw[edgeB,color=blue] (leaf1)--(leg12);
	\draw[edgeB,color=red] (leaf1)--(leg13);
	\draw[edgeF,color=blue] (leaf2)--(leg21);
	\draw[edgeB,color=darkgreen] (leaf2)--(leg23);
	\draw[edgeF,color=red] (leaf3)--(leg31);
	\draw[edgeF,color=darkgreen] (leaf3)--(leg32);
	\end{tikzpicture}
}
\newcommand{\treeExampleProperlyLeggedSym}{
	\begin{tikzpicture}[baseline=(basepoint)]
	\node(basepoint) at (0,0.2) {};
	
	\coordinate[root] (root) at (0,0);
	\coordinate[leafs] (leaf1) at (\x*\sPL,0.25*\sPL);
	\coordinate[leafs]  (leaf2) at (-\x*\sPL,0.25*\sPL);
	\coordinate[leafs] (leaf3) at (0,0.5*\sPL);

	\coordinate[] (leg21) at (-2*\x*\sPL, 0.5*\sPL);
	\coordinate[] (leg23) at (-1*\x*\sPL, 0.6*\sPL);
	\coordinate[] (leg12) at (2*\x*\sPL, 0.5 *\sPL);
	\coordinate[] (leg13) at (1*\x*\sPL,0.6*\sPL);
	\coordinate[] (leg31) at (-\x*\sPL,0.75*\sPL);
	\coordinate[] (leg32) at (\x*\sPL,0.75*\sPL);

	\coordinate[] (leg34) at (0,0.85*\sPL);
	\coordinate[] (leg43) at (0,0.35*\sPL);

	\draw[edgeAbolds] (root)--(leaf1);
	\draw[edgeAbolds] (root)--(leaf2);
	\draw[edgeA] 	  (leaf2)--(leaf3);

	\draw[edgeB,color=red] (leaf1)--(leg12);
	\draw[edgeB,color=blue] (leaf1)--(leg13);
	\draw[edgeF,color=red] (leaf2)--(leg21);
	\draw[edgeB,color=darkgreen] (leaf2)--(leg23);
	\draw[edgeF,color=blue] (leaf3)--(leg31);
	\draw[edgeF,color=darkgreen] (leaf3)--(leg32);

	\draw[edgeB,color=gray] (leaf3)--(leg34);
	\draw[edgeF,color=gray] (root)--(leg43);
	\end{tikzpicture}
}

\newcommand{\treeExamplePAMa}{
	\begin{tikzpicture}[baseline=(basepoint)]
	\node(basepoint) at (0,0.0) {};
	
	\coordinate[root] (root) at (0,0);
	\coordinate[leafs] (node1) at (-0.15,0.25);
	\coordinate[leafs] (node2) at (0.15,0.25);

	\draw[edgeAbolds] (root)--(node1) node [midway, below left=-0.15] {$\scriptscriptstyle {1}$};
	\draw[edgeAbolds] (root)--(node2) node [midway, below right=-0.15] {$\scriptscriptstyle {1}$};
	\end{tikzpicture}
}
\newcommand{\treeExamplePAMb}{
	\begin{tikzpicture}[baseline=(basepoint)]
	\node(basepoint) at (0,0.0) {};
	
	\coordinate[root] (root) at (0,0);
	\coordinate[leafs] (node1) at (-0.15,0.25);
	\coordinate[leafs] (node2) at (0.15,0.25);

	\draw[edgeAbolds] (root)--(node1) node [midway, below left=-0.15] {$\scriptscriptstyle {2}$};
	\draw[edgeAbolds] (root)--(node2) node [midway, below right=-0.15] {$\scriptscriptstyle {2}$};
	\end{tikzpicture}
}
\newcommand{\treeExamplePAMc}{
	\begin{tikzpicture}[baseline=(basepoint)]
	\node(basepoint) at (0,0.0) {};
	
	\coordinate[root] (root) at (0,0);
	\coordinate[leafs] (node1) at (-0.15,0.25);
	\coordinate[leafs] (node2) at (0.15,0.25);

	\draw[edgeAbolds] (root)--(node1) node [midway, below left=-0.15] {$\scriptscriptstyle {3}$};
	\draw[edgeAbolds] (root)--(node2) node [midway, below right=-0.15] {$\scriptscriptstyle {3}$};
	\end{tikzpicture}
}
\newcommand{\treeExamplePAM}{
	\begin{tikzpicture}[baseline=(basepoint)]
	\node(basepoint) at (0,0.0) {};
	
	\coordinate[leafs] (root) at (0,0);
	\coordinate[leafs] (node1) at (-0.15,0.25);

	\draw[edgeA] (root)--(node1);
\end{tikzpicture}
}
\newcommand{\treePhia}{
	\begin{tikzpicture}[baseline=(basepoint)]
	\node(basepoint) at (0,0.0) {};
	
	\coordinate[node] (root) at (0,0);
	\coordinate[node] (node0) at (0,0.3);
	\coordinate[leafs] (node1) at (-0.15,0.25);
	\coordinate[leafs] (node2) at (0.15,0.25);
	\coordinate[leafs] (node3) at (-0.15,0.55);
	\coordinate[leafs] (node4) at (0.15,0.55);

	\draw[edgeA] (root)--(node0);
	\draw[edgeA] (root)--(node1);
	\draw[edgeA] (root)--(node2);
	\draw[edgeA] (node0)--(node3);
	\draw[edgeA] (node0)--(node4);
\end{tikzpicture}
}
\newcommand{\treePhib}{
	\begin{tikzpicture}[baseline=(basepoint)]
	\node(basepoint) at (0,0.0) {};
	
	\coordinate[node] (root) at (0,0);
	\coordinate[node] (node0) at (0,0.3);
	\coordinate[leafs] (node1) at (-0.15,0.25);
	\coordinate[leafs] (node2) at (-0.15,0.55);
	\coordinate[leafs] (node3) at (-0.0,0.6);
	\coordinate[leafs] (node4) at (0.15,0.55);

	\draw[edgeA] (root)--(node0);
	\draw[edgeA] (root)--(node1);
	\draw[edgeA] (node0)--(node2);
	\draw[edgeA] (node0)--(node3);
	\draw[edgeA] (node0)--(node4);
\end{tikzpicture}
}
\newcommand{\treeExampleKPZa}{
	\begin{tikzpicture}[baseline=(basepoint)]
	\node(basepoint) at (0,0.1) {};
	
	\coordinate[root] (root) at (0,0);
	\coordinate[leafs] (leaf1) at (0.18,0.25);
	\coordinate[leafs] (leaf2) at (0,0.3);
	\coordinate[node] (node) at (-0.18,0.25);
	\coordinate[leafs] (leaf3) at (-0.36,0.5);
	\coordinate[leafs] (leaf4) at (-0.18,0.5);

	\draw[edgeAbolds] (root)--(node);
	\draw[edgeA] (root)--(leaf1);
	\draw[edgeAbolds] (root)--(leaf2);
	\draw[edgeAbolds] (node)--(leaf3);
	\draw[edgeAbolds] (node)--(leaf4);
	\end{tikzpicture}
}
\newcommand{\treeExampleKPZb}{
	\begin{tikzpicture}[baseline=(basepoint)]
	\node(basepoint) at (0,0.1) {};
	
	\coordinate[root] (root) at (0,0);
	\coordinate[leafs] (leaf1) at (0.18,0.25);
	\coordinate[leafs] (leaf2) at (0,0.3);
	\coordinate[node] (node) at (-0.18,0.25);
	\coordinate[leafs] (leaf3) at (-0.36,0.5);
	\coordinate[leafs] (leaf4) at (-0.18,0.5);

	\draw[edgeAbolds] (root)--(node);
	\draw[edgeAbolds] (root)--(leaf1);
	\draw[edgeA] (root)--(leaf2);
	\draw[edgeAbolds] (node)--(leaf3);
	\draw[edgeAbolds] (node)--(leaf4);
	\end{tikzpicture}
}
\newcommand{\treeExampleKPZc}{
	\begin{tikzpicture}[baseline=(basepoint)]
	\node(basepoint) at (0,0.1) {};
	
	\coordinate[root] (root) at (0,0);
	\coordinate[leafs] (leaf1) at (0.18,0.25);
	\coordinate[leafs] (leaf2) at (0,0.3);
	\coordinate[node] (node) at (-0.18,0.25);
	\coordinate[leafs] (leaf3) at (-0.36,0.5);
	\coordinate[leafs] (leaf4) at (-0.18,0.5);

	\draw[edgeA] (root)--(node);
	\draw[edgeAbolds] (root)--(leaf1);
	\draw[edgeAbolds] (root)--(leaf2);
	\draw[edgeAbolds] (node)--(leaf3);
	\draw[edgeAbolds] (node)--(leaf4);
	\end{tikzpicture}
}
\newcommand{\treeExampleAsym}{
	\begin{tikzpicture}[baseline=(basepoint)]
	\node(basepoint) at (0,0.1) {};
	
	\coordinate[leafs] (node1) at (0,0);
	\coordinate[leafs] (node2) at (0.23,0.20);
	\coordinate[leafs] (node3) at (0,0.40);
	\coordinate[leafs] (node4) at (0.23,0.60);

	\draw[edgeA] (node1)--(node2);
	\draw[edgeAbolds] (node2)--(node3);
	\draw[edgeA] (node3)--(node4);
	\end{tikzpicture}
}
\newcommand{\treeExampleRootA}{
	\begin{tikzpicture}[baseline=(basepoint)]
	\node(basepoint) at (0,0.1) {};
	
	\coordinate[leafs] (root) at (0,0);
	\coordinate[leafs] (leaf1) at (0.18,0.25);
	\coordinate[node] (node) at (-0.18,0.25);
	\coordinate[leafs] (leaf2) at (-0.36,0.5);
	\coordinate[leafs] (leaf3) at (0,0.5);

	\draw[edgeA] (root)--(node);
	\draw[edgeA] (root)--(leaf1);
	\draw[edgeAbolds] (node)--(leaf2);
	\draw[edgeAbolds] (node)--(leaf3);
	\end{tikzpicture}
}
\newcommand{\treeExampleRootB}{
	\begin{tikzpicture}[baseline=(basepoint)]
	\node(basepoint) at (0,0.1) {};
	
	\coordinate[node] (root) at (0,0);
	\coordinate[leafs] (leaf1) at (0.18,0.25);
	\coordinate[leafs] (node) at (0,0.32);
	\coordinate[leafs] (leaf2) at (-0.18,0.25);
	\coordinate[leafs] (leaf3) at (-0.18,0.5);

	\draw[edgeAbolds] (root)--(node);
	\draw[edgeAbolds] (root)--(leaf1);	
	\draw[edgeA,color=red] (root)--(leaf2);
	\draw[edgeA] (leaf2)--(leaf3);
	\end{tikzpicture}
}
\newcommand{\treeExampleKPZaKernel}{
	\begin{tikzpicture}[baseline=(basepoint)]
	\node(basepoint) at (0,0.1) {};
	
	\coordinate[root] (root) at (0,0);
	\coordinate[node] (leaf1) at (0.18,0.25);
	\coordinate[node] (leaf2) at (0,0.3);
	\coordinate[node] (node) at (-0.18,0.25);
	\coordinate[node] (leaf3) at (-0.36,0.5);
	\coordinate[node] (leaf4) at (-0.18,0.5);

	\node [above left=-0.16 of leaf3] 	{$\scriptscriptstyle x_1$};
	\node [above=-0.11 of leaf4] 	{$\scriptscriptstyle x_2$};
	\node [above=-0.11 of leaf2] 	{$\scriptscriptstyle \,\,\,x_3$};
	\node [above right=-0.2 of leaf1] 	{$\scriptscriptstyle x_4$};

	\draw[edgeAbolds] (root)--(node);
	\draw[edgeA] (root)--(leaf1);
	\draw[edgeAbolds] (root)--(leaf2);
	\draw[edgeAbolds] (node)--(leaf3);
	\draw[edgeAbolds] (node)--(leaf4);
	\end{tikzpicture}
}

\def\a{0.9}
\newcommand{\treeExampleWeakEdges}{
	\begin{tikzpicture}[baseline=(basepoint)]
	\node(basepoint) at (0,0.1) {};

	\coordinate[root] (root) at (0,0);
	\coordinate[root] (node) at (0,0.9*\a);
	\coordinate[root] (leaf1) at (-0.4*\a, 0.45*\a);
	\coordinate[root] (leaf2) at (0.4*\a, 0.45*\a);
	\coordinate[root] (leaf3) at (-0.4*\a, 1.35*\a);
	\coordinate[root] (leaf4) at (0.4*\a,1.35*\a);

	\coordinate[root] (weak1) at (-0.5*\a,0.9*\a);
	\coordinate[root] (weak2) at (-0.3*\a,0.9*\a);
	\coordinate[root] (weak3) at (0 ,1.4*\a);
	\coordinate[root] (weak4) at (0.6*\a,0.9*\a);
	\coordinate[root] (weak5) at (0.4*\a,0.9*\a);
	\coordinate[root] (weak6) at (0.2*\a,0.9*\a);

	\draw[edgeA] (root) -- (node);
	\draw[edgeA] (root) -- (leaf1);
	\draw[edgeA] (root) -- (leaf2);
	\draw[edgeA] (node) -- (leaf3);
	\draw[edgeA] (node) -- (leaf4);

	\draw[edgeA,densely dotted,color=red] (leaf1) -- (weak1);
	\draw[edgeA,densely dotted,color=red] (leaf1) -- (weak2);
	\draw[edgeA,densely dotted,color=blue] (leaf3) -- (weak1);
	\draw[edgeA,densely dotted,color=blue] (leaf3) -- (weak2);
	\draw[edgeA,densely dotted,color=blue] (leaf3) -- (weak3);
	\draw[edgeA,densely dotted,color=darkgreen] (leaf4) -- (weak3);
	\draw[edgeA,densely dotted,color=darkgreen] (leaf4) -- (weak4);
	\draw[edgeA,densely dotted,color=darkgreen] (leaf4) -- (weak5);
	\draw[edgeA,densely dotted,color=darkgreen] (leaf4) -- (weak6);
	\draw[edgeA,densely dotted,color=orange] (leaf2) -- (weak4);
	\draw[edgeA,densely dotted,color=orange] (leaf2) -- (weak5);
	\draw[edgeA,densely dotted,color=orange] (leaf2) -- (weak6);
	\end{tikzpicture}
}

\newcommand{\treeExampleKPZaa}{
	\begin{tikzpicture}[baseline=(basepoint)]
	\node(basepoint) at (0,0.1) {};
	
	\coordinate[root]  (root) at (0,0);
	\coordinate[leafsa] (leaf1) at (0.18,0.25);
	\coordinate[leafs] (leaf2) at (0,0.3);
	\coordinate[node]  (node) at (-0.18,0.25);
	\coordinate[leafs] (leaf3) at (-0.36,0.5);
	\coordinate[leafsa] (leaf4) at (-0.18,0.5);

	\draw[edgeAbolds] (root)--(node);
	\draw[edgeA] (root)--(leaf1);
	\draw[edgeAbolds] (root)--(leaf2);
	\draw[edgeAbolds] (node)--(leaf3);
	\draw[edgeAbolds] (node)--(leaf4);
	\end{tikzpicture}
}
\newcommand{\treeExampleKPZaaEaa}{
	\begin{tikzpicture}[baseline=(basepoint)]
	\node(basepoint) at (0,0.1) {};
	
	\coordinate[root]  (root) at (0,0);
	\coordinate[leafsaE] (leaf1) at (0.18,0.25);
	\coordinate[leafs] (leaf2) at (0,0.3);
	\coordinate[node]  (node) at (-0.18,0.25);
	\coordinate[leafs] (leaf3) at (-0.36,0.5);
	\coordinate[leafsa] (leaf4) at (-0.18,0.5);
	
	\draw[edgeAbolds] (root)--(node);
	\draw[edgeA] (root)--(leaf1);
	\draw[edgeAbolds] (root)--(leaf2);
	\draw[edgeAbolds] (node)--(leaf3);
	\draw[edgeAbolds] (node)--(leaf4);
	\end{tikzpicture}
}
\newcommand{\treeExampleKPZaaEa}{
	\begin{tikzpicture}[baseline=(basepoint)]
	\node(basepoint) at (0,0.1) {};
	
	\coordinate[root]  (root) at (0,0);
	\coordinate[leafsa] (leaf1) at (0.18,0.25);
	\coordinate[leafsE] (leaf2) at (0,0.3);
	\coordinate[node]  (node) at (-0.18,0.25);
	\coordinate[leafs] (leaf3) at (-0.36,0.5);
	\coordinate[leafsa] (leaf4) at (-0.18,0.5);
	
	\draw[edgeAbolds] (root)--(node);
	\draw[edgeA] (root)--(leaf1);
	\draw[edgeAbolds] (root)--(leaf2);
	\draw[edgeAbolds] (node)--(leaf3);
	\draw[edgeAbolds] (node)--(leaf4);
	\end{tikzpicture}
}
\newcommand{\treeExampleKPZaaEb}{
	\begin{tikzpicture}[baseline=(basepoint)]
	\node(basepoint) at (0,0.1) {};
	
	\coordinate[root]  (root) at (0,0);
	\coordinate[leafsa] (leaf1) at (0.18,0.25);
	\coordinate[leafs] (leaf2) at (0,0.3);
	\coordinate[node]  (node) at (-0.18,0.25);
	\coordinate[leafsE] (leaf3) at (-0.36,0.5);
	\coordinate[leafsa] (leaf4) at (-0.18,0.5);
	
	\draw[edgeAbolds] (root)--(node);
	\draw[edgeA] (root)--(leaf1);
	\draw[edgeAbolds] (root)--(leaf2);
	\draw[edgeAbolds] (node)--(leaf3);
	\draw[edgeAbolds] (node)--(leaf4);
	\end{tikzpicture}
}
\newcommand{\treeExampleKPZaaEc}{
	\begin{tikzpicture}[baseline=(basepoint)]
	\node(basepoint) at (0,0.1) {};
	
	\coordinate[root]  (root) at (0,0);
	\coordinate[leafsa] (leaf1) at (0.18,0.25);
	\coordinate[leafs] (leaf2) at (0,0.3);
	\coordinate[node]  (node) at (-0.18,0.25);
	\coordinate[leafs] (leaf3) at (-0.36,0.5);
	\coordinate[leafsaE] (leaf4) at (-0.18,0.5);
	
	\draw[edgeAbolds] (root)--(node);
	\draw[edgeA] (root)--(leaf1);
	\draw[edgeAbolds] (root)--(leaf2);
	\draw[edgeAbolds] (node)--(leaf3);
	\draw[edgeAbolds] (node)--(leaf4);
	\end{tikzpicture}
}
\newcommand{\treeExampleKPZaaEEaa}{
	\begin{tikzpicture}[baseline=(basepoint)]
	\node(basepoint) at (0,0.1) {};
	
	\coordinate[root]  (root) at (0,0);
	\coordinate[leafsaEE] (leaf1) at (0.18,0.25);
	\coordinate[leafs] (leaf2) at (0,0.3);
	\coordinate[node]  (node) at (-0.18,0.25);
	\coordinate[leafs] (leaf3) at (-0.36,0.5);
	\coordinate[leafsa] (leaf4) at (-0.18,0.5);
	
	\draw[edgeAbolds] (root)--(node);
	\draw[edgeA] (root)--(leaf1);
	\draw[edgeAbolds] (root)--(leaf2);
	\draw[edgeAbolds] (node)--(leaf3);
	\draw[edgeAbolds] (node)--(leaf4);
	\end{tikzpicture}
}
\newcommand{\treeExampleKPZaaEEa}{
	\begin{tikzpicture}[baseline=(basepoint)]
	\node(basepoint) at (0,0.1) {};
	
	\coordinate[root]  (root) at (0,0);
	\coordinate[leafsa] (leaf1) at (0.18,0.25);
	\coordinate[leafsEE] (leaf2) at (0,0.3);
	\coordinate[node]  (node) at (-0.18,0.25);
	\coordinate[leafs] (leaf3) at (-0.36,0.5);
	\coordinate[leafsa] (leaf4) at (-0.18,0.5);
	
	\draw[edgeAbolds] (root)--(node);
	\draw[edgeA] (root)--(leaf1);
	\draw[edgeAbolds] (root)--(leaf2);
	\draw[edgeAbolds] (node)--(leaf3);
	\draw[edgeAbolds] (node)--(leaf4);
	\end{tikzpicture}
}
\newcommand{\treeExampleKPZaaEEb}{
	\begin{tikzpicture}[baseline=(basepoint)]
	\node(basepoint) at (0,0.1) {};
	
	\coordinate[root]  (root) at (0,0);
	\coordinate[leafsa] (leaf1) at (0.18,0.25);
	\coordinate[leafs] (leaf2) at (0,0.3);
	\coordinate[node]  (node) at (-0.18,0.25);
	\coordinate[leafsEE] (leaf3) at (-0.36,0.5);
	\coordinate[leafsa] (leaf4) at (-0.18,0.5);
	
	\draw[edgeAbolds] (root)--(node);
	\draw[edgeA] (root)--(leaf1);
	\draw[edgeAbolds] (root)--(leaf2);
	\draw[edgeAbolds] (node)--(leaf3);
	\draw[edgeAbolds] (node)--(leaf4);
	\end{tikzpicture}
}
\newcommand{\treeExampleKPZaaEEc}{
	\begin{tikzpicture}[baseline=(basepoint)]
	\node(basepoint) at (0,0.1) {};
	
	\coordinate[root]  (root) at (0,0);
	\coordinate[leafsa] (leaf1) at (0.18,0.25);
	\coordinate[leafs] (leaf2) at (0,0.3);
	\coordinate[node]  (node) at (-0.18,0.25);
	\coordinate[leafs] (leaf3) at (-0.36,0.5);
	\coordinate[leafsaEE] (leaf4) at (-0.18,0.5);
	
	\draw[edgeAbolds] (root)--(node);
	\draw[edgeA] (root)--(leaf1);
	\draw[edgeAbolds] (root)--(leaf2);
	\draw[edgeAbolds] (node)--(leaf3);
	\draw[edgeAbolds] (node)--(leaf4);
	\end{tikzpicture}
}
\newcommand{\treeExampleKPZaaEd}{
	\begin{tikzpicture}[baseline=(basepoint)]
	\node(basepoint) at (0,0.1) {};
	
	\coordinate[root]  (root) at (0,0);
	\coordinate[leafsaE] (leaf1) at (0.18,0.25);
	\coordinate[leafsE] (leaf2) at (0,0.3);
	\coordinate[node]  (node) at (-0.18,0.25);
	\coordinate[leafs] (leaf3) at (-0.36,0.5);
	\coordinate[leafsa] (leaf4) at (-0.18,0.5);
	
	\draw[edgeAbolds] (root)--(node);
	\draw[edgeA] (root)--(leaf1);
	\draw[edgeAbolds] (root)--(leaf2);
	\draw[edgeAbolds] (node)--(leaf3);
	\draw[edgeAbolds] (node)--(leaf4);
	\end{tikzpicture}
}
\newcommand{\treeExampleKPZaaEe}{
	\begin{tikzpicture}[baseline=(basepoint)]
	\node(basepoint) at (0,0.1) {};
	
	\coordinate[root]  (root) at (0,0);
	\coordinate[leafsaE] (leaf1) at (0.18,0.25);
	\coordinate[leafs] (leaf2) at (0,0.3);
	\coordinate[node]  (node) at (-0.18,0.25);
	\coordinate[leafsE] (leaf3) at (-0.36,0.5);
	\coordinate[leafsa] (leaf4) at (-0.18,0.5);
	
	\draw[edgeAbolds] (root)--(node);
	\draw[edgeA] (root)--(leaf1);
	\draw[edgeAbolds] (root)--(leaf2);
	\draw[edgeAbolds] (node)--(leaf3);
	\draw[edgeAbolds] (node)--(leaf4);
	\end{tikzpicture}
}
\newcommand{\treeExampleKPZaaEf}{
	\begin{tikzpicture}[baseline=(basepoint)]
	\node(basepoint) at (0,0.1) {};
	
	\coordinate[root]  (root) at (0,0);
	\coordinate[leafsaE] (leaf1) at (0.18,0.25);
	\coordinate[leafs] (leaf2) at (0,0.3);
	\coordinate[node]  (node) at (-0.18,0.25);
	\coordinate[leafs] (leaf3) at (-0.36,0.5);
	\coordinate[leafsaE] (leaf4) at (-0.18,0.5);
	
	\draw[edgeAbolds] (root)--(node);
	\draw[edgeA] (root)--(leaf1);
	\draw[edgeAbolds] (root)--(leaf2);
	\draw[edgeAbolds] (node)--(leaf3);
	\draw[edgeAbolds] (node)--(leaf4);
	\end{tikzpicture}
}

\newcommand{\treeExampleKPZaaEg}{
	\begin{tikzpicture}[baseline=(basepoint)]
	\node(basepoint) at (0,0.1) {};
	
	\coordinate[root]  (root) at (0,0);
	\coordinate[leafsa] (leaf1) at (0.18,0.25);
	\coordinate[leafsE] (leaf2) at (0,0.3);
	\coordinate[node]  (node) at (-0.18,0.25);
	\coordinate[leafsE] (leaf3) at (-0.36,0.5);
	\coordinate[leafsa] (leaf4) at (-0.18,0.5);
	
	\draw[edgeAbolds] (root)--(node);
	\draw[edgeA] (root)--(leaf1);
	\draw[edgeAbolds] (root)--(leaf2);
	\draw[edgeAbolds] (node)--(leaf3);
	\draw[edgeAbolds] (node)--(leaf4);
	\end{tikzpicture}
}
\newcommand{\treeExampleKPZaaEh}{
	\begin{tikzpicture}[baseline=(basepoint)]
	\node(basepoint) at (0,0.1) {};
	
	\coordinate[root]  (root) at (0,0);
	\coordinate[leafsa] (leaf1) at (0.18,0.25);
	\coordinate[leafsE] (leaf2) at (0,0.3);
	\coordinate[node]  (node) at (-0.18,0.25);
	\coordinate[leafs] (leaf3) at (-0.36,0.5);
	\coordinate[leafsaE] (leaf4) at (-0.18,0.5);
	
	\draw[edgeAbolds] (root)--(node);
	\draw[edgeA] (root)--(leaf1);
	\draw[edgeAbolds] (root)--(leaf2);
	\draw[edgeAbolds] (node)--(leaf3);
	\draw[edgeAbolds] (node)--(leaf4);
	\end{tikzpicture}
}
\newcommand{\treeExampleKPZaaEi}{
	\begin{tikzpicture}[baseline=(basepoint)]
	\node(basepoint) at (0,0.1) {};
	
	\coordinate[root]  (root) at (0,0);
	\coordinate[leafsa] (leaf1) at (0.18,0.25);
	\coordinate[leafs] (leaf2) at (0,0.3);
	\coordinate[node]  (node) at (-0.18,0.25);
	\coordinate[leafsE] (leaf3) at (-0.36,0.5);
	\coordinate[leafsaE] (leaf4) at (-0.18,0.5);
	
	\draw[edgeAbolds] (root)--(node);
	\draw[edgeA] (root)--(leaf1);
	\draw[edgeAbolds] (root)--(leaf2);
	\draw[edgeAbolds] (node)--(leaf3);
	\draw[edgeAbolds] (node)--(leaf4);
	\end{tikzpicture}
}
\newcommand{\treeExampleKPZaaEj}{
	\begin{tikzpicture}[baseline=(basepoint)]
	\node(basepoint) at (0,0.1) {};
	
	\coordinate[root]  (root) at (0,0);
	\coordinate[leafsaE] (leaf1) at (0.18,0.25);
	\coordinate[leafsE] (leaf2) at (0,0.3);
	\coordinate[node]  (node) at (-0.18,0.25);
	\coordinate[leafsE] (leaf3) at (-0.36,0.5);
	\coordinate[leafsa] (leaf4) at (-0.18,0.5);
	
	\draw[edgeAbolds] (root)--(node);
	\draw[edgeA] (root)--(leaf1);
	\draw[edgeAbolds] (root)--(leaf2);
	\draw[edgeAbolds] (node)--(leaf3);
	\draw[edgeAbolds] (node)--(leaf4);
	\end{tikzpicture}
}
\newcommand{\treeExampleKPZaaEk}{
	\begin{tikzpicture}[baseline=(basepoint)]
	\node(basepoint) at (0,0.1) {};
	
	\coordinate[root]  (root) at (0,0);
	\coordinate[leafsaE] (leaf1) at (0.18,0.25);
	\coordinate[leafsE] (leaf2) at (0,0.3);
	\coordinate[node]  (node) at (-0.18,0.25);
	\coordinate[leafs] (leaf3) at (-0.36,0.5);
	\coordinate[leafsaE] (leaf4) at (-0.18,0.5);
	
	\draw[edgeAbolds] (root)--(node);
	\draw[edgeA] (root)--(leaf1);
	\draw[edgeAbolds] (root)--(leaf2);
	\draw[edgeAbolds] (node)--(leaf3);
	\draw[edgeAbolds] (node)--(leaf4);
	\end{tikzpicture}
}
\newcommand{\treeExampleKPZaaEl}{
	\begin{tikzpicture}[baseline=(basepoint)]
	\node(basepoint) at (0,0.1) {};
	
	\coordinate[root]  (root) at (0,0);
	\coordinate[leafsaE] (leaf1) at (0.18,0.25);
	\coordinate[leafs] (leaf2) at (0,0.3);
	\coordinate[node]  (node) at (-0.18,0.25);
	\coordinate[leafsE] (leaf3) at (-0.36,0.5);
	\coordinate[leafsaE] (leaf4) at (-0.18,0.5);
	
	\draw[edgeAbolds] (root)--(node);
	\draw[edgeA] (root)--(leaf1);
	\draw[edgeAbolds] (root)--(leaf2);
	\draw[edgeAbolds] (node)--(leaf3);
	\draw[edgeAbolds] (node)--(leaf4);
	\end{tikzpicture}
}
\newcommand{\treeExampleKPZaaEm}{
	\begin{tikzpicture}[baseline=(basepoint)]
	\node(basepoint) at (0,0.1) {};
	
	\coordinate[root]  (root) at (0,0);
	\coordinate[leafsa] (leaf1) at (0.18,0.25);
	\coordinate[leafsE] (leaf2) at (0,0.3);
	\coordinate[node]  (node) at (-0.18,0.25);
	\coordinate[leafsE] (leaf3) at (-0.36,0.5);
	\coordinate[leafsaE] (leaf4) at (-0.18,0.5);
	
	\draw[edgeAbolds] (root)--(node);
	\draw[edgeA] (root)--(leaf1);
	\draw[edgeAbolds] (root)--(leaf2);
	\draw[edgeAbolds] (node)--(leaf3);
	\draw[edgeAbolds] (node)--(leaf4);
	\end{tikzpicture}
}
\newcommand{\treeExampleKPZaaEn}{
	\begin{tikzpicture}[baseline=(basepoint)]
	\node(basepoint) at (0,0.1) {};
	
	\coordinate[root]  (root) at (0,0);
	\coordinate[leafsaE] (leaf1) at (0.18,0.25);
	\coordinate[leafsE] (leaf2) at (0,0.3);
	\coordinate[node]  (node) at (-0.18,0.25);
	\coordinate[leafsE] (leaf3) at (-0.36,0.5);
	\coordinate[leafsaE] (leaf4) at (-0.18,0.5);
	
	\draw[edgeAbolds] (root)--(node);
	\draw[edgeA] (root)--(leaf1);
	\draw[edgeAbolds] (root)--(leaf2);
	\draw[edgeAbolds] (node)--(leaf3);
	\draw[edgeAbolds] (node)--(leaf4);
	\end{tikzpicture}
}

\newcommand{\triplekernel}{
	\begin{tikzpicture}[baseline=(basepoint)]
	\node(basepoint) at (0,0.0) {};
	
	\coordinate[nodeSmall]  (root) at (0,0);
	\coordinate[nodeSmall] (leaf1) at (0.14,0.25);
	\coordinate[nodeSmall] (leaf2) at (0,0.3);
	\coordinate[nodeSmall] (leaf3) at (-0.14,0.25);

	\draw[edgeB, color=red]	(root)--(leaf1);
	\draw[edgeB, color=blue] 	(root)--(leaf2);
	\draw[edgeB, color=darkgreen] 	(root)--(leaf3);
	\end{tikzpicture}
}

\newcommand{\treeExampleKPZaashiftz}{
	\begin{tikzpicture}[baseline=(basepoint)]
	\node(basepoint) at (0,0.1) {};
	
	\coordinate[root]  (root) at (0,0);
	\coordinate[leafsa] (leaf1) at (0.18,0.25);
	\coordinate[leafsred, color=treesDarkred] (leaf2) at (0,0.3);
	\coordinate[node]  (node) at (-0.18,0.25);
	\coordinate[leafs] (leaf3) at (-0.36,0.5);
	\coordinate[leafsa] (leaf4) at (-0.18,0.5);
 
	\draw[edgeAbolds] (root)--(node);
	\draw[edgeA] (root)--(leaf1);
	\draw[edgeAbolds] (root)--(leaf2);
	\draw[edgeAbolds] (node)--(leaf3);
	\draw[edgeAbolds] (node)--(leaf4);
	\end{tikzpicture}
}

\def\abc{1.6}
\newcommand{\treeExampleKPZaacontractionz}{
	\begin{tikzpicture}[baseline=(basepoint)]
	\node(basepoint) at (0,0.25) {};
	
	\coordinate[root] 	(root)  at (0,0);
	\coordinate[node] (leaf1) at (0.18 *\abc,0.25*\abc);
	\coordinate[node]	(leaf2) at (0,0.3*\abc);
	\coordinate[node]	(node)  at (-0.18*\abc,0.25*\abc);
	\coordinate[node]	(leaf3) at (-0.36*\abc,0.5*\abc);
	\coordinate[node] (leaf4) at (-0.18*\abc,0.5*\abc);

	\draw[edgeAbolds] (root)--(node);
	\draw[edgeA] (root)--(leaf1);
	\draw[edgeAbolds] (root)--(leaf2);
	\draw[edgeAbolds] (node)--(leaf3);
	\draw[edgeAbolds] (node)--(leaf4);

	\coordinate[] (tmp) at (-0.18*\abc,0.65*\abc);
	\coordinate[] (tmp2) at (0.1*\abc,0.65*\abc);

	\draw[edgeB, color=red]			(leaf2) to (leaf1);
	\draw[edgeB, color=darkgreen]	(leaf2) to [out=90,in=0]  (tmp) to  (leaf3);
	\draw[edgeB, color=blue]		(leaf2) to  (leaf4);
	\end{tikzpicture}
}
\newcommand{\treeExampleKPZaacontractionbz}{
	\begin{tikzpicture}[baseline=(basepoint)]
	\node(basepoint) at (0,0.25) {};
	
	\coordinate[root] 	(root)  at (0,0);
	\coordinate[node] (leaf1) at (0.18*\abc,0.25*\abc);
	\coordinate[node]	(leaf2) at (0,0.3*\abc);
	\coordinate[node]	(node)  at (-0.18*\abc,0.25*\abc);
	\coordinate[node]	(leaf3) at (-0.36*\abc,0.5*\abc);
	\coordinate[node] (leaf4) at (-0.18*\abc,0.5*\abc);

	\draw[edgeAbolds] (root)--(node);
	\draw[edgeA] (root)--(leaf1);
	\draw[edgeAbolds] (root)--(leaf2);
	\draw[edgeAbolds] (node)--(leaf3);
	\draw[edgeAbolds] (node)--(leaf4);

	\coordinate[] (tmp) at (-0.18*\abc,0.65*\abc);
	\coordinate[] (tmp2) at (0.1*\abc,0.65*\abc);

	\draw[edgeB, color=blue]			(leaf2) to (leaf1);
	\draw[edgeB, color=darkgreen]	(leaf2) to [out=90,in=0]  (tmp) to  (leaf3);
	\draw[edgeB, color=red]		(leaf2) to  (leaf4);
	\end{tikzpicture}
}

\newcommand{\treeExampleVariance}{
	\begin{tikzpicture}[baseline=(basepoint)]
	\node(basepoint) at (0,-0.1) {};
	
	\coordinate[nodeS] 	(node1)  at (-0.1,0);
	\coordinate[nodeS] 	(node2)  at (0.4,0);
	\coordinate[nodeS] 	(cnt1)  at (1,-0.4);
	\coordinate[nodeS] 	(cnt2)  at (1,0.2);
	\coordinate[nodeS] 	(cnt3)  at (1,0.4);
	\coordinate[nodeS] 	(node3)  at (1.6,0);
	\coordinate[nodeS] 	(node4)  at (2.1,0);

	\coordinate[] 	(cnt1a)  at (1,-0.35);
	\coordinate[] 	(cnt2a)  at (1,0.15);

	\draw[edgeA, color=darkgreen] (node1)--(node2);
	\draw[edgeA, color=darkgreen] (node2)--(cnt1);
	\draw[edgeA, color=darkgreen] (node2)--(cnt2);
	\draw[edgeA, color=darkgreen] (node2)--(cnt3);
	\draw[edgeA, color=darkred] (node3)--(cnt1);
	\draw[edgeA, color=darkred] (node3)--(cnt2);
	\draw[edgeA, color=darkred] (node3)--(cnt3);
	\draw[edgeA, color=darkred] (node3)--(node4);

	\draw[dotted] (cnt1a) -- (cnt2a);
	\end{tikzpicture}
}

\newcommand{\treeKPZVa}{
	\begin{tikzpicture}[baseline=(basepoint)]
	\node(basepoint) at (0,0.0) {};
	
	\coordinate[leafs] (root) at (0,0);
	\coordinate[leafsx] (root) at (0,0);
	\coordinate[leafs] (leaf1) at (-0.15,0.21);

	\draw[edgeA] (root)--(leaf1);
	\end{tikzpicture}
}
\newcommand{\treeKPZVb}{
	\begin{tikzpicture}[baseline=(basepoint)]
	\node(basepoint) at (0,0.0) {};
	
	\coordinate[leafs] (root) at (0,0);
	\coordinate[leafs] (leaf1) at (-0.15,0.21);
	\coordinate[leafsx] (leaf1) at (-0.15,0.21);

	\draw[edgeA] (root)--(leaf1);
	\end{tikzpicture}
}
\newcommand{\treeKPZVc}{
	\begin{tikzpicture}[baseline=(basepoint)]
	\node(basepoint) at (0,0.0) {};
	
	\coordinate[root] (root) at (0,0);
	\coordinate[leafs] (leaf1) at (-0.15,0.21);
	\coordinate[leafs] (leaf2) at (0.15,0.21);
	\coordinate[leafsx] (leaf2) at (0.15,0.21);

	\draw[edgeAbolds] (root)--(leaf1);
	\draw[edgeAbolds] (root)--(leaf2);

	\coordinate[leafs] (leaf2) at (0.15,0.21);
	\coordinate[leafsx] (leaf2) at (0.15,0.21);
	\end{tikzpicture}
}
\newcommand{\treeKPZVd}{
	\begin{tikzpicture}[baseline=(basepoint)]
	\node(basepoint) at (0,0.0) {};
	
	\coordinate[leafsb] (root) at (0,0);
	\coordinate[leafsx] (root) at (0,0);
	\coordinate[leafs] (leaf1) at (-0.15,0.21);
	\coordinate[leafs] (leaf2) at (0.15,0.21);

	\draw[edgeAbolds] (root)--(leaf1);
	\draw[edgeAbolds] (root)--(leaf2);

	\coordinate[leafsb] (root) at (0,0);
	\coordinate[leafsx] (root) at (0,0);
	\end{tikzpicture}
}
\newcommand{\treeKPZVe}{
	\begin{tikzpicture}[baseline=(basepoint)]
	\node(basepoint) at (0,0.0) {};
	
	\coordinate[root] (root) at (0,0);
	\coordinate[leafs] (leaf1) at (-0.15,0.21);
	\coordinate[leafs] (leaf2) at (0.15,0.21);

	\draw[edgeA] (root)--(leaf1);
	\draw[edgeAbolds] (root)--(leaf2);
	\end{tikzpicture}
}
\newcommand{\treeKPZVf}{
	\begin{tikzpicture}[baseline=(basepoint)]
	\node(basepoint) at (0,0.0) {};
	
	\coordinate[root] 	(root) 	at (0,0);
	\coordinate[leafs] 		(leaf1) at (-0.15,0.21);
	\coordinate[root] 	(node) 	at (0.15,0.21);
	\coordinate[leafs] 		(leaf2) at (0,0.42);

	\draw[edgeAbolds] (root)--(leaf1);
	\draw[edgeAbolds] (root)--(node);
	\draw[edgeAbolds] (node)--(leaf2);
	\end{tikzpicture}
}
\newcommand{\treeKPZVg}{
	\begin{tikzpicture}[baseline=(basepoint)]
	\node(basepoint) at (0,0.0) {};
	
	\coordinate[leafs] 	(root) 	at (0,0);
	\coordinate[root] 	(node) 	at (0.15,0.21);
	\coordinate[leafs] 		(leaf1) at (0,0.42);

	\draw[edgeA] (root)--(node);
	\draw[edgeAbolds] (node)--(leaf1);
	\end{tikzpicture}
}

\def\sch{1.5}
\def\schleg{1.6}
\newcommand{\treeExampleKPZh}{
	\begin{tikzpicture}[baseline=(basepoint)]
	\node(basepoint) at (0,0.1) {};
	
	\coordinate[root] 	(root)  at (0,0);
	\coordinate[leafs] 	(leaf1) at (0.18*\sch,0.25*\sch);
	\coordinate[leafs] 	(leaf2) at (0,0.3*\sch);
	\coordinate[node] 	(node)  at (-0.18*\sch,0.25*\sch);
	\coordinate[leafs] 	(leaf3) at (-0.36*\sch,0.5*\sch);
	\coordinate[leafs] 	(leaf4) at (-0.18*\sch,0.5*\sch);

	\coordinate[]		(leg11) at (0.26*\schleg, 0.43*\schleg);
	\coordinate[]		(leg12) at (0.3 *\schleg, 0.41*\schleg);
	\coordinate[]		(leg13) at (0.36*\schleg, 0.38*\schleg);

	\coordinate[]		(leg21) at (-0.06*\schleg, 0.48*\schleg);
	\coordinate[]		(leg22) at (0 	*\schleg, 0.5*\schleg);
	\coordinate[]		(leg23) at (0.06*\schleg, 0.48*\schleg);

	\coordinate[]		(leg31) at (-0.42*\schleg, 0.68*\schleg);
	\coordinate[]		(leg32) at (-0.36*\schleg, 0.70*\schleg);
	\coordinate[]		(leg33) at (-0.30*\schleg, 0.68*\schleg);

	\coordinate[]		(leg41) at (-0.24*\schleg, 0.68*\schleg);
	\coordinate[]		(leg42) at (-0.18*\schleg, 0.70*\schleg);
	\coordinate[]		(leg43) at (-0.12*\schleg, 0.68*\schleg);

	\draw[edgeAbolds] (root)--(node);
	\draw[edgeA] (root)--(leaf1);
	\draw[edgeAbolds] (root)--(leaf2);
	\draw[edgeAbolds] (node)--(leaf3);
	\draw[edgeAbolds] (node)--(leaf4);

	\draw[edgeB,color=red]		(leaf1)--(leg11);
	\draw[edgeB,color=blue]		(leaf1)--(leg12);
	\draw[edgeB,color=darkgreen]	(leaf1)--(leg13);

	\draw[edgeF,color=red]		(leaf2)--(leg21);
	\draw[edgeB,color=orange]	(leaf2)--(leg22);
	\draw[edgeB,color=cyan]		(leaf2)--(leg23);

	\draw[edgeF,color=blue]		(leaf3)--(leg31);
	\draw[edgeF,color=orange]	(leaf3)--(leg32);
	\draw[edgeB,color=btube]		(leaf3)--(leg33);

	\draw[edgeF,color=darkgreen](leaf4)--(leg41);
	\draw[edgeF,color=cyan]		(leaf4)--(leg42);
	\draw[edgeF,color=btube]		(leaf4)--(leg43);
	\end{tikzpicture}
}
\newcommand{\treehlega}{
	\begin{tikzpicture}[baseline=(basepoint)]
	\node(basepoint) at (0,0) {};
	
	\coordinate[] 	(root)  at (0,0);
	\coordinate[] 	(leaf1) at (0,0.2);

	\draw[edgeB, color=btube] (root)--(leaf1);
	\end{tikzpicture}
}
\newcommand{\treehlegb}{
	\begin{tikzpicture}[baseline=(basepoint)]
	\node(basepoint) at (0,0) {};
	
	\coordinate[] 	(root)  at (0,0);
	\coordinate[] 	(leaf1) at (0,0.2);

	\draw[edgeF, color=btube] (root)--(leaf1);
	\end{tikzpicture}
}
\newcommand{\treehlegc}{
	\begin{tikzpicture}[baseline=(basepoint)]
	\node(basepoint) at (0,0) {};
	
	\coordinate[] 	(root)  at (0,0);
	\coordinate[] 	(leaf1) at (0,0.2);

	\draw[edgeB, color=darkgreen] (root)--(leaf1);
	\end{tikzpicture}
}
\newcommand{\treehlegd}{
	\begin{tikzpicture}[baseline=(basepoint)]
	\node(basepoint) at (0,0) {};
	
	\coordinate[] 	(root)  at (0,0);
	\coordinate[] 	(leaf1) at (0,0.2);

	\draw[edgeF, color=darkgreen] (root)--(leaf1);
	\end{tikzpicture}
}
\newcommand{\treehlege}{
	\begin{tikzpicture}[baseline=(basepoint)]
	\node(basepoint) at (0,0) {};
	
	\coordinate[] 	(root)  at (0,0);
	\coordinate[] 	(leaf1) at (0,0.2);

	\draw[edgeF, color=cyan] (root)--(leaf1);
	\end{tikzpicture}
}
\newcommand{\treehlegf}{
	\begin{tikzpicture}[baseline=(basepoint)]
	\node(basepoint) at (0,0) {};
	
	\coordinate[] 	(root)  at (0,0);
	\coordinate[] 	(leaf1) at (0,0.2);

	\draw[edgeF, color=orange] (root)--(leaf1);
	\end{tikzpicture}
}
\newcommand{\treehlegg}{
	\begin{tikzpicture}[baseline=(basepoint)]
	\node(basepoint) at (0,0) {};
	
	\coordinate[] 	(root)  at (0,0);
	\coordinate[] 	(leaf1) at (0,0.2);

	\draw[edgeF, color=blue] (root)--(leaf1);
	\end{tikzpicture}
}

\newcommand{\treeExampleKPZhL}{
	\begin{tikzpicture}[baseline=(basepoint)]
	\node(basepoint) at (0,0.1) {};
	
	\coordinate[root] 	(root)  at (0,0);
	\coordinate[leafs] 	(leaf1) at (-0.1*\sch,0.25*\sch);
	\coordinate[leafs] 	(leaf2)  at (0.1*\sch,0.25*\sch);
	
	\coordinate[]		(leg1) at (-0.14*\schleg, 0.42*\schleg);
	\coordinate[]		(leg2) at (0.14*\schleg, 0.42*\schleg);
	
	\draw[edgeAbolds] (root)--(leaf1);
	\draw[edgeAbolds] (root)--(leaf2);

	\draw[edgeF,color=btube]		(leaf1)--(leg1);
	\draw[edgeB,color=btube]		(leaf2)--(leg2);
	\end{tikzpicture}
}
\newcommand{\treeExampleKPZhLd}{
	\begin{tikzpicture}[baseline=(basepoint)]
	\node(basepoint) at (0,0.1) {};
	
	\coordinate[root] 	(root)  at (0,0);
	\coordinate[leafsa] 	(leaf1) at (-0.1*\sch,0.25*\sch);
	\coordinate[leafs] 	(leaf2)  at (0.1*\sch,0.25*\sch);
	
	\coordinate[]		(leg1) at (-0.14*\schleg, 0.42*\schleg);
	\coordinate[]		(leg2) at (0.14*\schleg, 0.42*\schleg);
	
	\draw[edgeAbolds] (root)--(leaf1);
	\draw[edgeAbolds] (root)--(leaf2);

	\draw[edgeF,color=btube]		(leaf1)--(leg1);
	\draw[edgeB,color=btube]		(leaf2)--(leg2);
	\end{tikzpicture}
}
\newcommand{\treeExampleKPZhR}{
	\begin{tikzpicture}[baseline=(basepoint)]
	\node(basepoint) at (0,0.1) {};
	
	\coordinate[root] 	(root)  at (0,0);
	\coordinate[leafs] 	(leaf1) at (0.18*\sch,0.25*\sch);
	\coordinate[leafs] 	(leaf2) at (0,0.3*\sch);
	\coordinate[node] 	(node)  at (-0.18*\sch,0.25*\sch);

	\coordinate[]		(leg11) at (0.27*\schleg, 0.43*\schleg);
	\coordinate[]		(leg12) at (0.3 *\schleg, 0.41*\schleg);
	\coordinate[]		(leg13) at (0.35*\schleg, 0.39*\schleg);

	\coordinate[]		(leg21) at (-0.05*\schleg, 0.48*\schleg);
	\coordinate[]		(leg22) at (0 	*\schleg, 0.5*\schleg);
	\coordinate[]		(leg23) at (0.05*\schleg, 0.48*\schleg);

	\coordinate[]		(leg31) at (-0.24*\schleg, 0.46*\schleg);
	\coordinate[]		(leg32) at (-0.28*\schleg, 0.43*\schleg);
	\coordinate[]		(leg33) at (-0.32*\schleg, 0.40*\schleg);
	\coordinate[]		(leg34) at (-0.36*\schleg, 0.37*\schleg);

	\draw[edgeAbolds] (root)--(node);
	\draw[edgeA] (root)--(leaf1);
	\draw[edgeAbolds] (root)--(leaf2);

	\draw[edgeB,color=red]		(leaf1)--(leg11);
	\draw[edgeB,color=blue]		(leaf1)--(leg12);
	\draw[edgeB,color=darkgreen]	(leaf1)--(leg13);

	\draw[edgeF,color=red]		(leaf2)--(leg21);
	\draw[edgeB,color=orange]	(leaf2)--(leg22);
	\draw[edgeB,color=cyan]		(leaf2)--(leg23);

	\draw[edgeF,color=blue]		(node)--(leg31);
	\draw[edgeF,color=orange]	(node)--(leg32);
	\draw[edgeF,color=cyan]	(node)--(leg33);
	\draw[edgeF,color=darkgreen](node)--(leg34);

	\end{tikzpicture}
}

\newcommand{\cherry}{
	\begin{tikzpicture}[baseline=(basepoint)]
	\node(basepoint) at (0,0.0) {};
	
	\coordinate[root] (root) at (0,0);
	\coordinate[leafs] (node1) at (-0.12,0.20);
	\coordinate[leafs] (node2) at (0.12,0.20);

	\draw[edgeA] (root)--(node1) node [midway, below left=-0.15] {};
	\draw[edgeA] (root)--(node2) node [midway, below right=-0.15] {};
	\end{tikzpicture}
}
\newcommand{\phiCA}{
	\begin{tikzpicture}[baseline=(basepoint)]
	\node(basepoint) at (0,0.1) {};
	
	\coordinate[root] (root) at (0,0);
	\coordinate[root] (node) at (0,0.22);
	\coordinate[leafs] (leaf1) at (-0.13,0.18);
	\coordinate[leafs] (leaf2) at (-0.13,0.40);
	\coordinate[leafs] (leaf3) at (0,0.44);
	\coordinate[leafs] (leaf4) at (0.13,0.40);

	\draw[edgeA] (root)--(node);
	\draw[edgeA] (root)--(leaf1);
	\draw[edgeA] (node)--(leaf2);
	\draw[edgeA] (node)--(leaf3);
	\draw[edgeA] (node)--(leaf4);
	\end{tikzpicture}
}

\def\phih{0.24}
\newcommand{\quadcherry}{
	\begin{tikzpicture}[baseline=(basepoint)]
	\node(basepoint) at (0,0.0) {};
	
	\coordinate[root] (root) at (0,0);
	\coordinate[leafs] (leaf1) at (-0.18,0.20);
	\coordinate[leafs] (leaf2) at (-0.06,\phih);
	\coordinate[leafs] (leaf3) at (0.06 ,\phih);
	\coordinate[leafs] (leaf4) at (0.18 ,0.20);

	\draw[edgeA] (root)--(leaf1);
	\draw[edgeA] (root)--(leaf2);
	\draw[edgeA] (root)--(leaf3);
	\draw[edgeA] (root)--(leaf4);
	\end{tikzpicture}
}
\newcommand{\quadcherryA}{
	\begin{tikzpicture}[baseline=(basepoint)]
	\node(basepoint) at (0,0.0) {};
	
	\coordinate[root] (root) at (0,0);
	\coordinate[leafsred] (leaf1) at (-0.18,0.20);
	\coordinate[leafsred] (leaf2) at (-0.06,\phih);
	\coordinate[leafsred] (leaf3) at (0.06 ,\phih);
	\coordinate[leafs] (leaf4) at (0.18 ,0.20);

	\draw[edgeA] (root)--(leaf1);
	\draw[edgeA] (root)--(leaf2);
	\draw[edgeA] (root)--(leaf3);
	\draw[edgeA] (root)--(leaf4);
	\end{tikzpicture}
}
\newcommand{\quadcherryB}{
	\begin{tikzpicture}[baseline=(basepoint)]
	\node(basepoint) at (0,0.0) {};
	
	\coordinate[root] (root) at (0,0);
	\coordinate[leafsred] (leaf1) at (-0.18,0.20);
	\coordinate[leafsred] (leaf2) at (-0.06,\phih);
	\coordinate[leafsred] (leaf3) at (0.06 ,\phih);
	\coordinate[leafsred] (leaf4) at (0.18 ,0.20);

	\draw[edgeA] (root)--(leaf1);
	\draw[edgeA] (root)--(leaf2);
	\draw[edgeA] (root)--(leaf3);
	\draw[edgeA] (root)--(leaf4);
	\end{tikzpicture}
}
\newcommand{\quadcherryC}{
	\begin{tikzpicture}[baseline=(basepoint)]
	\node(basepoint) at (0,0.0) {};
	
	\coordinate[root] (root) at (0,0);
	\coordinate[leafsred] (leaf1) at (-0.18,0.20);
	\coordinate[leafs] (leaf2) at (-0.06,\phih);
	\coordinate[leafs] (leaf3) at (0.06 ,\phih);
	\coordinate[leafs] (leaf4) at (0.18 ,0.20);

	\draw[edgeA] (root)--(leaf1);
	\draw[edgeA] (root)--(leaf2);
	\draw[edgeA] (root)--(leaf3);
	\draw[edgeA] (root)--(leaf4);
	\end{tikzpicture}
}
\newcommand{\quadcherryD}{
	\begin{tikzpicture}[baseline=(basepoint)]
	\node(basepoint) at (0,0.0) {};
	
	\coordinate[root] (root) at (0,0);
	\coordinate[leafs] (leaf1) at (-0.10,0.22);
	\coordinate[leafs] (leaf2) at (0.10,0.22);

	\draw[edgeA] (root)--(leaf1);
	\draw[edgeA] (root)--(leaf2);
	\end{tikzpicture}
}
\newcommand{\quadcherryE}{
	\begin{tikzpicture}[baseline=(basepoint)]
	\node(basepoint) at (0,0.0) {};
	
	\coordinate[root] (root) at (0,0);
	\coordinate[leafsred] (leaf1) at (-0.10,0.22);
	\coordinate[leafs] (leaf2) at (0.10,0.22);

	\draw[edgeA] (root)--(leaf1);
	\draw[edgeA] (root)--(leaf2);
	\end{tikzpicture}
}
\def\snl{1.2}
\newcommand{\treeExampleKPZNL}{
	\begin{tikzpicture}[baseline=(basepoint)]
	\node(basepoint) at (0,0.1) {};
	
	\coordinate[root] (root) at (0,0);
	\coordinate[node]  (node1) at (-0.18*\snl,0.25*\snl);
	\coordinate[leafsyellow] (node2) at (0.18*\snl,0.25*\snl);
	\coordinate[leafsgreen] (leaf1) at (-0.27*\snl,0.5*\snl);
	\coordinate[leafsgreen] (leaf2) at (-0.09*\snl,0.5*\snl);
	\coordinate[leafs] (leaf3) at ( 0.09*\snl,0.5*\snl);
	\coordinate[leafs] (leaf4) at ( 0.27*\snl,0.5*\snl);

	\coordinate[] (noise1) at (-0.27*\snl,0.75*\snl);
	\coordinate[] (noise2) at (-0.09*\snl,0.75*\snl);
	\coordinate[] (noise3) at ( 0.09*\snl,0.75*\snl);
	\coordinate[] (noise4) at ( 0.27*\snl,0.75*\snl);

	\coordinate[] (leg1) at (-0.37*\snl,0.73*\snl);
	\coordinate[] (leg2) at (-0.19*\snl,0.73*\snl);
	\coordinate[] (leg3) at ( 0.40*\snl,0.45*\snl);

	\draw[edgeAbolds] (root)--(node1);
	\draw[edgeAbolds] (root)--(node2);
	\draw[edgeAbolds] (node1)--(leaf1);
	\draw[edgeAbolds] (node1)--(leaf2);
	\draw[edgeAbolds] (node2)--(leaf3);
	\draw[edgeAbolds] (node2)--(leaf4);

	\draw[edgeE] (leaf1)--(noise1);
	\draw[edgeE] (leaf2)--(noise2);
	\draw[edgeE] (leaf3)--(noise3);
	\draw[edgeE] (leaf4)--(noise4);

	\draw[edgeB] (leaf1)--(leg1);
	\draw[edgeB] (leaf2)--(leg2);
	\draw[edgeB] (node2)--(leg3);
	\end{tikzpicture}
}

\newcommand{\treeExampleKPZm}{
	\begin{tikzpicture}[baseline=(basepoint)]
	\node(basepoint) at (0,0.1) {};
	
	\coordinate[leafsa] (leaf1) at (0,0);
	\coordinate[leafs] (leaf2) at (-0.21,0.21);
	\coordinate[leafs] (leaf3) at (0 	,0.42);
	\coordinate[leafsa] (leaf4) at (-0.21,0.63);
	
	\draw[edgeA] (leaf1)--(leaf2);
	\draw[edgeA] (leaf2)--(leaf3);
	\draw[edgeA] (leaf3)--(leaf4);
	\end{tikzpicture}
}
\newcommand{\treeExampleKPZn}{
	\begin{tikzpicture}[baseline=(basepoint)]
	\node(basepoint) at (0,0.1) {};
	
	\coordinate[leafs] (leaf1) at (0,0);
	\coordinate[leafs] (leaf2) at (-0.21,0.21);
	\coordinate[leafsa] (leaf3) at (0 	,0.42);
	\coordinate[leafsa] (leaf4) at (0.21,0.21);
	
	\draw[edgeA] (leaf1)--(leaf2);
	\draw[edgeA] (leaf2)--(leaf3);
	\draw[edgeA] (leaf1)--(leaf4);
	\end{tikzpicture}
}
\newcommand{\treeExampleKPZmcpmL}{
	\begin{tikzpicture}[baseline=(basepoint)]
	\node(basepoint) at (0,0.05) {};
	
	\coordinate[leafs] (leaf1) at (0,0);
	\coordinate[leafs] (leaf2) at (0,0.3);
	\coordinate[leafsx] (leaf2) at (0,0.3);
	
	\draw[edgeA] (leaf1)--(leaf2);
	\end{tikzpicture}
}
\newcommand{\treeExampleKPZmcpmR}{
	\begin{tikzpicture}[baseline=(basepoint)]
	\node(basepoint) at (0,0.1) {};
	
	\coordinate[leafsa] (leaf1) at (0,0);
	\coordinate[node] (leaf2) at (-0.21,0.21);
	\coordinate[leafsa] (leaf3) at (0 	,0.42);
	
	\draw[edgeA] (leaf1)--(leaf2);
	\draw[edgeAbolds] (leaf2)--(leaf3);
	\end{tikzpicture}
}
\newcommand{\treeExampleKPZncpmR}{
	\begin{tikzpicture}[baseline=(basepoint)]
	\node(basepoint) at (0,0.1) {};
	
	\coordinate[node] (leaf1) at (0,0);
	\coordinate[leafsa] (leaf2) at (-0.21,0.21);
	\coordinate[leafsa] (leaf3) at (0.21,0.21);
	
	\draw[edgeAbolds] (leaf1)--(leaf2);
	\draw[edgeA] (leaf1)--(leaf3);
	\end{tikzpicture}
}

\newcommand{\treeExampleKPZmcpmiL}{
	\begin{tikzpicture}[baseline=(basepoint)]
	\node(basepoint) at (0,0.05) {};
	
	\coordinate[leafs] (leaf1) at (0,0);
	\coordinate[leafs] (leaf2) at (0,0.3);
	\coordinate[leafs] (leaf2) at (0,0.3);
	
	\draw[edgeA] (leaf1)--(leaf2);
	\end{tikzpicture}
}
\newcommand{\treeExampleKPZmcpmiR}{
	\begin{tikzpicture}[baseline=(basepoint)]
	\node(basepoint) at (0,0.05) {};
	
	\coordinate[leafsa] (leaf1) at (0,0);
	\coordinate[node] (leaf2) at (-0.21,0.15);
	\coordinate[leafsa] (leaf3) at (0 	,0.3);
	
	\draw[edgeA] (leaf1)--(leaf2);
	\draw[edgeA] (leaf2)--(leaf3);
	\end{tikzpicture}
}
\newcommand{\treeExampleKPZncpmiR}{
	\begin{tikzpicture}[baseline=(basepoint)]
	\node(basepoint) at (0,0) {};
	
	\coordinate[node] (leaf1) at (0,0);
	\coordinate[leafsa] (leaf2) at (-0.21,0.21);
	\coordinate[leafsa] (leaf3) at (0.21,0.21);
	
	\draw[edgeA] (leaf1)--(leaf2);
	\draw[edgeA] (leaf1)--(leaf3);
	\end{tikzpicture}
}

\newcommand{\introCherryA}{
	\begin{tikzpicture}[baseline=(basepoint)]
	\node(basepoint) at (0,0.2) {};
	\node(delta) at (-0.23,0.68) {\tiny$\delta$};

	\coordinate[root] (root) at (0,0);
	\coordinate[root] (node1) at (-0.24,0.40);
	\coordinate[root] (node2) at (0.24,0.40);

	\coordinate[root] (node) at (0,0.64);

	\draw[edgeA] (root)--(node1);
	\draw[edgeA] (root)--(node2);

	\draw[edgeB] (node1) to [out=90,in=180] (node);
	\draw[edgeG] (node2) to [out=90,in=0] (node);
	
	\end{tikzpicture}
}
\newcommand{\introCherryB}{
	\begin{tikzpicture}[baseline=(basepoint)]
	\node(basepoint) at (0,0.2) {};
	\node(delta) at (-0.23,0.68) {\tiny$\delta$};
	\node(delta) at (0.26,0.68) {\tiny$\delta$};

	\coordinate[root] (root) at (0,0);
	\coordinate[root] (node1) at (-0.24,0.40);
	\coordinate[root] (node2) at (0.24,0.40);

	\coordinate[root] (node) at (0,0.64);

	\draw[edgeA] (root)--(node1);
	\draw[edgeA] (root)--(node2);

	\draw[edgeB] (node1) to [out=90,in=180] (node);
	\draw[edgeB] (node2) to [out=90,in=0] (node);
	
	\end{tikzpicture}
}
\newcommand{\introCherryC}{
	\begin{tikzpicture}[baseline=(basepoint)]
	\node(basepoint) at (0,0.2) {};
	\node(delta) at (-0.23,0.68) {\tiny$\lambda$};
	\node(delta) at (0.26,0.68) {};

	\coordinate[root] (root) at (0,0);
	\coordinate[root] (node1) at (-0.24,0.40);
	\coordinate[root] (node2) at (0.24,0.40);

	\coordinate[root] (node) at (0,0.64);

	\draw[edgeA] (root)--(node1);
	\draw[edgeA] (root)--(node2);

	\draw[edgeB] (node1) to [out=90,in=180] (node);
	\draw[edgeG] (node2) to [out=90,in=0] (node);
	
	\end{tikzpicture}
}

\newcommand{\introCherryD}{
	\begin{tikzpicture}[baseline=(basepoint)]
	\node(basepoint) at (0,0.2) {};
	\node(delta) at (-0.23,0.68) {\tiny$\lambda$};
	\node(delta) at (0.26,0.68) {\tiny$\delta$};

	\coordinate[root] (root) at (0,0);
	\coordinate[root] (node1) at (-0.24,0.40);
	\coordinate[root] (node2) at (0.24,0.40);

	\coordinate[root] (node) at (0,0.64);

	\draw[edgeA] (root)--(node1);
	\draw[edgeA] (root)--(node2);

	\draw[edgeB] (node1) to [out=90,in=180] (node);
	\draw[edgeB] (node2) to [out=90,in=0] (node);
	
	\end{tikzpicture}
}
\newcommand{\introCherryE}{
	\begin{tikzpicture}[baseline=(basepoint)]
	\node(basepoint) at (0,0.2) {};
	\node(delta) at (-0.23,0.68) {\tiny$\lambda$};
	\node(delta) at (0.26,0.68) {\tiny$\lambda$};

	\coordinate[root] (root) at (0,0);
	\coordinate[root] (node1) at (-0.24,0.40);
	\coordinate[root] (node2) at (0.24,0.40);

	\coordinate[root] (node) at (0,0.64);

	\draw[edgeA] (root)--(node1);
	\draw[edgeA] (root)--(node2);

	\draw[edgeB] (node1) to [out=90,in=180] (node);
	\draw[edgeB] (node2) to [out=90,in=0] (node);
	
	\end{tikzpicture}
}
\newcommand{\treeKPZa}{
	\begin{tikzpicture}[baseline=(basepoint)]
	\node(basepoint) at (0.21,0.1) {};
	
	\coordinate[leafs] (leaf1) at (0,0);
	\end{tikzpicture}
}
\newcommand{\treeKPZb}{
	\begin{tikzpicture}[baseline=(basepoint)]
	\node(basepoint) at (0,0.1) {};
	
	\coordinate[root] (node1) at (0,0);
	\coordinate[leafs] (leaf1) at (-0.21,0.21);
	\coordinate[leafs] (leaf2) at (+0.21,0.21);
	
	\draw[edgeAbolds] (node1)--(leaf1);
	\draw[edgeAbolds] (node1)--(leaf2);
	\end{tikzpicture}
}
\newcommand{\treeKPZc}{
	\begin{tikzpicture}[baseline=(basepoint)]
	\node(basepoint) at (0,0.1) {};
	
	\coordinate[] (node1) at (0,0);
	\coordinate[] (node2) at (-0.21,0.21);
	\coordinate[leafs] (leaf1) at (-0.42,0.42);
	\coordinate[leafs] (leaf2) at (0., 0.42);
	\coordinate[leafs] (leaf3) at (0.21,0.21);
	
	\draw[edgeAbolds] (node1)--(node2);
	\draw[edgeAbolds] (node1)--(leaf3);
	\draw[edgeAbolds] (node2)--(leaf1);
	\draw[edgeAbolds] (node2)--(leaf2);
	\end{tikzpicture}
}

\newcommand{\treeKPZd}{
	\begin{tikzpicture}[baseline=(basepoint)]
\node(basepoint) at (0,0.1) {};

\coordinate[] (node1) at (0,0);
\coordinate[] (node2) at (-0.21,0.21);
\coordinate[leafs] (leaf2) at (0., 0.42);
\coordinate[leafs] (leaf3) at (0.21,0.21);

\draw[edgeAbolds] (node1)--(node2);
\draw[edgeAbolds] (node1)--(leaf3);
\draw[edgeAbolds] (node2)--(leaf2);
\end{tikzpicture}
}
\newcommand{\treeKPZe}{
	\begin{tikzpicture}[baseline=(basepoint)]
	\node(basepoint) at (0,0.1) {};
	
	\coordinate[] (node1) at (0,0);
	\coordinate[] (node2) at (-0.21,0.21);
	\coordinate[] (node3) at (-0.42,0.42);
	\coordinate[leafs] (leaf1) at (-0.63,0.63);
	\coordinate[leafs] (leaf2) at (-0.21,0.63);
	\coordinate[leafs] (leaf3) at (0., 0.42);
	\coordinate[leafs] (leaf4) at (0.21,0.21);
	
	\draw[edgeAbolds] (node1)--(node2);
	\draw[edgeAbolds] (node1)--(leaf4);
	\draw[edgeAbolds] (node2)--(node3);
	\draw[edgeAbolds] (node2)--(leaf3);
	\draw[edgeAbolds] (node3)--(leaf2);
	\draw[edgeAbolds] (node3)--(leaf1);
	\end{tikzpicture}
}
\newcommand{\treeKPZf}{
	\begin{tikzpicture}[baseline=(basepoint)]
	\node(basepoint) at (0,0.1) {};
	
	\coordinate[] (node1) at (0,0);
	\coordinate[] (node2) at (-0.21,0.21);
	\coordinate[] (node3) at (+0.21,0.21);
	\coordinate[leafs] (leaf1) at (-0.31,0.42);
	\coordinate[leafs] (leaf2) at (-0.11,0.42);
	\coordinate[leafs] (leaf3) at (0.11, 0.42);
	\coordinate[leafs] (leaf4) at (0.31,0.42);
	
	\draw[edgeAbolds] (node1)--(node2);
	\draw[edgeAbolds] (node1)--(node3);
	\draw[edgeAbolds] (node2)--(leaf1);
	\draw[edgeAbolds] (node2)--(leaf2);
	\draw[edgeAbolds] (node3)--(leaf3);
	\draw[edgeAbolds] (node3)--(leaf4);
	\end{tikzpicture}
}
\newcommand{\treePairingExample}{
	\begin{tikzpicture}[baseline=(basepoint)]
	\node(basepoint) at (0,0.0) {};
	
	\coordinate[node] (root) at (0,0);
	\coordinate[node] (node0) at (0,0.4);
	\coordinate[leafs] (node1) at (-0.2,0.25);
	\coordinate[leafsa] (node2) at (0.2,0.25);
	\coordinate[leafsc] (node3) at (-0.2,0.65);
	\coordinate[leafsd] (node4) at (0.2,0.65);

	\coordinate[] (nodeL1) at (-0.4, 0.3);
	\coordinate[] (nodeL2) at (0.4, 0.3);
	\coordinate[] (nodeL3) at (-0.4, 0.7);
	\coordinate[] (nodeL4) at (0.4, 0.7);

	\draw[edgeA] (root)--(node0);
	\draw[edgeA] (root)--(node1);
	\draw[edgeA] (root)--(node2);
	\draw[edgeA] (node0)--(node3);
	\draw[edgeA] (node0)--(node4);

	\draw[leg] (nodeL1)--(node1);
	\draw[leg] (nodeL2)--(node2);
	\draw[leg] (nodeL3)--(node3);
	\draw[leg] (nodeL4)--(node4);
\end{tikzpicture}
}
\newcommand{\treePairingExampleA}{
	\begin{tikzpicture}[baseline=(basepoint)]
	\node(basepoint) at (0,0.0) {};
	
	\coordinate[node] (root) at (0,0);
	\coordinate[node] (node0) at (0,0.4);
	\coordinate[leafs] (node1) at (-0.2,0.25);
	\coordinate[leafsa] (node2) at (0.2,0.25);
	\coordinate[leafsc] (node3) at (-0.2,0.65);
	\coordinate[leafsd] (node4) at (0.2,0.65);

	\coordinate[] (nodeL1) at (-0.4, 0.3);
	\coordinate[] (nodeL2) at (0.4, 0.3);
	\coordinate[] (nodeL3) at (-0.4, 0.7);
	\coordinate[] (nodeL4) at (0.4, 0.7);

	\coordinate[node,color=lightred] (nodeA) at (-0.40, 0.9);
	\coordinate[node,color=lightgreen] (nodeB) at (0.40, 0.9);
	\coordinate[node] (nodeXa) at (-0.5, 0.55);
	\coordinate[node] (nodeXb) at (0.5, 0.55);
	\coordinate[node] (nodeXc) at (0, 0.8);
	\coordinate[node] (nodeXd) at (0, 1.0);

	\draw[edgeA] (root)--(node0);
	\draw[edgeA] (root)--(node1);
	\draw[edgeA] (root)--(node2);
	\draw[edgeA] (node0)--(node3);
	\draw[edgeA] (node0)--(node4);
	\draw[edgeA] (node3)--(nodeA);
	\draw[edgeA] (node4)--(nodeB);

	\draw[leg] (nodeL1)--(node1);
	\draw[leg] (nodeL2)--(node2);
	\draw[leg] (nodeL3)--(node3);
	\draw[leg] (nodeL4)--(node4);

	\draw[edgeA] (node1)--(nodeXa);
	\draw[edgeA] (node2)--(nodeXb);
	\draw[edgeA] (nodeA)--(nodeXa);
	\draw[edgeA] (nodeB)--(nodeXb);
	\draw[edgeA] (nodeA)--(nodeXc);
	\draw[edgeA] (nodeB)--(nodeXc);
	\draw[edgeA] (nodeA)--(nodeXd);
	\draw[edgeA] (nodeB)--(nodeXd);
\end{tikzpicture}
}

\newcolumntype{L}{>{$}c<{$}} 

\def\emptyset{{\centernot\ocircle}}
\def\restr{\mathord{\upharpoonright}}

\numberwithin{equation}{section}

\setlist{nosep}

\def\ex{\mathrm{ex}}
\def\PPi{\mathbf\Pi}
\def\one{\mathbf 1}

\let\comm\comment

\def\Zcan{Z_{\mathrm{c}}}
\def\Nabla_#1{\nabla_{\!#1}}
\def\Moll{\mathord{\mathrm{Moll}}}

\def\geo{{\text{\rm\tiny geo}}}
\def\BPHZ{{\text{\rm\tiny BPHZ}}}

\def\Vec{\mathop{\mathrm{Vec}}}
\def\Nabla_#1{\nabla_{\!#1}}
\def\s{\mathbf{s}}
\def\p{\mathbf{p}}
\def\q{\mathbf{q}}
\def\n{\mathbf{n}}
\def\linspace#1{\Vec #1}
\def\cl{{\text{\rm\tiny cl}}}

\long\def\maybeNotNeeded#1{}

\DeclareSymbol{generic}{0}{
\draw (0,0.6) node[xi] {};
}

\DeclareSymbol{Xi2}{-2}{\draw (-1,-0.25) node[leafs] {} -- (0,1) node[leafs] {};} 
\DeclareSymbol{Xi2mixed}{-2}{\draw (-1,-0.25) node[leafsc] {} -- (0,1) node[leafs] {};} 
\DeclareSymbol{Xi2mixedt}{-2}{\draw (-1,-0.25) node[leafsd] {} -- (0,1) node[leafs] {};} 
\DeclareSymbol{Xi2mixed2}{-2}{\draw (-1,-0.25) node[leafs] {} -- (0,1) node[leafsc] {};} 
\DeclareSymbol{Xi2mixed2t}{-2}{\draw (-1,-0.25) node[leafs] {} -- (0,1) node[leafsd] {};} 

\DeclareSymbol{I1Xitwo}
{0}{\draw[edgeAbolds] (0,0) -- (1,1.25) node[leafs] {};
\draw[edgeAbolds] (0,0) node[root] {} -- (-1,1.25) node[leafs] {};}

\DeclareSymbol{I1Xitwomixed}
{0}{\draw[edgeAbolds] (0,0) -- (1,1.25) node[leafs] {};
\draw[edgeAbolds] (0,0) node[root] {} -- (-1,1.25) node[leafsc] {};}

\DeclareSymbol{I1Xitwomixedt}
{0}{\draw[edgeAbolds] (0,0) -- (1,1.25) node[leafs] {};
\draw[edgeAbolds] (0,0) node[root] {} -- (-1,1.25) node[leafsd] {};}

\DeclareSymbol{blue}{-2.8}{\node[leafs] {};}
\DeclareSymbol{dblue}{-2.8}{\node[leafs,fill=darkblue] {};}

\DeclareSymbol{smXi}{-2.8}{\node[leafs,fill=lightred] {};}
\DeclareSymbol{smXit}{-2.8}{\node[leafs] {};}

\DeclareSymbol{Xi}{-2.8}{\node[leafc] {};}
\DeclareSymbol{Xit}{-2.8}{\node[leafb] {};}
\DeclareSymbol{Xih}{-2.8}{\node[leafd] {};}

\DeclareSymbol{2I1Xi4c1}{2}{
\draw[edgeAbolds] (-1,1) -- (-1.5,2.5) node[leafs,fill=darkblue] {};
\draw[edgeAbolds] (-1,1) -- (-0.5,2.5) node[leafs] {};
\draw[edgeAbolds] (1,1) -- (0.5,2.5) node[leafs,fill=darkblue] {};
\draw[edgeAbolds] (1,1) -- (1.5,2.5) node[leafs] {};
\draw[edgeAbolds] (0,0) -- (-1,1) node[root] {};
\draw[edgeAbolds] (0,0)  node[root] {}  -- (1,1) node[root] {};
}

\DeclareSymbol{2I1Xi4c2}{2}{
\draw[edgeAbolds] (-1,1) -- (-1.5,2.5) node[leafs,fill=darkblue] {};
\draw[edgeAbolds] (-1,1) -- (-0.5,2.5) node[leafs,fill=darkblue] {};
\draw[edgeAbolds] (1,1) -- (0.5,2.5) node[leafs] {};
\draw[edgeAbolds] (1,1) -- (1.5,2.5) node[leafs] {};
\draw[edgeAbolds] (0,0) -- (-1,1) node[root] {};
\draw[edgeAbolds] (0,0)  node[root] {}  -- (1,1) node[root] {};
}

\DeclareSymbol{0}{-2}{
\coordinate[leafs] (0,1.5) {};
}
\DeclareSymbol{2}{2}{
\draw[edgeAbolds] (0,0) node[root]{} -- (0.5,1.5) node[leafs] {};
\draw[edgeAbolds] (0,0) node[root]{} -- (-0.5,1.5) node[leafs] {};
}
\DeclareSymbol{11}{2}{
\draw[edgeAbolds] (-0.5,1.5) -- (-0.0,2.8) node[leafs] {};
\draw[edgeAbolds] (0,0) node[root]{} -- (0.6,1.4) node[leafs] {};
\draw[edgeAbolds] (0,0) node[root]{} -- (-0.6,1.4)  node[root] {};
}
\DeclareSymbol{2x}{2}{
\draw[edgeAbolds] (0,0) node[root]{} -- (-0.5,1.5) node[leafs] {};
\draw[edgeAbolds] (0,0) node[root]{}-- (0.5,1.5) node[leafs] {};
\coordinate[leafsx] (root) at (0.5,1.5);
}
\DeclareSymbol{21}{2}{
\draw[edgeAbolds] (-0.5,1.5) -- (-1,3) node[leafs] {};
\draw[edgeAbolds] (-0.5,1.5) -- (0,3) node[leafs] {};
\draw[edgeAbolds] (0,0) node[root]{}-- (0.5,1.5) node[leafs] {};
\draw[edgeAbolds] (0,0) node[root]{} -- (-0.5,1.5)  node[root] {};
}
\DeclareSymbol{40}{2}{
\draw[edgeAbolds] (-1,1) -- (-1.5,2.5) node[leafs] {};
\draw[edgeAbolds] (-1,1) -- (-0.5,2.5) node[leafs] {};
\draw[edgeAbolds] (1,1) -- (0.5,2.5) node[leafs] {};
\draw[edgeAbolds] (1,1) -- (1.5,2.5) node[leafs] {};
\draw[edgeAbolds] (0,0) -- (-1,1) node[root] {};
\draw[edgeAbolds] (0,0)  node[root] {}  -- (1,1) node[root] {};
}
\DeclareSymbol{211}{2}{
\draw[edgeAbolds] (-1,3) -- (-1.5,4.5) node[leafs] {};
\draw[edgeAbolds] (-1,3) -- (-0.5,4.5) node[leafs] {};
\draw[edgeAbolds] (-0.5,1.5) -- (-1,3) node[root]{};
\draw[edgeAbolds] (-0.5,1.5) -- (0,3) node[leafs] {};
\draw[edgeAbolds] (0,0) node[root]{} -- (-0.5,1.5) node[root] {};
\draw[edgeAbolds] (0,0) node[root]{}  -- (0.5,1.5) node[leafs] {};
}

\def\SS{\mathscr{S}}

\def\f#1#2{\textstyle{#1\over #2}}
\let\d\partial
\def\spacetime{{\DD}}

\newtheorem{assumption}{Assumption}
\newtheorem{example}[lemma]{Example}

\articletype{RESEARCH ARTICLE}
\jyear{2021}
\jdoi{10.1017/fmp.2021.18}

\begin{document}

\begin{Frontmatter}

\title{The support of singular stochastic PDEs}

\author[1]{Martin Hairer}\orcid{0000-0002-2141-6561}
\author[1]{Philipp Sch\"onbauer}
\authormark{Martin Hairer and Philipp Sch\"onbauer}

\address[1]{\orgname{Imperial College London}, \orgaddress{\street{180 Queen's Gate}, \city{London} \postcode{SW7 2AZ}, \country{United Kingdom}};\email{m.hairer@imperial.ac.uk, philipp.schoenbauer@gmx.net}}

\received{12 December 2019}
\accepted{31 October 2021}


\abstract{We obtain a generalisation of the Stroock--Varadhan support theorem for a large class of
systems of subcritical singular stochastic PDEs driven by a noise that is either white or 
approximately self-similar. The main problem that we face is the presence of renormalisation.
In particular, it may happen in general that different renormalisation procedures yield solutions with
different supports. One of the main steps in our construction is the identification of a 
subgroup $\CH$ of the renormalisation group such that any renormalisation procedure determines 
a unique coset $g\circ\CH$. The support of the solution then only depends on this coset and 
is obtained by taking the closure of all solutions obtained by replacing the driving 
noises by smooth functions in the equation that is renormalised by some element of $g\circ\CH$.

One immediate corollary of our results is that the $\Phi^4_3$ measure in finite volume
has full support and that the associated Langevin dynamic is exponentially ergodic.}

\end{Frontmatter}

\localtableofcontents

\vspace*{14pt}

\section{Introduction}

The purpose of this article is to provide a far-reaching generalisation of the 
support theorem of Stroock and Varadhan \cite{SupportThm}.
Recall that this result can be formulated as follows. Let 
$\{V_i\}_{i=0}^m$ be a finite collection of vector fields on $\R^n$
that have bounded first and second derivatives and consider the solution
$x$ to the system of stochastic differential equations given by
\begin{equ}[e:SDE]
dX = V_0(X)\,dt + \sum_{i=1}^m V_i(X)\circ dW_i(t)\;,
\end{equ}
where the $W_i$ are i.i.d.\ standard Wiener processes and $\circ$ denotes 
Stratonovich integration \cite{Strat}. Write $\P_x$ for the law of the solution
to \eqref{e:SDE} with initial condition $X_0 = x$ on $\CC(\R_+, \R^n)$.
It follows from the Wong--Zakai theorem that, if we write $X^{(\eps)}$
for the solution to the random ODE
\begin{equ}[e:SDEeps]
\dot X^{(\eps)} = V_0(X^{(\eps)}) + \sum_{i=1}^m V_i(X^{(\eps)})\, \dot W_i^{(\eps)}\;,
\end{equ}
for $W^{(\eps)}$ a smooth approximation to $W$ (for example convolution with 
a smooth mollifier), then $X^{(\eps)} \to X$ in probability. On the other hand, 
for any fixed $\eps > 0$, the topological support of the law $\P_x^{(\eps)}$
of $X^{(\eps)}$ is contained in the closure $R_x$ of the range of the 
continuous map $\CI_x \colon \CC^1(\R_+, \R^m) \to \CC(\R_+, \R^n)$
which maps any $\CC^1$ function $W^{(\eps)}$ to the solution to \eqref{e:SDEeps}.

Since the topological support is lower semi-continuous under weak convergence, this
immediately implies that one also has $\supp \P_x \subset R_x$. What 
Stroock and Varadhan proved in \cite{SupportThm} is that one actually has
$\supp \P_x = R_x$. Our aim is to generalise such a statement to a wide class of
singular stochastic PDEs. 

The general framework used in this article is that 
of \cite{BrunedHairerZambotti2016,BrunedChandraChevyrecHairer2017}. Loosely, speaking, we
consider systems of SPDEs of the form
\begin{align}\label{eq:singular:SPDE}
\partial_t u_i = \CL_i u_i
+
F_i(u, \nabla u, \ldots)
+
\sum_{j\le n} F_i^j(u, \nabla u, \ldots)\xi_j \,,
\qquad{i \le m}
\end{align}
where the $\CL_i$ denote homogeneous differential operators on $\R^d$,
the spatial variable takes values in the torus $\T^d$, and the $\xi_i$ 
denote driving noises that are of the form $\xi_i = \CK_i \star \eta_i$
where $\eta_i$ denotes space-time white noise (or possibly noise that is white in
space and constant in time) and $\CK_i$ is a kernel
which is self-similar in a neighbourhood of the origin and smooth otherwise.
The $F_i^j$ are local nonlinearities in the sense that the value of 
$F_i^j(u, \nabla u, \ldots)$ at a given space-time point is a smooth function of
$u$ and finitely many of its derivatives evaluated at that same point.
We will assume throughout that the system \eqref{eq:singular:SPDE} is locally
subcritical in the sense of \cite{BrunedHairerZambotti2016}.
\begin{remark}
The choice 	$\xi_i = \CK_i \star \eta_i$ covers many interesting examples in which $\xi_i$ is the solution of a linear equation driven by $\eta_i$; in this case $\CK_i$ should be chosen as the Green's function. For our support theorem we do not need that $\CK_i$ is actually the Green's function of a PDE, but we do need the kernel to be homogeneous under rescaling. This assumption will be used heavily throughout this article, compare Assumption~\ref{ass:kernelhomo}.
\end{remark}

It was shown in \cite{BrunedHairerZambotti2016} that one can associate to such 
an equation in a  natural way a nilpotent Lie group
$\CG_-$\label{idx:RG}, usually called the \textit{renormalisation group} in this context,
as well as a construction of the following type.
Write $\CX$ for a suitable space of right-hand sides for \eqref{eq:singular:SPDE} (i.e.\ an element of $\CX$
consists of the nonlinearities $F_i$ as well as $F_i^j$ that can be described by a regularity structure
built from a fixed complete subcritical ``rule'' as in \cite[Sec.~5]{BrunedHairerZambotti2016}) and write $\CX_0 \subset \CX$ for 
the ``deterministic right-hand sides'', i.e.\ those elements such that $F_i^j \equiv 0$.
 
One then has a map $\Upsilon\colon \CG_- \times \CX \to \CX_0$
such that $(g,F) \mapsto F + \Upsilon(g,F)$ yields a representation of $\CG_-$ on $\CX$.
(See Remark~\ref{rem:reprOmegaF} for more details.)

Furthermore, given any natural regularisation $\xi^\eps$ of $\xi$, one can find a sequence
of elements $g_\eps \in \CG_-$ such that the solutions to 
\begin{align}\label{eq:singular:SPDE:reg}
\partial_t u^\eps_{i} = \CL_i u^\eps_i
+
F_i(u^\eps, \nabla u^\eps, \ldots)
+
\sum_{j\le n} F_i^j(u^\eps, \nabla u^\eps, \ldots)\xi^\eps_j
+
\bigl(\Upsilon (g_\eps,F) \bigr)_i(u^\eps, \nabla u^\eps, \ldots)\;,
\end{align}
subject to suitable initial conditions $u^\eps_i(0,\cdot) = u^{\eps,(0)}_i$, 
converge to a limit $u$. (The convergence takes place in probability in a
space of Hölder continuous trajectories with possible finite-time blow-up.)
These limits have a restricted uniqueness property in the sense that, for any other
regularisation $\tilde \xi^\eps$ of $\xi$ one can find a sequence of elements
$\tilde g_\eps \in \CG_-$ such that the solutions to \eqref{eq:singular:SPDE:reg} 
with $\xi^\eps$ replaced
by $\tilde \xi^\eps$ and $g_\eps$ replaced by $\tilde g_\eps$ converge to the
same limit. 

\begin{remark}\label{rmk:initialcondition}
As in \cite[Sec.~2.7]{BrunedChandraChevyrecHairer2017}, the initial condition $u^{\eps,(0)}_i$ is dependent on $\eps$ and 
taken of the form $u^{\eps,(0)} = v^{(0)} + \CS^-_\eps(\xi)(0,\cdot)$, where $\CS^-_\eps(\xi)$ is a stationary process
representing the rough part (i.e.\ the non function-valued part) of the solution. 
In particular, it is in general not possible to choose as initial condition a deterministic smooth function unless solutions
themselves are function-valued in which case $\CS^-_\eps \equiv 0$. An interesting equation where this happens is 
the so-called $\Phi^4_{4-\delta}$ equation, 
see \cite[Sec.~2.8.2]{BrunedChandraChevyrecHairer2017} and 
Sections~\ref{sec:Phi4} and \ref{sec:Phi44}.
For many interesting examples, including generalised KPZ and generalised PAM, 
this issue is not apparent and the initial condition can be chosen as any deterministic function 
(or even distribution) with sufficient regularity. An exceptional case is $\Phi^4_3$ where $\CS^-_\eps \ne 0$ but one can compensate this by choosing $v^{(0)}$ appropriately, compare Section~\ref{sec:Phi43}.
\end{remark}

\begin{remark}\label{rem:reprOmegaF}
Writing $\TT_-$ for the set of trees of negative degree associated to the class of SPDEs under consideration
(see Section~\ref{sec:trees} below), 
one can explicitly set 
\begin{equ}
\Upsilon\colon (g,F) \mapsto \sum_{\tau \in \TT_-} {g(\tau) \over S(\tau)} \Upsilon^F[\tau]\;,
\end{equ}
where the space of nonlinearities $\CP^{\FL_+}$, combinatorial factor $S(\tau)$ and
evaluation map $\Upsilon^F$ are defined in \cite[Sec.~2.7]{BrunedChandraChevyrecHairer2017}.
In the right-hand side, we identify elements of $\CG_-$ with maps $\TT_- \to \R$ (characters on the 
free unital algebra generated by $\TT_-$). 

In this context, the purpose of $\Upsilon$ is to provide a formula for the 
counterterms required to renormalise our equation. As already noted in 
\cite{BrunedChandraChevyrecHairer2017}, the same map also provides an expression for the 
expansion of the ``abstract solution'' to our SPDE in the corresponding regularity 
structure. This is strongly reminiscent of the expression of the 
Taylor expansion of the solution to an ODE in terms of a sum over trees \cite{Butcher}.
\end{remark}

We call a choice of $(\xi^\eps, g_\eps)_{\eps > 0}$ a \textit{renormalisation procedure}
and we consider two such procedures to be equivalent if they yield 
the same limit process for any system of SPDEs driven by $\xi^\eps$
belonging to a suitable class of systems of the same form as the original one.
(See \cite{BrunedChandraChevyrecHairer2017} for the definition of this class of equations given a ``rule''
in the sense of \cite{BrunedHairerZambotti2016}.)
Given two renormalisation procedures $(\xi^\eps, g_\eps)$ and 
$(\tilde \xi^\eps, \tilde g_\eps)$, it turns out that it is always possible to find one single element $f \in \CG_-$
such that $(\tilde \xi^\eps, \tilde g_\eps)$ is equivalent to $(\xi^\eps, f \circ g_\eps)$.
Given any fixed choice of compactly supported kernel $\CK_i$ such that $(\d_t - \CL_i)\CK_i = \delta$ in a 
neighbourhood of the origin and any choice $\xi^\eps$ of smooth approximation to $\xi$ (by
convolution with a compactly supported mollifier, but this could in principle be more general),
there is a distinguished choice of $g^{(\eps)}_\BPHZ$ (depending on $\xi^\eps$), which we call the ``BPHZ renormalisation'',
see \cite{BrunedHairerZambotti2016}. In particular, this has the property that the $(\xi^\eps, g^{(\eps)}_\BPHZ)$ are
all equivalent for different choices of $\xi^\eps$, so that we can talk about ``the'' BPHZ 
solution to \eqref{eq:singular:SPDE}.

At first sight, the natural generalisation of Stroock and Varadhan's result
for a system of equations of the type \eqref{eq:singular:SPDE} may be that the support
of the solutions starting at $u$ coincides with the closure $R_u$ of the 
set of all solutions to \eqref{eq:singular:SPDE} with the $\xi_j$ replaced by smooth 
controls.
A moment of thought reveals that this cannot be the case
for the simple reason that the formal expression \eqref{eq:singular:SPDE} only determines
a solution theory up to a choice of renormalisation procedure and different
renormalisation procedures may produce solutions with different supports. This is already
apparent in the case of SDEs where an expression like
\begin{equ}
\dot x = V_0(x) + V_i(x)\,\xi_i\;,
\end{equ}
(summation over repeated indices is implicit)
may be interpreted either in the Itô sense or in the Stratonovich sense, yielding
solution theories with distinct supports in general.

It is also not difficult to see that in general one cannot hope to obtain the support of 
\eqref{eq:singular:SPDE} as the closure $R_u^g$ of the 
set of all solutions to \eqref{eq:singular:SPDE:reg} with the $\xi_j^\eps$ 
replaced by smooth controls and $g_\eps$ replaced by some fixed element
$g$ of the renormalisation group. Indeed, consider the system of SPDEs given by
\begin{equ}[e:example]
\d_t u = \d_x^2 u + \xi\;,\qquad \d_t v = \d_x^2 v + (\d_x u)^2\;.
\end{equ}
The relevant part of the renormalisation group for this equation is
simply $(\R,+)$, with the renormalised equation being of the form
\begin{equ}[e:exampleRenorm]
\d_t u = \d_x^2 u + \xi\;,\qquad \d_t v = \d_x^2 v + \bigl((\d_x u)^2 - c\bigr)\;.
\end{equ}
For any fixed value of $c$, solutions to \eqref{e:exampleRenorm} with 
smooth $\xi$ and vanishing initial condition are such that $v$ is
bounded below by $-ct$. However, the solution to \eqref{e:example}
should really be interpreted as the limit as $\eps \to 0$ to the 
solution to \eqref{e:exampleRenorm} with $\xi$ replaced by $\xi_\eps$ and
$c$ replaced by $c_\eps$ for a suitable choice of $c_\eps \to +\infty$.

Furthermore, it was already remarked in \cite{Hairer2012} (in a slightly different
setting) that, for any fixed smooth $h$,  the solutions to 
\begin{equ}[e:oscillation]
\d_t u = \d_x^2 u + h + a \eps^{-1}\cos(\eps^{-1} x)\;,\qquad 
\d_t v = \d_x^2 v + (\d_x u)^2 - c\;,
\end{equ}
converge as $\eps \to 0$ to those of
\begin{equ}
\d_t u = \d_x^2 u + h \;,\qquad 
\d_t v = \d_x^2 v + (\d_x u)^2 - (c - \tilde c a^2)\;,
\end{equ}
for some fixed positive constant $\tilde c$. In other words, it is possible
to emulate a \textit{decrease} in the renormalisation constant $c$ (but not an increase!)
by adding a small (in a distri\-bu\-tional sense) highly oscillatory term
to $h$. This suggests that the support of the solution to \eqref{e:example}
is given by the closure of the set of all solutions to 
\begin{equ}[e:supportKPZ2]
\d_t u = \d_x^2 u + h \;,\qquad 
\d_t v = \d_x^2 v + (\d_x u)^2 - c\;,
\end{equ}
for any choice of smooth function $h$ and any choice of constant $c \in \R$.
As a matter of fact, by considering perturbations of $h$ of the type
\eqref{e:oscillation}, but with an additional modulation of the highly oscillatory term,
we will see in Theorem~\ref{theo:gKPZ} below that, whatever the choice of renormalisation procedure, solutions
to \eqref{e:example} have full support, so that this example exhibits some weak
form of ``hypoellipticity''.

\subsection{The main theorem}

We consider subcritical SPDEs of the form \eqref{eq:singular:SPDE} 
such that Assumptions~\ref{ass:main:reg} and~\ref{ass:noises:derivatives:products} below hold.
Subcriticality ensures that one can construct a problem dependent regularity structure as in \cite{BrunedHairerZambotti2016}, and Assumptions~\ref{ass:main:reg} and~\ref{ass:noises:derivatives:products} guarantee by \cite[Thm.~2.33]{ChandraHairer2016} the convergence of the sequence of admissible models $\hat Z^\eps$ to a random limit model~$\hat Z$, where $\hat Z^\eps$ denotes the renormalised canonical lift of the regularised noise~$\xi^\eps$, see Section~\ref{sec:kernels}.
Furthermore, we can only expect a support theorem to hold if the integration kernels associated to our equations are homogeneous on small scales, and in order to not overcomplicate the presentation, we assume that our Green's functions are self-similar under rescaling, compare Assumption~\ref{ass:kernelhomo}. 
For convenience we also restrict to the case of independent (space or space-time) Gaussian white noises $\xi_i$ (but compare Remark~\ref{rem:noises}). 
Our assumptions ensure that equation (\ref{eq:singular:SPDE}) can be lifted to an abstract fixed point problem as in \cite[Thm.~7.8]{Hairer2014}. 
%
Finally, we need a technical assumption on the trees that appear in our regularity structure, which for ease of this introduction we will not comment on and instead refer the interested reader to Assumptions~\ref{ass:technical} and~\ref{ass:CVz} in Section~\ref{sec:technical:assumpation}.

In order to have a well behaved solution map,
it is convenient to be in the slightly more restrictive setting of \cite{BrunedChandraChevyrecHairer2017}, 
which guarantees in particular that the reconstructed solution to the abstract fixed point problem for $\hat Z^\eps$ 
satisfies the regularised and renormalised SPDE (\ref{eq:singular:SPDE:reg}).
We thus assume for the sake of the main results, Theorems~\ref{thm:mainThm} and \ref{thm:main:sv}, that the full assumptions of \cite{BrunedChandraChevyrecHairer2017} are satisfied. 
\begin{assumption}\label{ass:main:thm}
We assume that 
\cite[Eqn.~2.5, Ass.~2.6, Ass.~2.8, Ass.~2.13, Ass.~2.15, Ass.~2.16]{BrunedChandraChevyrecHairer2017} are satisfied and that
Assumptions~\ref{ass:main:reg}--\ref{ass:zero-homo} given in Section~\ref{sec:assumptions} and Assumptions~\ref{ass:CJHopfIdeal} and \ref{ass:CHBPHZcharacters} 
given in Section~\ref{sec:CJ} hold.
\end{assumption}

\begin{remark}
We will show in Section~\ref{sec:constraints} below that Assumptions~\ref{ass:CJHopfIdeal} and \ref{ass:CHBPHZcharacters} 
are implied by Assumptions~\ref{ass:main:reg}--\ref{ass:CVz} (without requiring the additional assumptions
of \cite{BrunedChandraChevyrecHairer2017}).
\end{remark}

Our main result then is a support theorem for the BPHZ renormalised model $\hat Z$ or indeed
any model differing from $\hat Z$ by the action of an element of the renormalisation
group $\CG_-$ associated to the class of equations under consideration. If we denote 
by $Z(h)$ the canonical lift of any $h \in \CC_0^\infty$, a slightly informal version of
our main result reads as follows

\begin{theorem}
There exists a subgroup $\CH \subset \CG_-$ and a left coset $g\CH$ of $\CH$ such that
\begin{equ}
\supp \hat Z = \overline{\{ \CR^h Z(f)\,:\, f \in \CC_0^\infty,\, h \in g\CH\}}\;.
\end{equ}
\end{theorem}

One important remark here is that $\CH$ is \textit{not} determined solely by the
regularity structure of the problem. Instead, it also incorporates information about
the symmetries satisfied by the integration kernels associated to the problem. 
Extracting this ``rigid'' algebraic data out of ``soft'' analytic data is one
of the main difficulties of this article.

We also have a more concrete statement at the level of solutions which we state now.
Regarding solutions, our support theorem applies for $u$ 
in any space $\CX= \bigoplus_{i} \CX_i$ such that 
the solution operator (mapping the space of admissible models for the regularity structure $\CT$ into $\CX$) is continuous. For instance, one could define the space $\CX_i$ as a version of the usual H\"older spaces allowing for finite time blow up as in \cite{BrunedChandraChevyrecHairer2017}. 
In situations where we know a priori that the solution survives until some deterministic
time $T>0$ almost surely, 
one can take alternatively for $\CX_i$ the usual H\"older--Besov spaces $\CC^{-\shalf + \beta_i -\kappa}_\fsL((0,T) \times \T^e)$. (Here $\beta_i>0$ and the scaling $\fsL : \{0, \ldots, e\} \to \N$ are determined by the linear part $\partial_t - \CL_i$ of our equations, see Assumption~\ref{ass:kernelhomo}. The statement holds for any $\kappa>0$.)
The main theorem of this article is the following description of the topological support of $u$.

\begin{theorem}\label{thm:mainThm}
Under Assumption~\ref{ass:main:thm}, let $u^\eps$ denote the classical solutions to the regularised and renormalised equation (\ref{eq:singular:SPDE:reg}) with noise $\xi^\eps$ and renormalisation constants $c^\eps_\tau = h\circ g^\eps_\BPHZ(\tau)$ for some fixed $h \in \CG_-$, and let $u := \lim_{\eps \to 0}u^\eps$. Then, one has the identity
\[
\supp u=\bigcap_{\eps>0} \overline{\bigcup_{\delta<\eps} \supp u^\delta}
\]
in $\CX$. 
\end{theorem}

In Theorem~\ref{thm:main} below we show that Assumptions~\ref{ass:main:reg}--\ref{ass:zero-homo} 
and \ref{ass:CJHopfIdeal}, \ref{ass:CHBPHZcharacters} imply an analogous support theorem for the random models
associated to $u$ and $u^\delta$. More precisely, we show
that
\begin{equ}[e:suppModel]
\supp \hat Z=\bigcap_{\eps>0}\overline{\bigcup_{\delta<\eps} \supp \hat Z^{\delta}}\;,
\end{equ}
where $\hat Z^\delta$ denotes the BPHZ model associated to the noise $\xi^\delta$
and $\hat Z$ denotes its limit, namely the BPHZ model associated to the limiting white 
noise $\xi$.

Once we know \eqref{e:suppModel}, Theorem~\ref{thm:mainThm} is a direct consequence of the 
continuity of the solution operator given in \cite[Thm~2.21]{BrunedChandraChevyrecHairer2017},
combined with the fact that, given a measure $\mu$ and a continuous map $F$, $\supp F^*\mu$
is given by the closure of $F(\supp \mu)$.
It will become clear from our proof that for a ``tweaked'' choice of renormalisation constants $\tilde c^\eps_\tau = k^\eps \circ h \circ g^\eps_\BPHZ(\tau)$ with $k^\eps \to \one^*$ as $\eps \to 0$ one can show that, denoting by $\tilde u^\eps$ the classical solution to the system (\ref{eq:singular:SPDE}) with renormalisation constants $\tilde c^\eps$, one still has $\tilde u^\eps \to u$ in probability in $\CX$, but one has the stronger statement
\[
\supp \tilde u^\eps\subseteq \supp u
\]
for any $\eps>0$. 

We also have a characterisation of the support in the spirit of Stroock and Varadhan's support theorem for 
SDEs \cite{SupportThm}.
The ``correct'' way to resolve the issue of divergent renormalisation constants 
in such a description turns out to be the following.

\begin{theorem}\label{thm:main:sv}
Under Assumption~\ref{ass:main:thm}, let $h \in \CG_-$ and $u$ be as in Theorem~\ref{thm:mainThm}.
There exists a Lie subgroup $\CH \ssq \CGm$ of the renormalisation group and a character $f \in \CGm$ independent
of the choice of $h$ in Theorem~\ref{thm:mainThm}
such that the following holds. The support $\supp u$ is given by the closure of all solutions $\vphi$ to 
\begin{align}\label{eq:control:problem}
\partial_t \vphi_{i} = \CL_i \vphi_i
+
F_i(\vphi, \nabla \vphi, \ldots)
+
\sum_{j\le n} F_i^j(\vphi, \nabla \vphi, \ldots) \psi_j
+
\bigl(\Upsilon_i k \bigr)(\vphi, \nabla \vphi, \ldots)\;,
\end{align}
for any character $k \in h \circ f \circ \CH$, initial condition $\vphi(0,\cdot)  \in \Phi_0^k$\footnote{See Remark~\ref{rmk:initialcondition:controlproblem} below for the definition of this set.},
and smooth deterministic functions $\psi_j$, $j=1, \ldots, n$ 
(depending only on space if $\xi_j$ is purely spatial white noise).
Here we write $\Upsilon_i k := (\Upsilon M^{k}\Omega_F )_i$ for simplicity.
%
\end{theorem}

Theorem~\ref{thm:main:sv} follows from Proposition~\ref{prop:constantsInSupport} below, the properties of the shift operator, Theorem~\ref{thm:translationoperator}, and the continuity of the solution operator. The Lie subgroup $\CH$ is given as the annihilator of a finite number of linear ``constraints'' 
between the renormalisation constants. We refer the reader to Definition~\ref{def:CH} for a
precise definition. The tweaking by $f$ is necessary, since the BPHZ characters only respect 
these constraints up to order $1$ (a by-product of the fact that we use truncated integration kernels 
for its definition.)
%
%


\begin{remark}\label{rmk:initialcondition:controlproblem}
In case that we are in a situation in which we are allowed to choose the initial condition $u^{\eps,(0)} = v^{(0)}$ deterministically and independent of $\eps$, the initial condition of the control problem (\ref{eq:control:problem}) has to coincide with this choice, so that we have to set $\Phi_0^k = \{ v^{(0)} \}$.

In order to cover also the case that the initial condition $u^{\eps,(0)}$ is a perturbation to $\CS^-_\eps(\xi)(0,\cdot)$, compare Remark~\ref{rmk:initialcondition}, we make use of the fact that $\CS^-_\eps$ can be written as an explicit continuous function of 
the model $\hat Z^\eps$, compare \cite[Prop.~5.22, Eqn.~6.10]{BrunedChandraChevyrecHairer2017}. In the notation of that paper we 
define $\Phi_0^k$ as the set of all functions of the form 
$v^{(0)} + (\CR^Z \CP^Z \tilde U)(0,\cdot) \in (\CC^\infty(\T^e))^m$ 
where $Z$ is a renormalised canonical lift
$Z = \CR^k \Zcan(\psi)$ with 
$\psi \in \CC_c^\infty(\R\times\T^e)^n$.
(The fact that we can choose the initial condition independently of the $\psi_j$'s appearing in (\ref{eq:control:problem}) comes from the fact that $(\CR^Z \CP^Z \tilde U)(0,\cdot)$ only depends on the value of $\psi$ on negative times, while in the equation (\ref{eq:control:problem}) only the behaviour of $\psi$ for positive times matters.)
\end{remark}

\begin{remark}
It suffices to prove Theorems~\ref{thm:mainThm} and~\ref{thm:main:sv} for $h=\one^*$. This follows, since by \cite[Thm.~2.13]{BrunedChandraChevyrecHairer2017} there exists a action $(F,h) \mapsto h \circ F$ of the renormalisation group $\CGm$ onto the collection of vector fields $F=(F_i)$, which leaves the class of vector fields consider in \cite{BrunedChandraChevyrecHairer2017} invariant, and is such that $F_i + \Upsilon_i (h \circ g) = h \circ F_i  + \Upsilon_i g$ for any $h,g \in \CGm$. Therefore, changing renormalisation can simply be viewed as changing the non-linearity.
\end{remark}

\begin{remark}\label{rmk:a-priori}
The set $f \circ \CH$ used in Theorem~\ref{thm:main:sv} is in some sense the largest set of characters such that we can guarantee that the solution to (\ref{eq:control:problem}) is in the support of $u$. In many situations we know a-priori that there exists a smooth approximation $\xi^\eps = \xi\star \rho^\eps$ as above with the property that the BPHZ characters $g^\eps$ take values in a fixed subset $K \ssq f \circ \CH$. In this case, combining Theorem~\ref{thm:mainThm} and Theorem~\ref{thm:main:sv} implies that the support $\supp u$ is given by the closure of the set of all solutions to the control problem (\ref{eq:control:problem}) with $k \in K$.
\end{remark}

\begin{remark}
The classical Stroock--Varadhan support theorem can be viewed as the case $d=0$ of our result
with $\CL_i = 0$.
In this case, one has $\CG_- \simeq (\R,+)$\footnote{Strictly speaking one has
$\CG_- \simeq (\R^{n \times n},+)$, but only multiples of the identity matrix preserve 
the natural symmetries given by invariance under permutation 
of indices.} and, in the notations of Theorem~\ref{thm:main:sv}, 
\begin{equ}
(\Upsilon_i c)(u) = c F_k^j(u) \d_k F_i^j(u) \;,\qquad c \in \R \simeq \CG_-\;,
\end{equ}
with summation over $j$ and $k$ implied. Furthermore, using the BPHZ model (and therefore setting
$h = 0$) leads to solutions in the Itô sense. Since there is only one renormalisation constant
in this case and the ``heat kernel'' is given by the Heaviside function which is non-trivial,
Definition~\ref{def:CH} readily leads us to the conclusion that $\CJ$ is the unit ideal, so
that its annihilator is given by $\CH = \{0\}$. 

Theorem~\ref{thm:main:sv} then states that there exists some constant $c$ such that 
the support of the Itô solutions to
\begin{equ}
du_i = F_i(u)\,dt + \sum_{j \le n} F_i^j(u)\,dW_j\;,
\end{equ}
is given by the closure of all solutions to
\begin{equ}[e:supportSDEs]
\dot u_i = F_i(u) + \sum_{j \le n} F_i^j(u)\,\psi_j + c \sum_{k \le m} \sum_{j \le n} F_k^j(u) \d_k F_i^j(u)\;,
\end{equ}
with smooth controls $\psi_j$. Note that the correct value of $c$ (corresponding to the
character $f$ in the statement) is not specified by the theorem. 
On the other hand, one can explicitly compute the ``BPHZ character'' in this case
and show that (again identifying $\CG_-$ with $\R$) it converges as $\eps \to 0$ to $-\f12$,
and we conclude from Remark~\ref{rmk:a-priori} 
that $c = -\f12$ in \eqref{e:supportSDEs}, thus recovering the
Stroock--Varadhan support theorem.
\end{remark}

Before we proceed, let us briefly discuss how these results compare to the existing literature.
There are of course many support theorems for stochastic PDEs that do \textit{not} require renormalisation,
see for example \cite{BMS,support3,support1,support2}. In all of these cases, 
the statement is the one that one would expect,
namely that the support is given by the closure of all solutions obtained by replacing the noises
by suitable controls. In the case of singular SPDEs, information on the support follows
in some special cases. For example, Jona-Lasinio and Mitter \cite{Jona} construct solutions 
to a type of Langevin equation for the
$\Phi^4_2$ measure by using Girsanov's theorem, which yields full support as an immediate byproduct.
One of the earliest result on the support in cases that cannot be dealt with in this way is the work
\cite{Friz2016} by Chouk and Friz where the authors consider a generalised parabolic Anderson model
of the form $\d_t u = \Delta u + g(u) \xi$ in dimension~$2$ and show that a suitably renormalised
version of it has support given by the closure of all solutions to control problems of 
the type $\d_t u = \Delta u + g(u) \phi + c (gg')(u)$ with $\phi$ a smooth function (constant in time) and
$c$ an arbitrary constant. This can be viewed as a special case of our result in a situation
where $\CH = \CG \approx (\R,+)$. The way we deal with the presence of renormalisation,
while inspired by \cite{Friz2016}, substantially differs from the construction given there.
See the discussion at the start of Section~\ref{sec:shift} for more details.

Using similar techniques, Tsatsoulis and Weber \cite{Tsatsoulis2018} showed that the $\Phi^4_2$ dynamic has full 
support.  
Finally, proofs of support theorems for stochastic \emph{ordinary} differential equations based on rough path techniques are by now very classical. It was already mentioned in \cite{Lyons1998} that the continuity properties of the solution map can be used for a straightforward proof of a support theorem, provided one has a support theorem for the enhanced Brownian motion. The latter was shown in a series of results, see for instance  \cite{LEDOUX2002} (for a support theorem for rough paths in the p-variation topology), \cite{Friz2005} (in H\"older topology), \cite{Friz2006} (for enhanced fractional Brownian motions) and \cite{friz2010} (for an implementation using deterministic shifts). For an introduction to the topic and more details see \cite[Sec.~9.3]{FrizHairer2014} or \cite[Cha.~19]{friz_victoir_2010}.

\subsection{Applications}

\subsubsection{The \texorpdfstring{$\Phi^p_d$}{Phi p d} equation}\label{sec:Phi4}

The $\Phi^p_d$ equation formally is given by
\begin{align}\label{eq:phip2}
\partial_t u = \Delta u + \sum_{1 \le k \le p-1} a_k u^k + \xi
\end{align}
with space-time white noise $\xi$ on $\domain = \R \times \T^d$. This equation is subcritical in the
sense of \cite{Hairer2014,BrunedHairerZambotti2016} provided that $p < 2d/(d-2)$. As pointed out above, in a formal sense, 
one can also consider \eqref{eq:phip2} in dimension ``$d-\eps$'', either by replacing 
$\Delta$ by $-(-\Delta)^{1+\eps}$ or by convolving $\xi$ with a slightly regularising Riesz kernel.
We will restrict ourselves here to the cases $d=2$ and $p$ even, $d =3$ and $p=4$, as well
as $d = 4-\eps$ and $p=4$. We denote by ``the'' solution to \eqref{eq:phip2} the BPHZ solution in
the sense of \cite{BrunedHairerZambotti2016,ChandraHairer2016} for any fixed truncation $K$ of the heat kernel. All statements below are
independent of the choice of cutoff.

Note that in dimension $d=2$ Assumption~\ref{ass:CVz} below is violated, 
but as pointed out in Remark~\ref{rmk:assumptions}, Assumption~\ref{ass:CVz} can be replaced by Assumptions~\ref{ass:CJHopfIdeal} and~\ref{ass:CHBPHZcharacters}, which are trivially true in this case (one has $\CJ:=\{ 0 \}$ and $\CH = \CG_-$).
In dimension $d=3$ all assumptions are satisfied. However, the `black-box' theorem of
\cite{BrunedChandraChevyrecHairer2017} only allows us to start the approximate equation at a perturbation of $\CS^-_\eps(\xi)(0,\cdot)$, compare Remark~\ref{rmk:initialcondition} (in this case $\CS^-_\eps(\xi)(0,\cdot)$ is in law a smooth approximation to the Gaussian free field). As was already noticed in \cite[Sec.~9.4]{Hairer2014}, this issue can be circumvented, but this requires to work with a model topology which is slightly stronger than the usual one. We show in Section~\ref{sec:Phi43} that the support theorem still holds for this topology.
If we emulate dimension 
$d=4-\eps$ by slightly regularising the noises, then our 
assumptions on the noises are violated (since they are no longer white), but it is again possible 
to resolve this issue, see Section~\ref{sec:Phi44}. 
We will be interested in showing ergodicity of \eqref{eq:phip2}, so that we will
always assume that $a_{p-1} <0$. Under this condition, we have the following consequence of Theorem~\ref{thm:main:sv}.
\begin{theorem}\label{thm:Phi4}
Let $u_0 \in \CC^{\eta}(\T^3)$, where $\eta > -\frac{2}{3}$ if $d=2,3$ and $\eta>-(\eps\wedge {2\over 3})$ if $d=4-\eps$.
Let $u$ denote the solution to the $\Phi^p_d$ equation with the combinations of $p$ and $d$ mentioned above, with initial condition
$u_0 + \CS^-(\xi)(0,\cdot)$ (in the sense of Remark~\ref{rmk:initialcondition}).
Then for any $T>0$
 $u$ has full support in $\CC^\alpha_\fs((0,T) \times \T^d)$ for $\alpha = \frac{2-d}{2}-\kappa$ for any 
$\kappa>0$.

For $d=2,3$, let $\alpha$ as above and consider the solution $u$ with fixed initial condition 
$u_0 \in \CC^{\eta}(\T^3)$ for some $\eta \in (-\frac{2}{3},\alpha]$. Then,  $u$ has support in $\CC([0,T],\CC^\eta(\T^d))$
given by all functions with value $u_0$ at time $0$.
\end{theorem}
\begin{proof}
Global existence for these equation was shown in \cite{Tsatsoulis2018} in $d=2$ and \cite{Weber2016,MoinatWeber2018} in $d=3$. 
For $d=4-\eps$ it will be a consequence of a forthcoming paper \cite{CMW2019}. The first statement then follows 
directly from Theorem~\ref{thm:main:sv},
which shows that any trajectory can be realised since the equation is driven by additive noise.

The second statement does not follow immediately since the topology of our model space is too weak for the solution
map to be continuous as a map with values in $\CC([0,T],\CC^\eta(\T^d))$. We show in Section~\ref{sec:Phi43} below that 
one can endow it with a slightly stronger topology in such a way that the solution map becomes continuous and
our support theorem still holds.
\end{proof}

A particular application of our support theorem in dimension $d \le 3$  is to the uniqueness of the invariant measure and exponential convergence to this measure. 

\begin{corollary}\label{cor:Phi4}
Assume that $p$, $d\le 3$ and $a_{p-1}$ are as above. Then the $\Phi^p_d$ equation
admits a unique invariant measure $\mu$ on $\CC^\alpha(\T^d)$. 

Moreover, if $p \ge 4$, then we have uniform exponential convergence of the dynamical model 
to the invariant measure in the following sense. Let $u$ be the solution starting from $u_0$ as 
in Theorem~\ref{thm:Phi4}. Then
\begin{equ}[e:SG]
\| (u_t)_*\P - \mu \|_{\TV} \le 1 \wedge C\exp (- \lambda t) \;,
\end{equ}
for some $C,\lambda>0$, uniformly over $t \ge 0$ and $u_0 \in \CC^\alpha(\T^d)$. (Here, $f_*\P$ denotes the
pushforward of the measure $\P$ under the random variable $f$.)
\end{corollary}
\begin{proof}
It follows from Doeblin's theorem (see for instance \cite[Thm.~3.6]{Hairer2016notes} with $V=0$)
that it suffices to show that for some $t>0$ one has\footnote{We normalise the total variation norm so that mutually singular probability measures have distance $1$.}
\begin{align}\label{eq:harris:app}
\| (u_t^{v})_*\P - (u_t^{w})_*\P \|_{\TV} \le 1-\delta 
\end{align}
for some $\delta>0$ and all $v,w \in \CC^\alpha(\T^d)$. Here $u^{v}$ denotes the solution to (\ref{eq:phip2}) with initial condition $v$.

As a consequence of the ``coming down from infinity'' property, see \cite[Eq.~3.24]{Tsatsoulis2018} for $d=2$, \cite[Eq.~1.27]{Weber2016} for $d=3$ (see also \cite{MoinatWeber2018}), there exists a compact set $K \ssq \CC^\alpha(\T^d)$ such that
\[
\inf_{v \in \CC^\alpha(\T^d)} \P[u_1^{v} \in K] \ge \half\;.
\]
By the strong Feller property for $\Phi^p_d$ shown in \cite{HairerMattingly2016} (see also \cite{Tsatsoulis2018} for $d=2$),  the transition probabilities are continuous in the total variation norm, so that for some $\eps>0$ one has
\[
\| (u_1^{v})_*\P - (u_1^{w})_*\P \|_{\TV} \le \half 
\]
for any $v$, $w$ in the centred $\eps$-ball $B_\eps$ in $\CC^\alpha(\T^d)$. Again by continuity of the transition probabilities and compactness of $K$ the infimum
\[
\rho:=\inf_{v \in K} \P[u^{v}_1 \in B_\eps]
\]
is attained for some $\bar v \in K$ and, by Theorem~\ref{thm:Phi4}, one has $\rho > 0$. It follows that (\ref{eq:harris:app}) holds for $t=3$ with $\delta = \frac{1}{4} \rho$.
\end{proof}

\begin{remark}
We have to restrict to $d \le 3$ in Corollary~\ref{cor:Phi4} since it is not known that the solution to $\Phi^4_{4-\eps}$ is a Markov process (although it is expected). 
Actually, at the current state it is even unclear if one can start the equation at a fixed deterministic initial condition (compare Remark~\ref{rmk:initialcondition} for a discussion of this issue) or evaluate the solution at a fixed positive time.
\end{remark}

\subsubsection{The generalised KPZ equation}
\label{sec:gKPZ}

A natural analogue to the class of SDEs
\eqref{e:SDE} is given by the class of stochastic PDEs recently studied in
\cite{String,BrunedHairer2019} that can formally be written as
\begin{align}\label{e:gKPZ}
\d_t u = \d_x^2 u + \Gamma(u)(\d_xu, \d_x u) + h(u) + \sum_{i=1}^m \sigma_i(u)\,\xi_i\;,
\end{align}
where $u \colon \R_+ \times S^1 \to \R^n$, the $\xi_i$ denote independent space-time white noises, $h:\R^n \to \R^n$ and $\sigma_i :\R^n \to \R^n$ are smooth functions and
$\Gamma$ is a smooth map from $\R^n$ into the space of symmetric bilinear maps 
$\R^n \times \R^n \to \R^n$. 
This should be viewed as a connection on $\R^n$, which
is why we use the customary symbol $\Gamma$ for it, and it gives rise to a notion of covariant differentiation:
\begin{equ}[e:connectionGamma]
(\nabla_X Y)^i(u) = X^j(u)\d_j Y^i(u) + \Gamma^{i}_{j,k}(u)X^j(u)Y^k(u)\;,
\end{equ}
for any two smooth vector fields $X,Y \colon \R^n \to \R^n$.

One problem when trying to even guess 
the form of a support theorem for an equation like \eqref{e:gKPZ} is that there is 
typically no canonical notion of solution associated to it. Instead, one has a whole family
of solution theories that can be parametrised by a \textit{renormalisation group} $\CG_-$.
This already happens for SDEs where one has a natural one-parameter family of solution theories
which include solutions in the sense of It\^o, Stratonovich, backwards It\^o, etc, so that 
$\CG_- = (\R,+)$ in this case.  
While $\CG_-$ is always a finite-dimensional Lie group, it can be quite large in general: even
after taking the $x \leftrightarrow -x$ symmetry and the fact that the noises $\xi_i$ are
Gaussian and i.i.d.\ into account, one has $\CG_- = (\R^{54},+)$ in the case of \eqref{e:gKPZ}
(at least for $n$ large enough, see \cite[Prop.~6.8]{BrunedHairer2019}). Furthermore, there is typically
no na\"ive analogue of the Wong-Zakai theorem: if one simply replaces $\xi$ by a mollified
version $\xi^{(\eps)}$, the resulting sequence of solutions $u^{(\eps)}$ typically fails
to converge to any limit whatsoever. Instead, one needs to modify the right-hand side of the equation
in an $\eps$-dependent way in order to obtain a well-defined limit.

In some cases, imposing additional desirable properties on the solution theory
results in a reduction of the number of degrees of freedom, but still leads to mollifier-dependent 
counterterms.
For example, it is shown in \cite{BrunedHairer2019} that \eqref{e:gKPZ} admits a natural \textit{one}-parameter 
family of 
solution theories, all of them which satisfy all of the following properties simultaneously:
\begin{itemize}
\item The usual chain rule holds in the sense that, if $u$ solves \eqref{e:gKPZ}
and $v = \phi(u)$ for some diffeomorphism $\phi\colon \R^n \to \R^n$, then $v$ solves
the equation obtained from \eqref{e:gKPZ} by formally performing the corresponding change 
of variables as if the $\xi_i$ were smooth. (This is analogous to the property of Stratonovich 
solutions to SDEs.)
\item If $\{\tilde \sigma_j\}_{j=1}^{\tilde m}$ is a collection of smooth vector fields on $\R^n$ such that
\begin{equ}
\sum_{i=1}^m \sigma_i(u) \otimes \sigma_i(u)
= \sum_{j=1}^{\tilde m} \tilde \sigma_j(u) \otimes \tilde \sigma_j(u)\;,
\end{equ}
then the solution to \eqref{e:gKPZ} is identical in law to the solution with the $\sigma_i$ 
replaced by the $\tilde \sigma_i$. (This is analogous to the property of It\^o 
solutions to SDEs.)
\item Given \eqref{e:gKPZ}, there exists a collection of $12$ vector fields\footnote{The number $12$ is the
dimension of the space $\CV^{\mathrm{nice}}$ in \cite[Sec.~1.2, Rem.~3.13]{BrunedHairer2019}} $W_i$ on $\R^n$
such that, for any mollifier $\rho$, there exist constants $c_i^{(\eps)}$ such that,
setting $\xi_i^{(\eps)} = \rho_\eps \star \xi_i$, solutions to \eqref{e:gKPZ} are 
given by $u = \lim_{\eps \to 0}u_\eps$ with
\begin{equ}[e:gKPZ:reg]
\d_t u_\eps = \d_x^2 u_\eps + \Gamma(u_\eps)(\d_xu_\eps, \d_x u_\eps) + h(u_\eps) + \sum_{i=1}^m \sigma_i(u_\eps)\,\xi_i^{(\eps)} - \sum_{j=1}^{12}c_j^{(\eps)} W_j(u_\eps) \;.
\end{equ}
Furthermore, the $W_j$ are such that, for every $u_\star \in \R^n$ such that 
$\Gamma(u_\star) = 0$ and $D \sigma_i(u_\star) = 0$ (for $i > 0$), one has 
$W_j(u_\star) = 0$ for every $j$.
\end{itemize}


Given \eqref{e:gKPZ}, we then define a number of auxiliary vector fields. First,
for $\mu,\nu=1,\ldots,m$, we set 
$$
X_{\mu\nu}(u) = (\nabla_{\sigma_\mu}\sigma_\nu)(u)\;,
$$
and we also write $V_\star$ for the vector field $H_{\Gamma,\sigma}$ defined in \cite[Eq.~1.9]{BrunedHairer2019}.
We then use the $X_{\mu\nu}$ to define two additional vector fields as follows:
\begin{equ}
V = X_{\mu\mu}\;,\quad
\hat V = \nabla_{X_{\mu\nu}} X_{\mu\nu}\;,\quad
\end{equ}
with implied summation over repeated indices.

As already mentioned above, this class of equations admits a one-parameter canonical family
of solution theories that combine the formal properties of both ``Stratonovich'' and ``Itô'' solutions.
We fix once and for all one of these solution theories and call it henceforth ``the'' solution to
\eqref{e:gKPZ}. Again, our statement is independent of the precise choice of solution theory
as long as it belongs to the canonical family. (Actually, this can be further weakened, see Section~\ref{sec:gKPZproof}.) 
Under the assumption that $\Gamma$, 
$h$ and $\sigma$ are smooth functions, we have the following result, the proof of which
is postponed to Section~\ref{sec:gKPZproof}.

\begin{theorem}\label{theo:gKPZ}
Let $u$ be the solution to (\ref{e:gKPZ}) with deterministic 
initial condition $u(0) = u_0 \in \CC^\alpha(\T)$ 
for some $\alpha \in (0, \half)$. Then, there exists a constant $\hat c$ 
such that the support of the law of $u$ in $\CC^\alpha(\R_+\times \T)$ 
is given by the closure of all solutions to 
\begin{equs}\label{e:controlKPZ}
\partial_t u^i
&=
\partial_x^2 u^i
+
\Gamma^{i}_{j,k}(u) \partial_x u^j \partial_x u^k
+
h^i ( u ) \\
&\quad +
\hat c \hat V^i(u) + K_\star V_\star^i(u) + K V^i(u)
+
\sigma_\mu^i ( u ) \eta^\mu
\end{equs}
for arbitrary smooth controls $\eta^\mu$ and arbitrary constants $K, K_\star$.
\end{theorem}

\begin{remark}
The appearance of the additional constants $K$ and $\hat c$ in \eqref{e:controlKPZ}
may seem strange at first, although we have of course already seen in the discussion
preceding \eqref{e:supportKPZ2} that one cannot expect to obtain the support of $u$
by simply replacing noises by smooth controls in \eqref{e:controlKPZ}.
\end{remark}

\begin{remark}
At this stage, we do not know whether one actually has $\hat c = 0$ (which would be natural)
or whether the description given above even depends on the value of $\hat c$. We do however know that
both terms $V_\star$ and $V$ are required for the result to hold, as follows from the example
\begin{equ}
\d_t u_1 = \d_x^2 u_1 + \xi\;,\quad
\d_t u_2 = \d_x^2 u_2 + (\d_x u_1)^2\;,\quad
\d_t u_3 = \d_x^2 u_3 + (\d_x u_2)^2\;,
\end{equ}
with $u(0) = 0$ say.
In this case, $V \propto (0,1,0)$ and $V_\star \propto (0,0,1)$, so that 
Theorem~\ref{theo:gKPZ} (when combined with Lemma~\ref{lem:oscillating} below) shows that 
the law of $u$ has full support, while we would have 
$u_2(t) \ge C_2 t$, $u_3(t) \ge C_3 t$ if we placed some constraints on the possible
values of $K_\star$ and $K$.
\end{remark}


\subsection{Outline}

All equations in our setting can be lifted to abstract fixed point problems \cite[Thm.~7.8]{Hairer2014} in a problem dependent regularity structure $\CT$. Exploiting the continuity of the solution map (mapping the space of admissible models $\CM_0$, see Section~\ref{sec:trees}, continuously into some solution space $\CX$), we can redirect our focus towards showing Theorem~\ref{thm:main}, which gives a characterisation of the topological support of random models in complete analogy with Theorem~\ref{thm:mainThm}.   
We are interested in random models $\hat Z$ obtained as the limit of a sequence of smooth random models $\hat Z = \lim_{\eps \to 0} \hat Z^\eps$.  The upper bound for the support of $\hat Z$ then follows from elementary probability theory arguments. 
The basic idea to show the lower bound is to fix a deterministic model $Z$ (for which we want to show $Z \in \supp \hat Z$) and to construct a sequence of ``shifts'' $\xi+\zeta_\delta$ of the underlying Gaussian noise $\xi$ by a smooth random function $\zeta_\delta = \zeta_\delta(\xi)$ such that the ``shifted model'' $\hat Z_\delta$, formally given by $\hat Z_\delta(\xi) = \hat Z(\xi + \zeta_\delta)$, converges to $Z$ almost surely as $\delta \to 0$. Since $\supp \hat Z_\delta \ssq \supp \hat Z$ for any $\delta>0$ 
(this is not completely obvious since $\zeta_\delta$ is not adapted in general so Girsanov's theorem need not apply, but see Lemma~\ref{lem:supporttranslation} for a proof) and $\supp \hat Z$ is closed, this shows that $Z \in \supp \hat Z$. While this is the broad strategy already used in \cite{BMS,Friz2016,Tsatsoulis2018}, the identification
of a suitable shift $\zeta_\delta$ is significantly more involved in this case.

We want to consider random shifts for reasons outlined in detail below (most crucially, our shifted noises are still of the type considered \cite{ChandraHairer2016}). It is then not even clear a priori what we mean by ``shifted model'', since the law of $\xi + \zeta_\delta(\xi)$ is not necessarily absolutely continuous with respect to the law of $\xi$, so that simply evaluating the random limit model $\hat Z$ at $\xi + \zeta_\delta(\xi)$ is in general not well-defined. 
Instead we rely on a purely analytic shift operator $T_f$ (Theorem~\ref{thm:translationoperator}, see also \cite[Thm.~3.1]{HairerMattingly2016}), acting continuously on the space of admissible models and satisfying $\hat Z(\xi + f) = T_f \hat Z(\xi)$ for \emph{deterministic}, smooth, compactly supported functions $f$
(in which case $\hat Z(\xi + f)$ is well-defined by the Cameron-Martin theorem), and we call $\hat Z_\delta(\xi):=T_{\zeta_\delta(\xi)} \hat Z(\xi)$ the shifted model.
From the deterministic continuity of the shift operator we infer in particular that any shift maps the support of $\hat Z$ into itself (this also works for random shifts, see Lemma~\ref{lem:supporttranslation}), so that we are left to find the set of models $Z$ for which a shift as above can be constructed.

For the type of statement we are looking for, it suffices to consider models $Z$ of the form $Z = \CR^h \Zcan(f)$ for some tuple of smooth functions $f = (f_i)_{i \le m}$, where  $f_i \in \CC_c^\infty(\R \times \T^e)$ for any $i \le m$, and some character $h$ in the renormalisation group $\CGm$. (See Section~\ref{sec:reg:structures} for the notation used here; $f \mapsto \Zcan(f)$ denotes the canonical lift, $\CR:\CGm \times \CMz \to \CMz$ denotes the action of the renormalisation group onto the set of admissible models.) In fact, since the shift operator commutes with the action of the renormalisation group (Theorem~\ref{thm:translationoperator}), it suffices to consider $f=0$ in the sense that we aim to find a set $\suppH \ssq \CGm$ which is as large as possible such that for any $h \in \suppH$ one can find
a sequence of smooth random shifts $\zeta_\delta$ such that
\begin{equ}[e:wantedConvergence]
\lim_{\delta \to 0} T_{\zeta_\delta} \hat Z(\xi) = \CR^h \Zcan(0)\;,
\end{equ}
where the limit is taken in the sense of convergence in probability in the space of models. Actually, since our proof draws on the results of \cite{ChandraHairer2016}, we will automatically have convergence in $L^p$ for any $p\ge1$.

Since the limit we aim for as $\delta\to 0$ is deterministic, we are left to choose $\zeta_\delta$ 
in such a way that the variance of the models goes to zero, while the expected value has the correct 
behaviour in the limit. The first point is ensured if $\xi+\zeta_\delta \to 0$ in a strong enough sense,
which will be formalised in Definition~\ref{def:smooth:noise}. Note that the space of noises introduced
there is a subset of the one used in \cite{ChandraHairer2016}, and our distance (\ref{eq:noise:norm}) is 
stronger, see Lemma~\ref{lem:bound:eta}. Our noises always live in a fixed inhomogeneous Wiener chaos with 
respect to some fixed Gaussian noise, which in particular allows us to work with a \emph{linear} space 
of noises and our distance is an actual norm on this space. 
The main issue is then to obtain (\ref{eq:shiftedtreeexpectationA}), namely to ``control'' the expected value $\hat\Upsilon^{\delta}\tau := \E T_{\zeta_\delta} \hat {\PPi}^{\xi}\tau(0)$ of the finite number of trees $\TT_-$ of negative homogeneity, so that in the limit $\delta \to 0$ they equal $h(\tau)$. Here $\hat {\PPi}^{\xi}$ denotes the renormalised canonical lift of $\xi$ and $T_{\zeta_\delta}$ is as above the shift operator acting on the space of admissible models. 

These two properties are obviously necessary for the convergence 
\eqref{e:wantedConvergence} in $L^2$ in the space of models. To see this, note that if we write $\PPi^g$ 
for the model $\CR^g \Zcan(0)$, then one has $\PPi^g \tau (0) = g(\tau)$ for any $\tau \in \TT_-$. 
With a bit more effort (Proposition~\ref{prop:sequencetoconstant}) it is possible to see that they are also sufficient.
At this stage there are two main problems left to be solved, which we address respectively in Sections~\ref{sec:constraints} and~\ref{sec:shift}.
\begin{enumerate}
\item What is the set $\suppH$ of characters $h$ such that we can find a shift $\zeta_\delta$ as above? In particular, we have to show that this set is large enough to ``almost'' contain the BPHZ character $g^\eps$ (up to an $o(1)$ tweaking, see the remark below Theorem~\ref{thm:mainThm} or the second statement of Theorem~\ref{thm:main}).
\item Given $h \in \suppH$, how does one construct a shift $\zeta_\delta$ such that 
$\xi+\zeta_\delta \to 0$ in some suitable space of admissible noises $\SMz$ 
(see Definition~\ref{def:smooth:noise} below) and
such that $\lim_{\delta \to 0} \E T_{\zeta_\delta} \hat {\PPi}^{\xi}\tau(0) = 
h(\tau)$ for every $\tau \in \TT_-$? 
\end{enumerate}
Let us first discuss the second question, since our solution to this problem motivates the choice of $\suppH$. It is natural to make the ansatz $\zeta_\delta = -\xi^{\delta} + k^\delta$, see Section~\ref{sec:ext:reg:str}, where $\xi^{\delta}$ is a smooth approximation of $\xi$ at scale $\delta$ and $k^\delta$ is a random, centred, stationary, and smooth function living only on high frequencies, or equivalently on small scales (think of scales much smaller than $\delta$). The last property will ensure weak convergence of $k^\delta$ to $0$ as $\delta \to 0$. 
If we simply chose $k^\delta=0$, then the quantity $\hat\Upsilon^\delta\tau$ of some fixed tree $\tau \in \TT_-$ would in general blow up, as shown in the following example.

\begin{example}
Consider the ``cherry'' $\tau=\cherry$ appearing in the regularity structure associated to the $\Phi^4_3$ 
equation. Setting $\zeta_\delta = -\xi^{\delta}$ (so $k^\delta = 0$)
and using the fact that by definition of the BPHZ character one has 
$\E \hat {\PPi} ^\xi \cherry (0)  = 0$, one has
\begin{align}\label{eq:ex:cerry}
\hat\Upsilon^\delta \cherry = -2\, \introCherryA + \introCherryB \simeq -\delta^{-1}.
\end{align}
Here we use Feynman diagrams on the right-hand side to encode real constants in the
same way as for example in \cite{Hairer2017} or \cite[Sec.~10.5]{Hairer2014}. 
Straight lines represent the heat kernel, 
dotted lines represent the $\delta_0$-distribution, and wavy lines represent an 
approximation to $\delta_0$ at scale $\delta>0$.

To see how a ``high frequency perturbation'' can solve this issue, consider adding a term of the form $k^\delta = a^{\delta} \xi^{\lambda}$ with $\lambda=\lambda_\delta \ll \delta$ and $a^\delta \in \R$. Similar to (\ref{eq:ex:cerry}) one obtains
\begin{align*}
\hat\Upsilon^\delta \cherry 
&= -2\, \introCherryA + \introCherryB + 
	2 a^\delta \introCherryC -
	2 a^\delta \introCherryD +
	(a^\delta)^2 \introCherryE
\\
&\simeq
	-2\delta^{-1} + \delta^{-1} + 2 a^\delta \lambda^{-1} - 2 a^\delta \delta^{-1} + (a^\delta)^2 \lambda^{-1}.
\end{align*}
Fix now a number $h(\cherry) \in \R$. 
Then provided $\lambda \ll \delta$ one can find $a^\delta$ such that $\hat\Upsilon^\delta \cherry = h(\cherry)$. 
To see this, observe that in the regime
$\lambda \ll \delta$ and $a^\delta \ll 1$ one has 
$\delta^{-1} \ll \lambda^{-1}$ and $(a^\delta)^2 \ll a^\delta$, so that the third term above dominates all other terms, and one can solve the fixed point problem
\[
a^\delta 
=
\half
\big ( \introCherryC \big)^{-1}
\big(
	h(\cherry)
	+ 2\introCherryA - \introCherryB
	+ 2 a^\delta \introCherryD 
	-(a^\delta)^2 \introCherryE
\big).
\]
\end{example}
\begin{remark}
In the above example the term that ended up dominating the quantity $\hat\Upsilon^\delta \cherry$ was the tree in which exactly one white noise was replaced by the highly oscillating perturbation $k^\delta$, while all other noises remained white. We will tailor our shift so that the trees with this property will always represent the dominating part, see Sections~\ref{sec:shift:WC} and~\ref{sec:fpa}, in particular Lemma~\ref{lem:dominatingpart} and Lemma~\ref{lem:Fthenondominatingpart}. 
\end{remark}

This strategy is complicated by two hurdles. Firstly, one has to control various trees simultaneously, and it is a priori not clear that a perturbation designed to control one tree does not destroy the desired expected value of another. Indeed, it is not hard to see that with our strategy we are in general not able to control all trees $\tau \in \TT_-$ at the same time to arbitrary values $h(\tau)$, but we have to respect certain linear constraints between them. See Examples~\ref{ex:const:first}--\ref{ex:const:last} for examples of such linear constrains in the context of various interesting SPDEs. (It is a crucial insight that these constraints are ``almost'' satisfied by the BPHZ character, see the outline below and Assumption~\ref{ass:CHBPHZcharacters}.)

The second problem comes from the fact that we also have to bound the expected values of trees with more than two leaves. If one tries to use high frequency perturbations which are \emph{Gaussian}, then in general trees with one white noise replaced by such a perturbation would not dominate the expression $\hat\Upsilon^\delta\tau$. There are even trees for which these expressions vanish identically for any Gaussian shift $\srn_\delta$. An example is the tree $\quadcherry$ from the $\Phi^5_2$ equation, for which we obtain (in case of a \emph{Gaussian} shift $\srn_\delta$)
\begin{align}\label{eq:Phi43vanish}
\hat\Upsilon^\delta \,\quadcherry
=
\enUpsilon^\delta \left( 4\,\quadcherryA + \quadcherryB \right).
\end{align}
Here, red nodes are new noise types and should be thought of as  placeholders for the shift $\srn_\delta$.
 Formally, the trees on the right-hand side of \eqref{eq:Phi43vanish}, which we call ``shifted trees'',
 are elements of an enlarged regularity structure $\sCT$, see Section~\ref{sec:ext:reg:str}. 
The renormalisation group $\CGm$ acts naturally on 
$\sCT$ by only considering contractions of original trees. 
In this way one can build for any $\eps,\delta>0$ a ``renormalised'' model $\hat{\PPi}^{\xi_\eps,\zeta_\delta}_{\mathrm{en}}$, which converges in the limit $\eps \to 0$ to a model $\hat{\PPi}^{\xi,\zeta_\delta}_{\mathrm{en}}$, and we introduce the notation $\enUpsilon^\delta \tau := \E\hat{\PPi}^{\xi, \zeta_\delta}_{\mathrm{en}} \tau (0)$. (Note that $\hat{\PPi}^{\xi, \zeta_\delta}_{\mathrm{en}}$ is very different from the BPHZ renormalisation $\hat{\PPi}^{\xi_\eps,\zeta_\delta}$ on the large regularity structure, in which case these quantities would vanish by definition of the BPHZ character.)
We will define just after \eqref{e:defShiftOperator} below a shift operator $\SS: \CT \to \sCT$, formally given by replacing blue nodes with red nodes in all possible ways, and we will show in Lemma~\ref{lem:shiftoperator} that $\hat\Upsilon^\delta = \enUpsilon^\delta \SS$.

In the above example $\enUpsilon^\delta$ vanishes on any ``shifted'' tree which does not appear on the right-hand side of (\ref{eq:Phi43vanish}). To clarify why, let us write $\eUpsilon^{\eps,\delta} \tau := \E \PPi^{\xi_\eps,\srn_\delta} \tau (0)$, where $\PPi^{\xi_\eps,\srn_\delta}$ denotes the canonical lift of 
$(\xi_\eps,\srn_\delta)$ (think of $\eps \ll \delta$) to a model in the enlarged regularity structure. Using 
Eq.~\ref{lem:shift:operator:co:product} below one shows that 
\begin{align}\label{eq:ex:quadcherry}
\enUpsilon^\delta \quadcherryC = \lim_{\eps \to 0} \bigl(\eUpsilon^{\eps,\delta} \quadcherryC - 3 \eUpsilon^{\eps,\delta} \quadcherryD \,\eUpsilon^{\eps,\delta} \quadcherryE\bigr)
=
0.
\end{align}
The second identity in (\ref{eq:ex:quadcherry}) only holds if $\srn_\delta$ is Gaussian in general. This can be seen by using Wick's rule of calculating the expected values of all trees involved, which shows that it
identically vanishes for any fixed $\eps > 0$.
Note also that the renormalisation constant of this tree vanishes identically, i.e.\ one has $g^\eta(\quadcherry)=0$ for any smooth Gaussian noises $\eta$ and the BPHZ character $g^\eta$, but the expectation after shifting the noise \emph{does not} vanish and with the choice $\zeta_\delta=-\xi^\delta$ would blow up as $\delta\to 0$.

One could now try to use shifted trees with more than one shifted noise to dominate the expression, which however leads to two issues which seem difficult to resolve. First, in general it would now be \emph{subtrees} of $\tau$ that dominate the behaviour of the shifted tree (in the example above, it would be $\cherry$), and one may see constraints between these trees. Contrary to the constraints we end up with, such constraints (between trees of different homogeneity with different number of leaves) are not seen at the level of the BPHZ characters. Second, while the equation we needed to solve above for $\cherry$ was a perturbation of a \emph{linear} equation, we would now have to solve a polynomial equation, which introduces non-linear constraints (for example $(a^\delta)^2$ is always positive) and it is not clear if these polynomial equation can be solved (to worsen the matter, recall that we need to control various trees simultaneously, so that we end up with a \emph{system} of polynomial equations).

We opt for a different way. We introduce a shift $k^\delta$ such that trees with one noise replaced by a shifted noise gives a non-vanishing contribution. We ensure this by choosing $k^\delta$ such that the cumulant of $(k^\delta,\xi,\ldots,\xi)$, with $m(\tau):=\#L(\tau)-1$ instances of white noise $\xi$, does not vanish.
(Here $L(\tau)$ denotes the number of ``leaves'' of $\tau$.) The easiest way to guarantee this is to choose  $k^\delta$ in the $m(\tau)$-th homogeneous Wiener chaos with respect to $\xi$.

\begin{example}\label{ex:shift:dominating:intro}
Consider the tree $\tau=\treeExampleKPZaa$ from the generalised KPZ equation, where we draw $\leafs$ and $\leafsa$ to distinguish two different (hence independent) noise types. In this case we would choose our shift
\[
k^\delta := a^\delta \sintf_{[\leafslow,\,\leafsalow,\,\leafsalow]}(\triplekernel)\,,
\]
where $\triplekernel \in \bar\CC_c^\infty(\bar \domain \times \bar\domain^3)$ is a suggestive way to write a kernel of the form \[\triplekernel(x;x_1, x_2, x_3) = {\color{darkgreen}K}(x - x_1){\color{blue}K}(x-x_2) {\color{red}K}(x-x_3)\] for some kernels ${\color{darkgreen}K},{\color{blue}K},{\color{red}K} \in \CC_c^\infty(\bar\domain)$, and $\sintf_{[\leafslow,\,\leafsalow,\,\leafsalow]}$ denotes a third order stochastic integral with respect to the joint law of $(\xi_{\leafslow}, \xi_{\leafsalow})$, see \eqref{eq:stochasticintegral_cm}. 
One then has the following graphical representation 
\begin{align}\label{eq:ex:shift:dominating:intro}
\enUpsilon^\delta
\treeExampleKPZaashiftz
=
a^\delta
\treeExampleKPZaacontractionz + a^\delta \treeExampleKPZaacontractionbz \;.
\end{align}
(Here, a dark red node represents an instance of $k^\delta$.)
We would now rescale the kernels ${\color{darkgreen}K},{\color{blue}K},{\color{red}K}$ to a scale $\lambda=\lambda_\delta \ll \delta$ at a homogeneity $\alpha$ which is determined by the homogeneity $|\treeExampleKPZaa|_\fs= -\kappa$ and $m(\treeExampleKPZaa)=3$, see (\ref{eq:alpha}).

(See Section~\ref{sec:notation} for the definition of the domain $\bar\domain$.)
\end{example}

The strategy outlined above is implemented in Section~\ref{sec:shift} as follows. In Section~\ref{sec:ext:reg:str} we will construct an enlarged regularity structure, containing additional noise types \eqref{idx:sFLm}, large enough to be able to represent the regularised noise $\xi^\delta_\Xi$ (for any noise type $\Xi$) and the highly oscillating perturbation $k_{(\Xi,\tau)}^\delta$ (for any tree $\tau$ and noise type $\Xi \in \ft(L(\tau))$ appearing in $\tau$). We will will so construct the \emph{shift operator} $\SS:\CT \to \sCT$ as in \eqref{e:defShiftOperator}
below. We 
determine the set of trees $\oSS[\tau]$ in the image of the shift operator which will dominate the expected value in Definition~\ref{def:oSS}.
In Section~\ref{sec:shift:WC} we construct in (\ref{eq:eta}) a ``highly oscillating perturbation'' 
$\eta_{(\Xi,\tau)}^\delta$ in the $m(\tau)$-th Wiener chaos for any tree $\tau\in\FT_-$ (see below for 
the definition of $\FT_- \ssq \TT_-$) and any $\Xi \in \ft(L(\tau))$. The kernel $K_{(\Xi,\tau)}^\delta$ 
(with respect to Gaussian integration) of this perturbation will be a rescaled version of a fixed 
kernel $\Phi_{(\Xi,\tau)}$, see (\ref{eq:rescalingoperator}), at a homogeneity $\alpha_{(\Xi,\tau)}$, 
see (\ref{eq:alpha}), to a scale $\lambda^\delta_\tau$ (we will discuss shortly the choice of these scales). 
The kernels $\Phi_{(\Xi,\tau)}$ will 
be chosen along the lines of Example~\ref{ex:shift:dominating:intro} above (there is a 
slight subtlety here in case of 
log-divergencies, see Example~\ref{ex:log} below, which we ignore for the sake of this introduction).

A key result is Lemma~\ref{lem:dominatingpart} which determines the behaviour of the ``dominating'' trees $\taua \in \oSS[\tau]$. 
It will be useful to introduce the function $F_\tau(a,\lambda) := \hat\Upsilon^\delta \tau$ for 
$\tau \in \FT_-$,  see (\ref{eq:Ftau}), where $a=a^\delta_\tau$ and $\lambda=\lambda^\delta_\tau$, 
$\tau \in \FT_-$. In Proposition~\ref{prop:choice_of_a} in Section~\ref{sec:fpa} we will then, for 
fixed $\lambda$, recast the equation $F_\tau(a,\lambda) = h(\tau)$ for $\tau \in \FT_-$ into a fixed point 
problem for $a$. This problem will be a small perturbation of a solvable linear problem (\emph{linear} 
because of the definition of $k^\delta$, \emph{solvable} thanks to Lemma~\ref{lem:dominatingpart}, 
\emph{small perturbation} thanks to Lemma~\ref{lem:Fthenondominatingpart}) which is therefore straightforward to 
solve. The tricky issue is that in order for Lemma~\ref{lem:Fthenondominatingpart} to hold one needs 
to choose the scales $\lambda_\tau^\delta$ carefully. In Section~{\ref{sec:monotonestatements}} we 
will determine an order $\le$ on the set of trees $\FT_-$, and we will choose the scales such that
$\lambda_\tau^\delta \ll \lambda_{\taua} ^\delta$ whenever $\tau,\taua \in \FT_-$ with $\tau \le \taua$. 
To formalise this idea, we introduce in Definition~\ref{def:attainable} the notion of an \emph{attainable} 
statement, and we show at the end of Section~\ref{sec:fpa} that the necessary bound of 
Lemma~\ref{lem:Fthenondominatingpart} is attainable in this sense.

We now outline how we will address the first point above, i.e.\ how to define the set $H$, which we will do in Section~\ref{sec:constraints}. Every tree $\tau \in \CT$ can be mapped onto a function $\CK\tau: \R^{(d+1)L(\tau)} \to \R$, see (\ref{eq:CK}). One should think of $\CK\tau$ as the function obtained by anchoring the root to the origin and integrating out all other vertices, except for the leaves.
\begin{example}
In the case of the $\Phi^4_3$ equation, one has for instance 
$$
	\left(\CK\phiCA\right)(x_1, \ldots, x_4) = \int dy K(x_1 ) K(y) K(x_2 - y) K(x_3 - y) K(x_4 - y),
$$
where we identify the set of leaves $L(\phiCA)\simeq\{1,2,3,4\}$ with `$1$' denoting the leaf directly 
attached to the root, and where $K$ denotes a truncation of the heat kernel.
\end{example}
Denote now by $\hat\CK\tau$ the function defined in same way, but with $K$ replaced by the actual (i.e.\ not 
truncated) heat kernel $\hat K$ (we will later write $\CK_{\hat K}\tau$ for this). It is a priori not 
clear that these integrals are well-defined on large scales, but we will show in 
Theorem~\ref{thm:evaluation:trees:largescale} that at least for trees $\tau$ of non-positive homogeneity 
this is always the case. Let us furthermore write $\CK_\sym \tau$ and $\hat\CK_\sym\tau$ for the 
kernels obtained from $\CK \tau$ and $\hat \CK \tau$ by symmetrisation under spatial reflections $(t,x) \mapsto (t,-x)$ and permutation of the variables. (If $\tau$ contains more 
than one noise type, one should only symmetrise variables corresponding to the same noise type.)

From the discussion above, it is clear that we cannot hope to control two trees $\tau,\taua$ independently if $\hat\CK_\sym\tau$ and $\hat\CK_\sym\taua$ are linearly dependent. To make this more clear, consider the following example.
\begin{example}
Continuing Example~\ref{ex:shift:dominating:intro}, one has 
\begin{align}\label{ex:shift:dominating:intro:2}
\enUpsilon^\delta
\treeExampleKPZaashiftz
=
2a^\delta \int dx_1\cdots dx_4 
\left(
	\CK_\sym\treeExampleKPZaa
\right)
( x_1 , \ldots , x_4 )
\triplekernel(x_3;x_1,x_2,x_4)
\;.
\end{align}
where we identify the leaves of $\tau:=\treeExampleKPZaa$ with $\{1,2,3,4\}$ from left to right. Since one should think of $\triplekernel$ as being rescaled to scales $\lambda \ll \delta$, only the small scale behaviour of $\CK_\sym\tau$ matters, which is (essentially) the behaviour of the self-similar kernel $\hat\CK_\sym\tau$. (The last statement is justified by Lemma~\ref{lem:dominatingpart}, where we show that the difference between (\ref{ex:shift:dominating:intro:2}) with $\CK_\sym$ and $\hat\CK_\sym$ vanishes in the limit $\delta\to 0$.) It follows that if $\taua \in\TT_-$ is another tree carrying the same noise types as $\tau$ and such that $\hat\CK_\sym\taua = c \hat\CK_\sym \tau$ for some $c \in \R$, then the shifted trees which are dominating (i.e.\ elements of $\oSS[\tau]$ and $\oSS[\taua]$) satisfy the same linear relation in the limit $\delta\to 0$.
\end{example}

Motivated by this example, we introduce in Definition~\ref{def:CH} an ideal $\CJ\subset \CTm$ 
generated by linear combination of trees $\sigma \in \CT$ carrying the same noise types and such 
that $\hat\CK_\sym \sigma=0$. Here we introduce the notation $\CTm$ for the free, unital, 
commutative algebra generated by $\TT_-$. We recall at this point \cite{BrunedHairerZambotti2016} 
that $\CTm$ is naturally endowed with a Hopf algebra structure with coproduct $\cpmh$ (the character 
group of $\CTm$ is precisely the renormalisation group $\CGm$ already mentioned above), 
see Section~\ref{sec:trees} for details and precise references.

We show in Section~\ref{sec:constraints} that $\CJ$ is a Hopf ideal, see Assumption~\ref{ass:CJHopfIdeal}. 
The crucial implication is that its annihilator $\CH$ is a Lie subgroup of $\CGm$. 
We show further Assumption~\ref{ass:CHBPHZcharacters}, which states that the BPHZ character $g^\eps$ of the 
regularised noise $\xi^\eps$ ``almost'' belongs to this group, in the sense that one has 
$g^\eps \in f^\eps \circ \CH$, for a sequence of characters $f^\eps \in \CGm$ which converges to a 
finite limit $\fxi$ as $\eps \to 0$. 
It is crucial to note that we show this also for a class of
non-Gaussian approximations $\xi^\eps$ which is rich enough to contain the shift $\zeta_\delta$.
Assumption~\ref{ass:CHBPHZcharacters}  finally justifies 
the assertion made above that $g^\eps$ ``almost'' satisfies the linear constraints. (In a perfect world, 
$g^\eps$ would satisfy these constraints precisely. 
The discrepancy stems from the fact that we 
use truncated kernels to define $g^\eps$.)
Moreover, we have identified that the set $H \ssq \CGm$ for which we can construct a shift as above is equal to the coset $\fxi \circ \CH$. It may be useful to observe that while the character $\fxi$ is not uniquely defined, the coset $\fxi \circ \CH$ is unique. 

Section~\ref{sec:constraints} shows that a under a technical Assumption~\ref{ass:CVz} the Assumptions~\ref{ass:CJHopfIdeal} and~\ref{ass:CHBPHZcharacters} always hold. 
The latter two are formulated as assumptions (rather than theorems), 
since there are a range of interesting equations in which Assumption~\ref{ass:CVz} is violated, while one can simply show Assumptions~\ref{ass:CJHopfIdeal} and~\ref{ass:CHBPHZcharacters} by hand. 
(Examples are the $\Phi_2^p$ equations discussed in Section~\ref{sec:Phi4} and the 2D parabolic Anderson model.)
The general proof, assuming Assumption~\ref{ass:CVz} and given in Section~\ref{sec:constraints}, is motivated and outlined at the beginning of this section. 

We are left to link the two constructions outlined above. In Definition~\ref{def:FT} we will define a set $\FT_- \ssq \TT_-$ which is a maximal set with the property that $\linspace{\FT_-}$ and $\CJ$ are linearly independent (in other words, one has $\linspace{\FT_-} \oplus \CJ = \CT_-$ where $\oplus$ denotes the direct sum of vector spaces). 
For any fixed character $h \in \fxi \circ \CH$ we will tailor a shift of the noise in Section~\ref{sec:shift} (see outline above) such that $\hat\Upsilon^\delta \tau \to h(\tau)$ as $\delta \to 0$ for any $\tau \in \FT_-$. Using the fact that $\linspace{\FT_-}$ has a complement in $\linspace{\TT_-}$ which is a subset of the ideal $\CJ$, we will show in Proposition~\ref{prop:sequencetoconstant} that the sequence of shifted models converge to $\CR^h \Zcan(0)$ almost surely, which shows in particular that 
\[
\{ \CR^h \Zcan(f) : h \in \fxi \circ \CH , f \in \CC_c^\infty \} \ssq \supp \hat Z.
\]
Philosophically, Proposition~\ref{prop:sequencetoconstant} fills in the ``gap'' between $\FT_-$ and $\TT_-$, in the sense that we do not need any a priori information how the shifted models behaves on trees $\tau \in \TT_- \backslash \FT_-$. This step relies of course on the relation between the set $\FT_-$ and the ideal $\CJ$, and the fact that we choose $h \in \fxi \circ \CH$, where $\CH$ is the annihilator of $\CJ$. What is less obvious, it also uses crucially the fact that $\CH$ is indeed a sub\emph{group} of $\CGm$ (see Assumption~\ref{ass:CJHopfIdeal}).
By Assumption~\ref{ass:CHBPHZcharacters}, the ``tweaked'' BPHZ character $\fxi \circ (f^\eps)^{-1} \circ g^\eps$ is an element of $\fxi\circ \CH$ for any $\eps>0$, and using that $\fxi \circ (f^\eps)^{-1} \to \one^*$ as $\eps \to 0$ concludes the proof.

\section{Notations and assumptions}\label{sec:assumptions}

\subsection{Conventions on notation} \label{sec:notation}\label{sec:multisets}

For any integer $M\in\N$ we write $[M]:=\{1,\ldots, M\}$ with the convention that $[0] = \emptyset$. We fix a spatial dimension $e \ge 1$ and a space-time domain \label{idx:lambda}$\spacetime:= \R \times \T^{e}$.
 We assume that either all noises $\xi_j$ in \eqref{eq:singular:SPDE} are space-time white noises, or they are all purely spatial white noises. 
In the first case, we define \label{idx:domain}$\domain:=\spacetime$ as the space-time domain with dimension $d:=e+1$, while in the second case we let $\domain := \T^{e}$ be the purely spatial domain with dimension $d:=e$. In either case, we define $\bar\domain := \R^d$, so that $\domain$ can be identified with the factor space of $\bar\domain$ modulo a suitable discrete group of translations. Given a distribution $u$ on $\domain$ we can naturally view $u$ as a distribution on \label{idx:bardomain}$\bar\domain$ by periodic extension.

For any integer $m\in \N$ we write $\CD'(\R^m)$ for the space of distributions and $\CC_c^\infty(\R^m)$ for the space of compactly supported, smooth functions on $\R^m$. For any distribution $u$ and any multiindex  $k\in\N^d$ we denote by $D^k u$ the $k$th distributional derivative of $u$. 
In the sequel, test functions that are compactly supported in the difference of their variables but invariant under simultaneous translations of all their arguments will play an important role. We capture this in the following definition.

\begin{definition}\label{def:barCinfty}
For any finite set $L$ we define the space $\bar\CC^\infty_c(\bar\domain^{L})$ as the set of smooth functions $\phi\in\CC^\infty(\bar\domain^L)$ such that both of the following properties are satisfied.
\begin{enumerate}
\item The function $\phi$ is invariant under simultaneous translation of all variables by any vector $h\in\bar\domain$. In other words, we postulate that one has the identity
\[
\phi((x_u)_{u\in L})=\phi((x_u+h)_{u\in L})
\]
for any $h\in\bar\domain$ and any $x\in \bar\domain^{L}$.
\item There exists $R>0$ such that $\phi((x_u)_{u\in L})=0$ for any $x\in\bar\domain^{L}$  such that for some $u,v\in L$ one has $|x_u-x_v|> R$.
\end{enumerate}
We will consider the usual topology of test-functions on this space.
\end{definition}

\subsubsection*{Scalings}

We write $\fsL$ for the scaling on $\spacetime$ (which we used already in the formulation of our main results, Theorems~\ref{thm:mainThm} and~\ref{thm:main:sv}). Here $\fsL$ is determined by the integration kernels, see Assumption~\ref{ass:kernelhomo}. We will mostly work with the scaling \label{idx:fs}$\fs:[d] \to \N$, defined by restricting $\fsL$ to $\domain$.
We write $|\fs|:=\sum_{i=1}^d \fs(i)$ for the effective dimension. 
For a multi-index $k \in \N^{\{1, \ldots d \}}$ we write $|k|_\fs := \sum_{i =1}^d \fs(i)k_i$, and for $z \in \domain$ we write $|z|_\fs := \sum_{i=1}^d |z_i|^{\frac{1}{ \fs(i) } }$. We use the convention that sums of the form 
\[
\sum_{|k|_\fs\le r} (\cdots)
\]
always run over all multi-indices $k \in \N^d$ with $|k|_\fs \le r$.
Finally, for any $x \in \bar\domain$, $\phi \in \CC^\infty(\bar\domain)$ and $\lambda>0$ we define $\lambda^{-\fs} x \in \bar\domain$ and $\phi^{(\lambda)} \in \CC^\infty(\bar\domain)$ by 
\begin{align}\label{eq:rescaling}
(\lambda^{-\fs}x)_i := \lambda^{-\fs(i)} x_i\,, \;\; i \le d
\qquad\text{ and }\qquad
\phi^{(\lambda)}(x) := \lambda^{-|\fs|} \phi( \lambda^{-\fs} x ).
\end{align}

\subsubsection*{Multisets}

Let $A$ be a finite set. A \emph{multiset} $\cm$ with values in $A$ is an element of $\N^A$ (i.e.\ a map
$A \to \N$ counting the number of occurrences of each element). Given two multisets $\cm , \cn \in \N^A$ we write $(\cm \setminus \cn)_a := (\cm_a-\cn_a) \lor 0$ for any $a \in A$.  We also naturally identify a subset $B\ssq A$ with the multiset $\I_B : A \to \{0,1\}$. Given a function $f: A \to \R$ we write $f(\cm):=\sum_{a \in \cm}f(a):=\sum_{a \in A}\cm(a)f(a)$.
Given any finite set $I$ and a map $\varphi : I \to A$ we write $[I,\varphi]$ for the multiset with values in $A$ given by
\label{idx:multiset[,]}
\begin{align}\label{eq:notation:multiset}
[I,\varphi]_a := \#\{ i \in I : \varphi(i) = a \}
\end{align}
for any $a \in A$.
Given a finite multiset $\cm$, it will be useful to define the index set
\label{idx:multisetd}
\begin{align}\label{eq:notation:multiset:as:set}
\td(\cm):=\{(a,k) : a \in A, 1\le k \le \cm(a) \}\subset A\times \N\;.
\end{align}

It will be useful to consider functions $f$ with the property that their domain is intuitively given by $M^{\cm}$ for some set $M$ and some multiset $\cm$. Given sets $M$ and $N$, we write $f: M^{\cm} \to N$ as
a shorthand for a function $f: M^{\td(\cm)}  \to N$ which is symmetric in the sense that 
$f(x_j) = f(x_{\sigma(j)})$ for every permutation $\sigma$ of $\td(\cm)$  
preserving the ``fibres'' $\{a\} \times \N$ for all $a \in A$.
Note that if $\cm = [I,\varphi]$, then any $f:M^{[I,\varphi]} \to N$ can be identified with a function 
$f_I : M^I \to N$ by choosing any bijection $\psi: I \to \td(\cm)$ with the property that $\psi_1 = \varphi$, 
and setting $f_I((x_i)_{i \in I}) := f ( (x_{\psi^{-1}(a,k)})_{(a,k) \in \td(\cm)})$. The symmetry of $f$ guarantees that $f_I$ is
independent of the choice of bijection $\psi$.
If $M$ and $N$ are subsets of the Euclidean space, we use the notation $\CC^\infty(M^\cm,N)$, etc., with the obvious meaning.

Another way of viewing a multiset $\cm : A \to \N$ is to fix an arbitrary total order $\preceq$ on $A$ and implicitly identify $\cm$ with the tuple \label{idx:multisettilde}$\tilde\cm \in A^{\#\cm}$ defined as the (unique) order preserving map $\tilde\cm:[\#\cm] \to A$ such that $\#\{ i: \tilde\cm_i = a \} = \cm(a)$
for every $a \in A$.

\begin{remark}\label{rmk:multisets}
We now have three equivalent representations of multisets: $\cm : A \to \N$, $\td(\cm) \ssq A \times \N$ and $\tilde\cm : [\#\cm] \to A$. We will mostly working with the first, but depending on the context, it will be helpful to have the notations $\td(\cm)$ and $\tilde\cm$ at hand.
\end{remark}

\subsection{Regularity structures}
\label{sec:reg:structures}

Our driving noises $\xi$ are indexed by a finite sets of noise types $\FL_-$. These noises $\xi_\Xi$, $\Xi \in \FL_-$, should be thought of as independent Gaussian noises whose law is self-similar under rescaling. For simplicity, we will restrict to Gaussian space or space-time white noises (but see Remark~\ref{rem:noises}). The components of our equation are indexed by a finite set of kernel-types $\FL_+$ and to any component $\ft \in \FL_+$ we associate an integration kernel $K_\ft \in \CC_c^\infty( \bar\domain \setminus \{0\})$ satisfying the ``usual'' assumptions, see Section~\ref{sec:kernels}. 
We equip \label{idx:FL}$\FL := \FL_+ \sqcup \FL_-$ with two homogeneity assignments \label{idx:homo}$|\cdot|_\fs:\FL_\star \to \R_\star$ and \label{idx:fancyhomo}$\fancynorm{\cdot}_\fs : \FL_\star \to \R_\star \sqcup \{0\}$ for $\star \in \{+,-\}$, where we think of $\fancynorm{\cdot}_\fs$ as the ``real homogeneity'' of the noises (for instance $|\Xi|_\fs = -\shalf$ for space-time white noise), and we assume that
\[
|\ft|_\fs = \fancynorm{\ft}_\fs \,, \quad \text{ for } \ft \in \FL_+ \,
\qquad
|\Xi|_\fs = \fancynorm{\Xi}_\fs - \kappa \,, \quad \text{ for } \ft \in \FL_+
\]
for some $\kappa>0$ (small enough).

Recall \cite[Def.~5.7]{BrunedHairerZambotti2016} that a rule $R$ is a collection $(R(\ft))_{\ft \in \FL_+}$ that assigns to any kernel-type $\ft \in \FL_+$ a set $R(\ft)$ of multisets with values in $\FL\times \N^{d}$.
In order to lift our problem to the abstract level of regularity structures, we assume that we are given a normal, subcritical (with respect to $|\cdot|_\fs$) and complete (c.f.\ \cite[Def.~5.7, Def.~5.14,  Def.~5.22]{BrunedHairerZambotti2016}) rule $R$ which is ``rich enough'' to treat the system at hand. (Such a rule is not hard to work out by hand in situations which are simple enough. For more involved examples we refer the reader to \cite{BrunedChandraChevyrecHairer2017}.)

In \cite[Def.~5.26]{BrunedHairerZambotti2016} the authors constructed an (extended) regularity structure \label{idx:CTex}$\CTex$ based on the rule $R$. We also write \label{idx:CT}$\CT \ssq \CTex$ for the reduced regularity structure obtained as in \cite[Sec.~6.4]{BrunedHairerZambotti2016}.
(We will actually work with a slightly simplified extended decoration, compare Section~\ref{sec:trees} below.) We extend the homogeneity assignments $|\cdot|_\fs$ and $\fancynorm{\cdot}_\fs$ to homogeneity assignments $|\cdot|_+$  and $\fancynorm{\cdot}_+$ (respectively $|\cdot|_-$ and $\fancynorm{\cdot}_-$) on $\CT^\ex$ in the usual way, taking into account (respectively neglecting) the extended decoration. On the reduced structure $\CT$ we set $|\cdot|_\fs:=|\cdot|_+=|\cdot|_-$ and $\fancynorm{\cdot}_\fs := \fancynorm{\cdot}_+ = \fancynorm{\cdot}_-$.
We also write \label{idx:TT}$\TT^\ex$ and $\TT$ for the set of trees in $\CT^\ex$ and $\CT$, respectively, so that $\CT^\ex$ and $\CT$ are freely generated by $\TT^\ex$ and $\TT$ as linear spaces. 

\subsubsection{Trees and algebras}\label{sec:trees}

Given a rooted tree $T$, we define a total order $\le$ on the vertex set $V(T)$ of $T$ by setting $u \le v$ if and only if $u$ lies on the unique shortest path from $v$ to the root $\rho_T$, and we write edges $e \in E(T)$ as order pairs $e=(e^\uparrow, e^\downarrow)$ with $e^\uparrow \ge e^\downarrow$. If $u \in V(T) \backslash\{\rho_T \}$, then there exists a unique edge $e \in E(T)$ such that $u = e^\uparrow$, and in this case we write $u^\downarrow := e$.  

Basis elements $\tau \in \TT^\ex$ can be written as typed, decorated trees $\tau=(T^{\fn,\fo}_\fe, \ft)$, where $T$ is a rooted tree with vertex set $V(T)$, edge set $E(T)$ and root $\rho_T$, the map $\ft : E(T) \to \FL$ assigns types to edges, and the decorations $\fn,\fe,\fo$ are maps $\fn:N(T) \to \N^d$, $\fe : E(T) \to \N^d$ and $\fo : N(T) \to (-\infty,0]$. We call $\fo$ the \emph{extended decoration}. Here we define the decomposition of the set of edges into \label{idx:L(T)}\label{idx:K(T)}$E(T) = L(T) \sqcup K(T)$ with $e \in L(T)$ (resp. $e \in K(T)$) if and only if $\ft(e) \in \FL_-$ (resp. $\ft(e) \in \FL_+$), and we write $N(T) \ssq V(T)$ 
for the set of $u \in V(T)$ such that there does not exist $e \in L(\tau)$ such that $u=e^\uparrow$.
We will often abuse notation slightly and leave the type map $\ft$ and the root $\rho_\tau$ implicit. 
Recall that it follows from the fact that $R$ is normal (c.f.\ \cite[Def.~5.7]{BrunedHairerZambotti2016}) that elements $u \in V(T)\backslash N(T)$ are leaves of the tree $T$. 

Given a typed, decorated tree $\tau$ as above, $k \in \N^{d}$ and $\ft \in \FL_+$ we write
$
\CJ_\ft^k \tau
$
for the planted, decorated, typed tree obtained from $\tau$ by attaching an edge $e = (\rho(\tau), \rho({\CJ_\ft^k \tau}))$ with type $\ft$ to the root $\rho(\tau)$ and $\fe(e)=k$, and moving the root $\rho({\CJ_\ft^k \tau})$ to the new vertex. 

We frequently use the Hopf algebras \label{idx:CTm}$\CTm$ and \label{idx:CTmex}$\CTmex$ associated to negative renormalisation \cite[Eq.~5.23, Sec.~6.4]{BrunedHairerZambotti2016}. The character group \label{idx:CGm}$\CGm$ of $\CTm$ is called \emph{renormalisation group}, and we write $\circ$ for the group product. We denote by \label{idx:TTmex}$\TTm$ the set of trees of $\tau \in \TT$ with $|\tau|_\fs < 0$ and such that $\tau$ is not planted, so that $\CTm$ is freely generated as a unital, commutative algebra from $\TTm$. 
We will also frequently use the algebras \label{idx:CTmhat}$\CTmhat$ and \label{idx:CTmhatex}$\CTmhatex$  \cite[Def.~5.26]{BrunedHairerZambotti2016} which are freely generated as a unital, commutative algebra by $\TT$ and $\TTex$, respectively.

Recall \cite[Prop.~5.35, Cor.~6.37]{BrunedHairerZambotti2016} that the algebras $\CTm$ and $\CTmex$ endowed with the coproduct \label{idx:coproduct}$\cpmh$ are Hopf algebras, and $\CTmhatex$ with the coaction $\cpm : \CTmhatex \to \CTmex \otimes \CTmhatex$ is a comodule. Finally, we write $\antipode : \CT_-^\ex \to \hat\CT_-^\ex$ for the twisted antipode \cite[Prop.~6.6]{BrunedHairerZambotti2016}.


With this notation, we make the following assumption, which guarantees that the analytic BPHZ theorem of \cite{ChandraHairer2016} can be applied.

\begin{assumption}\label{ass:main:reg}\label{ass:first}
For any tree $\tau \in \TT$ with $K(\tau)\ne\emptyset$ one has 
\begin{align}
|\tau|_\fs > 
	\big( -\shalf \big) 
	\lor \max_{u \in L(\tau)} |\ft(u)|_\fs
	\lor \big( - |\fs| - \min_{\Xi \in \FL_-} |\Xi|_\fs \big).
\end{align}
We also impose that for any $\tau \in \TT$ and any $e \in K(\tau)$ one has $|\ft(e)|_\fs - |\fe(e)|_\fs > 0$.
\end{assumption}

We also make the simplifying assumption on the rule that we do not allow products or derivatives of noises to appear on the right-hand side of the equation. As was already remarked in \cite{ChandraHairer2016} and \cite{BrunedChandraChevyrecHairer2017}, such an assumption does not seem to be crucial but simplifies certain arguments.

\begin{assumption}\label{ass:noises:derivatives:products}
We assume that for any $\ft \in \FL$ and any $N \in R(\ft)$ there exists at most one pair $(\Xi,k) \in \FL_- \times \N^{d+1}$ such that $N_{(\Xi,k)} \ne 0$, and in this case $k=0$ and $N_{(\Xi,0)}=1$. 
\end{assumption}

\subsubsection{Kernels and models}
\label{sec:kernels}

We assume that for any $\ft \in \FL_+$ we are given a Green's function $\greens_\ft \in \CC^\infty(\bar\domain \backslash \{0\})$, and we make the following assumption.

\begin{assumption} \label{ass:kernelhomo}
We assume that for any kernel-type $\ft \in \FL_+$ the kernel $P_\ft$ is invariant under rescaling in the sense that 
\[
 \lambda^{-|\fsL| + |\ft|_\fs} \greens_\ft(\lambda^{-\fsL} \cdot) = \greens_\ft.
\]
for any $\lambda>0$. Furthermore, in case that the $\xi_\Xi$'s are purely spatial white noises,  we assume that $|\fs|-|\ft|_\fs>\fs_0$.
\end{assumption}

The last property ensures that in case of purely spatial white noise the time integral 
${\hat K} _\ft(x):= \int_{-\infty}^\infty \greens_\ft(t,x) dt$ is well-defined and  self-similar under scaling $\lambda^{-|\fs| + |\ft|} {\hat K}_\ft( \lambda^\fs \cdot) = \hat K_\ft$ for any $\lambda>0$.
To avoid case distinctions, we set \label{idx:hatK}$\hat K:= \greens$ in case of space-time white noise.

It follows from Assumption~\ref{ass:kernelhomo} that ${\hat K}_\ft$ can be decomposed into ${\hat K}_\ft = K_\ft + R_\ft$ with $R_\ft \in \CC^\infty(\bar\domain)$ 
and such that $K_\ft \in \CC_c^\infty(\bar\domain\backslash\{0\})$ (smooth functions with bounded support) satisfies \cite[Ass.~5.1, Ass.~5.4]{Hairer2014}. It will be convenient in Section~\ref{sec:rigidities} to assume that $K_\ft = \hat K_\ft \phi$, where $\phi \in \CC_c^\infty(\bar\domain)$ is symmetric under $x_i \to -x_i$ for any $i\le d$ and equal to $1$ in a neighbourhood of the origin. 
Given the kernel assignment $(K_\ft)_{\ft \in \FL_+}$ we recall the definition of admissible models \cite[Def.~2.7,~Def.~8.29]{Hairer2014}. 
We call a model $Z=(\Pi,\Gamma)$ \emph{smooth} if $\Pi_x \tau \in \CC^\infty(\domain)$ for any $\tau \in \TT^\ex$ and some (and therefore any) $x \in \domain$, and we call $Z$ \emph{reduced} if $\Pi_x \tau$ does not depend on the extended decoration of $\tau$. 

Given an admissible \cite[Def.~6.8]{BrunedHairerZambotti2016} and reduced linear map 
$\PPi : \CT^\ex \to \CC^\infty(\domain)$ we write \label{idx:CZ}$\CZ(\PPi)$ for the model constructed as in \cite[Eqs~6.11, 6.12]{BrunedHairerZambotti2016}, whenever this is well-defined, and we write \label{idx:CMs}$\CM_\infty$ for the set of smooth, reduced, admissible models for $\CT^\ex$. We write \label{idx:CMz}$\CM_0$ for the closure of $\CM_\infty$ in the space of models. We write \label{idx:Omegainfty}$\Omega_\infty:= \Omega_\infty(\FL_-):= \CC^\infty(\domain)^{\FL_-}$ and, given $f \in \Omega_\infty$, we write \label{idx:canlift}$\Zcan(f)=Z^f= \CZ(\PPi^f)$ for the canonical lift of $f$ to a model $Z^f \in \CM_\infty$, c.f.\ \cite[Rem.~6.12]{BrunedHairerZambotti2016}.

%

\subsubsection{Renormalised models}
\label{sec:BPHZ:theorem}
Recall \cite[Eq.~6.23]{BrunedHairerZambotti2016} that for a smooth noise $\eta$ (which we assume to be stationary and centred, with all its derivatives having moments of all orders) we can define a character \label{idx:Upsilon}$\Upsilon^\eta$ on $\hat\CT_-^\ex$ by setting $\Upsilon^\eta := \E(\PPi^\eta \tau)(0)$ for any tree $\tau\in\hat\CT_-^\ex$, and extending this linearly and multiplicatively, where $\PPi^\eta$ denotes the canonical lift of $\eta$ to an admissible random model.
The BPHZ character $g^\eta \in \CG_- \subset (\CT_-^\ex)^*$ is then given by
\label{idx:BPHZcharacter}
\begin{align}\label{eq:BPHZ:character}
g^\eta  := \Upsilon^\eta \antipode \;,
\end{align}
with $\antipode \colon \CT_-^\ex \to \hat\CT_-^\ex$ denoting the ``twisted antipode'' as given in \cite[Eq.~6.8]{BrunedHairerZambotti2016}.
A character $g \in \CG_-$ defines a renormalisation map $M^g:\CTex \to \CTex$ by
\label{idx:Mg}
\[
M^g := (g \otimes \Id) \cpm, 
\]
and we recall that the BPHZ renormalised model $\hat Z^\eta = \CZ(\hat {\PPi}^\eta)$ for a smooth noise $\eta$ is given by \cite[Thm.~6.17]{BrunedHairerZambotti2016}
\begin{align}\label{eq:lift:BPHZ}
\hat{\PPi}^\eta \tau := \PPi^\eta M^{g^\eta}\tau
\end{align}
for any $\tau \in \CT^\ex$. 
Finally, note that one has a continuous action \label{idx:CRg}$g \mapsto \CR^g$ of the renormalisation group $\CG_-$ onto the space $\CM_0$ of admissible models, given by 
\begin{align}\label{eq:CR:M}
\CR^g \CZ(\PPi) := \CZ( \PPi M^g).
\end{align}
(The fact that $\CM_0$ is stable under this action is not obvious but was shown in \cite[Thm.~6.15]{BrunedHairerZambotti2016}.) 

\begin{remark}\label{rem:renormgroup}
We will work with the convention that the renormalisation group product on $\CGm$ is given by 
\[
g \circ h := (g \otimes h) \cpmh.
\]
With this convention one obtains $M^{h \circ g} = M^g M^h$ for any $g,h \in \CGm$, which follows from a quick computation
\begin{multline*}
M^{h \circ g} = \left( (h \circ g) \otimes \Id \right) \cpmh
=
(h \otimes g \otimes \Id) (\cpmh \otimes \Id) \cpmh
=
\\
(h \otimes g \otimes \Id) ( \Id \otimes \cpmh ) \cpmh
=
(h \otimes M^g) \cpmh
=
M^g M^h,
\end{multline*}
so that the group $\mathfrak{R}$ of ``matrices'' $M^g$ acting on $\CT$ is naturally 
identified with the opposite group $\CG_-^{\mathrm{op}}$.
Note however that the action $\CR$ of $\CG_-$ onto the space of models 
satisfies $\CR^{g \circ h} = \CR^g \CR^h$ for any $g,h \in \CGm$.
\end{remark}


A central role will be played by the following ``shift operator''.
\begin{theorem}\label{thm:translationoperator}
	For any $h\in\Omega_\infty$ there exists a continuous operator $T_h:\CM_0\to\CM_0$ with the property that for any $f\in\Omega_\infty$ and any $g\in\CG_-$ the canonical lift $\Zcan(f)$ of $f$ satisfies 
	\begin{equ}[e:propsTh]
	\renorm{g}  \Zcan(f+h)= T_{h} \renorm g \Zcan(f) = \renorm g T_h \Zcan(f).
	\end{equ}
	
	Moreover, this operator is continuous as a map $\Omega_\infty\times\CM_0\to\CM_0:(h,Z)\mapsto T_h Z$ where we endow the space $\Omega_\infty\times\CM_0$ with the product topology. 
	We call $T_h$ the \emph{shift operator}.
\end{theorem}
\begin{proof}
The construction of $T_h$ and the verification of \eqref{e:propsTh} as well as its continuity
are obtained very similarly to the verification of Assumption~10 in the proof
of \cite[Thm.~5.1]{HairerMattingly2016}, so we only give a sketch of the proof.

Consider first an enlarged set of types $\gFL$ such that $\gFLp = \FL_+$,
but $\gFLm = \{\Xi, \gXi\,:\, \Xi \in \FL\}$. In other words, every original noise-type 
$\Xi$ comes with a new ``shifted'' noise type $\gXi$.
We then define a regularity structure $\gCT$ in the same way as $\CT$, but from an
enlarged rule $\gR$ obtained from $R$ by allowing to replace any number of 
noises by their corresponding ``shifted'' noises.
We also write $\gCM$ for the space of admissible models analogous to $\CM_0$, but for $\gCT$.
Finally, we fix $a > 0$ sufficiently large so that, setting $\deg \gXi = a$ for 
every shifted noise (the degrees of the original noises and kernels remain unchanged), 
one has $\gCTm = \CT_-$, i.e.\ all newly added basis vectors
of $\gCT$ have strictly positive degree. 

As a consequence of this last condition, it is straightforward to show by repeated
invocations of \cite[Prop.~3.31]{Hairer2014} and \cite[Thm.~5.14]{Hairer2014} that, given any
model $\CZ(\PPi) \in \CM_0$ and any $h \in \Omega_\infty$, 
there exists a unique model $\CY(h,\PPi) \in \gCM$ with the property that
\begin{equ}
\CY(h,\PPi)\tau = \PPi \tau \;,\quad \forall \tau \in \gCTm\;,\qquad
\CY(h,\PPi)\gXi = h_\Xi \;,\quad \forall \Xi \in \FL_-\;.
\end{equ}
Exactly as in \cite[Eq.~5.9]{HairerMattingly2016}, one can then construct a continuous
map $\CZ \colon \gCM \to \CM_0$ which has the effect of ``adding the function represented
by $\gXi$ to the distribution represented by $\Xi$''. Setting $T_h Z = \CZ(\CY(h,Z))$, the
claim now follows from \cite[Eqns~5.10 \&\ 5.11]{HairerMattingly2016}, combined with the
continuity of both $\CY$ and $\CZ$.
\end{proof}

\subsection{Driving noises}\label{sec:noises}

For simplicity we restrict to the case that our noises $(\xi_\Xi)_{\Xi \in \FL_-}$ are independent Gaussian white noises on $\domain$, so that one has
\[
\E[\xi_\Xi (\vphi) \xi_{\Xi'}(\vphi')] =  \langle\vphi, \vphi' \rangle_{L^2(\domain)} \I_{\Xi = \Xi'}\;,
\]
and we set $\fancynorm{\Xi}_\fs:= -\shalf$. We fix a smooth and compactly support function \label{idx:rho}$\rho\in\CC_c^\infty(\bar\domain)$ such that $\int\rho(x)dx=1$, and, recalling for any $\eps>0$ the notation $\rho^{(\eps)}$ from \eqref{eq:rescaling}, we define the random smooth noise \label{idx:noisereg}$\xi^\eps$ by setting 
\[
\xi_\ft^\eps:=\xi_\ft \star\rho^{(\eps)}
\]
for any $\ft\in\FL_-$. .

\begin{remark}\label{rem:noises}
We do this in order to not complicate the presentation unnecessarily. In principle the proof we give in this paper will hold (modulo some minor modifications) in the case that $\xi$ is a family of independent, stationary, centred Gaussian noises with ``self-similar'' covariance structure and the property that all smooth, compactly supported functions are included in the Cameron-Martin space. One can often relate these situations back to our setting by introducing a new kernel type, see for instance Section~\ref{sec:Phi44} where this is made precise for the $\Phi^4_{4-\eps}$ equation.
\end{remark}

It is well known that $\xi$ admits a version which is a random element of 
\label{idx:Omega}
\[
\Omega:=\bigoplus_{\Xi\in\FL_-} \CC_\fs^{|\Xi|_\fs}(\domain)\;.
\]
We denote the law of $\xi_\Xi$ on $\CC_\fs^{|\Xi|_\fs}(\domain)$ by $\Pwn$, and we write $\P:= \bigotimes_{\Xi \in \FL_-} \Pwn$ for the law of $\xi$ on $\Omega$. Since only the law of $\xi_\Xi$ is relevant in order to establish a support theorem, there is no loss of generality to assume that $\xi:\Omega\to\Omega$ denotes the canonical process.
We write $H:= L^2(\domain)^{\FL_-}\subseteq\Omega$ for the Cameron-Martin space of $\P$ and we recall the following well-known theorem.
 \begin{theorem}\label{thm:cameron-martin}(Cameron-Martin)
 	For any fixed $h\in H$, the laws of $\xi$ and $\xi+h$ under $\P$ are equivalent.
 \end{theorem}

Since smooth noises $\Omega_\infty$ are in general not in the Cameron-Martin space, we define the space of compactly supported smooth noises $\Omega_{\infty,c}:=\bigoplus_{\Xi\in\FL_-}\CC_c^\infty(\domain)$. It will often be convenient to identify functions $h \in \CC_c^\infty(\bar\domain)$ with the element of $\CC_c^\infty(\domain)$ obtained by symmetrisation.
We endow $\CC_c^\infty(\domain)$ with the usual topology (which induces convergence in the sense of test functions), and we define the seminorms
\[
\|f\|_{\alpha,K}:=\sup_{x\in K}|D^\alpha f(x)|
\]
for $K\subseteq\domain$ compact and $\alpha\in\N^d$. 

Recall \cite{Bogachev,Nualart2006} that there is a canonical isomorphism $h \mapsto \sintsimple(h)$ between the Cameron-Martin space
$H$ and a closed subspace $\CH_1$ of $L^2(\Omega,\P)$ with the property that $(
\sintsimple(h))_{h \in H}$ are jointly Gaussian
random variables. This extends to 
isomorphisms $\sintsimple_m$ between the symmetric tensor product $H^{\otimes_s m}$ and subspaces $\CH_m$ of $L^2(\Omega,\P)$ by setting
$\sintsimple_m(h\otimes\ldots\otimes h) = H_m(\sintsimple(h),\|h\|_H)$, where $H_m(x,c)$ denotes the $m$th Hermite polynomial
with parameter $c$. These maps extend to contractions on the full tensor product spaces $H^{\otimes m}$ by setting $\sintsimple_m(h_1 \otimes \ldots \otimes h_m):= \sintsimple_m (h_1 \otimes_s \ldots \otimes_s h_m)$.
We call $\sintsimple_m(h)$ the \emph{iterated integral} of $h\in H^{\otimes m}$ with respect to $\P$, and we write $\sintsimple_m(h)[\xi]$ if we want to emphasise the dependence of $\sintsimple_m(h)$ on the noise $\xi$.

We write $\pi_\Xi: L^2(\domain) \to H$ for the isometry given by $(\pi_\Xi h) _{\tilde\Xi} = h \one_{\tilde\Xi = \Xi}$ for any $h \in L^2(\domain)$. More generally, given $m \ge 1$ and a map $\tilde\cm: [m] \to \FL_-$ we write $\pi_{\tilde\cm} : \bigotimes^m L^2(\domain) \to \bigotimes^m H$ for the isometry which satisfies
\[
\pi_{\tilde\cm} (h_1 \otimes \ldots \otimes h_m) := (\pi_{\tilde\cm_1} h_1) \otimes \ldots \otimes (\pi_{\tilde\cm_m} h_m).
\]
%
We then introduce the notation
\begin{align}\label{eq:stochasticintegral_cm}
\sintsimple_{\tilde\cm}(h):=\sintsimple_m( \pi_{\tilde\cm} h)\;.
\end{align}
for any $h \in \bigotimes^m L^2(\domain)$.
We will mostly need a stochastic integral whose output is a smooth \emph{stationary function} on $\domain$ rather than just a number, and we define
\label{idx:sint}
\begin{align}
\sintf_{\tilde\cm} (h)(z) := \sintsimple_{\tilde\cm} \big(h(z- (\cdot)_1, \ldots, z- (\cdot)_m) \big).
\end{align}

Finally, recall from Remark~\ref{rmk:multisets} that given a total order $\preceq$ on $\FL_-$ (which we assume to be fixed once and for all) we obtain a map $\tilde\cm : [\#\cm] \to \FL_-$ for any multiset $\cm$. We then abuse notation slightly and write $\sintf_\cm := \sintf_{\tilde\cm}$.

\subsection{Non-Gaussian noises}\label{sec:non-gaussian-noises}

In this section, let $\sFLm$ be a finite set of noise types such that $\FLm \ssq \sFLm$. 
A possible choice is of course $\sFLm=\FLm$, but we do not require this here. 
One should rather think of $\sFLm$ as an enlarged set of noise types, see Section~\ref{sec:shift}.
The noises $(\eta_\snt)_{\snt \in \sFLm}$ which we will consider always take values in a fixed inhomogeneous Wiener chaos with respect to the (fixed) family of independent Gaussian white noises $\xi = (\xi_\Xi)_{\Xi \in \FLm}$. 
For technical reasons we restrict ourselves to a class of noises $\eta$ such that the kernels of $\eta_\Xi$ (in the Wiener chaos decomposition) has a relatively simple structure. For this we write $\CCone \ssq \CCinf$ for the space of smooth functions $\vphi \in \CCinf$ which are supported in a neighbourhood of $|\cdot|_\fs$ radius $1$ around the origin.
We also fix an integer $r \in \N$ larger than $\shalf$, and given a homogeneity $\alpha<0$ and a kernel $K \in \CC_c^\infty(\bar\domain\backslash\{0\})$ we write $\|K\|_\alpha \in [0,\infty]$ for the smallest constant such that 
\begin{align}\label{eq:norm:blowup:1}
|D^k K(x)| \le \|K\|_\alpha |x|_\fs^{\alpha - |k|_\fs}
\end{align}
for any $x \in \bar\domain \backslash\{0\}$ and multi-index $k \in \N^d$ with $|k|_\fs <r$, and such that
\begin{align}\label{eq:norm:blowup:2}
\int x^k K(x) dx \le \|K\|_\alpha 
\end{align}
for any $k \in \N^d$ with $|k|_\fs \le \roundup{-\alpha-|\fs| }$.

\begin{definition}\label{def:CYsimp}
For $n \in \N$ let $\CYsimp$ denote the space
\begin{align}\label{eq:CYsimp}
\CYsimp = \bigotimes_{i=0}^n \CCone.
\end{align}
For any $\bar\alpha = (\bar\alpha_i)_{i =0, \ldots, n} \in \R_-^{n+1}$ we define a norm on $\CYsimp$ by 
\begin{align}\label{eq:CYsimp:norm}
\| K_0 \otimes \ldots \otimes K_n \|_{\bar\alpha} := \prod_{i = 0}^n \| K_i \|_{\bar\alpha_i},
\end{align}
Finally, given $\alpha <0$ we define for $n\ge 2$ the norm
\begin{align}\label{eq:CYsimp:norm:2}
\| K \|_{\alpha} := \sup_{\bar\alpha} \| K \|_{\bar\alpha},
\end{align}
where the supremum on the right hand side runs over all $\bar\alpha \in \R_-^{n+1}$ such that $\sum_{i=0}^n \bar\alpha_i = \alpha - |\fs|$, $\bar\alpha_0>-|\fs|-1$ and $\bar\alpha_i>-|\fs|$ for $i = 1 , \ldots , n$.
For $n=1$ we define $\| K_0 \otimes K_1 \|_{\alpha} := \| K_0 \star K_1 \|_{\alpha}$.
\end{definition}

Elements $K \in \CYsimp$ define kernels $\CU K \in \CCinff{n}$ in the following way.

\begin{definition}\label{def:simple:kernel}
We define a linear map $\CU : \CYsimp \to \CCinff{n}$ by setting
\begin{align}\label{eq:simple:kernel}
\CU K (x_1, \ldots , x_n) = \int_{\bar\domain} dy K_0(y) K_1 (x_1 - y) \ldots K_n (x_n - y).
\end{align}
We call kernels of the form $\CU K$ \emph{simple kernels}, 
and we write $\CKsimp^n$\label{idx:CKsimpn} for the linear space generated by simple kernels in $n$ variables. 
\end{definition}

One should think of $\CU K$ as a kernel with respect to stochastic integration, see Definition~\ref{def:smooth:noise} below. 

\begin{remark}
One has an obvious isomorphism between $\CCinff{n}$ and $\barCCinff{n}$ given by identifying $K$ and $(x, x_1, \ldots, x_n) \mapsto K(x-x_1, \ldots, x-x_n)$. It will sometimes be useful to view simple kernels as elements of $\barCCinff{n}$ in this way, which we will do implicitly below.
\end{remark}

\begin{remark}
The ``kernels'' $K_i$ that we have in mind for $i = 1, \ldots , n$ are of the form $K_i = \lambda^{\beta + |\fs|}\phi_i^{(\lambda)}$ for some fixed test function $\phi_i$ and some $\lambda>0$, where $\beta := n^{-1} \alpha$, while $K_0$ will be of the form $K_0 := \phi_0^{(\lambda)}$ for some fixed test function $\phi_0$ integrating to zero. 
One then has $\|K_i\|_{\bar\alpha_i} \simeq \lambda^{\beta-\bar\alpha_i}$ and 
$\|K_0\|_{\bar\alpha_0} \simeq \lambda^{-\bar\alpha_0 -|\fs|}$ uniformly in 
$0 > \bar\alpha_i > -|\fs|$, $0 > \bar\alpha_0 > -|\fs|-1$ and $\lambda>0$, and thus
\begin{align}\label{eq:norm:simple:kernel}
\| K_0 \otimes \ldots \otimes  K_n \|_\alpha \lesssim 1
\end{align}
uniformly in $\lambda>0$. This is the type of kernel we will use when we define the shift of the noise in Section~\ref{sec:shift}. 

But we want the space of noises to be rich enough to  encode not only the shifts, but also an approximation to white noise itself. In this case one cannot choose $K_0$ to integrate to zero, which explains the slightly different definition of the norm $\|\cdot\|_\alpha$ on $\CYsimpp{1}$.
\end{remark}

We fix a homogeneity $\sfs:\sFLm \to \R_-$ with $\sfs \ge -\shalf-\kappa$ for $\kappa>0$ small enough, and we set $\beta^\snt_\cm := \sfs(\snt) - \#\cm \shalf$ for any $\snt \in \sFLm$ and any multiset $\cm$.

\begin{definition}\label{def:CYN}
For $N \in \N$ we denote by $\CYN$ the space of all families $K=(K_\cm^\snt)$ where $\snt \in \sFLm$ and $\cm$ runs over all multisets with values in $\FLm$ such that $\# \cm \le N$, and such that $K_\cm^\snt \in \CYsimpp{\#\cm}$.
On $\CYN$ we define the norm $\CYNnorm{\cdot}$ by setting
\begin{align}\label{eq:CYNnorm}
\CYNnorm{K} := \sum_{\cm,\snt} \| K_\cm^\snt \|_{\beta_\cm^\snt} 
\end{align}
We write $\CYNz$ for the closure of $\CYN$ under this norm.
\end{definition}
\begin{remark}
We will shortly interpret the kernels $K_\cm^\Xi$ as ``stochastic integration kernels'' which define a translation invariant noise in a fixed Wiener chaos, see Definition \ref{def:smooth:noise}. The norm defined in (\ref{eq:CYNnorm}) is then the natural norm to put on elements of $\CYN$. In particular,
\begin{align}\label{eq:varnorm}
\varnorm{K} := 
	\max_{\snt,\cm} \int_{\bar\domain^{\#\cm} } dx \int_{\bar\domain} dy
	\CU K_\cm^\snt(0, x_1, \ldots, x_{\#\cm}) \CU K_\cm^\snt(y, x_1, \ldots, x_{\#\cm}),
\end{align}
which corresponds to \cite[Eq.~A.15]{ChandraHairer2016}, is automatically bounded by $\CYNnorm{K}$. See the proof of Lemma~\ref{lem:bound:eta} for more details.
\end{remark}

\begin{definition}\label{def:smooth:noise}
For $N \in \N$ we denote by $\SMinfN = \SMinfN(\sFLm)$ the space of tuples $\eta=(\eta_\snt)_{\snt \in \sFLm}$ given by
\begin{align}\label{eq:smooth:noise}
\eta_\snt := \FKtN_\snt (K) := \sum_\cm \sintf_\cm(\CU K_\cm^\snt) 
\end{align}
for some $K \in \CYN$ with $\sintf_\cm$ as in \eqref{idx:sint}.
We call any $\eta \in \SMinfN$ a \emph{smooth noise}.
On $\SMinfN$ we define the norm
\begin{align}\label{eq:noise:norm}
\| \eta \|_{\sfs}
:=
\inf_{K} \| K \|_{\CYN},
\end{align}
where the infimum runs over all $K \in \CYN$ such that \eqref{eq:smooth:noise} holds,
and we denote by $\SMzN = \SMzN(\sFLm)$\label{idx:SMz} the closure of the set of simple smooth noises under this norm. 
(The space $\SMzN$ depends on $\sfs$, but we hide this dependence in the notation.)
It will be convenient to write $\SMinf := \bigcup_N \SMinfN$\label{idx:SMinf} and $\SMz := \bigcup_N \SMzN$.
\end{definition}

\begin{remark}
We will see in Lemma~\ref{lem:bound:eta} below that any smooth noise in our setting is a smooth noise in the sense of \cite{ChandraHairer2016}, and the distance $\|\cdot;\cdot\|_{\fc}$ considered there is dominated by $\|\cdot\|_\sfs$ (provided the \emph{cumulant homogeneity} $\fc$ is chosen appropriately, see below). One advantage of the restricted setting introduced here is that the spaces $\SMinf$ and $\SMz$ form \emph{linear spaces} and \eqref{eq:noise:norm} is indeed a \emph{norm} (this is very different from \cite{ChandraHairer2016}, where $\|\cdot;\cdot\|_{\fc}$ is not even a distance in the metric sense.) 
\end{remark}
\begin{remark}
One motivation behind this definition is that cumulants formed by noises of this type are represented by Feynman diagrams so we can use the results of \cite{Hairer2017}. This is of particular importance whenever we need results not covered in \cite{ChandraHairer2016} (for instance bounds on their large scale behaviour or conditions under which one does not see a $\log$-divergence for the renormalisation constant of $0$-order trees).
\end{remark}
In order to apply the results from \cite{ChandraHairer2016} we will have to bound cumulants of orders
higher than two. The assumptions in \cite{ChandraHairer2016} are formulated on objects called \emph{cumulant homogeneities}, see \cite[Def. A.14]{ChandraHairer2016}. We define now such a cumulant homogeneity $\fc$ consistent with $\sfs$. (Later on we will show that the shift of our noise is bounded uniformly by this cumulant homogeneity.)

Given a homogeneity assignment $\sfs$ we define a cumulant homogeneity \label{idx:sfsfc}$\fc=\phantom{}^{\sfs}\fc$ as follows. For any $M\in\N$, any map $\ft:[M]\to \sFLm$, any spanning tree $\T$ for $[M]$ and any interior vertex $\nu\in\interior\T$ we define the quantity
\begin{align}\label{eq:cumulanthomogeneity}
\fc^{(\ft,[M])}_\T(\nu):= -\bigg(\sum_{\mu\in\CC_\T(\nu)}\max_{u\in L(\T_\mu)}
\sfs{(\ft_u)} \bigg)+\max_{u\in L(\T_\nu)} \sfs{(\ft_u)} \I_{\nu\ne\rho_\T},
\end{align}
where $\CC_\T(\nu)$ denote the set of children of $\nu$ in $\T$ and $L(\T_\mu)$ denote the set of leaves $u$ of $\T$ such that $u\ge\mu$  with respect to the tree order. Note that in particular $L(\T_u)=\{u\}$ for any leaf $u\in L(\T)$.

%
%

\begin{remark}
In the notation of \cite{ChandraHairer2016}, we always set $\FL_{\cum}:=\FL_{\cum}^{\all}$.
\end{remark}

As a first result we check consistency \cite[Def. A.16]{ChandraHairer2016} of $\sfs$ and $^{\sfs}\fc$, and super-regularity of the shifted trees $\tau$. Here we call a tree $\tau$ ``shifted tree'' if there exists $\taua=(T^\fn_\fe,\ft) \in \CT$ such that $\tau = (T^\fn_\fe,\ft')$ where $\ft(e)=\ft'(e)$ for kernel-type edges $e \in K(\tau)$ and $\ft'(L(\tau)) \ssq \sFLm$. (The basis vectors of the larger regularity structure which we will construct in Section~\ref{sec:ext:reg:str} will be shifted trees in this sense.) The next lemma applies in particular in case of $\sFLm = \FLm$ and $\sfs = \fs$.

\begin{lemma}\label{lem:cumu_homo_consistent}
The cumulant homogeneity $^{\sfs}\fc$ is consistent with ${\sfs}$. Moreover, provided that $\sfs(\snt) \ge -\shalf-\kappa$ for any $\snt \in \sFLm$, any  shifted tree  is $(\phantom{}^{\sfs}\fc, |\cdot|_{\sfs})$-super-regular.
\end{lemma}
\begin{proof}
We first check consistency in the sense of \cite[Def~A.16]{ChandraHairer2016}. Let $M\in\N$ and  $\ft:[M]\to\sFLm$. The fact that
\[
\sum_{\nu\in\interior\T} \phantom{}^{\sfs}\fc^{\ft,[M]}_\T(\nu)=-|\ft([M])|_{\sfs}
\]
follows directly from the definition. 
To see point 3 of \cite[Def. A.16]{ChandraHairer2016}, let $\nu\in\interior\T$ such that $\nu\ne\rho_\T$. Then we have
\[
\sum_{\mu\in\interior\T,\mu\ge\nu}\phantom{}^{\sfs}\fc^{\ft,[M]}_\T(\mu)
=
-\sum_{i\in L(\T_\nu)} |\ft_i|_{\sfs}+\max_{i\in L(\T_\nu)} |\ft_i|_{\sfs}
<
-|\ft(L(\T_\nu))|_{\sfs}.
\]
To see the last point, let $M\ge 3$ and $\nu\in\interior\T$ with $|L(\T_\nu)|\le 3$. Then
\[
\sum_{\mu\in\interior\T,\mu\ge\nu}\phantom{}^{\sfs}\fc^{\ft,[M]}_\T(\mu)<|\scale|(|L(\T_\nu)|-1)
\]
since by assumption one has $|\ft|_{\sfs}>-|\fs|$ for any noise type $\ft \in \sFLm$.

We show next that any shifted tree $\tau = (T^\fn_\fe,\ft)$ is super-regular. Let $\taua  = (T^\fn_\fe,\ft')$ be as in the definition of shifted trees.
Since the tree $\tau \in \CT_-$ is $\fs$-super-regular by assumption, one has for any subtree $S\subseteq T$ with the property that $\# K(S)>1$ the estimate
\[
|(S^0_\fe,\ft)|_{ \sfs }\ge|(S^0_\fe,\ft')|_{\fs}>-\shalf.
\]
Furthermore, we have in the notation of \cite[Def~A.24]{ChandraHairer2016} the identity
\begin{align}\label{eq:cumulant:homo:hbar}
\hbar_{^{\sfs}\fc}(\ft(L(S)))=
-\max_{\substack{u\in L(S)}} |\ft(u)|_{\sfs}.
\end{align}
Choose now a noise type edge $v\in L(S)$ with the property that the maximum on the right-hand side of (\ref{eq:cumulant:homo:hbar}) is attained for $v$. If $v$ is such that $\ft(v) \ne \ft'(v)$, then one has
\[
|(S^0_\fe,\ft)|_{\sfs}
\ge
|(S^0_\fe,\ft')|_{\fs} + (|\ft(v)|_{\sfs}- |\ft'(v)|_{\fs})
>-\hbar_{^{\sfs}\fc}(L(S)),
\]
where we use the fact that by super-regularity of $\tau$ one has $|(S^0_\fe,\ft)|_{\fs}> |\ft(v)|_\fs$.

Finally, in the notation of \cite[Def. 2.26]{ChandraHairer2016} we have for any leaf-typed sets $A$ and $B$ 
\[
j_A(B)\ge\shalf-\kappa
\]
and for $\kappa>0$ small enough we have  $-\shalf+\kappa<|(S_\fe^0,\ft)|_{\sfs}$.
\end{proof}


%
%

We recall the notation $\|\eta \|_{N,\fc}$ and $\|\eta ; \bar \eta \|_{N,\fc}$ from \cite[Def. A.18 \& A.19]{ChandraHairer2016}. 

\begin{lemma}\label{lem:bound:eta}
Fix $N, \bar N \in \N$. Let $\sfss : \sFLm \to \R_-$ be a second homogeneity assignment such that $-\shalf-(\cumN+1)\kappa < \sfss < \sfs - \cumN\kappa$ and let $\fc:=\phantom{}^{\sfss}\fc$.
For any $\eta \in \SMinfN$ and $C>0$ one has
\begin{align}\label{eq:noise:kernel:bound}
\|\eta\|_{\cumN,\fc} \lesssim 
\| \eta \|_{\sfs}
\end{align}
uniformly over all noises $\eta \in \SMinfN$ with $\|\eta\|_{\sfs} \le C$.

If $\bar \eta \in \SM_\infty^N$ is another smooth noise, then one has
\begin{align}\label{eq:noise:kernel:bound:2}
\| \eta ; \bar \eta \|_{\cumN,\fc} \lesssim 
\| \eta ; \bar \eta \|_{\sfs}
\end{align}
uniformly over all noises $\eta,\bar \eta$ with $\|\eta\|_{\sfs} \lor \|\bar\eta\|_{\sfs} \le C$.
\end{lemma}
\begin{proof}
We only show (\ref{eq:noise:kernel:bound}), the bound (\ref{eq:noise:kernel:bound:2}) follows similarly.
Let $K \in \CYN$ be such that (\ref{eq:smooth:noise}) holds and such that $\CYNnorm{K} \le 2 \|\eta\|_\sfs$. To continue the proof, we introduce some notation from \cite{ChandraHairer2016}. Given $M \in \N$ we call $\T$ a \emph{spanning tree} for $M$ if $\T$ is a binary, rooted tree with set of leaves given by $L(\T) = [M]$. We denote by $\topcirc\T$ the set of interior nodes of $\T$ and we call an order-preserving map $\s:\topcirc\T \to \N$ a \emph{labelling}. 
Given a labelled spanning tree $(\T,\s)$ and a map $\ft : [M] \to \sFLm$ we introduce the notation
\[
\langle \fc^{\ft,[M]}_\T,\s \rangle := \sum_{\nu \in \topcirc \T} \fc^{\ft,[M]}_\T (\nu) \s(\nu),
\]
and the set $D(\T,\fs) \ssq \bar\domain^M$ as the set of $x \in \bar\domain^M$ such the
\[
C^{-1} 2^{-\s (k \land_\T l)} \le |x_k - x_l| \le
C 2^{-\s (k \land_\T l)}
\]
for any $1 \le k,l \le M$ and for some constant $C>0$ large enough. (Here $C>0$ is fixed but large enough so that the sets $D_{(\T,\s)}$ cover all of $\bar\domain^M$.)
With this notation one has 
\[
\| \eta \|_{\cumN,\fc} 
\le 
\varnorm{K} + 
\max_{M \le \cumN} \max_{\ft : [M] \to \sFLm} 
\sup{(\T,\s)}
\sup_{x \in D(\T,\s)}
|\E^c[ ( \eta _{ \ft(k) } (x_k) )_{k \le M} ]|
2^{- \langle \fc^{\ft,[M]}_\T,\s \rangle },
\]
where the first supremum runs over all labelled spanning trees $(\T,\s)$ for $M$.

We fix from now on $M \le \cumN$, a type map $\ft:[M] \to \sFLm$ and a spanning tree $\T$ for $M$. 
Writing $\E^c (X_k)_{k\le M}$ for the $M$th joint cumulant of a collection of random variables $X_k$, 
the cumulant of the noises $\eta_{\ft(k)} (x_k)$ can be bounded by 
\[
|\E^c[ ( \eta _{ \ft(k) } (x_k) )_{k \le M} ]|
\le
\sum_{\cm} 
|\E^c[ \sintf_{\cm(k)}( \CU K _{ \cm(k) }^{\ft(k)} (x_k) )_{k \le M} ]|
\] 
where the sum runs over all families $(\cm_k)_{k \in [M]}$ where each $\cm_k$ is a multiset with values in $\FLm$. We fix such a family from now on.
We then write $K^{(k)} := K_{\cm_k}^{\ft(k)} \in \CYsimpp{\#\cm(k)}$, so that it suffices to show that
\[
|\E^c[ \sintf_{\cm_k}(\CU  K ^{(k)} (x_k) )_{k \le M} ]|
\lesssim 
\prod_{k=1}^M \| K^{(k)} \|_{\beta^{(k)}}
2^{\langle \fc^{\ft,[M]}_\T,\s \rangle },
\]
uniformly over all labelling $\s$ and $x \in D_{(\T,\s)}$, where $\beta^{(k)} := \beta_{\cm_k}^{\snt(k)}$ is as in \eqref{eq:CYNnorm}.
It suffices to show this bound uniformly over all simple tensors $K^{(k)} = K^{(k)}_0 \otimes \ldots \otimes K^{(k)}_{m_k}$, where $m_k := \# \cm_k$, the general case follows from the definition of the tensor norm. 

We define $\Lambda:=\{ (k,l): 1 \le k \le M, 1 \le l \le m(k) \}$.
We think of $\Lambda$ as indexing the variables of the kernels $K^{(k)}$ which are integrated out by stochastic integration. We define $\CP$ as the set of pairings $P$ of $\Lambda$ with the following properties. We require that for any $\{(k,l), (m,n)\} \in P$ one has $k \ne m$ and $\snt(P):= \tilde\cm_{k} (l) = \tilde\cm_{m}(n)$. (The first condition reflects the fact that our noises take values in homogeneous Wiener Chaoses, so that self contractions do not need to be considered, the second condition reflects the fact that the Gaussian noises $\xi_\Xi$ are independent.) We also require that the pairing is connected, in the sense that if $\sim$ denotes the smallest equivalence relation on $[M]$ with the property that $k\sim l$ whenever there exists some $i,j$ such that $\{(k,i), (l,j)\} \in P$, then all elements of $[M]$ are equivalent.

The cumulant can then be written as
\[
\E^c[ \sintf_{\cm_k}(\CU  K ^{(k)} (x_k) )_{k \le M} ]
=
\sum_{P \in \CP} E_P(x),
\]
where
\begin{align}\label{eq:EP}
E_P(x) := \int_{\bar\domain^\Lambda} dy_\Lambda
	\prod_{k=1}^M
	\CU K^{(k)} ( ( x_k - y _{ (k,l) } ) _{l \le m_{\ft(k)} } )
	\prod_{P=\{a,b\} \in \CP}
	\delta(y_a - y_b),
\end{align}
and we will show that for any $P \in \CP$ one has
\begin{align}\label{eq:EP:bound}
|E_P(x)| \lesssim \prod_{k=1}^M \| K^{(k)}\|_{\beta^{(k)}}  \, 2^{\langle \fc^{\ft,[M]}_\T,\s\rangle}.
\end{align}

Let $J \ssq [M]$ denote the set of indices $k \le M$ with $m_k >1$, and we  write 
\[
E_P(x) = \int_{\bar \domain^M} dz_M \prod_{k \in J} K_0^{(k)}(x_k - z_k) \prod_{k \notin J} \delta_0(x_k - z_k) \tilde E_P(z) dz,
\]
with
\[
\tilde E_P(z) := 
\int_{\bar\domain^\Lambda} dy_\Lambda
	\prod_{k=1}^M
	\tilde\CU K^{(k)} ( ( x_k - y _{ (k,l) } ) _{l \le m_{\ft(k)} } )
	\prod_{P=\{a,b\} \in \CP}
	\delta(y_a - y_b),
\]
where we set
\begin{align*}
\tilde\CU K^{(k)} := 
\begin{cases}
\CU( \delta_0 \otimes K_1^{(k)} \otimes \ldots \otimes K_{m_k}^{(k)})  & \text{ if } k \in J \\
\CU(K_0^{(k)} \otimes K_1^{(k)})		& \text{ if } k \notin j.
\end{cases}
\end{align*}

It suffices to show that bound uniformly over kernels with $\| K_0^{(k)} \|_{-|\fs|} = 1$ for any $k$. Then, it suffices to show \eqref{eq:EP:bound} with $E_P$ replaced by $\tilde E_P$, the bound for $E_P$ can  be argued as in \cite[Sec.~B]{ChandraHairer2016}. 


By definition, for every choice of homogeneities $\beta^{(k)}_i$, $k = 1 , \ldots, M$, $i = 1, \ldots, m_k$, with $-|\fs| -\I_{k \notin J}< \beta^{(k)}_i < 0$  and $\sum_{i=1}^{m_k} \beta^{(k)}_i = \beta^{(k)}$ one has the bound
\[
|\tilde E_P(x)| \lesssim 
\left(
	\prod_{k=1}^M \prod_{i=1}^{m_k} \| \tilde K_i^{(k)} \|_{\beta^{(k)}_i}
\right)
\prod_{\{(k,i),(l,j)\} \in \CP}
	2^{- \s(k \land_\T l) \big((\beta^{(k)}_i + \beta^{(l)}_j + |\fs|) \land 0 \big) }.
\]
Here, we set $\tilde K^{(k)}_i := K^{(k)}_i$ if $k \in J$ and $\tilde K^{(k)}_1 := K^{(k)}_0 \star K^{(k)}_1 = \tilde\CU(K^{(k)})$ if $k \notin J$. 

Since by definition one has the estimate
\[
\prod_{k=1}^M \prod_{i=1}^{m_k} \| \tilde K_i^{(k)} \|_{\beta^{(k)}_i} \le \prod_{k=1}^M \|
  K^{(k)}\|_{\beta^{(k)}},
\]
it remains to find a choice of $\beta^{(k)}_i$ as above with the property that 
\begin{align}\label{eq:beta:fc}
-\sum_{\mu \ge \nu}\sum_{k,l: (k \land_\T l) = \mu } \Big((\beta^{(k)}_i + \beta^{(l)}_j + |\fs|) \land 0 \Big)
 \le
\sum_{\mu \ge \nu}
\fc^{t,[M]}_\T(\mu)
\end{align}
for any $\nu \in \topcirc\T$. Let $\bar k: \topcirc\T \backslash\{\rho_\T\} \to [M]$ be the injective map defined recursively by setting\footnote{If the $\argmin$ is not unique, we choose a minimizer arbitrarily.} $\bar k(\nu) := \argmin_{k \in L(\T_\nu)} \sfss({\ft(k)})$ if $\nu$ is maximal in $\topcirc\T$, and
\begin{align}\label{eq:k:bar}
\bar k(\nu) := \argmin\left( \sfss({\ft(k)}) : k \in L(\T_\nu) \backslash\{ \bar k(\mu) : \mu > \nu \}\right)
\end{align}
otherwise. (Recall that $L(\T_\nu)$ denotes the set of leaves $u \in L(\T)$ such that $u \ge \nu$.) Note that $\fc^{\ft,[M]}_\T(\nu) = -\sfss({\ft(\bar k(\nu))})$ for $\nu \in \topcirc\T \backslash\{\rho_\T\}$. Denote moreover by $k_1, k_2 \in[M]$ the two distinct elements of $[M]$ not in the range of $\bar k$, so that
$\fc^{t,[M]}_\T(\rho_\T) = -\sfss({k_1}) - \sfss({k_2})$. 

Conversely, denote by $\bar \nu(k) \in \topcirc\T$ the interior node of $\T$ with the property that $K^{(k)}$ ``collapses'' at $\bar\nu(k)$, i.e.\ $\bar\nu(k)$ is the maximum node $\nu$ with the property that whenever $\{ (k,i), (l,j)\} \in P$ one has $k\land_\T l \ge \nu$. Since we only have to consider ``connected'' pairings, it is clear that 
\[
\#\{ k \in [M] : \bar\nu(k) \ge \mu \} \le \# L(\T_\mu) -1
\]
for any $\mu \in \topcirc\T \backslash\{ \rho_\T \}$. Let finally $i(k) \in \{1 ,\ldots, m(k) \}$ denote some index such that $\{ (k,i(k)), (l,j) \} \in P$ for some $(l,j) \in \Lambda$ such $k \land_\T l = \bar\nu(k)$.

We also choose an arbitrary index $j(k) \in \{1 ,\ldots, m(k)\}$ such that $j(k) \ne i(k)$ whenever $k \in J$ (and hence $m(k)>1$).
With the choice
\[
\beta_i^{(k)} := - \shalf + (\sfs({ \ft(k) }) + \kappa) \I_{i = i(k)} - \kappa \I_{i = j(k)}
\]
one has $\sum_i \beta_i^{(k)} = \beta^{(k)}$ and $\beta_i^{(k)} > -|\fs| - \I_{k \notin J}$, so that it remains to show (\ref{eq:beta:fc}), which follows once we show that
\[
\sum_{k: \bar\nu(k) \ge \nu} \sfs({\ft(k)}) 
- \kappa M
\ge
\sum_{\mu \ge \nu} \sfss({\ft(\bar k(\mu))})
\]
It is clear from the fact that the numbers $\bar k(\mu)$ where recursively chosen to maximise $\sfs({\ft(k)})$, so that for any $A \ssq L(\T_\nu)$ with $\# A \le \#L(\T_\nu)-1$ one has
\[
\sum_{\mu \ge \nu} \sfss({\ft(\bar k(\mu))}) 
\le
\sum_{k \in A} \sfss({\ft(k)})
\le
\sum_{k \in A} \sfs({\ft(k)}) + \kappa \cumN.
\]
Since $\bar\nu(k) \ge \nu$ implies $k \in L(\T_\nu)$, the proof is finished.
\end{proof}

\subsection{Additional technical assumptions}\label{sec:technical:assumpation}

For the main result of this article we need a technical assumption that guarantees that ``logarithmic'' trees which appear (modulo polynomial decoration) as a subtree of another ``logarithmic'' tree are such that the BPHZ character vanishes automatically. This should also hold after we shift the noise. It turns out that in some examples (for instance generalised KPZ, see Section~\ref{sec:gKPZproof}), this is not true if we would consider arbitrary shifts. Instead we exploit certain (anti-)symmetries of our integration kernels, and for this we need the expectation of our noise to be invariant under these symmetries. To make this more concrete, we fix a finite symmetry group $\groupD \ssq \GL(d)$\label{symmetry}\label{idx:groupD} in $d$ dimensions.
The typical case one should have in mind (and suffices for our purpose) is when $\groupD$ is generated by finitely many
spatial reflections.

To incorporate this symmetry into our definitions, we make the following definition.

\begin{definition}\label{def:CYshift}\label{def:CYzsimp}
We denote by $\CYshift \ssq \CYsimp$ the set of $K \in \CYsimp$ such that $\CU K$ is invariant under simultaneous transformation of all variables by any $A \in \groupD$. We also write $\CYNshift \ssq \CYN$ for the space of all $K=(K_\cm^\Xi)$ such that $K_\cm^\Xi \in \CYshiftt{\#\cm}$ for any $\cm$ and any $\Xi$, and we write $\CYNzshift \ssq \CYNz$ for the closure of $\CYNshift$ under the norm (\ref{eq:CYNnorm}).
\end{definition}

Later on it will be convenient to also introduce the notation $\CYzsimpp{n} \ssq \CYshift$ for the linear space spanned by $K_0 \otimes \ldots \otimes K_n \in \CYshift$ such that $\int K_0 = 0$.
Note that for any $K \in \CYshift$ one can view $\CU K$ as an element $\barCCng$. Here, we let $\groupD$ act on $\bar\domain^n$ via $A (x_i)_{i \le n} := (A x_i)_{i \le n}$ for any $A \in \groupD$. The following definition will play an important role.

\begin{definition}\label{def:shifted:noise}
We write $\SMsinf \ssq \SMinf$\label{idx:SMsinf} for the subspace of noises given as in (\ref{eq:smooth:noise}) for some $K \in \CYNshift$, and we write $\SMsz \ssq \SMz$\label{idx:SMzinf} for the closure of $\SMsinf$ under the norm (\ref{eq:noise:norm}).
We call a smooth noise $\eta = (\eta_\Xi)_{\Xi \in \FL_-} \in \SMsinf$ 
a ``shifted smooth noise''.
\end{definition}



The terminology ``shifted noise'' will become clear in Section~\ref{sec:shift:WC}. 
In order to formulate our assumption, let \label{idx:CV}$\CV$  denote the set of trees $\tau \in \TT_-$ with $\fancynorm{\tau}_\fs=0$ and which are ``subtrees'' (modulo polynomial decoration) of a larger tree of zero homogeneity. More precisely, for any $\tau \in \CV$ there exists another tree $\sigma= S^\fn_\fe \in \TT_-$ with $\fancynorm{\sigma}_\fs=0$, a proper sub tree $\taua =\tilde T^{\fn}_{\fe}\ssq \sigma$ of 
$\sigma$ (``proper'' means that $E(\taua) \subsetneq E(\sigma)$) and a decoration $\tilde n: N(\taua) \to \N^d$  such that $\tau = {\tilde T}^{\tilde \fn}_\fe$. We also assume that $\tilde\tau$ is connected to its complement in $\sigma$ with more than one node, so that $\# \{ u \in N(\tilde\tau) : \exists e \in E(\sigma) \setminus E(\tilde\tau) \text{ with } u \in e\} > 1$.

\begin{assumption}\label{ass:zero-homo}\label{ass:technical}
We assume that for any $\tau \in \CV$ and (not necessarily Gaussian)  shifted smooth noise $\eta$ one has $g^{\eta} (\tau) = 0$. 
\end{assumption}

\begin{remark}
The only place where Assumption~\ref{ass:technical} is used is the proof of Lemma~\ref{lem:blowuprenormconstantzerohomo} below. 
Loosely speaking, it ensures that if $\tau$ is a tree of $0$ homogeneity and only one of its noises is made slightly more regular, then the renormalisation constant does not present any logarithmic divergences anymore. We need this to ensure that renormalisation constants of ``shifted'' trees are bounded by a constant only depending on the largest scale involved (we will have various shifts which are regularised on different scales).
The strategy we employ below relies on upper \emph{and lower} bounds of the blow-up behaviour of renormalisation constants, 
from which we deduce exactly which ``shifted'' tree is dominant. 
We do not show such a lower bound for log-divergences, which is why we need an additional assumption ensuring that there is only the ``main'' log-divergence and no log-subdivergence. 
\end{remark}

Finally, denote by $\CV_0$\label{idx:CV0} the set of $\tau \in \TT_-$ with $\fancynorm{\tau}_\fs = 0$ and $\# L(\tau) = 2$.
\begin{assumption}\label{ass:CVz}\label{ass:last}
We assume that for any $\tau \in \CV_0$ and (not necessarily Gaussian)  shifted smooth noise $\eta$ one has $g^{\eta} (\tau) = 0$. 
\end{assumption}

Assumption~\ref{ass:CVz} is needed in Section~\ref{sec:constraints} since the stability under removing the large-scale cutoff given in Theorem~\ref{thm:evaluation:trees:largescale} fails in general for $\tau \in \CV_0$. Note that for $\tau \in \CV_0$ one has $\E \PPi^\eta \tau (0) = -g^\eta(\tau)  = 0$.

\begin{remark}
We give an informal reason why the previous assumption is needed in Theorem~\ref{thm:evaluation:trees:largescale}. 
In this theorem we consider the evaluation from Definition~\ref{def:bar:Upsilon:large:scale}, which defines a constant based on the idea of integrating a tree $\tau$ with leaves $u_1, \ldots, u_n$ against a test function $\phi(u_1, \ldots, u_n)$.
We will assume in this context that $\phi$ is a function of the differences of its arguments 
and compactly supported in these differences (i.e. there exists $R>0$ so that $\phi(u_1, \ldots, u_n)=0$ whenever there exists $i,j \le n$ such that $|u_i - u_j|>R$).
Under this assumption we will show that this evaluation remains bounded as one removes 
the large-scale cutoff from the integration kernels. 
The proof relies on a counting argument, which we use to apply the results from the last section of \cite{Hairer2017}. One can think of this as a generalisation of the fact that $\int_{\R^d} (1+|x|)^p dx$ exists if and only if $p<-d$ to the case of generalised convolutions. 
In this analogy, the case $\tau \in \CV_0$ is similar to the situation of trying to integrate 
$(1+|x|)^{-d}$, which diverges on large scales.

Note that we do not have this problem for trees with more than two leaves, even if they are logarithmically divergent. This is because we assume that $\phi$ is compactly supported in all differences between its arguments, which in some sense means that we ``gain'' a degree $\frac{|\fs|}{2}$ for every leaf, as far as the power counting argument is concerned (equation \eqref{eq:psi_estimate} makes this more clear).
\end{remark}
\begin{remark}\label{rmk:assumptions}
One can replace Assumption~\ref{ass:CVz} by the weaker Assumptions~\ref{ass:CJHopfIdeal} and~\ref{ass:CHBPHZcharacters} introduced in Section~\ref{sec:CJ} below. We will show in Section~\ref{sec:constraints} that 
the former really implies the latter two. 
In some interesting examples, including SDEs, $\Phi^p_2$, Yang--Mills and the parabolic Anderson in two spatial dimensions, Assumptions~\ref{ass:CJHopfIdeal} and~\ref{ass:CHBPHZcharacters} can be shown ``by hand'' relatively easily, even though all of these examples violate Assumption~\ref{ass:CVz} above. 
However, for many more convoluted examples, including $\Phi^4_{4-\eps}$, generalised KPZ or the parabolic Anderson model in three dimensions, it seems difficult to show these assumptions by hand. 
We actually expect Assumptions~\ref{ass:CJHopfIdeal} and \ref{ass:CHBPHZcharacters} 
always to hold, so that one should be able to drop 
Assumption~\ref{ass:CVz} with a little more technical effort. 
\end{remark}

\section{A support theorem for random models}
\label{sec:support:thm:models}

Recall that we fix a Gaussian (space or space-time) white noise $\xi = (\xi_\Xi)_{\Xi \in \FL_-}$, which we can view as an element of $\SMsz$, see Definition~\ref{def:shifted:noise}. We also fix a smooth mollifier $\rho\in \CCinfg$ with $\int \rho = 1$, so that $\xi^{\eps}:=\xi \star \rho^{(\eps)} \in \SMsinf$ for any $\eps>0$ and one has $\xi^\eps \to \xi$ in $\SMsz$. We write $g^\eps := g^{\xi^\eps}$ for the BPHZ character \eqref{eq:BPHZ:character} and $\hat Z^\eps:= \hat Z^{\xi^\eps}$ and $\hat {\PPi}^\eps := \hat {\PPi}^{\xi^\eps}$ for the BPHZ-renormalised lift (\ref{eq:lift:BPHZ}) of $\xi^\eps$

\subsection{The ideal \texorpdfstring{$\CJ$}{J}}\label{sec:CJ}

Let us first introduce the following notation, which we will use heavily in the forthcoming sections. Given a kernel assignment $(G_\ft)_{\ft \in \FL_+}$ with $G_\ft \in \CC_c^\infty(\bar\domain\backslash \{0\})$ absolutely integrable, we define for any tree $\tau \in \CT$ a function $\CK_G\tau : \bar\domain^{L(\tau)} \to \R$ by
\label{ind:CK_G}
\begin{multline}\label{eq:CK}
\CK_G \tau ( x_ { L (\tau) } )
:=
 \int_{\bar\domain^{N(\tau)}} dx \,
	\delta(x_{\rho_\tau})
	\prod_{e\in K(\tau)} D^{\fe(e)} G_{\ft(e)}(x_{e^\downarrow}-x_{e^\uparrow}) 
\\
	\times \prod_{u\in N(\tau)} x_u^{\fn(u)} 
	\prod_{e \in L(\tau)} \delta ( x_e - x_{e^\downarrow} )\;.
\end{multline}
We also write $\CKhat := \CK_{\hat K}$.
Although $\hat K$ does not have bounded support, this is well-defined as a limiting
distribution obtained by removing a cutoff, see Theorem~\ref{thm:evaluation:trees:largescale} below.
Given additionally a smooth function $\vphi\in \bar\CC^\infty_c( \bar\domain^{L(\tau)} )$ it will be useful to introduce the notation
\begin{align}\label{eq:CK:ctr}
\ctrL{\tau}{G}{\vphi}
:=
\int_{\bar\domain^{L(\tau)}} dx
	\CK_G\tau ( x ) \vphi(x) \in \R\;.
\end{align}

\begin{example}
We can graphically represent the action of $\CK_G\tau$. For instance, we write (slightly informally)
\[
\Big(\CK_G \treeExampleKPZa\Big) (x_1, \ldots, x_4)
=
\treeExampleKPZaKernel,
\]
where we leave $G$ implicit on the right-hand side.
\end{example}

For two different trees $\tau,\taua$ one has by definition $L(\tau) \cap L(\taua) = \emptyset$, so that $\CK_G\tau$ and $\CK_G\taua$ have disjoint domains of definition. However if $\cm:=[L(\tau),\ft] = [L(\taua),\ft]$, then after symmetrising one can naturally view $\CK_G\tau$ and $\CK_G\taua$ as being defined on the same space $\bar\domain^\cm$. In particular, the notation \eqref{eq:CK:ctr} extends naturally to $\vphi \in \barCCcm$. This motivates the following definition.

\begin{definition}\label{def:tildePsi}
We write $\cutoffspace$ for the set of all families of test functions $(\psi_\cm)_\cm$, indexed by multisets $\cm$ with values in $\FLm$, such that  $\psi_\cm \in \barCCcmg$. We also write $\cutoffspacez$ for the set of $\psi \in \cutoffspace$ such that $\int\psi_\cm := \int_{\bar\domain^\cm} \delta(x_p)\psi_\cm(x_\cm) = 0$ for any $\cm$. Here we fix some arbitrary $p \in \td(\cm)$ (it is clear that this definition does not depend on the  choice of $p$). 
We then define an evaluation $\ctrL{\tau}{G}{\psi}$ for $\tau \in \CT$ and $\psi \in \cutoffspace$ by setting
\[
\ctrL{\tau}{G}{\psi}
:=
\ctrL{\tau}{G}{\psi_{[L(\tau),\ft]}}.
\]
\end{definition}
With this notation we now define an ideal $\CJcon$ as follows.

\begin{definition}\label{def:CJcon}\label{def:CH}
We define $\CJcon \ssq \CT_-$ as the ideal generated by all elements $\tau \in \Vec \TT_-$ such that
$\ctrLhat{\tau}{\psi} = 0$ for any $\psi \in \cutoffspacez$. We then denote by $\CH \ssq \CG_-$ the annihilator of $\CJ$ (given by the set of all characters $g\in\CG_-$ with the property that $g(\sigma)=0$ for all $\sigma\in\CJ$).
\end{definition}

Note that by definition $\CJcon$ is generated by linear combinations of trees (rather than linear combinations of products of trees). We will use this fact heavily below.

\begin{remark}
There is a natural norm on $\|\cdot\|_{\CK^+}$ on large scale kernel assignments $R$, see \eqref{eq:CKp:norm}, and writing $\CK^+_0$ for the closure of the space of smooth, compactly supported functions under this norm, one has indeed $\hat K-K \in \CK_0^+$. Moreover, it is not hard to show that $\ctrL{\tau}{K+R}{\psi}$ extends continuously to $R \in \CK^+_0$ for any $\tau \in \CT_-$ and any $\psi \in \cutoffspacez$, so that $\ctrLhat{\tau}{\psi}$ is well defined. 
The last claim follows from a straightforward counting argument as in \cite[Sec.~4]{Hairer2017}, which is carried out in Lemma~\ref{lem:super:regularity:implies:large:scale:bound} below. 
\end{remark}
\begin{remark}
We choose $\psi$ in the definition of $\cutoffspacez$ to integrate to zero, since the cumulants of our ``shifts'' will satisfy this property. This is needed to ensure weak convergence of the shift to zero.
\end{remark}

\begin{example}
Consider as an example the KPZ equation $\partial_t h = \Delta h + |\partial_x h|^2 + \xi$ where $\CT_-$ is generated by the trees
\[
\treeKPZa,
\treeKPZb,
\treeKPZc,
\treeKPZd,
\treeKPZe,
\treeKPZf.
\]
We show that $\CJ$ is the ideal generated by $\{ \leafs\}$.
First note that we have $\leafs \in \CJ$ since $\CK_G \leafs = 0$. Next, we recall that $\CJ$ is generated by linear combinations of trees with the same number of leaves. Furthermore, by simply rescaling $\phi(x) \mapsto \phi(\eps^{-1}x)$, we see that $\CJ$ is generated by linear combinations of trees of same number of leaves and same homogeneity.
It follows that no linear combination involving any of $\treeKPZb$, $\treeKPZc$ or 
$\treeKPZd$ belongs to $\CJ$.
The only non-trivial part is to deal with the remaining 2 trees $\treeKPZe, \treeKPZf$. We sketch the proof that they can not form a linear combination that takes values in $\CJ$.
For this fix test functions $\psi_{i,j} : \bar\D \to \R$ and consider test functions $\phi_\eps$ of the form $\phi_\eps(x_1, x_2, x_3, x_4) = \sum_{\sigma} \sum_{1 \le i < j \le 4} \psi_{i,j}^{\eps}(x_{\sigma(i)} - x_{\sigma(j)})$ where the first sum runs over all permutations $\sigma$ of $\{1,2,3,4\}$. 
Here $\psi_{i,j}^{\eps}(x) := \eps^{-\frac{3}{2}} \psi_{i,j}(\eps^{-\fs} x)$ if $(i,j)\in \{(1,2), (3,4)\}$ and $\psi_{i,j}^{\eps}(x) := \psi_{i,j}(x)$ otherwise. 
It is then not difficult to see that
\[
\langle\hat\CK\treeKPZe, \phi_\eps \rangle \sim \eps^{-1}
\qquad \text{ and }
\qquad
\langle\hat\CK\treeKPZf, \phi_\eps \rangle \sim \eps^{-2}\;.
\]
The reason for this is that the first tree only contains one subdivergence $\treeKPZb$ of degree $-1$, while the second tree contains two of them.
From this it follows that no linear combination of these two trees can be element of $\CJ$,
thus leading to the claim.
\end{example}


We now state the two assumptions that we are going to need for this section. The first assumes that the annihilator $\CH$ of $\CJ$ forms indeed a group.

\begin{assumption}\label{ass:CJHopfIdeal}
The ideal $\CJ$ is a Hopf ideal in $\CT_-$. In particular, its annihilator $\CH$ is a Lie subgroup of the renormalisation group $\CG_-$.
\end{assumption}

The next assumption relates the subgroup $\CH$ to the BPHZ characters associated to smooth shifted noise. We recall the notation $\SMsinf$ for the space of smooth shifted noises and $\SMsz$ for its closure under the norm (\ref{eq:noise:norm}).
In the following assumption we do not require the noise $\eta$ to be Gaussian.

\begin{assumption}\label{ass:CHBPHZcharacters}
There exists a continuous map $\SMsz \ni \eta \mapsto f^\eta \in \CG_-$ with the property that $f^{\mathbf 0} = \one^*$, and such that $g^\eta \in f^\eta \circ \CH$ for any $\eta \in \SMshifted$. Here $\mathbf 0$ denotes the $0$-noise.
\end{assumption}
We will see in Corollary~\ref{cor:CHsmallestgroup} below that $\CH$ is in fact the smallest Lie subgroup of $\CG_-$ that has the property described in Assumption~\ref{ass:CHBPHZcharacters}.
As was already pointed out in Remark~\ref{rmk:assumptions}, we will show in Section~\ref{sec:constraints} that the two assumptions given above are implied by Assumption~\ref{ass:CVz} (which is the only argument in the paper where Assumption~\ref{ass:CVz} is needed).


\subsection{A support theorem for random models}\label{sec:proof:CJ:hopf:ideal}\label{sec:supp:models}

From now on, we will always assume that Assumptions~\ref{ass:main:reg}--\ref{ass:zero-homo} 
and \ref{ass:CJHopfIdeal}, \ref{ass:CHBPHZcharacters} hold, except when
specified explicitly. The only exception is Section~\ref{sec:constraints}, where we prove that Assumptions~\ref{ass:CJHopfIdeal} 
and~\ref{ass:CHBPHZcharacters} are implied by Assumptions~\ref{ass:main:reg}--\ref{ass:CVz}.

Setting $\hat Z^\eps = \renorm{g^\eps} \Zcan(\xi^\eps)$ for the renormalised approximate model
and $\hat Z = \lim_{\eps \to 0}\hat Z^\eps$ for its limit, we can rewrite it as
\[
\hat Z^\eps=T_{\xi^\eps} \renorm{g^\eps} \Zcan(0).
\]
Models obtained by acting on the canonical lift of $0$ with the renormalisation operators will later play an important role, so we introduce the following notation.
\begin{definition}
For any character $g\in\CG_-$ we define the model $\mfc(g)$ by letting the renormalisation operator act on the canonical lift of $0$ to a model, i.e.\ we set
\label{idx:mfc}
\[
\mfc(g):=\renorm g \Zcan(0).
\]
\end{definition}

We will see in Lemma~\ref{lem:supporttranslation} below that the action of the translation operator maps the support into itself, so that the main part of the proof consists in understanding the set of characters $g\in\CG_-$ such that $\renorm g \Zcan(0)\in\supp \hat Z$. We will show that this set is a coset $f\circ\CH$ of the Lie subgroup $\CH$ of $\CG_-$ constructed above. 

\begin{proposition}\label{prop:constantsInSupport}
Let $\CH$ be the Lie subgroup of $\CG_-$ defined in Definition~\ref{def:CH} and let $\fxi\in\CG_-$ be the character defined in Assumption~\ref{ass:CHBPHZcharacters}. Then for any character $g\in \fxi\circ\CH$ in the left coset determined by $\fxi$ and $\CH$ one has
\[
\mfc(g)\in\supp \hat Z.
\]
\end{proposition}
Proposition~\ref{prop:constantsInSupport} follows from Proposition~\ref{prop:sequencetoconstant} below, which in turn relies on the constructions carried out in Sections~\ref{sec:renormalisationgroupargument} and~\ref{sec:shift}. Before we prove Proposition~\ref{prop:constantsInSupport} we show now that it implies a support theorem for random models, see Theorem~\ref{thm:main}.

\begin{remark}\label{rem:skew}
The converse of Proposition~\ref{prop:constantsInSupport} is not true in general. An example of this is given by the rough paths $\mathbf{B} = (B,\B + \M)$ where $\B^{i,j}_{s,t}:= \int_s^t (B^i_r- B^i_s) \circ dB^j_r$ denotes the Stratonovich lift and $\M_{s,t} = (t-s)M$ for a constant
(in time) and skew-symmetric matrix $M$. These rough paths are known to have support independent of $M$ but $\CH=\{\one^*\}$ is trivial in this case. We just sketch the argument here that the support is really independent of $M$. Consider deterministic smooth shifts $B^i_t \to B^i_t + \eps^{\half}\cos(\eps^{-1} t)$ and $B^j_t \to B^j_t + \eps^{\half} \sin(\eps^{-1} t)$. The translation operator transforms the second component into
\begin{multline*}
\B^{i,j,\eps}_{s,t} = \B^{i,j}_{s,t} + \eps^{-\half}\int_s^t (B^i_u - B^i_s) \cos(\eps^{-1} u) du
+
\\
\eps^\half\int_s^t  \cos(\eps^{-1} u ) - \cos(\eps^{-1} s) dB_u
+
\int_s^t ( \cos(\eps^{-1} u ) - \cos(\eps^{-1} s) ) \cos(\eps^{-1} u) du.
\end{multline*}
The last term converges to a positive constant times $(t-s)$ as $\eps \to 0$, while a quick computation shows that the two terms in the centre vanish in this limit.
\end{remark}

\begin{remark}
Proposition~\ref{prop:constantsInSupport} also gives information about limit models obtained from a different choice of renormalisation: For any $k\in\CG_-$ and any $g\in k\circ \fxi\circ \CH$ one has
\[
\mfc(g)\in \supp \renorm{k}\hat Z.
\]
\end{remark}

\maybeNotNeeded{
The following Lemma is immediate.
\begin{lemma}\label{lem:CZg}
Let $g\in\CG_-$ and denote by $(\Pi,\Gamma)=\CZ(\PPi)$ the model $\ZZ(g)$. Then $\ZZ(g)$ is the unique admissible model with the property that for any symbol $\tau=T^\fn_\fe\in\TT_-$ and any $x,z\in\R^d$ one has
\begin{align}\label{eq:ZZg}
\Pi_z T^\fn_\fe (x)=\sum_{\tilde \fn\le \fn}  
{\fn \choose \tilde\fn}
g(T^{\tilde\fn}_\fe) (x-z)^{\sum_{u\in N(T)} \fn(u)-\tilde\fn(u)}.
\end{align}
Moreover, one has the identity $\PPi \tau(x)=\Pi_0 \tau(x)$ for any $\tau\in\TT_-$ and any $x\in\R^d$.
\end{lemma}
\begin{proof}
If we denote by $(\Pi^\circ,\Gamma^\circ):=\CZ^{\ex}(\PPi^\circ):=Z^{\ex}(0)$ and by $f^\circ_z$ the character defined in \cite[Eq.~6.11]{BrunedHairerZambotti2016}, then it is easy to see that for any $\tau\in\CT$ such that $\tau\ne X^k$ for some $k\in\N^d$, we have $\Pi^\circ_z\tau=\PPi^\circ\tau\equiv 0$, and similarly for any $\tau\in \CT^{\ex}_+$ such that $\tau\ne X^k$, we have $f_z^\circ(\tau)=0$. Using the identity 
\[
\Pi_z \tau=(g\otimes\PPi^{\circ}\otimes f_z^\circ)(\Id\otimes\Delta^+_{\ex})\Delta^-_{\ex} \tau
=(g\otimes \Pi^\circ_z)\Delta^-_{\ex}\tau,
\]
the relation (\ref{eq:ZZg}) follows from the definition of the coproduct $\Delta^-_{\ex}$.

The uniqueness part follows as in \cite{Hairer2014} as a consequence of \cite[Thm.~3.31]{Hairer2014} and the definition of admissible models. 
\end{proof}
}
Before we state the main theorem, we derive some immediate identities.

\begin{lemma}\label{lem:translationoperatorZZ}
	For any character $g\in\CG_-$ and any smooth noise $h\in\SM_\infty$ one has the equality 
	\[
	\renorm{g} \Zcan(h)=T_{h} \mfc(g).
	\]
	In particular, one has the identity
	\[
	\hat Z^{\eps}[\xi]=T_{\xi^\eps} \mfc(g^{\eps})
	\]
	almost surely for any $\eps>0$.
\end{lemma}
\begin{proof}
	In order to see the first identity, it is enough to apply Theorem~\ref{thm:translationoperator} to $f\equiv 0$.
	The second claim follows from the fact that $\hat Z^\eps = \renorm{g^{\eps}} Z^\eps$ and
	$
	Z^\eps[\xi]=T_{\xi^\eps}\Zcan(0)
	$.
\end{proof}
The next lemma crucially states that shifting the noises by a \emph{random} smooth function maps the support into itself. 
\begin{lemma}\label{lem:supporttranslation}
	Let $h\in\Omega_\infty$. Then one has the identity
	\[
	T_h \hat Z[\xi]=\hat Z[\xi+h]
	\]
	almost surely. Moreover, if $h\in \SM_\infty$ is any smooth random noise, then one has the identity
	\[
	\supp T_h \hat Z\subseteq \supp \hat Z.
	\]
\end{lemma}
\begin{proof}
	We show the first statement. By Cameron-Martin's Theorem~\ref{thm:cameron-martin} it follows that the laws of $\xi$ and $\xi+h$ are equivalent. In particular, the right-hand side is well-defined $\P$-almost surely. To see the identity claimed in the statement, we use the fact that $T$ is jointly continuous in $h$ and $Z$. We then have
	\[
	\hat Z[\xi+h]=\lim_{\eps\to 0}\hat Z^\eps[\xi+h]=\lim_{\eps\to 0}T_{h^\eps}\hat Z^\eps[\xi]=T_h \hat Z[\xi].
	\]
	In order to see the second statement, let first $h\in\Omega_\infty$ be deterministic. In this case we exploit again the fact that the laws of $\xi$ and $\xi+h$ are equivalent, so that the laws of $\hat Z[\xi]$ and $T_h \hat Z[\xi]=\hat Z[\xi+h]$ are equivalent as well and $\supp \hat Z=\supp T_h \hat Z$. Using Lemma~\ref{lem:supportcontinuousmap} below, it follows that the continuous operator $T_h$ maps the support of $\hat Z$ into itself. 
		
	Let now $h\in \SM_\infty$ be random and let $A\subseteq \Omega$ be the set of full $\P$-measure with the property that $\hat Z(\omega)\in \supp\hat  Z$ for any $\omega\in A$. It then follows for $\omega\in A$ that $T_{h(\omega)} \hat Z(\omega)\in\supp \hat Z$. In particular we have $T_h \hat Z\in\supp \hat Z$ almost surely.
\end{proof}

In the previous proof we used the following lemma.
\begin{lemma}\label{lem:supportcontinuousmap}
	Let $X,Y$ be two Polish spaces, let $T:X\to Y$ be a continuous map and let $\mu$ be a probability measure on $X$. Then $\supp T_*\mu = \overline{T(\supp \mu)}$.
\end{lemma}
\begin{proof}
Since $x \in \supp \mu$ for $\mu$-almost every $x \in X$, it follows that $T(x) \in T(\supp  \mu)$ $\mu$-almost surely, and hence $\supp T_* \mu \ssq \overline{T(\supp \mu)}$. To see the inverse inclusion, let $y=T(x)\in T(\supp \mu)$ with $x\in\supp \mu$, and let $U$ be a neighbourhood of $y$ in $Y$. By continuity it follows that $T^{-1}(U)$ is a neighborhood of $x$ in $X$ and by the definition of the support, it follows that $\mu\{x:T(x)\in U\}=\mu(T^{-1}(U))>0$, and thus $y \in\supp T_* \mu$. This show that $T(\supp \mu) \ssq \supp T_*\mu$ and concludes the proof.
\end{proof}

Assuming Proposition~\ref{prop:constantsInSupport}, we can now state and prove the main theorem of this section.

\begin{theorem}\label{thm:main}
For any $\eps>0$ let $\hat Z^\eps$ denote the BPHZ renormalised lift of the regularised noise $\xi^\eps$ to a random admissible model and let $k\in\CG_-$ be any character. Then one has the identity
\begin{align}\label{eq:metatheoremmodel1}
\supp \renorm{k} \hat Z=\bigcap_{\eps>0}\overline{\bigcup_{\delta<\eps} \supp \renorm{k} \hat Z^{\delta}}.
\end{align}
Moreover, if we denote by $f^\eps:=f^{\xi^\eps}\in\CG_-$ the sequence of characters defined in Assumption~\ref{ass:CHBPHZcharacters}  (so that $f^\eps \to \fxi$ as $\eps \to 0$), then one has the stronger statement
\begin{align}\label{eq:metatheoremmodel2}
\supp \renorm{k} \hat Z=\overline{\bigcup_{\eps\in(0,1)}\supp \renorm{k\circ \fxi\circ(f^\eps)^{-1}}\hat Z^{\eps}}.
\end{align}

\end{theorem}
\begin{proof}\emph{(Assuming Proposition~\ref{prop:constantsInSupport})}
We first argue that (\ref{eq:metatheoremmodel1}) follows from (\ref{eq:metatheoremmodel2}). To see this, we introduce the sequence of characters $l^\eps \in \CG_-$ via the identity $k\circ \fxi\circ (f^\eps)^{-1}=l^\eps\circ k$, and we note that since $k\circ \fxi\circ (f^\eps)^{-1}\to k$ in $\CG_-$ it follows that $l^\eps\to \one^*$. By Lemma~\ref{lem:supportcontinuousmap} and the continuity of the action of the renormalisation group it follows that $\supp \renorm{k} \hat Z$ can be written as
\begin{equ}
\lim_{\eps\to 0}\overline {\bigcup_{\delta<\eps}\supp \renorm{l^\eps \circ k} \hat Z^\delta} 
= \lim_{\eps\to 0}\overline { \renorm{l^\eps} \bigcup_{\delta<\eps} \supp \renorm{k} \hat Z^\delta}
= \bigcap_{\eps>0}\overline{\bigcup_{\delta<\eps} \supp \renorm k { \hat Z^{\delta}} }.
\end{equ}

It remains to show  (\ref{eq:metatheoremmodel2}). The fact that $\supp \renorm{k} \hat Z$ is contained in the right-hand side follows trivially from the fact that 
\[
\renorm{k\circ \fxi\circ(f^\eps)^{-1}}\hat Z^{\eps}\to \renorm{k} \hat Z
\]
in probability in the space of models, so it remains to show the inverse inclusion. By Lemma~\ref{lem:translationoperatorZZ} we have the identity
\[
\supp \renorm{k\circ \fxi\circ(f^\eps)^{-1}} \hat Z^\eps = \overline{
	\{T_h \mfc(k\circ \fxi\circ \hat g^\eps): h\in\Omega_\infty\},
}
\] 
where we introduced the character $\hat g^\eps \in \CG_-$ via the identity $f^\eps \circ \hat g^\eps= g^\eps$. By Assumption~\ref{ass:CHBPHZcharacters} one has $\hat g^\eps \in \CH$, so that Proposition~\ref{prop:constantsInSupport} implies that
\[
\mfc(k\circ \fxi\circ \hat g^\eps)\in\supp \renorm k \hat Z\;.
\]
It remains to show that the translation operator $T_h$ leaves the support of $\renorm {k} \hat Z$ invariant, in the sense that for any smooth function $h\in\Omega_\infty$ one has 
\[
\supp T_h \renorm{k}\hat Z = \supp \renorm{k}\hat Z\;.
\]
This in turn is a corollary of  Lemma~\ref{lem:supporttranslation}, the fact that renormalisation and translation commute (see Theorem \ref{thm:translationoperator}) and Cameron--Martin's theorem.
\end{proof}

One consequence of Theorem~\ref{thm:main} is that the support of the limit model does in general depend on the choice of renormalisation $k \in \CG_-$. 
In the next result we show that for any fixed $k \in \CG_-$ there exists a Lie subgroup $\CH^k$ of $\CG_-$, such that changing renormalisation from $k$ to $l \circ k$ for some $l \in \CH^k$ does not change the support. More precisely, we have the following result.

\begin{corollary}
For any $k\in\CG_-$ let $\bar k := k\circ \fxi$, and denote by $\CH^k$ the Lie subgroup of $\CG_-$ obtained from $\CH$ by conjugation with $\bar k$, i.e.\ the subgroup given by
\[
\CH^k:= \bar k \circ \CH \circ \bar k^{-1}.
\]
Then, for any  $ l\in \CH^k$, one has
\[
\supp \renorm{l\circ k} \hat Z=\supp \renorm k \hat Z.
\]

Moreover, the groups $\CH^k$ are invariant under composing $k$ with any element of $\CH^k$, i.e. one has $\CH^k = \CH^{l \circ k}$ for any $l \in \CH^k$.
\end{corollary}
\begin{proof}
The supports of $\renorm k\hat Z$ and $\renorm{l\circ k} \hat Z$ are respectively characterised as the closure of all smooth translations of all models of the form $\mfc(h)$ and $\mfc(\tilde h)$ for some $h\in \bar k \circ \CH$ and some $\tilde h\in l\circ \bar k \circ \CH$. We are thus left to show that
\[
l \circ \bar k \circ \CH = \bar k \circ \CH,
\]
which is true if and only if $l\in\CH^k$.

The fact that $\CH^k$ is invariant under a change of renormalisation by $l\in\CH^k$ follows from 
the fact that
\begin{align*}
\CH^{l\circ k} 
&= l\circ ( \bar k\circ \CH\circ \bar k^{-1} )  \circ l^{-1} \\
&\subseteq \bar k \circ \CH\circ \bar k^{-1}  \circ 
( \bar k \circ \CH \circ \bar k^{-1} )
\circ \bar k \circ \CH \circ \bar k^{-1} =\CH^k\;,
\end{align*}
whence the claim follows.
\end{proof}
\begin{remark}
A consequence of the previous corollary is that, writing $e$ for the unit in $\CG_-$, the collection of  
cosets $\{ k \circ \CH^e  :  k\in\CG_-\}$
yields a foliation of $\CG_-$ into a family of manifolds of fixed dimension with the property that for any $k \in \CG_-$, the support of $\CR^l \hat Z$ is independent of $l \in k \circ \CH^e$.
\end{remark}

\subsection{Renormalisation group argument}\label{sec:renormalisationgroupargument}

In light of the last section it remains to show Proposition~\ref{prop:constantsInSupport}. For this we fix from now on a character $h\in \fxi\circ\CH$ and we will construct a sequence \label{idx:shift}$\rsn_\delta\in\SM_\infty$, $\delta>0$, of random smooth noises such that
\[
T_{\rsn_\delta}\hat Z \to \ZZ(h) \qquad\text{ in probability in }\CM_0
\]
as $\delta\to 0$. Together with the continuity of the translation operator and Lemma~\ref{lem:supporttranslation}, this immediately implies Proposition~\ref{prop:constantsInSupport}. This convergence essentially relies on two conditions. The first condition (\ref{eq:shiftednoisehomo}) guarantees that the noise cancels out in the limit $\delta\to 0$ and the second condition (\ref{eq:shiftedtreeexpectationA}) guarantees the correct behaviour of the expected values. Before stating the main proposition of this section, we introduce the following notation. 
\begin{definition}\label{def:sim}
Let $\sim$ denote the equivalence relation on $\TT_-$ given by setting $\tau \sim \tilde\tau$ if and only if $\fancynorm{\tau}_\fs = \fancynorm{\tilde\tau}_\fs$ and one has that the identity $[L(\tau),\ft]=[L(\tilde\tau),\ft]$ between multisets. We write $\TT_-/_\sim$ for the set of equivalence classes of $\TT_-$ with respect to $\sim$.
\end{definition}
We also fix an arbitrary total order \label{idx:prec}$\preceq$ on $\TT_-$ with the property that $\tau \preceq	 \tilde\tau$ whenever $\tau,\tilde\tau \in \TT_-$ are two trees such that either $\# E(\tau) < \# E(\tilde\tau)$, or $\# E(\tau) = \# E(\tilde\tau)$ and $\sum_{u \in N(\tau)}\fn(u) \le \sum_{u \in N(\tilde\tau)} \fn(u)$. 
We write $\TT_-^{\preceq\tau}\ssq \TT_-$ and $\TT_-^{\prec\tau}\ssq \TT_-$ for the set of trees $\tilde\tau \in\TT_-$ such that $\tilde\tau\preceq \tau$ and $\tilde\tau\prec \tau$, respectively. We denote the unital subalgebras of $\CT_-$ generated by $\TT_-^{\prec\tau}$ and $\TT_-^{\preceq\tau}$ by $\CT_-^{\prec\tau}$ and $\CT_-^{\preceq\tau}$ respectively and we point out that it follows from the properties of the coproduct $\cpmh$ that both of these algebras form Hopf algebras.

We use the total order $\preceq$ to select a subset of trees $\FT_- \ssq \TT_-$ in the following way.

\begin{definition}\label{def:FT}
For any equivalence class $\Theta \in \TT_-/_\sim$ we write $\TT_-(\Theta) \ssq \Theta$ for the set of trees $\tau \in \Theta$ with the property that there exists a linear combination of trees $\sigma \in \Vec {(\Theta \cap \TT_-^{\prec\tau})}$ such that
\[
\tau + \sigma \in \CJ.
\]
(Here, $\CJ$ is the ideal in $\CTm$ defined in Definition~\ref{def:CH}.) We also write $\FT_-(\Theta):=\Theta\backslash \TT_-(\Theta)$ and we define $\FT_- := \bigsqcup_{\Theta \in \TT_-/_\sim}\FT_-(\Theta)$.
\end{definition}

Later on in (\ref{eq:convergenceh}) we will be given a linear subspace $X$ of $\linspace{\TT_-}$ for which we can show relatively easily that $\CJ\cap \linspace{\TT_-} \ssq X$, and our goal will be to show that $X = \linspace{\TT_-}$. Definition~\ref{def:FT} is set up so that it suffices to show that $\FT_- \ssq X$. The total order $\preceq$ is chosen in such a way that we can show this inductively in the number of edges and the polynomial decoration of $\tau \in \FT_-$.

\begin{example}
Consider the case of 2D PAM equation, where $$\TT_- = \{\leafs, \treeExamplePAM,\treeExamplePAMa,\treeExamplePAMb \}$$ and where 
$\CJ$ is the ideal generated by $\leafs,\, \treeExamplePAM-\treeExamplePAMa-\treeExamplePAMb$. Here a bold edge with label $k=1,2$ denotes the derivative of the Poisson kernel with respect to $x_k$, and a circle denotes an instance of spatial white noise.
In this case we can choose the total order by setting $\leafs\preceq \treeExamplePAM\preceq\treeExamplePAMa\preceq\treeExamplePAMb$.
We then have $\TT_-/_\sim = \{ \{\leafs\}, \{\treeExamplePAM,\treeExamplePAMa,\treeExamplePAMb\} \}$, and further $\FT_-(\{\leafs\}) = \emptyset$ and $\FT_-(\{\treeExamplePAM,\treeExamplePAMa,\treeExamplePAMb\})
= \{\treeExamplePAM,\treeExamplePAMa\}$.

In particular, we have $\FT_-= \{\treeExamplePAM,\treeExamplePAMa\}$.
\end{example}

In Section~\ref{sec:shift} we will show the following proposition (see Proposition~\ref{prop:choice_of_a} below), for which we recall the notation $\Upsilon^\eta$ from Section~\ref{sec:BPHZ:theorem} and the spaces $\SMshifted$ and $\SMz$ of smooth shifted noises and (rough) noises from Definition~\ref{def:smooth:noise}.

\begin{proposition}\label{prop:shiftednoise}
There exists a sequence $\srn_\delta\in \SMshifted$, $\delta>0$, of smooth random noises such that 
\begin{equ}\label{eq:shiftednoisehomo}
\xi+\srn_\delta\to 0 \textqquad{ in }\SMz,
\end{equ}
and such that for any $\tau\in\FT_-$ one has
\begin{equ}\label{eq:shiftedtreeexpectationA}
\lim_{\delta\to 0}\lim_{\eps\to 0} \Upsilon^{\xi^\eps+\srn_\delta}M^{g^{\eps}}\tau=h(\tau).
\end{equ}
\end{proposition}

Given Proposition~\ref{prop:shiftednoise}, we can show the following result.

\begin{proposition}\label{prop:sequencetoconstant}
	Let $\srn_\delta$ be the sequence given by Proposition~\ref{prop:shiftednoise}. Then one has
	\[
	\lim_{\delta\to 0}T_{\srn_\delta} \hat Z= \mfc(h)
	\]
	in probability in the space of admissible models.
\end{proposition}

\begin{proof}
 We denote as before the BPHZ character for $\xi^\eps$ by $g^\eps:=g^{\xi^\eps}$, and we denote similarly by $g^{\eps,\delta}:=g^{\xi^\eps+\srn_\delta}$ the BPHZ character for the smooth noise $\xi^\eps+\srn_\delta$. We define a character $h^{\eps,\delta}\in\CG_-$ via the relation
\begin{align}\label{eq:character:h:g:g}
h^{\eps,\delta}\circ g^{\eps,\delta} =g^{\eps},
\end{align}
where $\circ$ denotes the group product in $\CGm$.
We show inductively with respect to $\prec$ that one has
\begin{align}\label{eq:convergenceh}
\lim_{\delta \to 0} \lim_{\eps \to 0} 
	h^{\eps,\delta}(\tau) = h(\tau)
\end{align}
Let first $\tau\in\FT_-$. Since for any tree $\tau \in \TT$ which contains at least one noise type edge one has
\[
\Upsilon^{\xi^\eps+\srn_\delta}M^{g^{\eps,\delta}} \tau\to 0
\]
in the limit $\eps\to 0$ and $\delta\to 0$ by Lemma~\ref{lem:convergenceboldPi} below, it follows that
\begin{align*}
h^{\eps,\delta}(\tau)
&=\Upsilon^{\xi^\eps+\srn_\delta}M^{g^{\eps,\delta}}M^{h^{\eps,\delta}}\tau
-
(h^{\eps,\delta}\otimes \Upsilon^{\xi^\eps+\srn_\delta}M^{g^{\eps,\delta}})(\cpmi- \Id\otimes\one)\tau
 \\
&= \Upsilon^{\xi^\eps+h_\delta}M^{g^{\eps}}\tau + o(1)\;,
\end{align*}
where $o(1)\to 0$ as $\eps \to 0 $ and $\delta \to 0$, so that (\ref{eq:convergenceh}) follows from (\ref{eq:shiftedtreeexpectationA}). 

Let now $\tau \in \TT_-\backslash \FT_-$. Let $\Theta \in \TT_-/_\sim$ be the equivalence class of $\tau$, and let $\tilde \tau\in \Vec {(\Theta \cap \TT_-^{\prec \tau})} $ such that $\sigma := \tau + \tilde \tau \in \CJ$.
We claim that $(f^\eps)^{-1}\circ h^{\eps,\delta}(\sigma)\to 0$ in the limit $\eps\to 0$ and $\delta\to 0$. Indeed, one has
\[
(f^\eps)^{-1} \circ h^{\eps,\delta}\circ f^{\eps,\delta} \circ \hat g^{\eps,\delta}=
\hat g^{\eps}\;,
\]
where $f^\eps,\hat g^\eps$ and $f^{\eps,\delta},\hat g^{\eps,\delta}$ are defined as in Assumption~\ref{ass:CHBPHZcharacters} for the noises $\xi^\eps$ and $\xi^\eps+\srn_\delta$, respectively, so that $f^\eps \circ \hat g^\eps = g^\eps$ and $f^{\eps,\delta} \circ \hat g^{\eps,\delta} = g^{\eps,\delta}$. By definition one has that $\hat g^{\eps,\delta}\in\CH$ and $\hat g^\eps\in\CH$, so that
\begin{align}\label{eq:fh}
(f^\eps)^{-1} \circ h^{\eps,\delta} \circ f^{\eps,\delta} = \hat g^\eps\circ(\hat g^{\eps,\delta})^{-1} \in \CH\;.
\end{align} 
By Assumption~\ref{ass:CHBPHZcharacters} the characters $f^\eps$ and $f^{\eps,\delta}$ converge to $\fxi$ and $\one^*$ in $\CG_-$, respectively. 
At this stage we would be done, if we knew a priori that $\lim_{\delta\to 0} \lim_{\eps \to 0}h^{\eps,\delta}$ exists in $\CG_-$ on $\CT_-^{\preceq\tau}$. By induction hypothesis, this is true on $\CT_-^{\prec\tau}$, so that it remains to show that $h^{\eps,\delta}(\tau)$ converges to \emph{something} in the limit $\eps\to 0$ and $\delta\to 0$. For this, note that (\ref{eq:fh}) vanishes when applied to $\sigma$, since $\sigma \in \CJ$. On the other hand, one has
\[
(\cpmh \otimes \Id)\cpmh\sigma \in 
	(\one \otimes \tau\otimes\one)  + 
	(\CT_-^{\preceq\tau}\otimes \CT_-^{\prec\tau}\otimes \CT_-^{\preceq\tau})\;,
\]
and we conclude using the induction hypothesis, which implies in particular that
\[
(f^\eps)^{-1} \otimes h^{\eps,\delta} \otimes f^{\eps,\delta} 
\]
converges on $\CT_-^{\preceq\tau}\otimes \CT_-^{\prec\tau}\otimes \CT_-^{\preceq\tau}$.

It follows that $\lim_{\delta\to 0}\lim_{\eps\to 0}(f^\eps)^{-1}\circ h^{\eps,\delta} = (\fxi)^{-1} \circ h$ on $\CT^{\preceq\tau}_-$. Since $\CT_-^{\preceq\tau}$ is a Hopf subalgebra of $\CT_-$, we conclude that one also has $\lim_{\delta\to 0}\lim_{\eps\to 0} h^{\eps,\delta}=h$ on $\CT^{\preceq\tau}_-$, and this concludes the proof of (\ref{eq:convergenceh}).

The remaining proof is now straightforward. We first compute
\begin{align*}
T_{\srn_\delta} \hat Z^\eps 
= \renorm{g^\eps} \Zcan(\xi^\eps+\srn_\delta)
=\renorm{h^{\eps,\delta}} \renorm{g^{\eps,\delta}}\Zcan(\xi^\eps + \srn_\delta).
\end{align*}
It follows from \cite[Thm.~2.33]{ChandraHairer2016} that $\lim_{\delta\to 0}\lim_{\eps\to 0}\renorm{g^{\eps,\delta}}\Zcan(\xi^\eps + \srn_\delta)=\Zcan(0)$ in probability in the space of models.
Using the fact that the renormalisation group $\CG_-$ acts continuously onto the space of admissible models, together with the fact that $\lim_{\delta\to 0}\lim_{\eps\to 0} h^{\eps,\delta}=h$, we obtain
\[
\lim_{\delta\to 0}\lim_{\eps\to 0} T_{\srn_\delta} \hat Z^\eps=\renorm h \Zcan(0)= \mfc (h),
\]
and this concludes the proof.
\end{proof}

\begin{lemma}\label{lem:convergenceboldPi}
For any  $\tau \in \CT$ the map $\eta \mapsto \E\PPi^\eta M^{g^\eta}\tau (0)$ is continuous as a map from $\SM_0$ into $\R$.
\end{lemma}
\begin{proof}
The continuity of the map $\eta \mapsto \E (\PPi^\eta M^{g^\eta} \tau)(\vphi)$ for any fixed test function $\vphi \in \CC_c^\infty(\domain)$ is a consequence of \cite{ChandraHairer2016}. To show the lemma, 
it thus suffices to find, for any fixed $\tau \in \CT$, a test function $\vphi$ such that $\E \PPi^\eta M^{g^\eta}\tau(0) = \E (\PPi^\eta M^{g^\eta}\tau)(\vphi)$ for any smooth noise $\eta \in \SM_\infty$. For this we recall that one has
\begin{align}\label{eq:stationary:model}
(\hat {\PPi}^\eta \tau)(z) \sim (\hat{\PPi}^\eta \otimes g_z) \Delta^+ \tau(0), \qquad z \in \domain,
\end{align}
where $\Delta^+$ denotes the coproduct on the structure group $\CG_+$ and $g_z \in \CG_+$ is defined by setting $g_z(X_i) := z_i$, and $g_z(\tau) = 0$ for any non-polynomial $\tau \in \CT_+$ (in the language of \cite[Def.~6.16]{BrunedHairerZambotti2016} this follows from the fact that $\hat{\PPi}^\eta$ is \emph{stationary}). It follows that
\[
\E\PPi^\eta M^{g^\eta}\tau (z)= ((\E \hat {\PPi}^\eta)\otimes g_z) \Delta^+ \tau (0)
= :
P^\eta(z),
\]
where $P^\eta$ is a polynomial depending on $\eta$ with $\deg P^\eta \le \sum_{ u \in N(\tau) }|\fn(u)|_\fs$. Since, for any fixed degree $N$, one can find a test function $\vphi$ integrating to $1$ and 
such that $\int x^k \vphi(x)\,dx = 0$ for all $0 < |k| \le N$, 
we conclude that $P^\eta(0) = \int P^\eta(z) \vphi(z)$, and thus
\[
\E\PPi^\eta M^{g^\eta}\tau (0)= \E (\hat {\PPi}^\eta \tau)(\vphi)\;,
\]
which finishes the proof.
\end{proof}

\maybeNotNeeded{
\begin{lemma}
Let $\eta^\eps\in\SM_\infty$ be a sequence of smooth noises such that $\eta^\eps\to 0$ in $\SM_0(\fc)$. Let also $\TT^\dagger \ssq \TT$ denote the set of trees $\tau$ containing at least one noise type edge. Then for any tree $\tau\in\TT^\dagger$ one has
\[
\E\PPi^{\eta^\eps}M^{g^{\eta^\eps}}\tau(0)\to 0
\]
as $\eps\to 0$.
\end{lemma}
\begin{remark}
The point here is that we allow $\tau$ to have positive homogeneity.
\end{remark}
}

\subsection{Corollaries}

We get the following characterisation of $\CH$.

\begin{corollary}\label{cor:CHsmallestgroup}
The group $\CH$ is the smallest Lie subgroup of $\CG_-$ with the property that the statement of Assumption~\ref{ass:CHBPHZcharacters} holds.
\end{corollary}
\begin{proof}
Let $\CK\subseteq\CG_-$ be any Lie subgroup of $\CG_-$ such that the statement of Assumption~\ref{ass:CHBPHZcharacters} holds and denote the corresponding characters for the noises $\xi^\eps$ and $\xi^\eps+\srn_\delta$ by $\tilde f^\eps$ and $\tilde f^{\eps,\delta}$. It then follows that $\tilde f^\xi:=\lim_{\eps\to 0}\tilde f^\eps$ exists and $\lim_{\delta\to 0}\lim_{\eps\to 0}\tilde f^{\eps,\delta}=\one^*$.

Let $h\in f^\xi\circ\CH$ be any character and let $\srn_\delta$ be the sequence of smooth shifts defined in Proposition~\ref{prop:shiftednoise}. Denoting as in the proof of Proposition~\ref{prop:sequencetoconstant} by $g^\eps$ and $g^{\eps,\delta}$ the BPHZ characters for $\xi^\eps$ and $\xi^\eps+\srn_\delta$ and $h^{\eps,\delta}\in\CG_-$ the character defined via the relation~(\ref{eq:character:h:g:g}), then it follows from (\ref{eq:convergenceh}) that $h^{\eps,\delta}\to h$. On the other hand $\CK$ is a subgroup, and by definition one has $(\tilde f^\eps)^{-1} \circ g^\eps \in \CK$ and $(g^{\eps,\delta})^{-1} \circ \tilde f^{\eps,\delta} \in \CK$, so that
\[
(\tilde f^\eps)^{-1} \circ h^{\eps,\delta} \circ \tilde f^{\eps,\delta} \in \CK.
\]
Note now that since $\CK$ is a Lie subgroup of a nilpotent (and therefore simply connected)
Lie group, it is closed, see for example the introduction of \cite{MR1201409}.
Since $\tilde f^\eps \to \tilde f^\xi$ and $\tilde f^{\eps,\delta}\to \one^*$, one has
$
h\in \tilde f^\xi \circ \CK
$, whence it follows that
\[
\fxi\circ\CH\subseteq \tilde f^\xi\circ \CK\;.
\]
Since the identity belongs to $\CH$ and $\CK$ is
a group, we conclude that $(\fxi)^{-1}\circ\tilde f^\xi \in \CK$ and therefore that  $\CH\subseteq\CK$.
\end{proof}

An interesting, although somewhat unrelated, corollary is the following statement.

\begin{corollary}\label{cor:singular:law}
Let $k,l\in\CG_-$ be two characters with $k\ne l$ and $k(\Xi) = l(\Xi) = 0$ for any noise type $\Xi \in \FL_-$. Then the laws of $\renorm k \hat Z$ and $\renorm l \hat Z$ are singular with respect to each other.
\end{corollary}
\begin{remark}
Even though the laws of $\renorm{k} \hat Z$ and $\renorm l \hat Z$ are mutually singular, their topological supports may still be the same.
\end{remark}
\begin{proof}
This is a Corollary of Proposition~\ref{prop:sequencetoconstant}. Indeed, for any random smooth noise $\srn$ such that $\srn:\Omega\to\Omega_\infty$ is continuous we denote by $\tilde T_\srn:\CM\to\CM$ the continuous map
\[
\tilde T_\srn \CZ(\PPi):= T_{\srn((\PPi\Xi)_{\Xi\in\FL_-})}\CZ(\PPi).
\]
This is well-defined since the map $\CZ(\PPi)\mapsto (\PPi\Xi)_{\Xi\in\FL_-}$ is continuous from $\CM$ into $\Omega$. 
In particular, for any $k\in\CG_-$ which acts trivially on $\FL_-$ one has the identity
\[
T_{\zeta(\xi)} \renorm k\hat Z(\xi) = \tilde T_\zeta \renorm k \hat Z(\xi)\;.
\]
If we now denote by $\srn_\delta$ the sequence defined in Proposition~\ref{prop:sequencetoconstant}, then it follows that, for a suitable subsequence $\delta \to 0$ sufficiently fast,
we have the $\P$-almost sure limit 
\begin{equ}
\lim _{\delta \to 0}\tilde T_{\srn_\delta} \renorm k \hat Z = \mfc(k)\;.
\end{equ}
Since $\mfc(k) \ne \mfc(l)$ (and both are deterministic), the claim follows. 
\end{proof}
\begin{remark}
In case of space-time white noise, we believe that the same statement holds for the laws of $\renorm k \hat Z$ and $\renorm l \hat Z$ restricted to any open subset $U$ of $\spacetime$ which contains the initial time-slice $\{t=0\}$. 
\end{remark}

\section{Constraints between renormalisation constants}
\label{sec:constraints}

The goal of this section is to show that Assumption~\ref{ass:CVz}, which we assume to hold in this section, implies Assumptions~\ref{ass:CJHopfIdeal} and~\ref{ass:CHBPHZcharacters} made at the beginning of Section~\ref{sec:support:thm:models}.

\begin{proposition}\label{prop:ass_implies_ass}
Assumption~\ref{ass:CVz} implies Assumptions~\ref{ass:CJHopfIdeal} and~\ref{ass:CHBPHZcharacters}.
\end{proposition}

\noindent Assumption~\ref{ass:CJHopfIdeal} is proven in Corollary~\ref{cor:CJ:hopf:ideal}, Assumption~\ref{ass:CHBPHZcharacters} follows from Lemma~\ref{lem:fnkn} and Lemma~\ref{lem:fnknbound}.
 To get a feeling for the ideal $\CJ$ first we consider a couple of examples of generators.

\begin{example}\label{ex:CJ:PAM}\label{ex:const:first}
In the case of the three dimensional PAM equation, one has 
\[
\treeExamplePAM
-
\treeExamplePAMa
-
\treeExamplePAMb
-
\treeExamplePAMc
\in \CJ.
\]
As above, a bold edge with label $k=1,2,3$ denotes the derivative of the Poisson kernel with respect to $x_k$, and a circle denotes an instance of spatial white noise.The reason for this is the relation
\[
K - \partial_1 K * \partial_1 K
-
\partial_2 K * \partial_2 K
-
\partial_3 K * \partial_3 K
=
0
\]
for the Poisson kernel $K$ in three dimensions. Note that the corresponding linear combination between the renormalisation constants does not vanish (since we work with a spatial truncation of the Poisson kernel), but it is easy to see that it is bounded uniformly in the limit.
\end{example}
\begin{example}
Another possible source of constraints is given by ``total derivatives''. For instance in case of the generalised KPZ equation one has
\[
\treeExampleKPZa
+
\treeExampleKPZb
+
\treeExampleKPZc
\in \CJ,
\]
where the white circles denote instances of white noise (the circles are allowed to denote different instances of white noise, but with the convention that circles that appear at the same position in the three trees correspond to the same white noise). These two classes of constraints were recently used quite systematically in \cite{MateSymmetries}.
\end{example}

\begin{example}
A third possible constraint comes from symmetries, for instance in case of 
\[
\partial_t u = -\Delta^2 u + g(u,\partial_x u) (-\Delta)^{1-\kappa} \eta
\]
with $\eta$ spatial white noise in 2 dimensions and $\kappa>0$, one has 
\[
\treeExampleAsym\in \CJ.
\]
Here an edge denotes the (truncation of the) Green's function for $\partial_t-\Delta^2$, the bold edge denotes its spatial derivative (say with respect to $x_1$) and a node $\leafs$ an instance of $(-\Delta)^{1-\kappa}\eta$.
\end{example}
\begin{example}\label{ex:const:last}
Finally, a possible source of constraints comes from moving the root. For instance, if one considers a 
couple of interacting forward-backward generalised KPZ equations, one has
\[
\treeExampleRootA \,- \,	\treeExampleRootB \in \CJ
\]
where a red edge denotes the backward heat kernel.
\begin{remark}
Forward-backward equations appear naturally in the context of the dual to the tangent equation when studying existence of densities for solutions to stochastic equations, see \cite{GassiatLabbe2017,Schoenbauer2018} for this construction in the context of SPDEs. Consider e.g. the KPZ equation $\partial_t h = \partial_x^2 h + |\partial_x h|^2 + \xi$ with tangent equation $\partial_t v = \partial_x^2 v + 2 \partial_x h \partial_x v + f$, where $f$ denotes a Cameron-Martin function. The dual of the tangent equation is given by
$-\partial_t w = \partial_x^2 w - 2\partial_x h \partial_x w + g$.
Some of the trees needed to solve the coupled equation for $(h,v,w)$ contain the backward heat kernel.
\end{remark}
\end{example}
\begin{remark}
It is unclear at this point whether all constrains that show up in reasonable examples are of the form described above. 
One could of course always construct more contrived examples by simply choosing the integration kernels themselves
to satisfy certain constraints. The approach chosen in this article aims for the largest possible generality 
while avoiding having to explicitly characterise these constraints. Instead, we show directly that the ideal 
generated by these constraints always has ``nice'' algebraic properties (Assumption~\ref{ass:CJHopfIdeal}) and 
that the BPHZ characters are ``well-behaved'' in the sense that they respect these constraints up to discrepancies of order~1 (Assumption~\ref{ass:CHBPHZcharacters}).
\end{remark}

We first generalise the notation \eqref{eq:CK} by including noises. We define the space
$\SM_\infty^\star := \SMsinf \sqcup \{\one\}$ and its closure $\SMsz$ under the norm (\ref{eq:noise:norm}).
Here, we let $\one$ act on any noise type $\Xi \in \FL_-$ by setting $\one(\Xi):=1$.\footnote{Note that $\one \notin \SM_\infty$, since $\E \one\ne 0$.}
With this notation, we now make the following key definition.
\label{idx:tildeUpsilon}
\begin{definition}\label{def:bar:Upsilon:large:scale:2}
Given a tree $\tau \in \CT$,  we define for any $\eta\in\SM_\infty^\star$, any $\psi \in \cutoffspace$ and any large-scale kernel assignment  $R=(R_\ft)_{\ft\in\FL_+}$ with $R_\ft \in \CC_c^\infty(\bar\domain)$ the constant
\begin{align}\label{eq:bar:Upsilon}
\evalnA{\eta}{\psi}{R} \tau:=
\ctrL{\tau}{K+R}{
\zeta
}
\end{align}
where $\zeta(x) := 	\Big( \E \prod_{e \in L(T)} \eta_{\ft(e)}(x_{e}) \Big) \, \psi (x)$.
We will write $\evalA{\psi}{R} := \evalnA{\one}{\psi}{R}$.
\end{definition}

As was already remarked below Definition~\ref{def:CJcon}, we will show in Theorem~\ref{thm:evaluation:trees:largescale} that for any $\tau \in \TT_-$ the limit $\evalnA{\eta}{\psi}{R} \tau$ does indeed exist as the smooth kernels $R_\ft$ approach $\hat K_\ft-K_\ft$, and we denote this limit by \label{idx:tildeUpsilonR}$\evalnA{\eta}{\psi}{}$. 
We write also $\evaA{\psi}:= \evalnA{\one}{\psi}{}$. All operators introduced here are multiplicatively extended to characters on the Hopf algebra $\CT_-$.

\begin{remark}
Note that convergence when $R_\ft$ approaches $\hat K_\ft-K_\ft$ relies on the smooth cut-off function $\psi$. This is the reason for introducing this cut-off in the definition \eqref{eq:bar:Upsilon}.
\end{remark}
The following lemma gives a useful alternative description of $\CJ$.
\begin{lemma}
Under Assumption~\ref{ass:CVz} the ideal $\CJ$ is generated by all $\tau \in \Vec{\TT_-}$ such that
$\evalnA{\eta}{\psi}{} \tau = 0$ for any $\eta \in \SM_\infty^\star$ and any $\psi \in \cutoffspace$.
\end{lemma}
\begin{proof}
Comparing \eqref{eq:bar:Upsilon} and Definition~\ref{def:CJcon}, we only have to show that \[
\ctrLhat{\tau}{\psi}
=0
\]
for any $\tau \in \CJ$ and any $\psi \in \cutoffspace$. Rescaling $\psi \to \psi^\eps$ and exploiting the homogeneous behaviour of the integration kernels $\hat K_\ft$ it suffices to consider linear combinations of trees $\tau=\sum_{i\le r} c_i \tau_i$ such that $[L(\tau_i),\ft]$ and $\alpha:=\homofancy{\tau_i}$ do not depend on $i\le r$. 
In particular, it suffices to consider the cases $\tau_i \in \CV_0$ for all $i\le r$ or $\tau_i \notin \CV_0$ for all $i\le r$. In the former case Assumption~\ref{ass:CVz} guarantees that $\ctrLhat{\tau}{\psi}=0$ for any $\psi \in \cutoffspace$.
 In the latter case, note that $\tau\in \CJ$ implies that 
\begin{align}\label{eq:small:1}
\ctrLhat{\tau}{\psi} = a(\tau) \int \psi
\end{align}
for some $a(\tau) \in \R$ and all $\psi \in \cutoffspace$. However, the transformation $\psi^\eps(x):= \eps^{-|\fs|} \psi(\eps^{-\fs} x)$ leaves the right hand side (\ref{eq:small:1}) invariant, while the left hand side is transformed as $\ctrLhat{\tau}{\psi^\eps}  = \eps^{\alpha + (\half \#L(\tau_i) - 1)|\fs| } \ctrLhat{\tau}{\psi}$, which is a contradiction unless $\alpha = -\#L(\tau_i) \shalf + |\fs|$. Unless $\#L(\tau_i)=2$ one has $\alpha\le -\shalf$, contradicting Assumption~\ref{ass:main:reg}. If $\#L(\tau_i)=2$, then one has $\alpha=0$ and thus $\tau \in \CV_0$, in contradiction with $\tau \notin \CV_0$.
\end{proof}

An important remark is that if the cutoff functions $\psi$ are chosen such that $\psi \equiv 1$ in a large enough neighbourhood of the origin, then one has the identity
\[
\evalnA{\eta}{\psi}{0}\tau
=
\E \PPi^\eta \tau (0)
=: 
\ev{\eta}\tau
\]
for any smooth noise $\eta \in \SM_\infty$. 

One may wonder what the function $\psi$ in this notation is trying to accomplish. We want to study the limit of $\evalnA{\eta}{\psi}{R}\tau$ in which $R_\ft$ converges to $\hat K_\ft - K_\ft$, where $\hat K_\ft: \bar\domain\to \R$ is the homogeneous extension of the integration kernel to the whole space, see Section~\ref{sec:kernels}. Without the cutoff function $\psi$, this quantity has no chance of converging in general. 
However, we will see that the presence of the cutoff $\psi$ is sufficient for this limit to exist. The fact that we cannot get rid of the large-scale cutoff completely is no surprise. Indeed, even for $\eta=\one$ this is not true:
\begin{example}
Consider the cherry tree $\cherry$ in $\Phi^4_3$. We obtain 
\[
\evalA{\psi}{R} \cherry = \int_{\R^+ \times \R^3} dx \int_{\R^+ \times \R^3} dy \, (K+R)(x) (K+R)(y) \psi(x-y)
\]
where we identify $\psi \in \bar\CC_c^\infty(\bar\domain^{L(\tau)}) \simeq \CC_c^\infty(\bar\domain)$.
We can only guarantee that this is finite as $R \to \hat K - K$ if $\psi$ is compactly supported.
\end{example}
On a more technical level, this issue is related to the bound on the degrees of tight partitions introduced in \cite[Sec.~4]{Hairer2017}.

There is however one big advantage of the large-scale cutoff introduced by $\psi$ over the one given by simply choosing a compactly
supported kernel $K$: 
the latter ``sees'' the interior structure of the tree $\tau$, while the former only ``sees'' the noise type edges.  When we work out properties of the ideal $\CJ$ later on, this becomes crucial, as it can happen that two trees with the property that the evaluation (\ref{eq:bar:Upsilon}) only 
differs due to the large scale cutoff have distinct interior structure (compare e.g. Example~\ref{ex:CJ:PAM}). However, such trees will always carry the same multiset 
of noise types, so that the cutoff introduced by $\psi$ as in (\ref{eq:bar:Upsilon}) will affect each of them in precisely the same way, 
thus not destroying exact identities between their renormalisation constants.

\subsubsection*{A motivating example}

The main difficulty in proving Assumptions~\ref{ass:CJHopfIdeal} and~\ref{ass:CHBPHZcharacters} is to determine the algebraic structure of a tree $\tau$ drawing only on the analytic information given by $\CK_{\hat K} \tau$. The strategy to show the first of these results, namely that $\CJ$ is a Hopf ideal, and the second one, namely that the BPHZ characters $g^\eta$ ``almost'' annihilate $\CJ$, are quite similar. The main step is to show how $\CJ$ interacts with the coproducts $\cpm$ and $\cpmi$, compare \eqref{eq:inclusion:CJ} and \eqref{eq:inclusion:CJ:hat} in Proposition~\ref{prop:interaction:wlCJ:Delta} below. The  interaction property of $\CJ$ with $\cpm$ gives immediately the Hopf ideal property, while the statement about the BPHZ characters needs a further argument carried out in Section~\ref{sec:constraints:rigidities}. 

Consider as an example two trees coming from the generalised PAM equation in 3D. Recall that in this equation we consider purely spatial white noise and the integration kernel is given by the 3D Poisson kernel $P$. One then has
\begin{equ}[e:inputCJ]
\tau_1 - \tau_2 := \treeExampleKPZm - \treeExampleKPZn \in \CJ\;,
\end{equ}
This can be seen by noting that $P$ is invariant under the transformation $x \mapsto -x$. Here we use different colours to indicate different (hence independent) white noises. As part of the proof of Assumptions~\ref{ass:CJHopfIdeal} and~\ref{ass:CHBPHZcharacters} we have to show respectively that 
\begin{align}\label{eq:outline:of:algebra}
\cpm \Big(
	\treeExampleKPZm - \treeExampleKPZn 
\Big)\in \CJ \otimes \CTm + \CTm \otimes \CJ
\qquad \text{ and }\qquad
\Big|g^\eta(\treeExampleKPZm) - g^\eta(\treeExampleKPZn)\Big| \lesssim 1,
\end{align}
where the second statement is uniform over $\eta$ with $\|\eta\|_\fs \le C$.

\begin{remark}
The reason why we have to bound the linear difference (as opposed to the ``difference'' with respect 
to the group operation) is that these two turn out to be the same in the present example. In general 
the second bound in (\ref{eq:outline:of:algebra}) does not hold and should be replaced with 
$ (\tilde f^\eta \circ g^\eta) (\tau_1 - \tau_2) = 0$ where $|\tilde f^\eta| \lesssim 1$. 
(Here $\tilde f^\eta = (f^\eta)^{-1}$ is the group inverse of the character defined in 
Assumption~\ref{ass:CHBPHZcharacters}. Since $f \mapsto f^{-1}$ is a uniformly bounded operation on $\CGm$, bounding 
$f^\eta$ and $\tilde f^\eta$ are equivalent.) Of course, boundedness is not quite sufficient and we will 
later show the stronger statement of continuity with respect to $\|\cdot\|_\fs$. (This is not equivalent 
since $\eta \mapsto f^\eta$ is not a linear map.)
\end{remark}

Let us first convince ourselves ``by hand'' that (\ref{eq:outline:of:algebra}) holds. To see the first statement, it suffices to note that
\[
\cpm \Big(
	\treeExampleKPZm - \treeExampleKPZn 
\Big)
=
\Big(
	\treeExampleKPZm - \treeExampleKPZn 
\Big) \otimes \one
+
\one \otimes \Big(
	\treeExampleKPZm - \treeExampleKPZn 
\Big)
+
\treeExampleKPZmcpmL
\otimes
\Big(\treeExampleKPZmcpmR
 -
\treeExampleKPZncpmR
\Big)
\]
and $\treeExampleKPZmcpmR
 -
\treeExampleKPZncpmR 
\in \CJ
$
holds with the same argument as above. (In fact both of these trees are individually elements of $\CJ$ for symmetry reasons. Actually, in this case one also has $\treeExampleKPZmcpmL \in \CJ$ for symmetry reasons. Both of these statements are however not generic. They would for instance not hold if the Poisson kernel was replaced by a non-symmetric kernel.) 
Here we draw a cross into the cricle to denote a polynomial decoration, and a bold edge denotes an edge carrying a derivative decoration.
For the second statement in (\ref{eq:outline:of:algebra}) one can calculate
\begin{align}\label{eq:intro:1}
g^\eta \Big(
	\treeExampleKPZm - \treeExampleKPZn 
\Big)
=
- \E \PPi^\eta \Big(
	\treeExampleKPZm - \treeExampleKPZn 
\Big)(0)
+
\E \PPi^\eta \treeExampleKPZmcpmiL(0)\,
\E \PPi^\eta 
\Big(\treeExampleKPZmcpmiR
 -
\treeExampleKPZncpmiR
\Big)(0)
=
0.
\end{align}
Note that the fact that this expression vanishes identically is not really intrinsic. For instance, the various Poison kernels could be associated to different components of the equation, and in principle we could choose different large-scale cutoff's, which would make the expression above non-zero (but it would remain order $1$). This may not seem like a natural thing to do, but it is sometimes unavoidable, compare Example~\ref{ex:CJ:PAM}.

The goal of this section is to automatise these arguments, drawing only on the information that $\evaA{\psi}(\tau_1 - \tau_2) = 0$ for any smooth function $\psi \in \bar\CC_c^\infty( \bar\domain^{[\leafslow,\leafslow,\leafsalow,\leafsalow]} )$. We write elements in the domain as $x = (\introx)$.
Let $\phi \in \CC_c^\infty(\bar\domain)$ be any smooth, symmetric (under $x \mapsto -x$) test function, define $\phi^{(\eps)}:= \eps^{-3}\phi(\eps^{-1} \cdot)$, and let $\psi^\eps(\introx) := \psi(\introx) \phi^{(\eps)}(\introxs)$. We also write $\check\psi(\introxa, \introxd) := \psi(\introxd,\introxd, \introxa)$, and we denote by $\evaA{\check\psi} \treeExampleKPZmcpmiR$ and $\evaA{\check\psi} \treeExampleKPZncpmiR$ the quantity defined analogously to (\ref{eq:bar:Upsilon}), but where the additional variable $\introxd$ corresponds to the node which was generated by contracting the subtree $\treeExampleKPZmcpmiL$ (in the current examples, the only node without a noise). This rather ad hoc notation is resolved later on by the introduction of legs, see below.
We then arrive at the following diagram	
\begin{center}\label{eq:diagram}
\begin{tabular}{ L L L L L }
\evaA{\psi^\eps} \treeExampleKPZm & - & \evaA{\psi^\eps} \treeExampleKPZn  & = & 0	
\\
- &   \phantom{\int^{\int^{\int^\int}}_{\int^{\int^\int}}}  &  - &	 &
\\
\evaA{\phi^{(\eps)}} \treeExampleKPZmcpmiL 
	\evaA{\check\psi} \treeExampleKPZmcpmiR 
	&  - & 
	\evaA{\phi^{(\eps)}} \treeExampleKPZmcpmiL 
		\evaA{\check\psi} \treeExampleKPZncpmiR 
	& = & 0
\\
 \lesssim 1 & \phantom{\int^{\int^{\int^\int}}} & \lesssim 1 & &
\end{tabular}
\end{center}

The equality in the first line is the analytic input we are given from \eqref{e:inputCJ}. The uniform bounds on the differences vertically are a consequence of the analytic BPHZ theorem \cite{Hairer2017}, see also  Proposition~\ref{prop:eps:beta:bound} below.
The equality on the second line is what we infer. Note that in a first step we only deduce a uniform bound, 
however we can make use of the fact that the integration kernels are homogeneous functions, so that we 
 know a priori that the expressions in the second line are proportional to $\eps^{\alpha}$ for some 
homogeneity $\alpha<0$. Both statements can hold simultaneously only if the quantity vanishes identically. It then follows in particular that
\begin{align}\label{eq:intro:2}
\evaA{\phi} \treeExampleKPZmcpmiL \evaA{\psi} \treeExampleKPZmcpmiR 
-
\evaA{\phi} \treeExampleKPZmcpmiL 
\evaA{\psi} \treeExampleKPZncpmiR 
=0
\end{align}
for any symmetric test function $\phi,\psi \in \CC_c^\infty(\bar\domain)$ (here we naturally identify $\bar\CC_c^\infty(\bar\domain^{ [ \leafslow,\leafslow ] })$ and $\bar\CC_c^\infty(\bar\domain^{ [ \leafsalow,\leafsalow ] })$ with the space of symmetric functions in  $\CC_c^\infty(\bar\domain)$). 
Note that in general there may be more than one divergent subtree. We then perform the strategy above with all possible divergent subtrees, by splitting the multiset $[\leafslow,\leafslow,\leafsalow,\leafsalow]$ in two parts in all possible ways (the derivation above would then correspond to $[\leafslow,\leafslow],[\leafsalow,\leafsalow]$). 

Comparing \eqref{eq:intro:2} and \eqref{eq:intro:1}, and using the fact that  the function $\phi(x-y):=\E[\eta(x)\eta(y)]$ is an element of $\CC_c^\infty(\bar\domain)$, we deduce that
\[
\hat g^\eta\Big(
	\treeExampleKPZm - \treeExampleKPZn 
\Big) = 0
\]
for any smooth, centred, stationary noise $\eta$. Here, $\hat g^\eta$ is a character which is defined similarly to the BPHZ character, but where the large-scale cutoff of the integration kernels is removed and instead a large-scale cutoff is introduced between any pair of nodes, see Definition~\ref{def:bar:Upsilon:large:scale} above and (\ref{eq:g:wl:eta:psi:R}) below. 

Let us review the outline so far from a more algebraic perspective. We have essentially proven that,
assuming that 
\begin{align}\label{eq:intro:CJhatCJ}
\treeExampleKPZm - \treeExampleKPZn \in \CJ \qquad\Rightarrow\qquad \cpmi \Big(
	\treeExampleKPZm - \treeExampleKPZn 
\Big)
\in \CJ \otimes \CTmhat + \CTm \otimes \hat\CJ,
\end{align}
where $\hat\CJ \ssq \CTmhat$ is an ideal defined analogously to $\CJ$. (We refrain from given a precise definition here, since there are some subtleties; most notably the fact that $\evaA{\eta}$ is in general not well-defined on trees of positive homogeneity. We refer to Definition~\ref{def:wlCJ} for the definition of an ideal that mirrors this idea.)
We then use that $\cpm = (\Id \otimes \p_-) \cpmi$ and $\CJ = \p_- \hat\CJ$ to conclude that $\CJ$ is a Hopf ideal, which concludes the outline of the proof of Assumption~\ref{ass:CJHopfIdeal}.
The remaining argument to conclude the outline of the proof of Assumption~\ref{ass:CHBPHZcharacters} is to bound the difference of $g^\eta$ and $\hat g^\eta$ with respect to the group product in $\CGm$, which we will do in Section~\ref{sec:constraints:rigidities}.

\subsubsection*{The problems ahead}
There are several points that complicate this line of argument in general:
\begin{itemize}
\item One can have more complicated sub-divergencies, in particular one can have divergent sub-forests instead of just single trees. To deal with this issue, we will introduce a test function for each pair of noises (Definition~\ref{def:FN}), which will give us the flexibility to trigger any sub-divergence by rescaling these test functions in all possible ways.
\item A bigger issue is the presence of derivatives hitting the test function. Implementing the above strategy without a proper algebraic framework leads to significant notational difficulties. Instead, we opt for a systematic extension of the algebraic framework by introducing the notion of ``legs'', against which our test functions are integrated. Formally, we do this via an extension of the regularity structure, see Section~\ref{sec:ext:reg:strct} for the details. The point here is that legs can have non-vanishing derivative decorations.
\item Every leg has a unique partner leg, and we call a tree \emph{properly legged} (Definition~\ref{def:propery:legged}), if for any pair of verices $u$ and $v$ with $u\ne v$, and both u and v carrying noises, there exists a unique leg incident to $u$ such that its partner is incident to $v$. We ultimately need to understand how this ``properly legged'' property interacts with the coproduct, which leads to the construction of algebras $\plCT$ and $\plCThat$, which are related to the algebras $\wCT$ and $\wCThat$ (we colour them to indicate that they are spaces generated by trees containing legs, c.f.\ Sec.~\ref{sec:ext:reg:strct}). We refer to Section~\ref{sec:algebra} for details.
\item Noises are in general indistinguishable. We need to distinguish them at the algebraic level to carry out the argument above, and only afterwards factor out the necessary ideals given by ``identifying'' noises that we made distinguishable (Definition~\ref{def:symCT}). (Actually, it suffices for us to break the symmetry at the level of legs.)
\item We need to make precise what exactly we need to subtract in general in order to see the cancellations inferred above. For this we need a general strategy of rescaling the test functions (c.f.\ \eqref{eq:phi:rescale},\eqref{eq:definition:deg:FJ}) and a general bound in the spirit of the BPHZ theorem (c.f.\ Proposition~\ref{prop:eps:beta:bound}). We draw here on the results of \cite{Hairer2017} rather than \cite{ChandraHairer2016}, since we deal with kernels of unbounded support.
\item Finally, we have to show that the evaluation $\evalnA{\eta}{\psi}{}\tau$ is well-defined, at least on a large enough set of trees $\tau$. We refer the reader to Theorem~\ref{thm:evaluation:trees:largescale} and Lemma~\ref{lem:super:regularity:implies:large:scale:bound}.
\end{itemize}

\subsubsection*{Outline of the section}
The plan is now as follows. We enlarge in Section~\ref{sec:ext:reg:strct} the regularity structures $\CT$ to a regularity structure $\wT$ by adding a sufficient number of new types (which we call ``leg types'', but are treated as noise types with just slightly negative homogeneity) and we allow any number of them (up to a large enough constant) to be incident to any node $u$ of any tree $\tau \in \CT$. We then construct spaces $\wCT$ and $\wCThat$ analogously to $\CTm$ and $\CTmhatex$. We show that one can remove the large-scale cutoff in the sense that $\evaA{\eta}\tau$ exists (at least for a large class of trees $\tau$) in Section~\ref{sec:evaluation:largescale}. The most cumbersome subsection is Section~\ref{sec:algebra}, in which we systematically factor out ideals in $\wCT$ and $\wCThat$, arriving eventually at the following sequences of spaces
\begin{center}
\begin{tabular}{ L L L L L L L L L }
\wCT & \stackrel{\wlP}{\to}& 	
	\wlCT 	 &\supseteq &
	\plCT 	&\stackrel{\symP}{\to}	&
	\symCT & \stackrel{\adzesymP{\adze}}{\to} &
	\pT
\\
\wCThat & \stackrel{\wlPhat}{\to}			& 
	\wlCThat & \supseteq 					&
	\plCThat &  \stackrel{\symPhat}{\to}	&
	\symCThat.
\end{tabular}
\end{center}

The spaces $\wlCT$ and $\wCThat$, see Definition~\ref{def:wCT}, are merely auxiliary spaces, and we 
will mostly be working with the subspaces $\plCT$ and $\plCThat$, see Definition~\ref{def:properly:legged:tree:algebra}, formed by  properly legged trees. So far symmetries of a tree, related to the fact that the same noise type appears multiple times, are not reflected in the legs, and we remedy this in $\symCT$ and $\symCThat$, see Definition~\ref{def:symCT}. Finally, dropping ``non-essential'' legs and identifying trees with non-vanishing derivative decoration on legs, we arrive at the space $\pT$, see Lemma~\ref{lem:CTadze:hopf:ideal}, which turns out to be  isomorphic as a Hopf algebra to $\CTm$, see Lemma~\ref{lem:hopfiso}.
An analytic result generalising the ``vertical'' cancellations in the diagram on page~\pageref{eq:diagram} will be derived in Proposition~\ref{prop:eps:beta:bound} in Section~\ref{sec:constraints:analytic}. A key result is Proposition~\ref{prop:interaction:wlCJ:Delta} in Section~\ref{sec:constraints:ideal}, making precise the idea of (\ref{eq:intro:CJhatCJ}) and in particular concluding the proof of Assumption~\ref{ass:CJHopfIdeal} that $\CJ$ is a Hopf ideal.
Finally, in Section~\ref{sec:constraints:rigidities} we compare the characters $\hat g^\eta$ and $g^\eta$, and show that their difference is continuous in the limit as $\eta$ approaches a rough limit noise, see Lemma~\ref{lem:fnkn} and Lemma~\ref{lem:fnknbound}

\subsection{Extension of the regularity structure}\label{sec:ext:reg:strct}

We assume that we are given a finite set \label{idx:legtypes}$\Legtype$, disjoint from $\FL$, elements of which we call \emph{leg types}. From an algebraic point of view,
we treat $\Legtype$ as a set of additional noise types, and we extend the homogeneity assignments $\homofancys{\cdot}$ and $\homos{\cdot}$ to $\wFLm:=\FL_-\sqcup \Legtype$ by setting $\homofancys{\legtype}:=0$ and $\homos{\legtype}:=-\kappa$ for some $\kappa>0$ small enough (to be specified shortly) whenever  $\legtype \in \Legtype$. From the extended set of types $\wFL := \wFLm \sqcup \FL_+$ we want to build a regularity structure $\wTex$ as in \cite[Sec.~5.5]{BrunedHairerZambotti2016}, for which we specify a rule $\wR$.

Let $M\in\N$ denote the maximum number of edges $\# E(\tau)$ for any $\tau \in \TT_-$. We first define the rule $\wRtilde$ by setting
\begin{align}\label{eq:wl:rule:tilde}
\wRtilde(\ft):=\{A\sqcup B: A\in R(\ft)\text{ and }B\subseteq \Legtype \times \{0\} \text{ is a \emph{set} with } \# B\le M \},
\end{align}
for any $\ft\in\FL_+$. Here $R$ denotes the rule used to construct the regularity structure $\CT$, see Section~\ref{sec:reg:structures}. Note that in (\ref{eq:wl:rule:tilde}) we only allow $B$ to be a proper set (or equivalently a multiset satisfying $B\le 1$), so that any tree conforming to $\wRtilde$ can be built from 
a tree conforming to $R$ by adding to every node up to $M$ edges of distinct types in $\Legtype$. 
Provided that $\kappa>0$ is small enough we obtain a normal and subcritical  \cite[Def.~5.14]{BrunedHairerZambotti2016} rule $\wRtilde$ in this way, and we  denote by $\wR$ its completion \cite[Prop. 5.21]{BrunedHairerZambotti2016}.
\begin{definition}
We denote by \label{idx:wT}\label{idx:wTex}$\wlTex$ (resp. $\wT$) the extended (resp. reduced) regularity structure constructed as in \cite[Sec.~5.5]{BrunedHairerZambotti2016} from the rule $\wR$. 
Furthermore, we denote by \label{idx:wCTex}$\wCTex$ and \label{idx:wCThat}$\wCThat$ the algebras constructed as in \cite[Def.~5.26, Def.~5.29]{BrunedHairerZambotti2016} starting from the regularity structure $\wlTex$. 
\end{definition}
As in \cite[Prop.~5.35]{BrunedHairerZambotti2016}, the space $\wCTex$ forms a Hopf algebra. We will mostly work with the factor Hopf algebra \label{idx:wCT}$\wCT$ of $\wCTex$ given by neglecting the extended decoration. We write \label{idx:wTT}$\wTT$ for the set of unplanted trees $\tau \in \wT$ of negative homogeneity, so that $\wCT$ is generated freely as a unital, commutative algebra from $\wTT$. 

For a tree $\tau\in \wTex$ we denote by \label{idx:LLegtype}$L_\Legtype(\tau)\subseteq E(\tau)$ the set of leg type edges, i.e.\ the set of $e \in E(\tau)$ such that $\ft(e)\in \Legtype$, and by $L(\tau)\subseteq E(\tau)$ the set of noise type edges of $\tau$, i.e.\ the set of $e \in E(\tau)$ such $\ft(e) \in \FL_-$. We will often call an edge of leg type simply a \emph{leg}. 
We write \label{idx:CLtau}$\CL(\tau) \ssq N(\tau)$ (resp. \label{idx:NLegtype}$\LTN(\tau)\subseteq N(\tau)$) for the set of nodes $u\in N(\tau)$ that are adjacent to at least one noise type (resp.\ leg type) edge, and we write \label{idx:hatCL}$\hat\CL(\tau) \ssq N(\tau)$ for the set of nodes $u \in N(\tau)$ with the property that $\fo(u) < 0$.

\begin{example}
In the following example, taken from the KPZ equation, we coloured kernel-type edges $e \in K(\tau)$ grey (they are bold because they carry a derivative decoration), noise type edges $e \in L(\tau)$ blue, we draw legs  $e \in L_\Legtype(\tau)$ as wavy lines, we colour nodes blue if they are elements of $\CL(\tau)$, and we draw nodes as squares (rather than circles) if they are elements of $\LTN(\tau)$
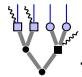
\[
\treeExampleKPZNL.
\]
This is the only example in this paper in which we make noise type edges explicit, since their position can always be inferred by $\CL(\tau)$. We will always make legs explicit (note that their position cannot be inferred from $\NLeg(\tau)$, as there may be more than one leg incident to the same node). Conversely, since $\NLeg(\tau)$ can be inferred from $L_\Legtype(\tau)$, we will not draw them explicitly as boxes in the forthcoming examples.
\end{example}
 
The space $\CT^\ex$ can be identified with the linear subspace of $\wlTex$ spanned by all trees $\tau \in \wlTex$ without legs. Similarly, the spaces $\hat\CT_-^\ex$, $\CT_-^\ex$,  and $\CT_-$ have natural interpretations as subalgebras of $\wCTex$, $\wCThat$ and $\wCT$, respectively (in this interpretation the latter two are Hopf subalgebras). Given any tree $\tau = (T^{\fn,\fo}_\fe, \ft) \in \wlTex$ we define a tree $\pi\tau \in \CT^\ex$\label{idx:pi} by simply removing all of its legs.
 The map $\pi$ extends to a linear map $\pi: \wlTex \to \CT^\ex$, and to an algebra morphism from the algebras $\wCTex$, $\wCThat$ and $\wCT$ onto $\CT_-^\ex$, $\hat\CT_-^\ex$ and $\CT_-$, respectively. 

Finally, we denote by \label{idx:wCG}$\wCG$ the character group of the Hopf algebra $\wCT$, which is canonically isomorphic to the reduced renormalisation group constructed in \cite[Thm.~6.28]{BrunedHairerZambotti2016}. 
There exists a subgroup of $\wCG$ isomorphic to $\CG_-$, given by the set of those characters that vanish on any 
tree $\tau$ with $L_\Legtype(\tau) \ne \emptyset$. (The isomorphism $\varphi:\CG_-\hookrightarrow\wCG$ is given explicitly by mapping $g\in\CG_-$ to a character $\varphi(g) \in \wCG$ given by setting $\varphi(g)(\tau) = g(\tau)$ for $\tau \in \TT_-$ and $\varphi(g)(\tau) = 0$ for $\tau \in \wTT \backslash \TT_-$.)

\begin{remark}\label{rmk:emebdding:character:group}
Since we view $\CTm$ as a subspace of $\wCT$, the embedding $\CGm \to \wCG$ is actually not ``canonical'', although the projection $\wCG \to \CGm$, given by restricting a character $g \in \wCG$ to $\CTm$, is canonical. The construction in the previous paragraph uses indirectly the fact that $\CTm$ is also naturally isomorphic to a factor Hopf algebra $\wCT/\ker \q$, where $\q : \wCT \to \CTm$ is defined by killing trees $\tau$ such that $\LTE(\tau) \ne \emptyset$. However, in the sequel we will continue to view $\CTm$ as a subalgebra of $\wCT$.
\end{remark}

\subsection{Large scale behaviour of renormalised trees}
\label{sec:evaluation:largescale}

We now fix a degree assignment $\deg_\infty(\ft) \in \R_-\sqcup\{-\infty \}$ for kernel types $\ft \in \FL_+$. In order to avoid case distinctions later on, we also set $\deg_\infty(\Xi):=0$ for any noise type $\Xi \in \FL_-$ and $\deg_\infty(\legtype):=-\infty$ for any leg type $\legtype \in \Legtype$. We write \label{idx:CK+infty}$\CK^+_\infty$ for the set of kernel assignments $R =(R_\ft)_{\ft\in\FL_+}$ such that $R_\ft:\bar\domain\to\R$ is smooth and compactly supported for any $\ft\in\FL_+$. We endow this space with the topology generated by the system of seminorms $\|\cdot\|_{\CK^+,\ft}$ for $\ft\in\FL_+$, where the latter is defined as the smallest constant such that
\begin{align}\label{eq:CKp:norm}
|D^k R_\ft(x)|\le \|R\|_{\CK^+,\ft} (1+|x|)^{\deg_\infty\ft}
\end{align}
for any $x\in\bar\domain$ and $k\in\N^d$ with $|k|_\fs<r$. We write \label{idx:CK+z}$\CK_0^+$ for the completion of $\CK_\infty^+$ with respect to the corresponding metric.
We extend the notation of (\ref{eq:CK}) and (\ref{eq:CK:ctr}) to the extended regularity structure, with $L(\tau)$ replaced by $L(\tau) \sqcup L_\Legtype(\tau)$, so that in particular one has $\CK_G \tau : \bar\domain ^{ L(\tau) \sqcup L_\Legtype(\tau) } \to \R$ and the integral in (\ref{eq:CK:ctr}) ranges over $\bar \domain ^{ L(\tau) \sqcup L_\Legtype(\tau) }$. Furthermore, we introduce the following space in analogy to Definition~\ref{def:tildePsi}.

\begin{definition}\label{def:Psi}
We write $\Legfn$ for the set of all families of test functions $(\psi_\cm)_\cm$, indexed by multisets $\cm$ with values in $\Legtype$, such that  $\psi_\cm \in \barCCcmg$.
\end{definition}

With this notation, we now define the following evaluations.

\begin{definition}\label{def:bar:Upsilon:large:scale}
We define for any tree $\tau=T^{\fn,\fo}_\fe\in \wTex$, any smooth noise $\eta\in\SM_\infty^\star$, any $\legfn \in \Legfn$, and any large-scale kernel assignment  $R=(R_\ft)_{\ft\in\FL_+}\in\CK^+_\infty$ the constant
\begin{align}\label{eq:bar:Upsilon:large:scale}
\bar\Upsilon^{\eta,\legfn}_R\tau:=
\ctrL{\tau}{K+R} { \zeta^ {\eta,\legfn,\tau} } \;,
\end{align}
where $\zeta^ {\eta,\legfn,\tau} \in \bar\domain^{L(\tau) \sqcup L_\Legtype(\tau)}$ is defined by
\[
\zeta ^ {\eta,\legfn,\tau} (x)
:=
\Big(\E
\prod_{u\in L(T)} \eta_{\ft(u)}(x_{u})\Big) \,
D^{\fe |_{L_\Legtype(\tau)}}
\psi_{[L_\Legtype(\tau),\ft]}
	(x_{L_\Legtype(\tau)})
\]
for any $x \in \bar\domain ^{L(\tau) \sqcup L_\Legtype(\tau)}$.
Moreover, we define the ``renormalised'' constant by 
\begin{align}\label{eq:evaluation:largescale:renorm}
\hat{\Upsilon}^{\eta,\psi}_R\tau:= \bar\Upsilon_R^{\eta,\psi} M^{g^\eta_\BPHZ}\tau
\end{align}
\end{definition}
Here, we use the notation $g^\eta_\BPHZ \in \CG_-$ for the BPHZ-character of the noise $\eta$ in the renormalisation group $\CG_-$, which we view naturally as a character in $\wCG$ as above. We also set $\eval{\psi}{R} := \evaln{\one}{\psi}{R}$ and $\eva{\psi} := \evaln{\one}{\psi}{\hat K - K}$.
\begin{remark}
One has $g^\eta_\BPHZ(\sigma) = 0$ for any $\sigma \in \wCT$ such that $L_\Legtype(\sigma) \ne \emptyset$, so that $M^{g^\eta_\BPHZ}$ maps $\tau$ onto the span of trees $\tilde\tau$ with the property that $[L(\tau),\ft] = [L(\tilde\tau),\ft]$ It follows that $\hat \Upsilon^{\eta,\psi}_R \tau$ really only depends on $\psi_{[L_\Legtype(\tau),\ft]}$.
We finally note that these notations do not depend on the extended decoration $\fo$.
\end{remark}

Our goal is to show that under some natural assumptions on the degree assignment $\deg_\infty$ the map $\evaln{\eta}{\psi}{R}$ extends 
continuously to any large-scale kernel assignment $R\in\CK_0^+$, and 
$\hat\Upsilon^{\eta,\psi}_R$ extends continuously to the set of pairs $(\eta,R)\in \SMsz \times \CK_0^+$. 
Such a statement can only be true if we make an assumption 
on the degree assignment $\deg_\infty$ and the positions of the legs, which is in complete analogy to \cite[Sec.~4]{Hairer2017}. We then consider partitions $\CP$ of the node set $N(\tau)$ such that $\#\CP\ge 2$ and such that there exists $P\in\CP$ with $\LTN(\tau)\subseteq P$. We call partitions of $N(\tau)$ that satisfy these properties \emph{tight} from now on. For any tight partition $\CP$ we  denote by $K(\CP)$ the set of kernel-type edges $e\in K(\tau)$ with the property that there does not exist $P\in\CP$ such that $e\subseteq P$, and we set 
\begin{align}\label{eq:deg:infty:parti}
\deg_\infty\CP:=
	\sum_{e\in K(\CP)} \deg_\infty \ft(e) +
	\sum_{u\in N(\tau)}|\fn(u)|_\fs +
	|\fs|(\#\CP-1).
\end{align}

Let $\SMstarz$ denote the closure of $\SMstar$ under the norm $\|\cdot\|_{\fs}$.
The key result of this section is the following theorem.

\begin{theorem}\label{thm:evaluation:trees:largescale}
Let $\tau\in \wTTex$ be such that $\deg_\infty\CP<0$ for any tight partition $\CP$ of $N(\tau)$. Then for any fixed $\legfn \in \Legfn$ and $\eta \in \SMstarz$ the evaluation
\begin{align}\label{eq:evaluation:largescale}
R \mapsto \bar\Upsilon^{\eta,\legfn}_R \tau
\end{align}
extends continuously to the space $\CK_0^+$. Moreover, the evaluation
\begin{align}\label{eq:evaluation:largescale:renorm:2}
(\eta,R) \mapsto \hat{\Upsilon}^{\eta,\legfn}_R \tau 
\end{align}
extends continuously to the space $\SMstarz \times \CK_0^+$. Finally, one has the bound
\begin{align}\label{eq:bound:evluation:largescales}
|\hat{\Upsilon}^{\eta,\legfn}_R\tau|
\lesssim 
\|\eta\|_{\fs}
\Big(
\prod_{e\in K(\tau)}
	\|R_{\ft(e)}\|_{\CK^+,\ft(e),r} + 1
\Big)
\end{align}
for $r\in\N$ any integer larger than $-\min \{ |\tau|_\fs : \tau \in \CT\}$, uniformly over all $(\eta,R) \in \SMstarz \times\CK_0^+$.\footnote{We set $\|\one\|_{N,\fc}:=1$.}
\end{theorem}
\begin{remark}
This theorem should be viewed as a generalisation of \cite[Thm.~4.3]{Hairer2017}. The main reason why it does not follow directly from 
\cite[Thm.~4.3]{Hairer2017} is the presence of higher-order cumulants. In principle, one could formulate a statement analogous to 
\cite[Thm.~4.3]{Hairer2017} for Feynman hyper-graphs which would then imply the statement of the above theorem. However, such a formulation 
is rather cumbersome, so that we refrain from carrying out this construction.
\end{remark}
\begin{proof}
This follows very similar to \cite[Thm.~4.3]{Hairer2017}. See Section~\ref{sec:Feynman-diagrmas} for a proof.
\end{proof}

The large-scale kernel assignment that we are interested in is given by $R_\ft=\hat K_\ft - K_\ft$ for any $\ft \in \FL_+$, so that we have to choose $\deg_\infty \ft:=\fancynorm\ft_\fs-|\fs|$ for any kernel-type $\ft \in \FL_+$. 
The next lemma shows that the assumption of super-regularity implies that the condition of Theorem~\ref{thm:evaluation:trees:largescale} holds automatically for a large class of trees $\tau \in \wTTex$. 

\begin{lemma}\label{lem:super:regularity:implies:large:scale:bound}
Let $\tau\in \wTTex$ be a tree with $\homofancyex{\tau} \le 0$ and such that $\CL(\tau)\cup \hat \CL(\tau) \subseteq \LTN(\tau)$.  Assume moreover that $\tau \notin \CV_0$ (see Section~\ref{sec:technical:assumpation}).
Then one has $\deg_\infty \CP < 0$ for any tight partition $\CP$ of $N(\tau)$.
\end{lemma}
\begin{proof}
Let $\tau=T^\fn_\fe\in \wTTex$, assume first that $\hat\CL(\tau) = \emptyset$, so that $\homofancyex{\tau}=\homofancy{\tau} \le 0$, and let $\CP$ be a tight partition of $N(T)$.
We denote by $P^\star\in\CP$ the set such that $\NLeg(\tau)\subseteq P^\star$, 
and therefore, by assumption, $\CL(\tau) \ssq P^\star$, and we write $\CP^\star:=\CP\backslash\{P^\star \}$. 
We need to show that
\[
\deg_\infty(\CP):=\sum_{e\in K(\CP)}\deg_\infty{\ft(e)}+\sum_{u\in N(T)} |\fn(u)|_\fs+ (\#\CP-1)|\scale|<0.
\]
Any $P \in \CP^\star$ is a subset of $N(\tau)$ and induces a subgraph $G_P= (V_P, E_P)$ of $\tau$, given by setting $V_P := P$ and $E_P$ is the set of edges $e \in K(\tau)$ such that $e \ssq P$.
It is sufficient to consider partitions $\CP$ that have the property that this induced subgraph is connected for any $P\in\CP^\star$ ; otherwise there exists a non-trivial way to write $P=P_1\sqcup P_2$ such that there does not exist an edge $e$ with the property that $e^\uparrow\in P_1$ and $e^\downarrow\in P_2$ or the other way around. One could then replace $P$ with $\{P_1,P_2\}$ in $\CP$ to create a tight partition $\CQ$ with $\deg_\infty(\CQ) = \deg_\infty(\CP) +|\fs|>\deg_\infty(\CP)$. We now claim that it is even sufficient to consider partitions $\CP$ with the property that any set $P\in\CP^\star$ contains only a single vertex. Indeed, assume that $P\in\CP^\star$ contains more than one vertex. Then $P$ induces a subtree $S\subseteq T$ that does not contain any $u\in\CL(T)$, and thus one has 
\[
\sum_{e\in K(T), e\subseteq P}\deg_\infty{\ft(e)}+(\#P-1)|\scale|=|S_\fe^0|_\fs>0
\]
by assumption.
With a virtually identical argument one can assume that the partition $\CP$ has the property that there exists a finite number of node-disjoint subtrees $S_1,\ldots,S_m$ of $T$ for some $m\ge 1$ such that for any $1\le i\le m$ one has that $L(S_i) \ne \emptyset$, such that $L(T) = \bigsqcup_{i \le m} L(S_i)$, and with the property that $P^\star=\bigsqcup_{i\le m} N(S_i)$. 
We also assume that the number of trees are minimal, so that for any $i \ne j$ the subgraph induced by the node set $N(S_i) \sqcup N(S_j)$ is not connected. It follows that $K(\CP) = K(T) \setminus \bigsqcup_{i \le m} K(S_i)$.

A straightforward calculation shows that
\[
\fancynorm{T^\fn_\fe}_\scale
	=\sum_{e\in K(T)}
		\deg_\infty{\ft(e)} +
		\sum_{u\in L(T)}\fancynorm{\ft(u)}_\fs +
		\sum_{u\in N(T)}\fn(u) +
		\# K(T)|\scale|\;,
\]
and a similar identity holds for any the subtree $S_i$ for any $i\le m$. Using the fact that $\# K(\CP) = \#K(T) - \sum_{i=1}^m \# K(S_i)$ we get the identity
\[
\deg_\infty(\CP)=\fancynorm{T^\fn_\fe}_\scale-\sum_{i\le m}\fancynorm{(S_i)^0_\fe}_\scale
-\#K(\CP)|\fs|+(\#\CP-1)|\fs|.
\]

Let $Q:= \{ e^\uparrow : e \in K(\CP) \}$, then our definitions show that
\[
Q = \big( \{ \rho(S_i) : i \le m \} \sqcup \CP^\star \big) \setminus \{ \rho_T \},
\]
so that $\# K(\CP) = \# Q = m + \#\CP - 2$, and thus
\[
\deg_\infty(\CP)=\fancynorm{T^\fn_\fe}_\scale-\sum_{i\le m}\fancynorm{(S_i)^0_\fe}_\scale
-(m-1)|\fs|.
\]
We now use the assumptions of the lemma which imply on the one hand that $\fancynorm{T^\fn_\fe}_\scale\le 0$ and on the other hand that $\fancynorm{(S_i)^0_\fe}_\scale\ge-\shalf$, from which it follows that
\begin{align}\label{eq:psi_estimate}
\deg_\infty(\CP)\le (1-\frac{m}{2})|\scale|\le 0\;,
\end{align}
with equality if and only if $m=2$, $\fancynorm\tau_\fs=0$, and $\fancynorm {S_i}_\fs=-\shalf$ for any $i=1,2$. By assumption, any tree $S \in\CT$ such that $\homofancys S=-\shalf$ is equal to some $\Xi \in \FL_-$, so that $\# L(\tau)=2$, and therefore $\tau \in \CV_0$.

Assume now that $\tau = (T^{\fn,\fo}_\fe, \ft) \in \wTTex$ is such that $\hat\CL(\tau) \ne \emptyset$. 
Since $\homofancyex{\tau} \le 0$, by \cite[Lem.~5.25]{BrunedHairerZambotti2016} there exists a tree $\bar\tau = (\bar T^{\bar \fn}_{\bar \fe}, \bar \ft) \in \wTT$ (that is, a tree with vanishing extended decoration such that $\homofancy{\bar\tau} \le 0$), a subforest $\CF \in \div(\tau)$ (here $\div(\tau)$\label{idx:divtau} denotes the set of subforests $\CF$ of $\tau$ with the property that any connected component $S$ of $\CF$ is of negative homogeneity; see Section~\ref{sec:i-forests}), and decorations $\fn_\CF$ and $\fe_\CF$ as in (\ref{eq:iforest:decoration:convention}) with the property that one has
\[
\tau = ((\bar T / \CF)^ {\bar \fn - \fn_\CF,  [\fo]_\CF } _{ \bar \fe + \fe_\CF } , \bar \ft ).
\]
(Note that necessarily $\tau$ has at least one divergent proper subtree, so that $\#L(\tau)>2$. In particular, $\tau$ is not the exceptional case from the first part of the proof.)
We let  $\varphi_{\bar T}^\CF : V(\bar T) \to V(T)$ be the map defined in (\ref{eq:iforests:vertex:map}), and we write $\varphi:= \varphi_{\bar T}^\CF |_{N(\bar \tau)}: N(\bar \tau) \to N(\tau)$ for the restriction of this map to the set of nodes $N(\bar \tau) \ssq V(\bar \tau)$. 

Let now $\CP$ be a tight partition of $N(T)$, 
and write again $P^\star \in \CP$ for the element such that $\CL(\tau) \cup \hat\CL ( \tau ) \ssq P^\star$.
 We define a partition $\CQ$ of $N( \bar \tau )$ by setting 
\[
\CQ := \{ \varphi^{-1}(P) : P \in \CP\}.
\]
Since $\CL( \bar \tau ) \ssq \varphi^{-1}(P^*)$, the partition $\CQ$ is tight, and by the first part of the proof one has $\deg_\infty \CQ <0$. It thus remains to note that $\deg_\infty\CP \le \deg_\infty \CQ$, which follows from the definition of $\deg_\infty \CP$ in (\ref{eq:deg:infty:parti}), the fact that one has $K(\CP) = K(\CQ)$, $\#\CP = \# \CQ$ and the fact that by definition $\sum_{ u \in N(\tau) } |\fn(u)|_\fs = \sum_{ u \in N(\bar \tau) } |\bar \fn(u) - \fn_\CF(u)|_\fs$.
\end{proof}

\begin{remark}\label{rmk:large:scale:V0}
The statement fails for trees $\tau \in \CV_0$. For such trees however one has $\evaln{\eta}{\psi}{R} \tau= \Upsilon^{\eta}\tau = - g^\eta(\tau) = 0$ for any $\eta \in \SMsinf$, where the first equality holds if $R=0$, and $\psi=1$ in a large enough neighbourhood of the origin (compare Lemma~\ref{lem:identity:giota:g}), and the last equality holds by Assumption~\ref{ass:technical}. Using the homogeneity of the integration kernels, it is possible to find a sequence $R_n \to \hat K-K$ so that $\evaln{\eta}{\psi}{R_n}$ vanished for any $n$, so that at least for this particular choice of $R_n$ and $\psi$ a statement analogue to \eqref{eq:evaluation:largescale:renorm:2} holds. We will make use of this fact in the proof of Lemma~\ref{lem:fnknbound} below.
\end{remark}

The statement of Lemma~\ref{lem:super:regularity:implies:large:scale:bound} does clearly not hold in general for trees $\tau \in \wT$ with positive homogeneity, if we only assume $\CL(\tau) \ssq \LTN(\tau)$. Keeping track of the ``location'' of contracted subtrees (and thus a sufficient criterion for the positions at which we have to attach legs) is the only reason why we keep track of the extended decoration $\fo$ instead of working directly with $\wT$. As mentioned in Lemma~\ref{lem:super:regularity:implies:large:scale:bound}, this is irrelevant for trees $\tau$ such that $\fancynorm{\tau}_- \le 0$, so that there is no need to keep the extended decoration when working with the Hopf algebra $\wCT$. It will therefore be convenient for us to work with the two spaces $\wCThat$ (keeping the extended decoration) and $\wCT$ (dropping the extended decoration), and we will view the operator $\cpm$ as acting between these space
\[
\cpm : \wCThat \to \wCT \otimes \wCThat
\]
by dropping the extended decoration on the left component.

\subsection{An algebraic construction} \label{sec:algebra}

We want to work with a  Hopf subalgebra (resp. subalgebra) of $\wCT$ (resp. $\wCThat$) generated by trees $\tau$ such that Theorem~\ref{thm:evaluation:trees:largescale} can be applied. In other words, we want to work with trees that contain enough legs so that the large-scale evaluation is well-defined.
Also, we would like to work with trees that are \emph{properly legged}, see Definition~\ref{def:propery:legged} below. Roughly speaking, we want every leg to have a unique ``partner''. 
For this we assume that we are given a type map \label{idx:imap}$\imap: \Legtype \to \FL_- \times \FL_-$ and an involution \label{idx:invo}$\Legtype \ni \legtype \mapsto \bar\legtype \in \Legtype$ that switches the components of $\imap$ in the sense that if $\imap(\legtype)=(\ft,\ft')$, then $\imap(\bar\legtype)=(\ft',\ft)$. To avoid case distinctions, we also assume that $\bar\legtype \ne \legtype$. 

\maybeNotNeeded{
It will be useful to fix arbitrary total orders $\le$ on $\FL_-$ and on $\Legtype$ which are compatible in the sense that $\legtype_1 \le \legtype_2$ whenever $\ft_1 < \ft_2$ or $\ft_1 = \ft_2$ and $\ft'_1 \le \ft'_2$, where $\imap(\legtype_i)=(\ft_i, \ft'_i)$ for $i=1,2$.
}
With this notation, we make the following key definition. Recall the notation $[\cdot,\cdot]$ for multisets from Section~\ref{sec:multisets}.

\begin{definition}\label{def:propery:legged}
We call a tree $\tau \in \wTTex$ \emph{properly legged} if $\homofancyex{\tau} \le 0$ and the following properties hold.
\begin{enumerate}
\item \label{item:properly:legged:legtypeunique}
Any leg type appears at most once, i.e.\ one has $[L_\Legtype( \tau ),\ft]\le 1$.
\item \label{item:properly:legged:typei}
For any noise type edge $u \in L(\tau)$ and any leg $e \in L_\Legtype(\tau)$ with $e^\downarrow=u^\downarrow$ one has $\imap_1(\ft(e))=\ft(u)$.
\item \label{item:properly:legged:leg_coupling}
For any leg $e \in L_\Legtype(\tau)$ there exists a leg $\bar e \in L_\Legtype(\tau)$, which we call the \emph{partner} of $e$, with $\ft(\bar e) = \overline{\ft(e)}$ and one has $e^\downarrow \ne \bar e^\downarrow$.
\item \label{item:properly:legged:leaf_coupling}
For any distinct $u,\bar u\in\CL(\tau)$ there exists a \emph{unique} leg $e$ with $e^\downarrow = u$ and $\bar e^\downarrow = \bar u$. 
\item \label{item:properly:legged:leaf_hat_coupling}
For any $u \in \hat \CL(\tau)$ and any $\bar u \in \CL(\tau)$\footnote{Recall that our assumptions imply $\CL(\tau)\cap \hat\CL(\tau) = \emptyset$, so that $u \ne \bar u$.} there exists\footnote{Note that we do not impose uniqueness here.} a leg $e\in L_\Legtype(\tau)$ such that $e^\downarrow = u$ and $\bar e^\downarrow = \bar u$.
\end{enumerate}
\end{definition}

\begin{remark}
The leg $e$ refered to in~\ref{item:properly:legged:leaf_hat_coupling}. in the previous definition is not assumed to be unique  (as opposed to~\ref{item:properly:legged:leaf_coupling}.):
Given a tree $\tau$ and a subtree $\sigma \ssq \tau$, then after contracting $\sigma$, we obtain a tree $\tilde\tau = \tau/\sigma$. Let $w \in \hat\CL(\tilde\tau)$ be the vertex generated by contracting $\sigma$. For any $u \in \CL(\tau) \backslash \CL(\sigma)$ and any $v \in \CL(\sigma)$ there will be a pair of legs $e,\bar e$ in $\tilde\tau$ with $e^\downarrow = u$ and ${\bar e}^\downarrow = w$.
\end{remark}

\begin{example}
Consider the following example of a properly legged tree:
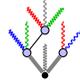
\[
\treeExampleProperlyLegged
\]
where straight lines denote kernel-type edges, circles denote noises (elements of $\CL(\tau)$) and coloured coiling edges denote legs. Here, we coloured legs which are partners with the same colour, but with different wavy patterns to make them distinguishable. Note that we could remove only the gray edges without loosing the property of being properly legged.
\end{example}

We will mainly work with algebras $\plCT$ and $\plCThat$ formed by properly legged trees. But if we would simply define these spaces as the algebras generated by properly legged trees, then they would not be closed under the action of the coproduct $\cpm$. The main problem here is that the previous definition enforces the existence of a partner for any leg $e \in L_\Legtype(\tau)$, and this property is not preserved under the coproduct. To circumvent this problem, we will define $\plCT$ and $\plCThat$ as subalgebras of factor algebras $\wlCT$ and $\wlCThat$, which are defined in the following way.

\begin{definition}\label{def:wCT}
Let $\wl{\CI}\subseteq \wCT$ and $\wl{\hat\CI}\subseteq \wCThat$ denote the ideals generated by the set of trees~$\tau\in \wCT$ or~$\tau\in \wCThat$ respectively such that there exists a leg $e \in L_\Legtype(\tau)$ without a partner. (Recall from \ref{item:properly:legged:leg_coupling} from Definition~\ref{def:propery:legged} that $\bar e \in L_\Legtype(\tau)$ is a partner of $e$ if $\ft(\bar e) = \overline{\ft(e)}$ and $e^\downarrow \ne \bar e^\downarrow$.)
Then we define
\begin{equ}
\wlCT:= \wCT /\wl{\CI}
\quad\text{ and }\quad
\wlCThat:= \wCThat /\wl{\hat\CI},
\end{equ}
and the canonical projections $\wlP:\wCT\to \wlCT$ and $\wlPhat: \wCThat \to\wlCThat$.
\end{definition}

Concerning $\wlCT$ and $\wlCThat$, we can now show the following Lemma.
\begin{lemma}\label{lem:interaction:WLCT:coproduct}
The ideal $\wl\CI$ forms a Hopf ideal in $\wCT$, so that in particular $\wlCT$ is a Hopf algebra, and the factor algebra $\wlCThat$ forms a co-module over the factor Hopf algebra $\wlCT$.
\end{lemma}
\begin{proof}
The lemma follows once we show the identities
\begin{align}
\label{eq:interaction:WLCT:coproduct:1}
\cpmh \wl\CI 		&\subseteq \wl\CI \otimes \wCT + \wCT \otimes \wl\CI \\
\label{eq:interaction:WLCT:coproduct:2}
\cpm \wl{\hat\CI} 	&\subseteq \wl\CI \otimes \wCThat + \wCT \otimes \wl{\hat\CI}.
\end{align}
We only show (\ref{eq:interaction:WLCT:coproduct:1}) since (\ref{eq:interaction:WLCT:coproduct:2}) follows with almost the same proof. 
Let $\tau =T^\fn_\fe\in \wl\CI$ be a tree and fix a leg $e \in L_\Legtype(\tau)$ 
such that all legs $\tilde e \in L_\Legtype(\tau)$ 
with the property that $\ft(\tilde e) = \overline{\ft(e)}$ satisfy $e^\downarrow = \tilde e^\downarrow$. 
By (\ref{eq:coproduct:ex}) we are left to show that for any forest $\CF \in \div(\tau)$ one has 
$\prod_{S \in \CF} S^{\fn_\CF+\pi\ce_\CF}_\fe 
\otimes (T/\CF)^{\fn-\fn_\CF,[\fo]_\CF}_{\fe+\fe_\CF} 
\in \wl\CI \otimes \wCT + \wCT \otimes \wl\CI$ 
for any choice of decorations $\fn_\CF,\fe_\CF$. For this we distinguish two cases. In the first case, writing $\bar \CF$ for the set of connected components of $\CF$,
there exist $S \in \bar\CF$ such that $e \in E(S)$. 
From this it follows that whenever $\ft(\tilde e) =\overline{\ft(e)}$ for some edge $\tilde e \in E(S)$, then one has $e^\downarrow = \tilde e^\downarrow$ and thus $S^{\fn_\CF+\pi\ce_\CF}_\fe \in \wl\CI$. In the second case, one has $e \notin E(\CF)$. 
In this case it suffices to note that whenever $\tilde e \notin E(\CF)$ is a leg with the property that $e^\downarrow = \tilde e^\downarrow$ in $T$, then this identity remains true in $T/\CF$ as well, so that in this case one has $(T/\CF)^{\fn-\fn_\CF}_{\fe+\fe_\CF, [\fo]_\CF}  \in \CI$.
\end{proof}

The canonical embedding $\wli:\wCT\to\wCThat$ induces an embedding \label{idx:wli}$\wli:\wlCT\to\wlCThat$. 
We denote by $\wlj:\wlCT\to\wCT$ and $\wljhat:\wlCThat\to\wCThat$ the obvious embeddings, so that the range of $\wlj$ is the algebra generated by trees with the property that any leg has at least one partner.
We now have the following analogue of \cite[Prop.~6.5]{BrunedHairerZambotti2016} in this setting.

\begin{proposition}\label{prop:twisted:antipode:pl}
There exists a unique algebra homomorphism $\wlCA:\wlCT\to\wlCThat$ with the property that the identity
\begin{align}\label{eq:definition:twisted:antipode:wl}
\CM(\wlCA\otimes\Id) \cpmwi =\one^\star
\end{align}
holds on $\wlCT$. Moreover, in terms of the usual twisted antipode $\wCA$, this operator is uniquely determined by the relation
\begin{align}\label{eq:identity:twisted:antipode:wl}
\wlCA\wlP=\wlPhat \wCA
\end{align}
on $\wCT$, or equivalently by the relation
\begin{align}\label{eq:identity:twisted:antipode:wl:2}
\wlCA=\wlPhat \wCA \wlj
\end{align}
on $\wlCT$.
\end{proposition}
\begin{proof}
The fact that (\ref{eq:definition:twisted:antipode:wl}) determines a unique algebra homomorphism follows easily via induction in the number of edges (see also the proof of \cite[Prop.~6.5]{BrunedHairerZambotti2016}). Since $\wlP$ is surjective, 
(\ref{eq:identity:twisted:antipode:wl}) defines a unique operator $\wlCA$, so that we are left to show that this operator 
also satisfies  (\ref{eq:definition:twisted:antipode:wl}). For this we use the identities $\cpm \wlPhat=(\wlP\otimes\wlPhat) \cpm$ on $\wlCThat$ and $\wli\wlP=\wlPhat\wli$, from which it follows that
\[
\CM(\wlCA\otimes\Id) \cpmwi \wlP
=
\CM(\wlCA\wlP\otimes\wlPhat) \cpmwi
\]
holds on $\wlCT$. Using (\ref{eq:identity:twisted:antipode:wl}), we can rewrite the right-hand side of this identity  as
\[
\wlPhat\Big(
\CM(\wlCA	\otimes\Id) \cpmwi
\Big)=\one^\star,
\]
and since $\wlPhat$ is a  surjective homomorphism the statement follows. 
The equivalence with (\ref{eq:identity:twisted:antipode:wl:2}) follows at once from the fact that 
$\wlP\wljhat=\Id$ on $\wlCT$.
\end{proof}

Later on we will mostly work with subalgebras $\plCT$ and $\plCThat$ of $\wlCT$ and $\wlCThat$ which are generated by properly legged trees.

\begin{definition}\label{def:properly:legged:tree:algebra}
We denote by $\plCT \subseteq \wlCT$ and $\plCThat \ssq \wlCThat$ the subalgebras generated by properly legged trees.
\end{definition}

\begin{remark}
Note that by definition any tree $\tau \in \plCThat$ satisfies $\homofancyex \tau \le 0$. By definition of $\homofancyex\tau$ and the coproduct $\cpm$ one has $\cpm: \plCThat \to \plCT \otimes \plCThat$, so that $\plCThat$ is a comodule over $\plCT$.
\end{remark}

One of the facts that motivate the definition of properly legged trees is that for any tree $\tau \in \plCThat$ there exists a one to one correspondence between forests $\CF \in \div(\pi\tau)$ and forests $\CG \in \div(\tau)$ with the property that $\CF$ and $\CG$ give non-vanishing contributions to the coproduct, see (\ref{eq:coproduct:pl}) below.
To state this correspondence we introduce the following notation. Given a tree $\tau \in \wTex$ and a forest $\CF \in \div(\pi \tau)$, we write $\forestlegs\CF$ for the forest of $\tau$ induced by the edge set 
\begin{equ}\label{eq:forests:legs}
E(\forestlegs\CF):= E(\CF) \sqcup \bigsqcup_{S \in \bar\CF}L_\Legtype(S)
\end{equ}
with $L_\Legtype(S) \subseteq L_\Legtype(T)$ defined as the set of legs $e \in L_\Legtype(T)$ 
with the property that $e^\downarrow,\bar e^\downarrow \in N(S)$, 
where $\bar e \in L_\Legtype(T)$ denotes the partner of $e$ in $T$ as before. Here and below we write $\bar\CF$ for the set of connected components of $\CF$. We sometimes write $\forestlegs{\CF}[\tau]$ if we want to emphasise the tree $\tau$.

With this notation, we have the following lemma, the proof of which is postponed to Section~\ref{sec:technical:proofs} below.

\begin{lemma}\label{lem:CT:pl}
The space $\plCT$ forms a Hopf subalgebra of $\wlCT$ and $\plCThat$ forms a co-module over $\plCT$. In particular, one has $\wlCA : \plCT \to \plCThat$. Moreover, one has that $\wli : \plCT \to \plCThat$. Finally, the coproduct $\cpm : \plCThat \to \plCT \otimes \plCThat$ is explicitly given by
\begin{align}\label{eq:coproduct:pl}
\cpm \tau = 
\sum_{\CF \in \div(\pi \tau)}
\sum_{\fn_{\forestlegs \CF}, \ce_{\forestlegs\CF} }
\frac{1}{\ce_{\forestlegs\CF} !}
\binom{\fn}{\fn_{\forestlegs\CF}}
\prod_{S \in \forestlegs\CF} S^{\fn_{\forestlegs\CF} + \pi\ce_{\forestlegs\CF}}_\fe
\otimes
(T/\forestlegs\CF)^{\fn-\fn_{\forestlegs\CF},  [\fo]_{\forestlegs\CF}}
_{\fe + \fe_{\forestlegs\CF}},
\end{align}
for any tree $\tau = T^{\fn,\fo}_\fe\in \plCThat$. 
\end{lemma}


We will work with embeddings $\iota:\CT_-\to\plCT$ with the property that any tree $\tau \in \CT_-$ is mapped onto a tree $\iota\tau \in \plCT$ with the property that $\pi\iota\tau = \tau$ and $\iota\tau$ is in some sense as simple as possible with this property. There is some freedom how to construct such embeddings, and many of the statements below do not depend on the choice of embedding, as long as certain conditions are met, which we summarise in the following definition. We choose this way, rather than simply fixing such an embedding, since it will be convenient in the proofs below to have some flexibility in this choice. 

\begin{definition}\label{def:admissible:embedding}
We call an algebra monomorphism $\iota:\CT_-\to\plCT$ an \emph{admissible embedding} if 
all of the following  properties hold for any $\tau\in\TT_-$.
\begin{itemize}
\item The tree $\iota\tau$ is constructed by attaching legs to $\tau$, i.e.\ one has $\pi \iota=\Id$ on $\CT_-$. 
\item There are only legs attached to nodes in $\CL(\tau)$, i.e.\ one has $\LTN( \iota \tau)=\CL( \iota \tau)$.
\item The derivative decoration vanishes on legs, i.e.\ all legs $e \in L_\Legtype(\tau)$ satisfy $\fe(e)=0$.
\end{itemize}
\end{definition}

We denote by \label{idx:adCT}$\adCT$ the subalgebra of $\plCT$ generated by all elements of the form $\iota\tau$ for some admissible embedding $\iota$ and some $\tau \in \CT_-$.
Note that $\adCT$ is not closed under the coproduct, so that in particular $\adCT$ does not form a Hopf algebra. We will write $\adCThat \subseteq \plCThat$ for the smallest subalgebra of $\plCThat$ with the property that $\cpmwi \adCT \subseteq \plCT \otimes \adCThat$.
It follows from the definition of $\adCT$ and (\ref{eq:coproduct:pl}) that one actually has 
$$\cpmwi :\adCT \to \adCT \otimes \adCThat.$$ 

%

\begin{example}
The following is an example of an admissible embedding: 
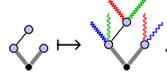
\[
\treeExampleProperlyNOLEGS \mapsto \treeExampleProperlyAd.
\]
\end{example}

The construction so far does not mirror the fact that noise types might appear multiple times on a given tree. In such a situation the cumulants built between noises satisfy certain symmetry constraints, and we want to mirror these symmetries 
at the level of the legs.
To this end, we perform the following construction.
\begin{definition}\label{def:Glegs}
We denote by $\Glegs$ the group of all permutations $\sigma$ of $\Legtype$ such that $\imap$ is invariant under $\sigma$ and with the property that $\overline{\sigma(\lt)}=\sigma(\bar\lt)$ for any $\lt \in \LT$.
\end{definition} 
We will often abuse notation and view elements $\sigma$ of $\Glegs$ as maps $\sigma : \FL \sqcup \LT \to \FL \sqcup \LT$ by extending $\sigma$ as the identity on $\FL$.
There exists an action $\CS$ of $\Glegs$ onto $\plCT$ and $\plCThat$ given by linearly and multiplicatively extending the map $(g,(\tau,\ft))\mapsto\CS^g(\tau,\ft):=(\tau,\CS^g\ft)$, and this action has the property that $\CS^g$ is an algebra automorphism on $\plCThat$, and a Hopf algebra automorphism on $\plCT$.
It will be useful to denote for $\lt_1,\lt_2 \in \LT$ with $\imap(\lt_1) = \imap(\lt_2)$ by $\lt_1 \lra \lt_2 \in \Glegs$ the group element given by sending $\legtype_1$ to $\legtype_2$, $\bar \legtype_1$ to $\bar\legtype_2$ (and vice versa), and letting ${(\legtype_1 \lra \legtype_2)}(\ft) := \ft$ for any $\ft \in \LT \backslash\{ \legtype_1, \legtype_2, \bar\legtype_1, \bar\legtype_2 \}$. Note that elements of the form $\lt_1 \lra \lt_2$ generate $\Glegs$.

\begin{example}
As an example, with $g = (\legRed \!\!\lra\!\! \legBlue)$ one has
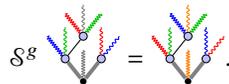
\[
\CS^{g}
\;
\treeExampleProperlyLegged
=
\treeExampleProperlyLeggedSym.
\]
\end{example}

With this notation we introduce the following definition.

\begin{definition}\label{def:symCT}
We denote by  $\wlS \subseteq \plCT$ and $\wlShat\subseteq \plCThat$ the ideals generated by all elements of the form $\CS^{g} \tau -\tau$ for some $\tau \in \plCT$ and $\tau \in \plCThat$ respectively, for some $g \in \Glegs$.

We then denote by 
\[
\symCThat:= \plCThat /\wlShat
\quad \text{ and } \quad
\symCT := \plCT / \wlS
\] 
the factor algebras with the canonical projections $\symPhat : \plCThat \to \symCThat$ and $\symP : \plCT \to \symCT$.
\end{definition}
As before, we define the natural embedding \label{idx:symfi}$\symfi : \symCT \to \symCThat$ induced by $\wli$.
We have the following lemma.
\begin{lemma}
The ideal $\wlS \subseteq \plCT$ forms  a Hopf ideal in $\plCT$, so that in particular the factor algebra $\symCT$ is a factor Hopf algebra. The algebra $\symCThat$ is a co-module over $\symCT$.
\end{lemma}
\begin{proof}
This follows from the fact that $\CS^g$ is a co-module and Hopf algebra automorphism on $\plCThat$ and $\plCT$, respectively.
\end{proof}

It will sometimes be convenient to view basis vectors $\tau \in \symCT$ (resp. $\tau \in \symCThat$) as basis vectors $\tau \in \plCT$ (resp. $\tau \in \plCT$). For this we simply fix, once and for all, a right inverse $\syminv : \symCT \to \plCT$ (resp. $\syminv:\symCThat \to \plCThat$) of the canonical projection, with the property that $\syminv$ maps trees onto trees.
Concerning the twisted antipode, we have the following analogue of Proposition~\ref{prop:twisted:antipode:pl}.

\begin{proposition}\label{prop:twisted:antipode:sym}
There exists a unique algebra homomorphism $\symCA: \symCT \to \symCThat$ such that the identity
\begin{align}\label{eq:definition:twisted:antipode:sym}
\CM (\symCA \otimes \Id) \cpmwi=\one^\star
\end{align}
holds on $\symCT$. Moreover, in terms of the operator $\wlCA$, this operator is uniquely determined by the relation
\begin{align}\label{eq:identity:twisted:antipode:sym}
\symCA \symP = \symPhat \wlCA
\end{align}
on $\plCT$.
\end{proposition}
\begin{proof}
The proof is identical to the proof of Proposition~\ref{prop:twisted:antipode:pl}.
\end{proof}

Finally, we have the following result.
\begin{lemma}\label{lem:iotasym}
The map $\symiota := \symP \iota : \CT_- \to \symCT$ is an algebra monomorphism and independent of the choice of admissible embedding $\iota : \CT_- \to \plCT$. 
\end{lemma}
\begin{proof}
The fact that $\symiota$ is independent of the admissible embedding follows directly from the definition. By definition $\symiota$ is a homomorphism of algebras, and the fact that $\symiota$ is one to one follows from the fact that $\wlS \subseteq \ker \pi$. 
\end{proof}

\subsubsection{Hopf algebra isomorphism}

We will now factor out a final ideal from $\symCT$ to obtain a factor Hopf algebra $\pT$ which is isomorphic 
as a Hopf algebra to $\CT_-$. There are two reasons why $\symCT$ is not already isomorphic 
to $\CT_-$. The first is that trees may contain more legs than needed to be properly legged, thus making $\symCT$ larger than $\CT_-$. The second reason is that legs may have non-vanishing derivative decoration, thus the decoration of trees in $\symCT$ is richer than in $\CT_-$.

To tackle the first issue, we denote for any tree $\tau \in \plCT$ by $\plQ \tau$ the tree obtained from $\tau$ by removing all legs which are not needed in order for $\tau$ to be properly legged. More precisely, suppose that $\tau= (T^\fn_\fe,\ft)$ is a properly legged tree. Then there exists (by definition of properly legged trees) for any pair of distinct vertices $u,v \in \CL(\tau)$ a leg $e \in L_\Legtype(\tau)$ such that $e^\downarrow = u$ and $\bar e^\downarrow = v$. We then call $e$ and $\bar e$ \emph{essential} legs, since we can not remove them from $\tau$ if we want the resulting tree to be properly legged. On the other hand, any leg $e \in L_\Legtype(\tau)$ such that either $e^\downarrow \notin \CL(\tau)$ or $\bar e^\downarrow \notin \CL(\tau)$ is called \emph{superfluous}, and we can remove it, together with its partner, while remaining properly legged.
We then set \label{idx:plQ}$\plQ \tau := (\tilde T^\fn_\fe,\ft)$, where $\tilde T \subseteq T$ denotes the subtree of $T$ obtained by removing all superfluous legs, and we extend $\plQ$ to a linear and multiplicative map, so that $\plQ : \plCT \to \adCT$ becomes an algebra homomorphism. (Note that $\plQ$ acts as the identity on the image of any admissible embedding $\iota:\CT_- \to \plCT$.) Composed to the left with the natural projection $\plCT \to \symCT$ and to the right with $\syminv : \symCT \hookrightarrow \plCT$, we obtain an algebra homomorphism $\plQ : \symCT \to \symCT$.

Similarly, we write \label{idx:plZ}$\plZ : \plCT \to \plCT$ for the multiplicative projection that kills trees with non vanishing derivative decoration on legs, formally given by 
\begin{align*}
\plZ(\tau) :=
\bigg\{ 
\begin{split}
	\tau		&\qquad \text{ if } \fe(e) = 0 \text{ for all } e \in L_\Legtype(\tau)\\
	0			&\qquad \text{ otherwise}
\end{split}
\end{align*}
for any tree $\tau \in \plCT$. As before, we use the same symbol for the map $\plZ:\symCT \to \symCT$ given by composing $\plZ$ with the natural projection and the embedding $\syminv$.

Finally, we denote by $\plQz: \plCT \to \plCT$ the multiplicative projection given by \label{idx:plQz}$\plQz = \plQ \plZ$. (Note that the order of the operators matters here.)
With this notation we now have the following
straightforward result, the proof of which is postponed 
to Section~\ref{sec:technical:proofs}. 
\begin{lemma}\label{lem:CTadze:hopf:ideal}
The ideal $\ker \plQz \subset \symCT$ is a Hopf ideal, so that in particular $\pT = \symCT / \ker \plQz$ 
is a Hopf algebra.
\end{lemma}

We now recall the projection $\pi:\wCT \to \CT_-$ given on a tree $\tau$ by simply removing all legs from $\tau$, which we naturally view as a projection $\pi:\symT{\adze} \to \CT_-$. 
Conversely, composing the embedding $\symiota:\CT_- \to \symCT$ of Lemma~\ref{lem:iotasym} with the canonical projection $\spadePsym : \symCT \to \pT$ yields an embedding $\CT_- \to \symT{\adze}$.
The next lemma shows that these two maps are actually Hopf algebra isomorphisms.

\begin{lemma}\label{lem:hopfiso}
The maps
\begin{align*}
\pi  : \pT \to \CT_-
\quad\text{ and }\quad
\hopfiso:=\spadePsym \symiota : \CT_- \to \pT 
\end{align*}
are Hopf algebra isomorphisms and one has $\hopfiso = \pi^{-1}$.
\end{lemma}
\begin{proof}
By construction the map $\pi\hopfiso$ is the identity on $\CT_-$. 
It is not hard to see that for any tree $\tau \in \CT_-$ the tree $\hopfiso \tau \in \pT$ is the unique tree in $\pT$ with the property that $\pi \sigma = \tau$. 
Note for this that any tree $\sigma\in \pT$ with the property that $\pi\sigma = \tau$ is the image of $\tau$ under an admissible embedding. The claim then follows from Lemma~\ref{lem:iotasym}. 
It follows that $\pi : \symT{\adze} \to \CT_-$ is an algebra isomorphism with inverse given  by $\hopfiso$. The fact that these maps are Hopf algebra isomorphisms follows from the explicit formula for the coproducts in (\ref{eq:coproduct:pl}) and (\ref{eq:coproduct:ex}), respectively.
\end{proof}

It will be useful to introduce the notation $\spadeP := \spadePsym \symP : \plCT \to \pT$ for the canonical projection. Also, for later use we point out that the projection $\plZ$ is also well-defined on $\plCThat$ and $\symCThat$. (Note however that $\plQ$ is not!) Note also that one actually has $\plQz:\plCT \to \adCT$, and $\plQz$ is the identity on $\adCT$.

\subsubsection{Evaluations and characters}
\label{sec:evaluation:characters}
We start by rephrasing the statement of Theorem~\ref{thm:evaluation:trees:largescale} into a form that is more suited to our analysis below. First we introduce the following terminology. We say that a typed set $(A,\ft)$ with $\ft: A \to \Legtype$ is \emph{properly typed} if $\ft$ is injective and such that 
every $a \in A$ has a \emph{partner} $\bar a \in A$ such that $\ft(\bar a)= \overline{\ft(a)}$.
We also call a subset $A \ssq \Legtype$ proper typed if for any $\lt \in A$ one has $\bar\lt \in A$. A multiset $\cm$ with values in $\Legtype$ is properly typed if it is a properly typed set.
Note that $(L_\Legtype(\tau),\ft)$ is properly typed for any properly legged tree $\tau$.
For the next definition we fix $\bar R> \max_{\tau \in \TT_-} \# K(\tau)$, and fix an element $\bar\psi \in \Psi$ such that $\bar\psi_\cm( x_\cm ) = 1$ for any $x_\cm \in \bar\domain^{ \td(\cm) }$ such that $|x_p-x_q|_\fs \le \bar R$ for any $p,q \in \td(\cm)$.

\begin{definition}\label{def:FN}
We denote by $\SN$ the set of all families $(\phi_{\legtype})_{\legtype \in \Legtype}$ of smooth functions
with the property that $\phi_{\legtype}\in \CCinfg$ 
and $\phi_{\legtype} = \phi_{\bar\legtype}(- \cdot)$ for any $\legtype \in \Legtype$. We also impose that $\phi_\legtype$ is supported in a centred scaled ball or radius $\bar R$ around the origin.
Any such family $\phi\in\SN$ determines an element $\hat\phi \in \Psi$ given by
\begin{align}\label{eq:Upsilon:bar:phi:psi}
\hat\phi_{A} ( x_A ):=
\bar\psi_A (x_A)
\prod_{ 
	\{\lt,\bar \lt\} \ssq A 
}
\phi_{ \lt }(x_{\lt} - x_{\bar \lt}),
\end{align}
for properly typed set $A \ssq \Legtype$.
We simply set $\hat\phi_{\cm}:=0$ if $\cm$ is not not properly typed.
\end{definition}
In (\ref{eq:Upsilon:bar:phi:psi}) there is one factor for leg-type $\lt$ and its partner $\bar \lt$. Note that the right-hand side is well-defined, since $\phi_{\lt}=\phi_{\bar\lt}(-\cdot)$. 
Without the function $\psi$ one would not have $\hat\phi \in \Psi$, since $\hat\phi$ would not be compactly supported in the differences of its arguments. 
Note however that $\bar \psi$ plays no role in the definition of the evaluation $\evaln{\eta}{\hat \phi}{R}$, since by (\ref{eq:bar:Upsilon:large:scale}) and (\ref{eq:CK}) only the function
\begin{align*}
(x_u) _{u \in \LTN} \mapsto \int dx_{\LTE(\tau)} \hat\phi(x_{\LTE(\tau)}) \prod_{e \in \LTE(\tau)} \delta_0( x_e - x_{e^\downarrow})
\end{align*}
enters the definition of $\evaln{\eta}{\hat\phi}{R}$, and by the support properties of $\phi_\lt$ this expression does not depend on the choice of $\bar\psi$ for any properly legged tree $\tau$. 
Finally, $\bar R$ is chosen such that one can find a tuple $\phi \in \SN$ with the property that (\ref{eq:identity:remove:cutoff}) holds (recall that the truncated integration kernels $K_\ft$ are supported in the centred ball of radius $1$).
We first have the following consequence of Theorem~\ref{thm:evaluation:trees:largescale}.
\begin{corollary}\label{cor:evaluation:trees:largescale}
Let $\tau\in \wTex$ be a properly legged tree, and let $\phi \in \SN$. Then for any $\eta \in \SM_\infty^\star$ the evaluation $R \mapsto \bar\Upsilon^{\eta,\hat\phi}_R \tau$ extends continuously to $\CK_0^+$, and the evaluation $(\eta,R) \mapsto \hat \Upsilon^{\eta,\hat\phi}_R \tau$ defined in (\ref{eq:evaluation:largescale:renorm:2}) extends continuously to the space $\SM_0^\star \times \CK^+_0$.  
\end{corollary}

We will abuse notation a bit and simply write $\hat \Upsilon^{\eta,\phi}_R \tau := \hat \Upsilon^{\eta, \hat\phi}_R \tau$ and similarly for $\bar \Upsilon^{\eta,\phi}_R$, $\eval{\phi}{R}$ and $\eva{\phi}$.
Given a smooth noise $\eta\in\SM_\infty^\star$, an element $\phi \in \SN$ and a large-scale kernel assignment  $R\in\CK^+_0$, we want to define a character $\wlchar{g^{\eta,\phi}_R}$
 on $\plCT$ which is defined analogously to the BPHZ character $g^\eta_\BPHZ$, but where the kernel assignment in the evaluations is replaced by $K+R$, and where we introduce a cutoff according to $\phi$.

To this 	end we first define a character on $\plCThat$ by linearly and  multiplicatively extending the evaluation $\bar \Upsilon^{\eta,\phi}_R$.
We then define a character $\wlchar{g^{\eta,\phi}_R}$ on $\plCT$ via the identity
\begin{align}\label{eq:g:wl:eta:psi:R}
\wlchar{g^{\eta,\phi}_R}(\tau) := \bar \Upsilon ^{\eta,\phi}_R \wlCA\tau.
\end{align}

Let \label{idx:SNsym}$\SN_{\sym}$ be defined as the set of families $\phi \in \SN$ which are invariant under $\Glegs$ in the sense that $\phi_{g(\legtype)} = \phi_\legtype$ for any $g \in \Glegs$.
This definition ensures that one has the identity
\[
\bar \Upsilon^{\eta,\phi}_R \CS^{g} = \bar \Upsilon^{\eta,\phi}_R
\]
for every $g \in \Glegs$, and hence the character $\bar \Upsilon^{\eta,\phi}_{R}$ vanishes on the ideal $\wlShat$ which we used in Definition~\ref{def:symCT} to define the factor algebra $\symCThat$. It follows that $\bar \Upsilon^{\eta,\phi}_{R}$ is well-defined on $\symCThat$ for any $\phi \in \SN_\sym$, and thus we can define a character $\wlchar{g^{\eta,\phi}_R}$ on $\symCT$ via the identity
\begin{align}\label{eq:g:sym:eta:psi}
\wlchar{g^{\eta,\phi}_R}(\tau) := \bar\Upsilon^{\eta,\phi}_R \symCA \tau.
\end{align}
(Comparing this with (\ref{eq:g:wl:eta:psi:R}) and Proposition~\ref{prop:twisted:antipode:sym} one has $\wlchar{g^{\eta,\phi}_R} \symP = \wlchar{g^{\eta,\phi}_R}$ on $\plCT$.)

The following lemma shows the relation between the characters $\wlchar{g^{\eta,\phi}_0}$ on $\plCT$ on the one hand, and the usual BPHZ character $g^\eta_\BPHZ$ on $\CT_-$ on the other hand.
\begin{lemma}\label{lem:identity:giota:g}
Let $\phi\in \SN$ and assume that for any $\legtype \in \Legtype$ one has that $\phi _\legtype=1$ in a large enough neighbourhood of the origin. Then for any $\eta\in\SM_\infty$ one has the identity
\begin{align}\label{eq:identity:giota:g}
\wlchar{g^{\eta,\phi}_0} = g^\eta \pi  \zeroP
\qquad \text{ and }\qquad
\wlchar{g^{\eta,\phi}_0} \iota = g^\eta_\BPHZ
\end{align}
on $\plCT$ and $\CT_-$, respectively, for any admissible embedding $\iota:\CT_- \to\plCT$.

Moreover, one has
\begin{align}\label{eq:identity:eval}
\evaln{\eta}{\phi}{R} M^{g^\eta} \wlj
=
\evaln{\eta}{\phi}{R} M^{\wlchar{g^{\eta,\phi}_0}}
\end{align}
on $\plCThat$. Here on the left hand side we view $\CGm \ssq \wCG$ as in Section~\ref{sec:ext:reg:strct}.
\end{lemma}
\begin{proof}
Let $\phi\in \SN$ be such that for any $\legtype \in \Legtype$ the test function $\phi_\legtype$ is $1$ in a neighbourhood 
of the origin which is large enough so that one has for any smooth noise $\eta\in\SM_\infty$ the identity
\begin{align}\label{eq:identity:remove:cutoff}
\bar\Upsilon^{\eta,\phi}_0  \tau = \Upsilon^{\eta} \pi \tau
\end{align}
for any tree $\tau \in \plCThat$ with the property that the derivative decoration $\fe$ vanishes on the set of legs $L_\Legtype(\tau)$. 

We first show (\ref{eq:identity:giota:g}). Note that with the same arguments that shows (\ref{eq:identity:remove:cutoff}) it also follows that one has $\bar\Upsilon^{\eta,\phi}_0 \tau=0$ for any tree $\tau \in \plCThat$ with the property that the derivative decoration $\fe$ does not vanish identically on the set of legs. Writing $\zeroP : \plCT \to \plCT$ and $\zeroPhat : \plCThat \to \plCThat$ for the multiplicative projections onto the respective subalgebra generated by trees $\tau \in \plCT$ and $\tau \in \plCThat$ respectively, with the property that the decoration $\fe$ vanishes identically on the set of legs of $\tau$, the two previous observations are equivalent to the identity
\begin{align}\label{eq:identity:remove:cutoff:2}
\bar \Upsilon^{\eta,\phi}_0 = \bar \Upsilon^{\eta,\phi}_0  \zeroPhat = \bar \Upsilon^\eta \pi  \zeroPhat
\end{align}
on $\plCThat$.

Noting that one has $\pi \zeroP \iota = \Id$ on $\CT_-$, we are left to show that $\wlchar{g^{\eta,\phi}_0} = g^\eta \pi  \zeroP$ on $\plCT$, which, with the aid of~(\ref{eq:identity:remove:cutoff:2}) and the definition of the respective character, follows once we show the identity
\[
\pi  \zeroPhat \wlCA = \tilde\CA_-^\ex \pi  \zeroP
\]
on $\plCT$. In order to see this, apply the operator on either side of this identity to some tree $\tau$ and proceed inductively in the number of edges of $\tau$. We then have the identities
\begin{align*}
\tilde\CA_-^\ex \pi  \zeroP \tau
&=
-\CM( \tilde\CA_-^\ex \otimes \Id) (\cpmi - \Id \otimes \one)
	\pi  \zeroP \tau \\
&=
-\CM(  \tilde\CA_-^\ex \otimes \Id ) (\pi  \zeroP  \otimes \pi  \zeroPhat )
	( \cpmwi - \Id \otimes \one)  \tau  \\
&=- \pi  \zeroPhat  \CM ( \wlCA \otimes \Id ) ( \cpmwi - \Id \otimes \one) \tau =  \pi  \zeroPhat  \wlCA \tau\;,
\end{align*}
and the claim follows.

We now show (\ref{eq:identity:eval}). Denote by $\qNL: \wCT \to \CTm$ the algebra homomorphism such that, for any tree $\tau \in \wCT$ one has $\qNL \tau = \pi \tau$ if $L_\Legtype(\tau) = \emptyset$ and $\qNL\tau = 0$ otherwise. Let furthermore denote by $\Qleg : \wCT \to \wCT$ the multiplicative projection which on a tree $\tau \in\wCT$ acts by removing all legs $e \in L_\Legtype(\tau)$ without a partner and such that $\fe(e) = 0$, and set $\Qlegz := \wlPhat \Qleg$. Then one has
\begin{align}\label{eq:noname:1}
\evaln{\eta}{\phi}{R} M^{g^\eta} \wlj
=
(g^\eta \qNL \otimes \evaln{\eta}{\phi}{R}) \cpmh \wlj 
=
(g^\eta \qNL \otimes \evaln{\eta}{\phi}{R} \Qlegz) \cpmh \wlj 
\end{align}
on $\plCThat$. The first equality is a consequence of the embedding $\CGm \ssq \wCG$, compare in Section~\ref{sec:ext:reg:strct}. The second equality is a consequence of the fact that the only legs $e$ appearing in the right component of this tensor product which do not have a partner are such that there exists a leg $\bar e$ with $\ft(\bar e) = \overline{\ft(e)}$ and $e^\downarrow = \bar e^\downarrow$. Since $\phi=1$ in a neighbourhood of the origin, if the derivative decoration of these legs is zero they do not contribute to the evaluation $\evaln{\eta}{\phi}{R}$, while in case that the derivative decoration does not vanish they kill the evaluation $\evaln{\eta}{\phi}{R}$. Either way, inserting the projection $\Qlegz$ does not change (\ref{eq:noname:1}). 

Next we note that one has
\begin{align}\label{eq:noname:2}
\evaln{\eta}{\phi}{R} M^{\wlchar{g^{\eta,\phi}_0}}
=
(g^\eta \pi \plZ \otimes \evaln{\eta}{\phi}{R}) \cpmh 
\end{align}
on $\plCThat$, where we used (\ref{eq:identity:giota:g}), and combining (\ref{eq:noname:1}) and (\ref{eq:noname:2}) we are left to show that
\begin{align}\label{eq:identity:eval:1}
(\qNL \otimes \Qlegz) \cpm \wlj 
=
(\pi \plZ \otimes \Id) \cpm
\end{align}
on $\plCThat$.

For this we use the forest expansion of $\cpm$  on $\wCThat$ given by (\ref{eq:coproduct:ex}) and on $\plCThat$ given by (\ref{eq:coproduct:pl}). First, due to the projection $\qNL$ on the right hand side of (\ref{eq:identity:eval:1}) the first sum in (\ref{eq:coproduct:ex}) can be restricted to $\CF \in \div(\pi\tau)$. The sum over all polynomial decorations is already identical, but (\ref{eq:coproduct:ex}) include a sum over edge decoration put on legs $e \in L_\Legtype(\underline\CF)$, where $\ul\CF$ is as in (\ref{eq:forests:legs}). Any term where $\fe_\CF$ does not vanish on such legs gets killed by $\Qlegz$, so that we can restrict the sum  over $\fe_{\CF}$ in (\ref{eq:coproduct:ex}) to $\fe_{\ul\CF}$ as in (\ref{eq:coproduct:pl}). 

Now fix $\CF \in \div(\pi\tau)$ and decorations $\fn_{\ul\CF}$ and $\fe_{\ul\CF}$.
We show that
\begin{align}\label{eq:noname:3}
\prod_{S \in \CF}  S^{\fn_{\ul\CF} + \pi\ce_{\ul\CF}}_\fe
\otimes
\Qlegz (T/\CF)^{\fn-\fn_{\ul\CF},  [\fo]_{\forestlegs\CF}}
_{\fe + \fe_{\ul\CF}}
=
\prod_{S \in \forestlegs\CF} \pi \plZ S^{\fn_{\forestlegs\CF} + \pi\ce_{\forestlegs\CF}}_\fe
\otimes
(T/\forestlegs\CF)^{\fn-\fn_{\forestlegs\CF},  [\fo]_{\forestlegs\CF}}
_{\fe + \fe_{\forestlegs\CF}},
\end{align}
which concludes the proof.

It there exists a leg $e \in L_\Legtype(\ul\CF)$ with $\fe(e)\ne 0$ then both sides of (\ref{eq:noname:3}) vanish. On the other hand, if $\fe(e)=0$ for all legs $e \in L_\Legtype(\ul\CF)$, then one can remove the projection $\plZ$ from the right hand side of (\ref{eq:noname:3}), which then becomes
\[
\prod_{S \in \CF} S^{\fn + \pi\ce_{\forestlegs\CF}}_\fe
\otimes
(T/\forestlegs\CF)^{\fn_{\forestlegs\CF}-\fn_{\forestlegs\CF},  [\fo]_{\forestlegs\CF}}
_{\fe + \fe_{\forestlegs\CF}},
\]
so that we are left to show that
\[
\Qlegz (T/\CF)^{\fn-\fn_{\ul\CF},  [\fo]_{\forestlegs\CF}} _{\fe + \fe_{\ul\CF}}
=
(T/\forestlegs\CF)^{\fn_{\forestlegs\CF}-\fn_{\forestlegs\CF},  [\fo]_{\forestlegs\CF}}
_{\fe + \fe_{\forestlegs\CF}},
\]
which is a consequence of the definition of $\Qlegz$.
\end{proof}


\subsection{An analytic result}
\label{sec:constraints:analytic}

In this section we are going to show an analytic result, Proposition~\ref{prop:eps:beta:bound}, which 
we will then use as a black box in the next section. 
Our goal is to study how the evaluations $\barg \Upsilon^{\eta,\phi}_R \tau$ for $\phi \in \SN$ behave when the smooth functions $\phi_\legtype$ for $\legtype \in \Legtype$ are rescaled to small scales. More concretely, assume that we are given a degree assignment $\deg : \Legtype \to \R_-\cup \{\redef\}$ that is invariant under conjugation.  For any family $\phi \in \SN$ we define a rescaled family $\phi^\eps \in \SN$ by setting
\begin{align}\label{eq:phi:rescale}
\phi_\legtype^\eps := 
\begin{cases}
	 \eps^{2\deg(\legtype)} \SS^{[\eps]} \phi_\legtype 	
			\quad &\text{ if }\deg(\legtype) \in \R_- \\
	 \phi_\legtype									\quad &\text{ if }\deg(\legtype)=\redef
\end{cases}
\end{align}
for any $\legtype \in \Legtype$ and $\eps>0$. Here, we define the rescaling operator $\SS^{[\eps]}$ by setting
\[
(\SS^{[\eps]}\varphi) (x):=
\varphi
		 ( \eps^{-\fs} x)
\]
for any $\varphi \in \CC_c^\infty(\bar\domain)$.

We will now describe a particular way to choose degree assignments $\deg : \Legtype \to \R_-\cup\{ \redef \}$. 
We fix an arbitrary homogeneity assignment $\boxnorm\cdot _\fs$ on $\FL_-$ with the property that one has $|\Xi|_\fs < \boxnorm\Xi _\fs < \fancynorm\Xi _\fs$ for any noise type $\Xi \in \FL_-$. For any set $\LTa \ssq \Legtype$ of leg types which is closed under conjugation we define a degree assignment $\LTdeg : \Legtype \to \R_-\cup\{ \redef \}$ by setting
\begin{align} \label{eq:definition:deg:I}
\LTdeg ( \legtype ) := 
\frac{
	\frac{1}{2}(
		\boxnorm{\imap_1(\legtype)}_\fs + \boxnorm{\imap_2(\legtype)}_\fs
	)
}
{-\frac{1}{2}+\sqrt{\frac{1}{4} + \# \LTa}}
\qquad\text{ if } \legtype \in \LTa,
\end{align}
and $\LTdeg( \legtype ):=\redef$ if $\legtype \notin \LTa$. The factor in (\ref{eq:definition:deg:I}) is chosen in such a way that one has for any tree $\tau \in \adCT$ the identity
\begin{align} \label{eq:definition:deg:I:explain}
\sum_{\legtype \in \LTa} \LTdeg(\legtype) = \sum_{e \in L(\tau)} \boxnorm{\ft(e)},
\end{align}
with $\LTa = \{ \ft(e) : e \in L_\Legtype(\tau) \}$ the set of leg types appearing in $\tau$.
Let us sketch the argument why (\ref{eq:definition:deg:I:explain}) is true. Since $\tau \in \adCT$ there are 
no superfluous legs in $\tau$, so that $\# \LTa = \# L(\tau) (\# L(\tau)-1)$. It follows that the denominator 
in (\ref{eq:definition:deg:I}) is simply given by $\# L(\tau) - 1$, which is equal to  $\{ e \in L_\Legtype(\tau) : e^\downarrow =u \}$
 for any $u \in \CL(\tau)$, and one has
\[
\sum_{\legtype \in \LTa} \LTdeg (\legtype) = 
\sum_{u \in L(\tau)}
	\sum_{e \in L_\Legtype(\tau) : e^\downarrow = u^\downarrow}
		\frac{\boxnorm{\ft(u)}_\fs}
			{\# L(\tau) - 1} 
=\sum_{e \in L(\tau)} \boxnorm{\ft(e)}.
\]

More generally, assume that we are given a system $\LTsys$ of non-empty, disjoint subsets of $\Legtype$ such that each $\LTa \in \LTsys$ is invariant under conjugation (we allow $\LTsys = \emptyset$, but we impose $\emptyset \notin \LTsys$). Then we define a degree assignment $\LTsysdeg$ by setting
\begin{align}\label{eq:definition:deg:FJ}
\LTsysdeg := \sum_{\LTa \in \LTsys} \LTdeg,
\end{align}
with the convention that $\alpha + \redef := \alpha$ for any $\alpha \in \R_- \cup \{ \redef \}$. We write $\LTsyss$ for the set of all systems $\LTsys$ as above.

We define for any $\LTsys \in \LTsyss$, any smooth tuple $\phi \in \SN$, and any large-scale kernel assignment $R \in \CK_0^+$ a character $\Lchar{\phi}{\LTsys}{R}$ on $\plCT$ by setting
\begin{align}\label{eq:h:CI:phi:R}
\Lchar{\phi}{\LTsys}{R}\tau := 
	- \sum_{\LTa \in \LTsys} \eval{\phi}{R} \LTp \tau
\end{align}
for any tree $\tau \in \plCT$, and extending this linearly and multiplicatively. Here, we introduce the linear (but not multiplicative!) projections \label{idx:LTp}$\LTp : \plCT \to \plCT$ onto the subspace of $\plCT$ spanned by all trees $\tau \in \adCT$ with the property that $\ft(L_\Legtype(\tau)) = \LTa$. In analogy to above we write $\Lch := \Lchar{ \phi }{ \LTsys }{ \hat K - K }$. 

\begin{remark}
We could make $P_\LTa$ multiplicative without changing (\ref{eq:h:CI:phi:R}). However, we will later on introduce the notation $P_\LTsys$, where $\LTsys$ is a system of subsets of $\LT$, in a similar way, projecting onto a subspace spanned by products of trees, see (\ref{eq:delta:PFI}). At this point we really want to consider the linear and not multiplicative projection, so we choose the definition introduced above to be consistent.
\end{remark}

The goal of the present section is to obtain bounds on the quantity
\[
(\Lche
	\otimes
    \eva{ \phi^\eps }
)
	\cpm
	\tau
\]
as $\eps \to 0$ for any $\tau \in \adCT$, where $\phi^\eps$ is defined as in (\ref{eq:phi:rescale}) for the degree assignment $\LTsysdeg$. 
\begin{example}
Let $\LTsys:= \{ \{\treehlega,\treehlegb\} , \{\treehlegc, \treehlegd\} \}$ and consider the following example from the generalised KPZ equation
\begin{multline*}
(\Lche \otimes \eva{\phi^\eps}) \cpm \treeExampleKPZh
=
\eva{ \phi^\eps } \treeExampleKPZh
-
\eva{\phi^\eps }\treeExampleKPZhL
\,
 \eva{\phi^\eps} \treeExampleKPZhR
\\
-
\sum_{\te \in \{ \treehlegd,\treehlege,\treehlegf,\treehlegg \}}
\eva{ \phi^\eps }
\treeExampleKPZhLd
\, 
\eva{ \phi^\eps }
D_\te \treeExampleKPZhR.
\end{multline*}
Here $D_{\treehlegd}$, etc., is a shortcut for putting a derivative decoration $(0,1)$ on the respective edge.
We also colour a node $u$ blue if it contains an $X$-decoration, i.e. $\fn(u) = (0,1)$.
One has $\# \{\treehlega,\treehlegb\} = 2$ so that (\ref{eq:definition:deg:I}) gives $\LTsysdeg(\treehlega) = -\shalf-\kappa$.
It follows that 
\[
\eva{\phi^\eps }\treeExampleKPZhL \simeq \eps^{-1-2\kappa}
\,, \qquad
\eva{\phi^\eps }\treeExampleKPZhLd \simeq \eps^{-2\kappa}
\]
as $\eps \to 0$. Note that the counter-terms cancel out precisely the subdivergence in the big tree. The fact that one has $\{\treehlegc, \treehlegd\} \in \LTsys$ changes nothing, since there is no subtree of negative homogeneity in the image of $P_{\{\treehlegc, \treehlegd\}}$.
\end{example}

There are some technical subtleties in the proof that require us to put certain assumptions on the test tuple $\phi$ in order for good bounds to hold, and we summarise these assumptions in the following definition. 
\begin{definition}\label{def:SN:FP}
Given a system $\LTsys \in \LTsyss$, and $\delta>0$, we define the set $\SN(\LTsys,\delta) \ssq \SN$ as the set of $\phi \in \SN$ such that both of the following properties hold for any $\legtype \in \Legtype$.
\begin{itemize}
\item If $\legtype \notin \bigsqcup \LTsys$, then one has that $\phi_\legtype = 0$ in the $\delta$-ball of the origin.
\item If $\LTsysdeg \legtype \le -\shalf$, then one has that $\int \phi_\legtype (x) dx = 0$.
\end{itemize}
We write $\SN(\LTsys)$ for the union of $\SN(\LTsys,\delta)$ over $\delta>0$.
\end{definition}

\def\CW{\Legtype}
\def\cw{\legtype}

Let us briefly comment why these assumptions will play a role later on. The first assumption ensures that under rescaling $\phi$ as in (\ref{eq:phi:rescale}) all subtrees $\sigma$ of a tree $\tau \in \adCT$ that trigger a divergence have the property that $\ft(L_\Legtype(\sigma)) \in \LTsys$. Without this assumption, one would have to consider additionally any subtree $\sigma$ with the property that $\ft(L_\Legtype(\sigma)) $ can be written as $\ft(L_\Legtype(\sigma))  = \bigsqcup_{i \le m} \LTa_i$ for some sets $\LTa_i \in \LTsys$. In particular, this assumption means that we never have to deal with nested divergencies. The second assumption above simply ensures that the test functions $\phi^\eps$ converge to $0$ in the distributional sense under the rescaling (\ref{eq:phi:rescale}). One always has $2\deg^\FJ \cw > -|\fs|-1$ for any $\cw \in \CW$, which follows from the assumption that $\homofancys{\Xi} \ge -\shalf$ for any $\Xi \in \FL_-$.

In order to state the next result, we need a final piece of notation. Let $\tau \in \adCT$ be a tree. Given $M \ssq \CL(\tau)$ we denote by $\LT(\tau, M)$ the set of leg types $\ft(e) \in \LT$ with $e\in L_\LT(\tau)$ such that both $e$ and $\bar e$ are incident to $M$. (Note in particular that $\LT(\tau,M)$ is closed under conjugation.) 
We also write \label{idx:Lambda(tau)}$\Lambda(\tau)$ for the set of all systems $\CM$ of disjoint, non-empty subsets of $\CL(\tau)$. (Note that one has $\emptyset \in \Lambda(\tau)$, but for any $\CM \in \Lambda(\tau)$ one has $\emptyset \notin \CM$.)
\begin{definition}\label{def:II:tau}
We write $\LTsyss(\tau)$ for the set of all $\LTsys \in \LTsyss$ of the form
\begin{align}\label{eq:II:tau}
\LTsys = \{ \LT(\tau,M) : M \in \CM \}
\end{align}
for some $\CM \in \Lambda(\tau)$.
\end{definition}

With these notations, we will show the following statement.

\begin{proposition}\label{prop:eps:beta:bound} 
Let $\tau \in \adCT$ be an admissible tree, let $\LTsys \in \LTsyss(\tau)$, and let $\phi \in \SN(\LTsys)$ be a tuple of smooth functions. Define the rescaled family $\phi^\eps$ as in (\ref{eq:phi:rescale}) for the degree assignment $\LTsysdeg$ defined in (\ref{eq:definition:deg:FJ}). Let finally $R,R' \in \{0, \hat K-K \}$ be large scale kernel assignments. Then, there exists $\beta>0$ such that one has the bound
\begin{align}\label{eq:bound:eps:beta}
\big|
\big(
\Lchar{\phi}{\LTsys}{R}
	\otimes
	\eval{\phi^\eps}{R'}
\big )
\cpmwi
	\tau
\big| 
\lesssim \eps^\beta
\end{align}
uniformly over $\eps>0$.
\end{proposition}


We will show Proposition~\ref{prop:eps:beta:bound} by applying the results of \cite{Hairer2017}, which ultimately comes down to comparing the character $h$ with the BPHZ character for a suitable space of Feynman diagrams. Since this proof is largely technical, we postpone it to Appendix~\ref{sec:proof:of:eps:beta:bound}.

\subsection{The ideal \texorpdfstring{$\CJ$}{J} is a Hopf ideal}
\label{sec:constraints:ideal}

We will construct an ideal $\wlCJ$ in $\plCT$ that is related to the ideal $\CJ$ given in Definition~\ref{def:CH} via the projection $\pi$ (Lemma~\ref{lem:CJ:wlCJ}). 
We will work below with the space $\linspace{\SN}$ of formal linear combinations of elements of $\SN$. The notation (\ref{eq:Upsilon:bar:phi:psi}) can be linearly extended
to an operator $\hat\cdot : \linspace{\SN} \to \Psi$, where $\Psi$ is as in Definition~\ref{def:Psi}. 

\begin{remark}
For a fixed properly typed set $A \ssq \Legtype$ the set $\{ \hat\phi_{A} : \phi \in \linspace\SN \}$ is dense in $\barCCAg$ (say with respect to the topology of $\CC^k$, for any $k>0$). Note however that the definition of $\hat\phi$ puts non trivial constraints between  $\hat\phi_A$ and $\hat\phi_B$ whenever $A \ssq B$.
\end{remark}

\begin{definition}
Given a linear combination $\phi \in \linspace{\SN}$, say $\phi = \sum_{i\le r}c_i \phi_i$ with $c_i \in \R$ and $\phi_i \in \SN$, we define the character $\eva{\phi}$ on $\plCThat$ by setting
\[
\eva{\phi} \tau 
:=
\sum_{i \le r}
c_i \eva{\phi_i} \tau,
\]
for any tree $\tau \in \plCThat$, and extending this linearly and multiplicatively.
\end{definition}

We now fix a partition $\Lpart$ of the set of leg types $\LT$ with the property that for any fixed $P \in \Lpart$ the noise type $\ft(P):=\imap_1(\lt) \in \FL_-$ does not depend on the representative $\legtype \in P$.
We then introduce the following terminology. 
\begin{definition}\label{def:good:trees}
We call a tree $\tau \in \plCT$ or $\tau \in \plCThat$ \emph{good} if there exists an injection $\zeta: \CL(\tau) \to \Lpart$ with the property that for any leg $e \in L_\LT(\tau)$ with $e^\downarrow \in \CL(\tau)$ one has $\ft(e) \in \zeta(e^\downarrow)$.
We write $\plCT[\Lpart]$ and $\plCThat[\Lpart]$ for the subalgebras generated by good trees.
\end{definition}

We can (and will) assume without loss of generality that $\LT$ is large enough and $\Lpart$ is such that there exists an admissible embedding $\iota: \CT_- \to \plCT$ mapping any tree $\tau\in \CT_-$ onto a good tree $\iota\tau \in \plCT$. We also note that these subalgebras are stable under the coproduct, namely one has
\[
\cpmh : \plCT[\Lpart] \to \plCT[\Lpart] \otimes \plCT[\Lpart]
\]
and 
\[
\cpm : \plCThat[\Lpart] \to \plCT[\Lpart] \otimes \plCThat[\Lpart].
\]
Moreover, the following simple lemma will be helpful, which contains the motivation for the preceding definition.
\begin{lemma}
Assume that $\tau ,\bar\tau \in\adCT$ are good trees such that $\ft(L_\LT(\tau)) = \ft(L_\LT(\bar\tau))$. Then $\LTsyss(\tau) = \LTsyss(\bar \tau)$, where $\LTsyss(\tau)$ is as in Definition~\ref{def:II:tau}.
\end{lemma}
\begin{proof}
Let $\zeta:\CL(\tau) \to \Lpart$ and $\bar\zeta:\CL(\bar \tau) \to \Lpart$ be the injections used in the definition of good trees. Observe that the condition of the lemma implies that $\zeta$ and $\bar\zeta$ have the same range, so that $\eta:= \bar\zeta^{-1} \circ \zeta$ defines a bijection from $\CL(\tau)$ to $\CL(\bar\tau)$. This induces a bijection from $\Lambda(\tau)$ to $\Lambda(\bar\tau)$, and the result follows immediately from Definition~\ref{def:II:tau}.
\end{proof}

With this notation, we define the following ideals.

\begin{definition}\label{def:wlCJ}
We define $\wlCJ \subseteq \plCT$ as the ideal generated by all $\sigma\in\plCT[\Lpart]$ 
with the property that 
\begin{equ}\label{eq:wlCJ}
\eva{\phi} \plQz\sigma=0\;,
\qquad \text{ for any } \phi\in \linspace{\SN},
\end{equ}
with $\plQz$ as in Lemma~\ref{lem:CTadze:hopf:ideal}. We also let $\wlCJad \ssq \plCT$ denote the ideal generated by all $\sigma \in \plCT[\Lpart] \cap \adCT$ such that (\ref{eq:wlCJ}) is satisfied.

Finally, we define $\wlCJhat \ssq \plCThat$ as the ideal generated by all $\sigma \in \plCThat[\Lpart]$  with the property that 
\begin{equ}\label{eq:wlCJhat}
\eva{\phi} \sigma=0
\qquad \text{ for any } \phi\in \linspace{\SN}.
\end{equ}
\end{definition}

Note that these ideals depend on $\Lpart$, but we think of $\Lpart$ as fixed from now on and hide this dependence in the notation. 
One has $\plQz \wlCJ = \plQz \wlCJad$, and since $\spadeP = \spadeP \plQz$ on $\plCT$, one has the identity
\begin{align}\label{eq:identity:adzeCJ}
\adzeCJ := \spadeP \wlCJ = \spadeP \wlCJad
\end{align}
as ideals on $\pT$.

\begin{remark}
We use $\plQz \sigma$ instead of just $\sigma$ in (\ref{eq:wlCJ}) to ensure that (\ref{eq:identity:adzeCJ}) holds. If $\eva{\phi}(\tau+\tilde \tau)=0$ for two trees $\tau,\tilde\tau$ and all $\phi \in \linspace{\SN}$, we easily infer that $\tau$ and $\tilde\tau$ contain the same leg types (unless the evaluation vanishes on both trees individually), but there is no reason for $\tau$ and $\tilde\tau$ to contain the same ``essential leg types'' (i.e.\ the set of types of essential legs), so that there is no obvious relation between $\plQz\tau$ and $\plQz\tilde\tau$. We cannot use $\plQz$ in (\ref{eq:wlCJhat}), since this projection is not well-defined on $\plCThat$. In particular is does \emph{not} hold that $\wlCJ = \wlcol{\p_-} \wlCJhat$. However, if $\tau= \sum_{i}c_i \tau_i \in \wlCJhat$ is a linear combination of trees with $|\tau_i|_- <0$, and we know a priori that all the $\tau_i$ contain the same essential leg types, then we can conclude that $\wlcol{\p_-} \tau \in \wlCJ$.
\end{remark}

We use $\phi \in \linspace{\SN}$ in the preceding definition rather than $\phi \in \SN$ so that Lemma~\ref{lem:ideal:generators} can be applied,
which ensures that the ideals $\wlCJ$ and $\wlCJhat$ are generated by linear combinations of trees, see Lemma~\ref{lem:wlCJ:generators}.
 Note that if $\sigma \in \plCT$ or $\sigma \in \plCThat$ is a linear combination of trees, then (\ref{eq:wlCJ}) for all $\phi \in \SN$ is equivalent to (\ref{eq:wlCJ}) for all $\phi \in \Vec \SN $. More generally, one has the following.

\begin{lemma}\label{lem:consistently:legged}
Let $\sigma \in \plCThat$ (respectively $\sigma \in \plCT$) be of the form
\begin{align}\label{eq:consistently:legged}
\sigma = \sum_{i \le r} c_i \prod_{j \le m} \tau_{i,j}
\end{align}
for some collection of trees $\tau_{i,j} \in \plCThat$ (respectively $\tau_{i,j} \in \plCT$) and some $m,r\ge 1$. Assume that the multisets $\fm_{i,j}:=[L_\Legtype(\tau_{i,j}), \ft]$ of leg types have the following two properties.
\begin{enumerate}
\item \label{item:consisitently:legged:1}
For fixed $j \le m$ the trees $\tau_{i,j}$ contain the same leg types for any $i \le r$, i.e.\ one has that $\fm_j := \fm_{i,j}$ is independent of $i \le r$.
\item \label{item:consistently:legged:2}
Any leg type appears at most once, i.e.\ one has that $\sum_{j \le m} \fm_j \le 1$ (in other words, the multisets $\fm_j$ are really sets and one has $\fm_j \cap \fm_k = \emptyset$ for any $j \ne k$).
\end{enumerate}
If (\ref{eq:wlCJhat}) (resp. (\ref{eq:wlCJ})) holds for $\sigma$ for any $\phi \in \SN$, then one has $\sigma \in \wlCJhat$ (respectively $\sigma \in \wlCJ$). 
\end{lemma}
\begin{proof}
We only show the statement for $\wlCJ$, the one for $\wlCJhat$ follows in the same way.
Assume without loss of generality that $\plQz \sigma = \sigma$. Let $l_0 \ge 1$ and let $\phi=\sum_{l \le l_0} \gamma_l \phi^l$ for some $\gamma_l \in \R$ and $\phi^l \in \SN$ for any $l \le l_0$. Given a finite sequence $\alpha : [m] \to [l_0]$, we define the tuple $\phi^\alpha \in \SN$ by setting
\begin{align*}
\phi^\alpha_\cw
:=
\begin{cases}
	\phi^{\alpha_j}_\cw 	&\qquad \text{ if } \cw \in \fm_j ,\, j \le m \\
	0						&\qquad \text{otherwise}
\end{cases}
\end{align*}
for any $\cw \in \CW$.
Note that this is well-defined, since it follows from point~\ref{item:consistently:legged:2}. of our assumptions that the relation $\cw \in \fm_j$ holds for at most one $j \le m$. It follows from a simple application of the binomial expansion and the representation (\ref{eq:consistently:legged}) that one has the identity
\begin{multline*}
\bar\Upsilon^{\phi} \sigma
=
\sum_{i \le r} c_i
\prod_{j \le m}
\sum_{l \le l_0}
\gamma_l
\eva{\phi^l} \tau_{i,j}
=
\\
\sum_{i \le r} c_i
\sum_{\alpha : [m] \to [l_0]}
\prod_{j \le m}
\gamma_{\alpha_j}
\eva{\phi^{\alpha_j}} \tau_{i,j}
=
\sum_{\alpha : [m] \to [l_0]}
\eva{\phi^\alpha}
\sigma.
\end{multline*}
In the last equality we used that one has $\eva{\phi^{\alpha_j}} \tau_{i,j} = \eva{\phi^{\alpha}} \tau_{i,j}$ for any $i \le r$ and $j \le m$.
\end{proof}

The next lemma is crucial since it shows that the ideals $\wlCJ$, $\wlCJhat$ and $\wlCJad$ are generated by linear combinations of trees.
\begin{lemma}\label{lem:wlCJ:generators}
The ideals $\wlCJ$ ($\wlCJhat$, $\wlCJad$) are generated by all $\sigma \in \plCT[\Lpart]$ ($\sigma \in \plCThat[\Lpart]$, $\sigma \in \plCT[\Lpart] \cap \adCT$), such that $\sigma$ can be written as a linear combination of good trees and such that (\ref{eq:wlCJ}) holds for any $\phi \in \SN$.
\end{lemma}
\begin{proof}
This follows from Lemma~\ref{lem:ideal:generators} applied to the algebras $\plCT[\Lpart]$ ($\plCThat[\Lpart]$, $\plCT[\Lpart]\cap \adCT$). Note that e.g. the set $\{\eva{\phi} \plQz : \phi \in \linspace\SN \}$ is indeed a \emph{linear} space of linear functionals when restricted to $\linspace{\big(\wTT \cap \plCT[\Lpart]\big)}$. This was the motivation for using $\linspace{\SN}$ in the definition of these ideals.
\end{proof}

We now have the following lemma.

\begin{lemma}\label{lem:CJ:wlCJ}
Let $\CJcon$ be the ideal defined in Definition~\ref{def:CH}. Then one has the identity
\begin{align}\label{eq:CJ:wlCJ}
\CJcon = \pi\adzeCJ = \pi\wlCJad
\end{align}
\end{lemma}
\begin{proof}
Since $\pi$ is an algebra homomorphism and both ideals are generated by linear combinations of trees (for $\adzeCJ$ this follows from Lemma~\ref{lem:CJ:wlCJ} and (\ref{eq:identity:adzeCJ}), for $\CJcon$ this follows from Definition~\ref{def:CJcon}), it suffices to show that for any linear combination of trees $\sigma \in  \symT{\adze}$ one has $\sigma \in \adzeCJ$ if and only if $\pi\sigma \in \CJcon$.

 	Let first $\sigma \in \plCT[\Lpart] \cap \adCT$ be as in Lemma~\ref{lem:wlCJ:generators} a linear combination of good trees $\sigma = \sum_{i \in I}c_i \sigma_i \in\wlCJad$ such that (\ref{eq:wlCJ}) holds. We assume without loss of generality that $\plQz \sigma = \sigma$.
	
	We first claim that is suffices to consider $\sigma$ such that the set of leg-types $L:=[L_\LT(\tau_i),\ft]$ does not depend on $i \in I$.
	Indeed, assume that the claim holds for all $\sigma$ with this property and let 
	$\sigma = \sum_{i \in I}c_i \sigma_i \in\wlCJad$ be as above a linear combination of trees, 
	but assume that $[\LTE(\tau_i),\ft]$ is not independent of $i\in I$. 
	We claim that there exists a proper non-empty subset $J\ssq I$ such that $\sum_{j \in J}c_j \sigma_j \in \wlCJad$, from which the result follows by induction. 
	For this let $\lta \in \LT$ let $J:=\{i \in I: \lta \in \LTE(\sigma_i)  \}$, and assume that $I \ne J \ne \emptyset$.
	Consider for any family $\phi \in \SN$ the family $\phi^\lambda$, $\lambda>0$, 
	defined by setting $\phi^\lambda_{\lta} := \lambda \phi_{\lta}$, $\phi^\lambda_{\barlta} := \lambda \phi_{\barlta}$ and
	 $\phi^\lambda_{\lt}:=\phi_{\lt}$ for any $\lt \in \LT \setminus\{ \lta, \barlta \}$. It follows that
\[
\eva{\phi^\lambda} \sigma_i = \eva{\phi} \sigma_i
\,,\qquad
\eva{\phi^\lambda} \sigma_j \to 0 
\] 
as $\lambda \to 0$, for any $i \in I\setminus J$ and $j \in J$, which implies that $\sum_{j \in J} c_j \sigma_j \in \wlCJad$. 
	
Hence, we also have that $[L(\tau_i),\ft]$ is independent if $i\in I$. Let now $\zeta_i : \CL(\tau_i) \to \Lpart$ be the injection as Definition~\ref{def:good:trees}, denote by $\Lparta$ its range (which is independent of $i$), and write as above $\ft(P) \in \FL_-$ for the ``type'' of $P \in \Lpart$.   Let $H \ssq \Glegs$ be the subgroup of those $g \in \Glegs$ with the property that $g(\lt) = \lt$ for any $\lt \notin L$. 
	For $\lt \in \LT$ let $P(\lt) \in \Lpart$ be such that $ \lt \in P(\lt)$. One has $\cm = [\Lparta,\ft]$, so that  for any $\phi \in \SN$ we can define a function 
	$\psi^\phi \in \barCCgg{\cm}$ by
	\begin{align} \label{eq:psi:phi}
		\psi^\phi( x_{\Lparta} ):=
		\sum_{g \in H}
		\prod_{ \lt, \bar\lt \in L }
			\phi_{ g(\lt) }(x_{P(\lt)} - x_{P(\bar\lt)}).
	\end{align}
	Recall that in the definition of multisets of the form $[\Lparta,\ft]$ we ``forget'' the domain $\Lparta$, so that one has indeed $[\Lparta,\ft] = [L(\tau_i),\ft]$ for any $i \in I$. Furthermore, $\psi^\phi$ is invariant under those perturbations of $\Lparta$ which leave the noise type $\ft(P)$ invariant for any $P \in \Lparta$. Hence, $\psi^\phi$ can indeed be viewed as having the domain $\bar\domain^\cm$.
	Finally, in the product on the right-hand side of (\ref{eq:psi:phi}) we have one (and only one) factor for each pair leg $\lt$ and its partner $\bar \lt$. Since by definition $g$ commutes with conjugation and $\phi_{\lt} = \phi_{\bar \lt}(- \cdot)$, there is no ambiguity in this notation.

We claim that the linear space $Y$ generated by functions of the form $\psi^\phi$ for some $\phi \in \SN$ is dense in the space $X$ of functions $\psi \in \barCCgg{\cm}$ 
which are supported in the set of $x_{\td(\cm)}$ with $|x_p-x_q| \le \bar R$ for any $p,q \in \td(\cm)$ 
with respect to uniform convergence. 
Then $Y$ is the linear space generated by functions $\psi \in X$ such that there exist functions 
$\psi_{\Xi,\Xia} \in \CCg$ with 
\[
\psi(x_{\cm}) = \prod_{(\Xi,k),(\Xia,l) \in \td(\cm), (\Xi,k) \ne (\Xia,l)} \psi_{\Xi,\Xia} (x_{(\Xi,k)} - x_{(\Xia,l)}).
\]
The claim now follows from Arzel\'a and Ascoli's theorem.

For any fixed compact $K \ssq \bar\domain^\cm$ and any $\tau \in \CT_-$ the evaluation $\psi \mapsto \evaA{\psi} \tau$ is continuous on the subspace of those $ \psi \in \bar\CC_c^\infty( \bar\domain^\cm )$ with $\supp \psi \ssq K$ with respect to uniform convergence. This follows from the second part Assumption~\ref{ass:main:reg}, which implies a bound on the small scales, and Proposition~\ref{lem:super:regularity:implies:large:scale:bound} and \cite[Sec.~4]{Hairer2017}, which implies a bound on the large scales. It now suffices to show $\tilde\Upsilon^{\psi^\phi} \pi\sigma = 0$ for any $\phi \in \SN$. This however follows from
	\[
	\evaA{\psi^\phi} \pi\sigma
	=
	\eva{\phi} \sum_{g \in H} \CS_g \sigma
	=	
	0,
	\]
	so that $\pi \sigma \in \CJcon$.

	The converse direction follows in almost the same way. Let $\sigma= \sum_{i\in I}c_i \sigma_i \in \CJcon$ be a linear combination of trees and let $\iota : \CT_- \to \plCT$ be an admissible embedding taking values in the set of good trees. Assume without loss of generality that the set of leg types $L:=L_\LT(\iota \sigma_i)$ does not depend on $i$, and let $H \ssq \Glegs$ be as above. It then suffices to show that
	\[
		\eva{\phi} \sum_{g \in H} \CS_g \iota \sigma = 0
	\]
	for any $\phi \in \SN$. Reversing the above arguments, we see that 
	\[
		\eva{\phi} \sum_{g \in H} \CS_g \iota \sigma
		=
		\eva{\psi^\phi} \sigma
		=
		0,
	\]
	which concludes the proof.
\end{proof}


We want to use Proposition~\ref{prop:eps:beta:bound} from the previous section. For this we need the following technical lemma, which shows that the ideals $\wlCJ$ and $\wlCJhat$ can alternatively be defined by considering only $\phi \in \SN(\LTsys)$ for some $\LTsys \in \LTsyss$, where $\SN(\LTsys)$ is as in Definition~\ref{def:SN:FP}.
\begin{lemma}\label{lem:CJ:limit:argument}
Let $\tau \in \plCT[\Lpart]$ ($\tau \in \plCT[\Lpart] \cap \adCT$, $\tau \in \plCThat[\Lpart]$), assume that $\tau$ can be written as a linear combination of trees 
$
\tau = \sum_{i=1}^r c_i \tau_i
$
with $r\ge 1$, $c_i\in\R$ and trees $\tau_i$ such that $\LTa:=[L_\LT(\tau_i),\ft]$ is independent of $i \le r$. Assume that there exists some system $\LTsys \in \LTsyss$ such that $\eva{\phi} \plQz\tau=0$ (resp. $\eva{\phi} \tau =0$ if $\tau \in \plCThat[\Lpart]$)  for any $\phi \in \SN(\LTsys)$. Then one has $\tau \in \wlCJ$ ($\tau \in \wlCJad$, $\tau \in \wlCJhat$).
\end{lemma}
\begin{proof}
We only show the statement about $\wlCJ$.

By Lemma~\ref{lem:consistently:legged}, we need to show that (\ref{eq:wlCJ}) holds for any $\phi \in \SN$. For this we recall the definition of $\eva{\phi}$ \eqref{eq:Upsilon:bar:phi:psi}, \eqref{eq:bar:Upsilon:large:scale}. It suffices to consider the case that the $\tau_i$'s contain only essential legs, and we naturally identify the sets of legs $L_\LT(\tau_i)$ with $\LTa$ for any $i \in I$. This identification induces a natural identification of the sets of noise type edges $L(\tau_i)$ with a subset $\Lparta \ssq \Lpart$ as in the proof of Lemma~\ref{lem:CJ:wlCJ}. Then, for any $P,Q \in \Lparta$ with $P \ne Q$ there exists a unique leg type $\lt(P,Q) \in \LTa$ such that $\lt(P,Q) \in P$ and $\bar\lt(P,Q) \in Q$. Conversely, for any leg type $\lt \in \LTa$ there exists a unique $P(\lt) \in \Lparta$ such that $\lt \in P(\lt)$.

For $\vphi \in \barCCgg{\LTa}$ let $\Pi \vphi \in \barCCgg{\Lparta}$ be defined by setting
\[
\Pi \vphi( x_{\Lparta})
:=
\int dx_{\LTa} \vphi(x_\Lparta) \prod_{\lt \in \LTa} \delta( x_\lt - x_{P(\lt)}).
\]
It follows that 
\[
	0
	=
	\eva{\phi}\tau
	=
	\langle \CK_{\hat K} \tau, D ^{\fe|_{\LTa}} \hat\phi_{\LTa} \rangle 
	=
	\langle \CK_{\hat K} \pi \tau, \Pi D ^{\fe|_{\LTa}} \hat\phi_{\LTa} \rangle.
\]
for any $\phi \in \SN( \LTsys )$. We need to show that this identity holds for any $\phi \in \SN$. 

Assume first that $\# \LTa>2$. 
We first claim that one has $\langle \CK_{\hat K} \pi \tau, \Pi  \hat\phi_{\LTa} \rangle=0$ for any $\phi \in \SN( \LTsys )$, that is, one can get rid of the derivative decoration. 
Indeed, let $R>0$ be such that $\supp \phi_\lt$ is included in the centred ball of radius $R$ for any $\lt \in \LTa$. Note that since $\CK_{\hat K} \pi\tau$ is homogeneous we may assume that $R$ is as small as we want, so that in particular, we may assume that $2 R \le \bar R$.
Fixing $\lt^* \in \LTa$, 
we see that for any $\tilde\phi \in \CCg$ 
such that $\tilde\phi$ vanishes inside the ball of radius $2R$ one has $\Pi D^{\fe|_\LTa}\hat\phi_\LTa = \Pi (D^{\fe|_\LTa}\hat\phi_\LTa + \hat \phi^*_\LTa)$, where $\phi^* _{\lt^*} := \tilde\phi$ and $\phi^*_\lt := \phi_\lt$ for $\lt \ne \lt^*$. Since any smooth function $\psi$ which is compactly supported in the centred ball of radius $R$ agrees with a function of the form $D^{\fe|_{\LTa}} \phi_\lt  + \tilde\phi$ (for $\tilde\phi$ as above) inside the ball of radius $2R$, the claim follows.
With precisely the same argument we can remove the second constrained coming from Definition~\ref{def:SN:FP}, so that
the equality $\langle \CK_{\hat K} \pi \tau, \Pi  \hat\phi_\LTa \rangle = 0$ holds for any $\phi \in \SN$ such that $\phi_\lt$ vanishes in a neighbourhood of the origin.
At this point is remains to note that $\CK_{\hat K} \pi \tau$ is a locally integrable function, so the condition that the $\phi_\fl$'s vanish around the origin can be removed by a limit argument in $L^\infty(\bar\domain)$.

The remaining case $\LTa = \{ \lt_0, \bar\lt_0 \} \in \LTsys$, so that $\LTsysdeg \lt_0 \le -\shalf$, needs a slightly different argument. Using a simple rescaling argument it is clear that it suffices to consider the case that $\alpha := \fancynorm{\tau_i}_\fs <0$ is independent of $i\le r$. (Note that in case $\alpha=0$ one has $\bar\Upsilon^\phi \tau = 0$ by Assumption~\ref{ass:zero-homo}, so that there is nothing to show.) Our integration kernels are homogeneous, so that we can write $\langle \CK_{\hat K}\tau, D^{\fe|_\LTa} \hat\phi_\LTa\rangle = \int \tilde K(x) \phi_{\fl_0}(x) dx$ for some function $\tilde K \in \CC^ \infty ( \bar \domain \setminus \{0\} )$ satisfying $\tilde K( \lambda x ) = \lambda^{\alpha} K(x)$ for all $\lambda>0$, $x \in \bar\domain \setminus \{0\}$. (Here we removed the derivative decoration by an integration by parts.) Since $\alpha>-|\fs|$ the function $\tilde K$ is locally integrable and we remove the constraint that $\phi_{\lt_0}$ vanishes around the origin by a limit argument in $L^\infty(\bar\domain)$. We still have the constraint $\int \phi_{\lt_0} = 0$ coming from Definition~\ref{def:SN:FP}, which implies that $\tilde K$ is a constant, and since $\alpha<0$, this actually implies that $\tilde K=0$ as required.
\end{proof}

With these preliminaries we can now show the following proposition, which is the main result of this section.

\begin{proposition}\label{prop:interaction:wlCJ:Delta}
One has the identities
\begin{align}
\label{eq:inclusion:CJ:hat}
\cpmwi
\wlCJad
&\subseteq
(\wlCJad \otimes \plCThat) + (\plCT \otimes \wlCJhat)
\\
\label{eq:inclusion:CJ}
\cpm
\wlCJad
 &\subseteq
(\wlCJad \otimes \plCT) + (\plCT \otimes \wlCJ) .
\end{align}
\end{proposition}

\begin{proof}
Let $\tau = \sum_{i=1}^r c_i \tau_i \in \wlCJad$ be a linear combination of trees with $\tau_i \in \plCT[\Lpart]  \cap \adCT$ and $c_i \in \R$ for $i\le r$. As before we can assume without loss of generality that the trees $\tau_i$ are such that the set of leg types $L := \ft(L_\LT(\tau_i))$ is independent of $i \le r$, and thus so is $\LTsyssa := \LTsyss(\tau_i)$ (recall \eqref{eq:II:tau} for the definition of this set).
By definition of the coproduct $\cpm$, it follows that one has the identity
\begin{align}\label{eq:delta:PFI}
\cpmwi
\tau
= 
(\sum_{\LTsys \in \LTsyssa} P_\LTsys \otimes \Id) \cpmwi \tau.
\end{align}
Here $P_\LTsys$ is the linear (but not multiplicative) projection of $\plCT$ onto the linear subspace $\plCT[\LTsys]$ spanned by all products of trees of the form $\tau = \prod_{\LTa \in \LTsys} \tau_{\LTa}$ with $\tau_\LTa \in \range P_{\LTa}$ (that is $\ft(L_\LT(\tau_\LTa)) = \LTa$). The projection $P_\LTsys$ is uniquely defined if we specify additionally that it diagonalises on the basis (in the sense of linear spaces) $\CB \ssq \plCT$ containing $\one$ and all possible products of trees.


The crucial step is to show that for any fixed $\LTsys \in \LTsyssa$ and any fixed  $\phi^L\in \SN(\LTsys)$ and $\phi^R \in \SN(\LTsys)$ one has
\begin{align}\label{eq:inclusion:CJ:hat:PFI}
\UOP{\phi^L,\phi^R} \tau
:=
(\eva{\phi^L} 
\otimes
\eva{\phi^R})
(P_\LTsys \otimes \Id) \cpmwi \tau 
=
0.
\end{align}
Actually, since no leg type appears in both the left and the right factor of this tensor product simultaneously, it is enough to show this claim for $\phi^L = \phi^R \in \SN(\LTsys)$.

More precisely: assume we have shown this special case. Then we construct a tuple $\phi \in \SN(\LTsys)$ by setting $\phi_{\lt} := \phi^L_\lt$ if there exists $P \in \LTsys$ such that one has $\lt \in P$, and $\phi_\lt:=\phi^R_\lt$ otherwise. It follows that $\phi \in \SN( \LTsys )$ and one has the identity
\[
\UOP{\phi} \tau
= 
\UOP{\phi^L,\phi^R} \tau,
\] 
where $\UOP{\phi}:=\UOP{\phi,\phi}$,
so that (\ref{eq:inclusion:CJ:hat:PFI}) follows indeed from the special case $\phi^L = \phi^R$.
In order to continue we fix a family $\phi \in \SN( \LTsys )$. For $\eps>0$ we define a rescaled family $\phi^\eps \in \SN( \LTsys )$ as in (\ref{eq:phi:rescale}) for the degree assignment $\LTsysdeg$ defined as in (\ref{eq:definition:deg:FJ}). With this notation, we define a function $f:(0,1] \to \R$ by setting
\[
f(\eps) :=
	\UOP{\phi^\eps} \tau
\]
for any $\eps \in (0,1]$. The proof of (\ref{eq:inclusion:CJ:hat:PFI}) is finished once we show that $f(1)=0$. 
\begin{lemma}
One has that 
\begin{align}\label{eq:f:eps:f:1}
|f(\eps)|\ge |f(1)|
\end{align}
for $\eps>0$ small enough.
\end{lemma}
\begin{proof}
We first note that one can write
\[
f(\eps)=
\UOP{\phi^\eps,\phi}  \tau.
\]
This follows from the fact that there is a projection $P_\LTsys$ hitting the left component of (\ref{eq:inclusion:CJ:hat}), which ensures that no leg type $\lt \in \bigcup \LTsys$ appears in the right component, together with the definition of $\LTsysdeg$ in (\ref{eq:definition:deg:FJ}).
On the other hand, by a simple change of variables one can exploit the homogeneity of the kernels $\hat K$, which implies that
\[
\eva{\phi^\eps}   
\sigma
=
\eps^{\homofancy\sigma} \,
\eva{\phi}   
\sigma
\]
for any fixed tree $\sigma \in \rng P_\LTa$ for any $\LTa \in \LTsys$. This is a consequence of the definition of $P_\LTa$ below (\ref{eq:h:CI:phi:R}) and $\LTsysdeg$.
As a consequence $f(\eps)$ can be written as a finite sum of terms $\sum_{j\le J}f_j(\eps)$ such that $f_j(\eps)=\eps^{\gamma_j}f_j(1)$ for some $\gamma_j\le 0$ and $\eps>0$, from which the statement of the lemma easily follows. 
\end{proof}

We now proceed to show (\ref{eq:delta:PFI}) by induction over $\# \LTsys$. For $\# \LTsys = 0$ one has the identity $(P_\LTsys \otimes \Id) \cpmwi \tau = \one \otimes \wli \tau$ so that (\ref{eq:inclusion:CJ:hat:PFI}) follows from the fact that $\tau \in \wlCJad$. Let now $\#\LTsys \ge 1$ and assume that (\ref{eq:inclusion:CJ:hat:PFI}) holds for any $\LTsysa$ with $\LTsysa \subsetneq \LTsys$. Then, using the induction hypothesis, we can rewrite $f(\eps)$ as
\begin{align}
f(\eps)=
\sum_{\LTsysa \subseteq \LTsys}
	(-1)^{\#\LTsysa}
(\eva{\phi^\eps}
\otimes
\eva{\phi^\eps})
(P_\LTsysa \otimes \Id) \cpmwi 
	\tau
= 
(\Lche
\otimes
\eva{\phi^\eps})
\cpmwi \tau,
\end{align}
where $\Lche$ denotes the character on $\plCT$ defined in (\ref{eq:h:CI:phi:R}), compare also (\ref{eq:h:trees:rewrite}).
Since $\phi \in \SN(\LTsys)$ by assumption, we conclude from Proposition~\ref{prop:eps:beta:bound} that there exists $\beta>0$ such that one has the estimate
\[
|f(\eps)|\lesssim \eps^\beta
\]
uniformly over $\eps \in (0,1)$.
Comparing this with (\ref{eq:f:eps:f:1}) it follows at once that one has $f(1)=0$, and this concludes the proof of (\ref{eq:inclusion:CJ:hat:PFI}). 

Since the left factor of (\ref{eq:inclusion:CJ:hat:PFI}) is an element of $\adCT$ one also has
\begin{align}\label{eq:inclusion:CJ:hat:PFI:Q}
(\eva{\phi^L} \plQz 
\otimes
\eva{\phi^R})
(P_\LTsys \otimes \Id) \cpmwi \tau = 0.
\end{align}
In order to see (\ref{eq:inclusion:CJ:hat}), we draw on the following simple lemma.
\begin{lemma}\label{lem:tensor:kernel}
Let $X$ and $Y$ be linear spaces and let $(f_i)_{i \in I}$ and $(g_j)_{j \in J}$ be families of linear functionals on $X$ and $Y$ respectively, for some index sets $I$ and $J$. Then one has
\[
\bigcap_{i,j}\ker (f_i \otimes g_j) = \Big(\bigcap_i \ker f_i \Big) \otimes Y + X \otimes \Big(\bigcap_j \ker g_j \Big)
\]
as subspaces of the algebraic tensor product $X \otimes Y$.
\end{lemma}
\begin{proof}
Denote the right and left-hand sides by $R$ and $L$, respectively. Let first $z \in R$. Then by definition we can write $z = z_1 + z_2$ with $(f_i \otimes \Id)(z_1) = (\Id \otimes g_j)(z_2) = 0$ for all $i \in I$ and $j \in J$. It follows that $(f_i \otimes g_j)(z_k) = 0$ for all $i\in I$, $j \in J$ and $k=1,2$, and thus $z \in L$.

Let now $z = \sum_{k=1}^K x_k \otimes y_k \in L$. We proceed inductively in $K$. For $K=1$ one has $f_i(x_1) g_j(y_1) = 0$ for all $i \in I$ and $j \in J$. Thus either $f_i(x_1)=0$ for all $i \in I$ or $g_j(y_1) = 0$ for all $j \in J$, and hence $x_1 \otimes y_1 \in R$. For $K>0$ we can assume that $x_K \otimes y_K \notin L$. In particular, there exists
\def\inode{{i_\node}}$\inode \in I$ such that $f_{\inode}(x_K) \ne 0$. Define $b_k := \frac{f_\inode(x_k)}{f_\inode{(x_K)}}$, so that by assumption one has
\begin{align}\label{eq:small}
\sum_{k=1}^K b_k g_j(y_k) = 0 \qquad \text{ for all }j \in J.
\end{align}
We can write
\[
z = \sum_{k=1}^{K-1} (x_k - b_k x_K) \otimes y_k + x_K \otimes \Big( \sum_{k=1}^K b_k y_k \Big)
\]
From (\ref{eq:small}) we deduce $x_K \otimes \Big( \sum_{k=1}^K b_k y_k \Big) \in R \ssq L$, so that $\sum_{k=1}^{K-1} (x_k - b_k x_K) \otimes y_k  \in L$. We conclude using the induction hypothesis.
\end{proof}
Applying this lemma to the families of linear functionals on $\plCT$ and $\plCThat$, given by $\eva{\phi}\plQz$ and $\eva{\phi}$, respectively, where $\phi$ ranges over $\SN(\LTsys)$ (and recalling Lemma~\ref{lem:CJ:limit:argument}), we conclude that (\ref{eq:inclusion:CJ:hat}) is a consequence of (\ref{eq:inclusion:CJ:hat:PFI:Q}).

In order to see (\ref{eq:inclusion:CJ}) we now use the identity
\[
\cpm = (\Id \otimes \wlcol{\p_-}) \cpmwi ,
\]
on $\plCT$. We still fix  $\LTsys \in \LTsyssa$. For $\lt \in \LT$ denote by $P(\lt) \in \PP$ the set such that $\lt \in P(\lt)$ (recall from above Definition~\ref{def:good:trees} that $\PP$ is a partition of the set of leg types). Let $\LTess$ denote the set of leg types $\lt \in L$ with the property that $\lt, \bar\lt \notin \bigcup_{\lt' \in\bigcup \LTsys} P(\lt')$. It follows that the right factor of $(P_\LTsys \otimes \wlcol{\p_-}) \cpmwi \tau$ takes values in the algebra generated by trees $\sigma$ such that the set of essential leg types of $\sigma$ is given by $\LTess$. Letting $\phi \to \phi^\eps$, where we rescale $\phi$ as in (\ref{eq:phi:rescale}) for the degree assignment $\deg^{\LTess}$, shows that 
\begin{align}\label{eq:bum}
\Big(
	\eva{\phi}
	\otimes
	\eva{\phi}
\Big)
 (\plQz P_\LTsys \otimes \wlcol{\p_-}) \cpmwi \tau = 0
\end{align}
for any $\phi \in \SN$.

Finally, letting $\phi_\legtype \to 1$  for any $\legtype \in \LTa := \Legtype\setminus (\LTess \cup \bigcup \LTsys)$ we can show that
\begin{align}\label{eq:bla}
	\Big( \eva{\phi} \otimes \eva{\phi} \Big)(\plQz P_\LTsys \otimes \plQz\wlcol{\p_-}) \cpmwi \tau = 0.
\end{align}
Indeed, recall that $\plQz = \plQ \plZ$. From Assumption~\ref{ass:main:reg} it follows that divergent subtrees $\sigma$ never touch noise type edges $e$, that is one has either $e \in L(\sigma)$ or $e^\downarrow \notin N(\sigma)$. It follows from this that the coproduct never produces a derivative decoration on noise type edges. In precisely the same way we see that the coproduct does not produce a derivative decoration on \emph{essential} legs on the right-hand side. Hence every tree on the right-hand side of $(P_\LTsys \otimes (\Id-\plZ)\wlcol{\p_-}) \cpmwi \tau$ contains at least one non-essential leg $e$ such that $\fe(e)>0$. Assume now that $\phi_\lt \equiv 1$ in a neighbourhood of the origin for $\lt \in \LTa$, and define $\phi_\lt^{\eps,N}:= \phi_\lt^\eps( N^{-1} \cdot)$ for $\lt \in \LTa$ and $\phi_\lt^{\eps,N}:=\phi_\lt^\eps$ for $\lt \in \LT\backslash\LTa$. we see that 
\begin{align*}
\Big(
	\eva{\phi^{\eps,N}} \otimes \eva{\phi^{\eps,N}} 
\Big)
(\plQz P_\LTsys \otimes (\Id-\plZ)\wlcol{\p_-}) \cpmwi \tau \to 0 \quad \text{ as } N \to \infty.
\end{align*}
In exactly the same way we see that
\begin{align*}
\Big(
	\eva{\phi^{\eps,N}} \otimes \eva{\phi^{\eps,N}} 
\Big)
(\plQz P_\LTsys \otimes (\Id-\plQ) \plZ\wlcol{\p_-}) \cpmwi \tau \to 0.
\end{align*}
On the other hand, the quantity
\begin{align*}
\Big(
	\eva{\phi^{\eps,N}} \otimes \eva{\phi^{\eps,N}} 
\Big)
(\plQz P_\LTsys \otimes \wlcol{A} \wlcol{\p_-}) \cpmwi \tau
\end{align*}
for $\wlcol{A} \in \{\Id, \plQz \}$ is independent of $N\in\N$ (for $\wlcol{A} = \Id$ this quantity vanishes by (\ref{eq:bum}), for $\wlcol A = \plQz$ independence of $N$ follows since the projection $\plQz$ removes non-essential legs on the right factor, but these are the only one that come with a type which we rescale). This concludes the proof.
\end{proof}


Finally, the key result of this section is the following corollary, which finishes the proof of Assumption~\ref{ass:CJHopfIdeal}.

\begin{corollary}\label{cor:CJ:hopf:ideal}
The ideal $\CJcon$ is a Hopf ideal in $\CT_-$.
\end{corollary}
\begin{proof}
By Lemma~\ref{lem:CJ:wlCJ} is suffices to show that $\adzeCJ$ is a Hopf ideal in $\symT{\adze}$. 
This in turn follows from (\ref{eq:inclusion:CJ}) and the identity (\ref{eq:identity:adzeCJ}).
\end{proof}


\subsection{Rigidities between renormalisation constants}\label{sec:rigidities}
\label{sec:constraints:rigidities}

In this section we are going to prove Assumption~\ref{ass:CHBPHZcharacters}.
We first build for any smooth shifted noise $\eta\in\SM_{\infty}$  characters $\hat g^\eta\in\CH$ and $f^\eta \in \CG_-$ such that $g^{\eta}= f^\eta\circ \hat g^\eta$. Recall that $\CH$ is the annihilator of the ideal $\CJ$ defined in Definition~\ref{def:CH}. We also recall for this the notation introduced in Section~\ref{sec:evaluation:largescale}, which we will use heavily in this section. As above, we always set $\deg_\infty\ft:=\fancynorm\ft_\fs-|\fs|$ for any kernel type $\ft \in \FL_+$, and we fix from now on the homogeneous large-scale kernel assignment $R_\ft:=\hat K_\ft-K_\ft$ for any $\ft\in\FL_+$. Recall that with this definition one has $(R_\ft)_{\ft\in\FL_+}\in\CK^+_0$.

Furthermore, we fix a smooth, symmetric under $\groupD$, compactly supported function $\varphi\in \CCg$ such that $\varphi\equiv 1$ in a neighbourhood of the origin. Given this function $\varphi$ we build an element $\phi\in\SN_\sym$ by setting $\phi_\cw=\varphi$ for any $\cw\in\CW$.

With this notation we introduce for any smooth noise $\eta \in \SM_\infty$ a character $\hat g^\eta\in\CG_-$ by setting
\[
\hat g^\eta := \wlchar{g ^{\eta,\phi} _R} \symiota,
\]
where $\wlchar{g ^{\eta,\phi} _R}$ is the character on $\symCT$ defined in (\ref{eq:g:sym:eta:psi}), and we define $f^\eta\in\CG_-$ by
\begin{align}\label{eq:definition:f}
g^\eta=f^\eta\circ\hat g^\eta.
\end{align}
One has the identity $\hat g^\eta = \wlchar{g ^{\eta,\phi} _R} \iota$ where $\wlchar{g ^{\eta,\phi} _R}$ is as in (\ref{eq:g:wl:eta:psi:R}) for \emph{any} admissible embedding $\iota : \CT_- \to \plCT$.
We assume that $\varphi\equiv 1$ holds in a large enough neighbourhood of the origin so that Lemma~\ref{lem:identity:giota:g} applies.
\begin{lemma}\label{lem:convergence:g:hat}
For any smooth shifted noise $\eta\in\SMsinf$ one has that $\hat g^\eta\in\CH$, where $\CH$ denotes the annihilator of $\CJ$, see Definition~\ref{def:CH}.
\end{lemma}
\begin{proof}
Fix an admissible embedding $\iota:\CT_- \to \plCT$. We have to show that $\hat g^\eta = \wlchar{g ^{\eta,\phi} _R} \iota$ vanishes on $\CJ$, so that is suffices to show that $\wlchar{g ^{\eta,\phi} _R}$ vanishes on $\iota \CJ \ssq \plCT$. 
Recalling that $\CJcon = \pi\spadeP \wlCJad$ we see that $\symP \iota\CJ = \symP \wlCJad$, and since the character $\wlchar{g ^{\eta,\phi} _R}$ is invariant under the symmetry group $\Glegs$, it suffices to show that $\wlchar{g ^{\eta,\phi} _R}$ vanishes on $\wlCJad$. For this we use the fact that $\wlCA \wlCJad \subseteq \wlCJhat$ (c.f.\ Proposition~\ref{prop:interaction:wlCJ:Delta}), and the fact that by definition the character $\bar\Upsilon^{\eta,\phi}_R$ vanishes on $\wlCJhat$.
\end{proof}

We are left to show that the map $\eta\mapsto f^\eta$ extends continuously to $\eta\in\SMsz$. A possible approach to show such a statement would be to use an inductive argument  in the number of edges of a tree $\tau \in \CT_-$, and to use the fact that we can re-write the definition of $f^\eta$ in (\ref{eq:definition:f}) as
\begin{align}\label{eq:identity:g:f:g:hat}
f^\eta\tau = g^\eta\tau - 
	(f^\eta\otimes\hat g^\eta)( \cpmh -\Id\otimes\one)\tau.
\end{align}
One could then exploit the properties of the coproduct from which it follows that the character $f^\eta$ on the right-hand side of (\ref{eq:identity:g:f:g:hat}) gets only hit by trees that have strictly fewer edges then $\tau$, so that one could try to match the diverging terms coming from $g^\eta$ and $\hat g^\eta$ on the right-hand side. At this point however, this approach leads to relatively complicated expressions, and our arguments are greatly simplified by bounding the linearised expression and using an integration argument. 

We first recast the problem into a problem of characters acting on $\symCT$.
\begin{lemma}
For $\eta \in \SMsinf$ let $\wlchar{\tilde f^\eta}$ be the character of $\symCT$ defined by
\begin{align}\label{eq:tildef:wl}
\wlchar{  {\tilde f}^\eta } \circ \wlchar{ g^{\eta,\phi}_0} = \wlchar{  g^{\eta,\phi}_R}.
\end{align}
If the map $\eta \mapsto \wlchar{  {\tilde f}^\eta }$ extends continuously to $\SMsz$, then so does the map $\eta \mapsto  f^\eta$.
\end{lemma}
\begin{proof}
For $\eta \in \SMshifted$ let $\tilde f^\eta := (f^\eta)^{-1}$, where the inverse is taken in the character group $\CGm$ of $\CTm$. The operation of taking inverses is a homeomorphism of $\CG_-$, so that is suffices to show that $\tilde f^\eta$ extends continuously to $\SMsz$. 
We claim that one has $\wlchar{  {\tilde f}^\eta } \symiota = \tilde f^\eta$, which concludes the proof, since the map $\symG \to \CGm$, $\wlg \mapsto \wlg \symiota$ is continuous. To see this claim, we are left to show that $\wlchar{  {\tilde f}^\eta } \symiota \circ g^\eta = \hat g^\eta$. Recall that one has
\begin{align}\label{eq:tildef:1}
\tilde f^\eta \circ g^\eta = \hat g^\eta
\qquad\text{ and }\qquad
\hat g^\eta = \wlchar{ g^{\eta,\phi}_R } \symiota \,, \quad
g^\eta = \wlchar{ g^{\eta,\phi}_0 }\symiota
\end{align}
so that we are left to show that
\begin{align}\label{eq:tildef:2}
(\wlchar{  {\tilde f}^\eta } 
\otimes 
\wlchar{ g^{\eta,\phi}_0} 
) (\symiota \otimes \symiota)
\cpmh
= 
(\wlchar{  {\tilde f}^\eta } 
\otimes 
\wlchar{ g^{\eta,\phi}_0} 
) \cpmh \symiota.
\end{align}
Note that the previous identity does not follow immediately, since $\symiota$ is not a Hopf algebra homomorphism. However, using Lemma~\ref{lem:CTadze:hopf:ideal} and Lemma~\ref{lem:hopfiso}, we can show that
\begin{align}\label{eq:tildef:3}
\cpmh \symiota
\in
(\symiota \otimes \symiota) \cpmh 
+
\ker \plQz \otimes \symCT + \symCT \otimes \ker\plQz.
\end{align}
Indeed, note first that $\ker\plQz = \ker\spadePsym$, so that with Lemma~\ref{lem:tensor:kernel} we are left to show that $(\spadePsym \otimes \spadePsym ) (\cpmh \symiota - (\symiota \otimes \symiota)\cpmh) = 0$ on $\symCT$. By Lemma~\ref{lem:hopfiso} the map $\spadePsym \symiota$ is a Hopf isomorphism and by Lemma~\ref{lem:CTadze:hopf:ideal} and the definition of a Hopf factor algebra one has $(\spadePsym \otimes \spadePsym)\cpmh = \cpmh \spadePsym$ on $\symCT$, hence (\ref{eq:tildef:3}) follows.

We now show (\ref{eq:tildef:2}), which concludes the proof. By the definition of admissible embeddings, the definition of $\cpmh$ and $\plQz$ one has $(\plQz \otimes \Id) \cpmh\symiota = \cpmh \symiota $ on $\CTm$, so that we deduce from (\ref{eq:tildef:3}) the stronger inclusion 
$\cpmh \symiota
\in
(\symiota \otimes \symiota) \cpmh 
+
\symCT \otimes \ker\plQz.
$
It remains to note that $\ker\plQz \ssq \ker \wlchar{ g^{\eta,\phi}_0}$, which follows since we chose $\phi=1$ in a large neighbourhood of the origin, compare Lemma~\ref{lem:identity:giota:g}.
\end{proof}

In order to continue, we define for $r>0$ the family of large-scale integration kernels $R^{(r)}=(R^{(r)}_{\ft})_{\ft\in\FL_+}$ by setting
\[
R^{(r)}_\ft(x) :=
 \hat K_\ft \phi ((r+1)^{-\fs} \cdot ) - K_\ft
\]
for any $\ft\in\FL_+$, where $\phi$ is as in Section~\ref{sec:kernels}. This particular way of removing the cutoff has the advantage that $R^{(r)}_\ft + K_\ft  = \hat K_\ft \phi((r+1)^{-\fs} \cdot)$ for any $r>0$, which will be helpful in the proof of Lemma~\ref{lem:fnknbound} below. We also denote by $\wlchar {g^\eta_r}$ the character of $\symCT$ defined by
\[
\wlchar {g^\eta_r} := \wlchar{g^{\eta,\phi}_{R^{(r)}}}.
\]
Note that one has $\lim_{r\to 0} R^{(r)}=0$ and $\lim_{r \to \infty}R^{(r)}=R$, so that it follows from Corollary~\ref{cor:evaluation:trees:largescale} that one has $ \wlchar {g^\eta_r}\to \wlchar {g^\eta_R}$ as $r\to\infty$ for any fixed $\eta\in\SM_\infty$. We define the character $\wlchar{\tilde f^\eta_r}$ analogue to above via the identity $\wlchar {\tilde f^\eta_r} \circ \wlchar {g^\eta_0} = \wlchar {g^\eta_r}$. It follows from the continuity of the group operation that one has $\wlchar{ \tilde f^\eta_0 } = \one^\star$ and $\wlchar{\tilde f^\eta_r} \to \wlchar{\tilde f^\eta}$ as $r\to \infty$. 
Moreover, it follows easily from the fact both $R_\ft$ and $\phi$ are smooth that the maps $r\mapsto \wlchar {g^\eta_r}$ and $r\mapsto \wlchar {\tilde f^\eta_r}$ are smooth functions in $r>0$ for any fixed smooth noises $\eta \in \SMshifted$. We are going to study a differential equation that $\wlchar{\tilde f^\eta_r}$ satisfies for $r>0$. To this end we introduce the following notation.

\begin{definition}
We call a linear map $k: \symCT\to\R$ an \emph{infinitesimal character} if for any $\tau_1,\tau_2\in \symCT$ one has $k(\tau_1\tau_2)=\one^*(\tau_1) k(\tau_2)+k(\tau_1)\one^*(\tau_2)$.
\end{definition}
Note that an infinitesimal character $k$ vanishes on elements which are not linear combination of trees. In particular, one has $k(\one)=0$, where $\one$ is the unity for multiplication. We extend the operation $\circ$ to act on any pair of linear maps $g,h: \symCT \to\R$ by setting $g\circ h:=(g\otimes h)\Delta$, where $\Delta$ denotes the coproduct of the Hopf algebra $\symCT$. With this definition $g\circ h$ is in particular well-defined whenever $g$ and $h$ are characters or infinitesimal characters, and in case both are infinitesimal characters, then $g\circ h-h\circ g$ is again an infinitesimal character. The following is well-known. 
\begin{lemma}
Let $\fg$ denote the space of infinitesimal characters of $\symCT$ and define the bi-linear map $[\cdot,\cdot]:\fg \times \fg \to \fg$ by
$
[k,l]:=k\circ l-l\circ k.
$
Then $\fg$ is the Lie algebra of the character group $\symG$ of $\symCT$.
\end{lemma} 
\begin{proof}
See for instance \cite[Thm.~3.9]{Bogfjellmo2018}.
\end{proof}
It is well known that the Lie algebra $\fg$ is naturally isomorphic to the tangent space $T_{\one^\star} \symG$ of $\symG$ at the co-unit $\one^\star \in \symG$ and for fixed $h \in \symG$ both right and left translations $k \mapsto k\circ h$ and $k \mapsto h \circ k$ induce isomorphisms between $\fg$ and the tangent space of $\symG$ at $h$.
We are going to study the differential equation
\begin{align}\label{eq:diffeq:f}
\partial_r \wlchar{\tilde f^\eta_r} =  \wlchar{k^\eta _r} \circ \wlchar{\tilde f^\eta _r} , \qquad\text{for }r>0,
\end{align}
with initial condition $\wlchar{\tilde f_0^\eta}:=\one^\star$. Note that the identity (\ref{eq:diffeq:f}) \emph{defines} an infinitesimal character $\wlchar{k_r^\eta} \in\symg$. 
The reason for studying equation (\ref{eq:diffeq:f}) is the following Lemma.
\begin{lemma}\label{lem:fnkn}
Assume that for any fixed $\eta\in\SMshifted$ the map $r\mapsto \wlchar{k^\eta_r}$ is an element of $L^1(0,\infty)$, and assume that this map extends to a continuous map $\eta\mapsto \wlchar{k^\eta_\cdot}$ from $\SMsz$ into $L^1(0,\infty)$. Then the map $\eta\mapsto \wlchar{\tilde f^\eta}$ extends continuously to $\SMsz$.
\end{lemma}
\begin{proof}
Let $\|\cdot\|$ denote a norm on $\fg$ and let $d(\cdot,\cdot)$ be the induced metric on $\symCT$. Then one has for any $\eta,\tilde\eta\in\SM_\infty$ the estimate
\[
d(\wlchar{\tilde f^\eta}, \wlchar{\tilde f^{\tilde\eta}}) 
	\le \exp \Big( \int_0^\infty \|\wlchar{k_r^\eta}- \wlchar {k_r^{\tilde\eta}}\| dr \Big),
\]
from which the statement follows immediately from the assumption of the lemma.
\end{proof}

Fix from now on a rough noise $\eta\in\SMsz$ and let $\eta^\eps\in\SMshifted$ be any sequence such that $\eta^\eps\to\eta$ in $\SMsz$ as $\eps\to 0$. We will use the simplified notation $\wlchar{g^\eps_r} := \wlchar{g^{\eta^\eps,\phi}_r}$, $\wlchar{ k^\eps_r } := \wlchar{ k^{\eta^\eps}_r }$, and similar for the other characters.
By Lebesgue's theorem it is sufficient to show that the sequence $\wlchar{ k^\eps_r }$ converges as $\eps\to 0$ for any fixed $r>0$, as well as the estimate $\int_0^\infty \sup_{\eps>0}\| \wlchar{ k^\eps_r } \|_{n\times n} dr<\infty$. This is equivalent to showing that there exist infinitesimal characters $\wlchar{ k_r }\in \symg$ such that one has
\begin{align}\label{eq:k:convergence}
\forall r>0:\, \wlchar{ k_r^\eps }(\tau) \to \wlchar{ k_r }(\tau) \text{ as } \eps \to 0 \,,
\qquad\text{ and }\qquad
\int _0 ^\infty \sup_{\eps>0} | \wlchar{ k_r^\eps }(\tau)| dr <\infty
\end{align}
for all $\tau \in \symCT$. In the remainder of this section we show (\ref{eq:k:convergence}), which completes the proof. 
For simplicity, we are going to write
\[
\doublebar\Upsilon^\eps_r:=\bar\Upsilon ^{\eta^\eps,\phi} _{ R^{(r)} }  
\]
for the character on $\symCThat$ from now on. With this notation, we have the following representation of the infinitesimal character $\wlchar{ k_r^\eps }$.

\begin{lemma}\label{lem:k:formula}
One has the identity
\[
\wlchar{ k^\eps_r }(\tau)
=
-(\wlchar{ k^\eps_r } \otimes \doublebar\Upsilon^\eps_r M^{\wlchar{ g^{\eps}_r} }) (\cpmwi - \Id \otimes \one) \tau
	- (\wlchar{ g^{\eps}_r } \otimes \partial_r \doublebar\Upsilon^\eps_r) \cpmwi \tau
\]
for any tree $\tau \in \symCT$.
\end{lemma}
\begin{remark}
The significance of this formula is that the right-hand side only depends on the character $\wlchar{ k_\tau^\eps }$ on proper subtrees of $\tau$. This identity is thus well adapted to an inductive argument, see the proofs of (\ref{eq:k:convergence:1}) and (\ref{eq:k:convergence:2}) below. 
\end{remark}
\begin{proof}

The key point is that by definition of the character $\wlchar{ g^{\eps} _{r} }$ in $\symCT$ and the definition of the twisted antipode $\symCA$ one has that
\begin{align}\label{eq:renom:vanishes:sym}
\doublebar\Upsilon ^{\eps} _{ r } 
	M^{\wlchar{ g^{\eps}_{r} }} 
	\symfi = 0
\end{align}
on $\symCT$ for any $r>0$, where we use the usual notation
$
M^g := (g \otimes \Id) \Delta_-^\ex
$
on $\symCThat$ for any character $g$ of $\symCT$. 
Differentiating (\ref{eq:renom:vanishes:sym}) with respect to $r$, one obtains
\begin{align*}
0
&= \partial_r  (\doublebar\Upsilon^\eps_r M^{\wlchar{ g^{\eps}_r }} \symfi) \\
&= (\partial_r \wlchar{ g^{\eps}_r } \otimes \doublebar\Upsilon_r^\eps ) \cpmwi  
	+ (\wlchar{ g^{\eps}_r } \otimes \partial_r \doublebar\Upsilon_r^\eps) \cpmwi \\
&= (\wlchar{k_r^\eps} \otimes \doublebar\Upsilon^\eps_r M^{\wlchar{ g^{\eps}_r} }) \cpmwi 
	+ (\wlchar{ g^{\eps}_r } \otimes \partial_r \doublebar\Upsilon^\eps_r) \cpmwi 
\end{align*}
on $\symCT$. In the last equality we used that
\begin{align*}
\partial_r \wlchar{ g^{\eps}_r } 
= \partial_r (\wlchar {\tilde f^\eps_r} \circ \wlchar {g^\eps_0}) 
= \wlchar{k_r^\eps} \circ \wlchar {\tilde f^\eps_r} \circ \wlchar {g^\eps_0} 
=  \wlchar{k_r^\eps} \circ  \wlchar{ g^{\eps}_r }.
\end{align*}
\end{proof}

As a consequence, we have the following sufficient condition for (\ref{eq:k:convergence}) to hold.

\begin{lemma}\label{lem:fnknbound}
Let $\tau \in \symCT$ be a tree and assume that (\ref{eq:k:convergence}) holds on the Hopf subalgebra $\symCT[\tau]$ generated by all trees $\sigma \in \symCT$ with strictly less edges than $\tau$. Then one has 
that 
\begin{align}\label{eq:k:convergence:1}
(\wlchar{ g^{\eps}_r }\otimes \partial_r \doublebar \Upsilon^\eps_r) \cpmwi \tau
\end{align}
converges to a finite limit $\eps\to 0$ for any $r>0$, and its supremum over $\eps \in (0,1)$ is moreover bounded in $L^1(0,\infty)$ as a function in $r$. Furthermore, for any properly legged tree $\sigma \in \symCThat$ with $|\sigma|_+<0$ and with strictly less edges than $\tau$ one has that 
\begin{align}\label{eq:k:convergence:2}
\doublebar\Upsilon^\eps_r M^{\wlchar{ g^{\eps}_r }} \sigma
\end{align}
converges to a finite limit as $\eps\to 0$, and is moreover bounded uniformly in $r>0$ and $\eps\in (0,1)$.
In particular, (\ref{eq:k:convergence}) holds for $\tau$.
\end{lemma}
\begin{remark}
The relative simplicity of (\ref{eq:k:convergence:1}) and (\ref{eq:k:convergence:2}) over the corresponding expressions one would get in the strategy outline in (\ref{eq:identity:g:f:g:hat}) is the main motivation for choosing this approach.
\end{remark}
\begin{proof}
Using Lemma~\ref{lem:k:formula}, it is clear that (\ref{eq:k:convergence:1}) and (\ref{eq:k:convergence:2}) imply (\ref{eq:k:convergence}). Note that in (\ref{eq:k:convergence:2}) it is sufficient to consider $\sigma$ with strictly less edges than $\tau$, since $\wlchar{k_r^\eps}$ is an infinitesimal character and vanishes on the unit element $\one$.

In order to see the converse, we first show (\ref{eq:k:convergence:2}). Recall that the large scale integration kernels $K^{(r)}$ converge to $\hat K - K$ in $\CK_0^+$ as $r \to \infty$ and the smooth noises $\eta^\eps$ converge to  $\eta$ in $\SMsz$ as $\eps \to 0$. Since $\sigma$ is properly legged by assumption, the convergence of the expression
\[
\lim_{\eps\to 0}\doublebar\Upsilon^\eps_r M^{\wlchar {g^{\eps}_0}} \sigma
\] 
and the uniform boundedness of $\doublebar\Upsilon^\eps_r M^{\wlchar {g^{\eps}_0}} \sigma$ in $\eps>0$ and $r>0$ are a consequence of Theorem~\ref{thm:evaluation:trees:largescale}, Lemma~\ref{lem:super:regularity:implies:large:scale:bound} and (\ref{eq:identity:eval}). (If $\sigma \in \CV_0$, then this expression vanishes for any $\eps,r>0$.)
The remaining obstacle is therefore the presence of the character $\wlchar{g ^\eps _r}$ instead of $\wlchar {g^{\eps}_0}$ in (\ref{eq:k:convergence:2}). However, by definition one has $\wlchar {\tilde f^{\eps}_r} \circ \wlchar {g^{\eps}_0} = \wlchar {g^{\eps}_r}$, so that it suffices to show that the character $\wlchar {\tilde f^{\eps}_r}$ restricted to $\symCT[\tau]$ is uniformly bounded in $\eps,r>0$ and converges as $\eps\to 0$ to a finite limit. This is a consequence of (\ref{eq:k:convergence}), which holds on $\symCT[\tau]$ by assumption.

In order to derive the bound (\ref{eq:k:convergence:1}) we make the following construction. Consider the extended  set of kernel types $\bar\FL_+:=\FL_+\sqcup\partial\FL_+$ where $\partial\FL_+:=\{\partial\ft:\ft\in\FL_+\}$ is a disjoint copy of $\FL_+$. We let $|\partial \ft|_\fs:=|\ft|_\fs$ for any $\ft\in\FL_+$, and we extend the rule $R$ to a rule $\bar R$ by allowing any kernel-type $\ft$ to be replaced by $\partial\ft$. 
We  denote by $\wlTexA$, $\symCTA$, and $\symCThatA$ the respective spaces constructed in Section~\ref{sec:evaluation:largescale} starting from the rule $\bar R$. 
Finally, we introduce a linear operator $\CD:\wlTex\to \wlTexA$ by setting for any tree $\tau=(T^{\fn,\fo}_\fe,\ft)$ 
\[
\CD\tau:=\sum_{e \in K(T)} (T^{\fn,\fo}_\fe,\partial^e \ft)
\]
where $\partial^e \ft:K(T)\sqcup L(T)\to\bar\FL_+\sqcup \FL_-$ is defined by setting $(\partial^e \ft)_{f}:=\ft_f$ for any $f\in K(T)\backslash\{e\} \sqcup L(T)$, and $(\partial^e\ft)_e:=\partial\ft_e$.
We extend this to a linear operator 
$\CD:\symCThat \to \symCThatA$ by imposing that the Leibnitz rule $\CD(\tau\sigma)= \CD(\tau) \sigma + \tau\CD(\sigma)$ holds.

Finally, we define the kernel assignments $K_{\partial\ft}=0$ and $R^{(r)}_{\partial\ft}:=\partial_r R^{(r)}_\ft$ for any $\ft\in\FL_+$, and we write again
$\doublebar\Upsilon^\eps_r:= \bar\Upsilon^{\eta^\eps,\phi}_{R^{(r)}}$ for the character on $\symCThatA$. 
It follows that (\ref{eq:k:convergence:1}) is equal to
\[
\doublebar\Upsilon^\eps_r M^{\wlchar {g^\eps_r}}\CD\tau,
\]
where we view $\gr$ as a character on $\symCTA$ by setting $\gr(\tau):=0$ for any tree $\tau \in \symCTA$ which contains an edge $e \in K(\tau)$ such that $\ft(e) \in \partial\FL_+$.

Using the induction hypothesis and an argument identical to before (using the identity $\wlchar {\tilde f^{\eps}_r} \circ \wlchar {g^{\eps}_0} = \wlchar {g^{\eps}_r}$ and the fact that by (\ref{eq:k:convergence}) the sequence $\wlchar{ \tilde f^{\eps}_r }$ is bounded), it is now sufficient to bound
\[
\doublebar\Upsilon^\eps_r M^{\wlchar{g^\eps_0}}\CD\tau,
\]
which is again bounded uniformly in $\eps>0$ and $r>0$ as a consequence of Theorem~\ref{thm:evaluation:trees:largescale}. It remains to show that this expression is absolutely 
integrable over $r\in (0,\infty)$ and that this integral is uniformly bounded in $\eps>0$. For this note 
that $\CD \tau$ satisfies the conditions of Theorem~\ref{thm:evaluation:trees:largescale} with the degree 
assignment $\deg_\infty \partial \ft:=\deg_\infty \ft := |\ft|_\fs - |\fs| = \fancynorm\ft_{\fs} - |\fs| - \kappa$ 
for $\ft \in \FL_+$. With this degree assignment however it follows that one has 
$\|R_{\partial\ft}^{(r)}\|_{\CK^+,\partial\ft}\lesssim r^{-\kappa-1}\land 1$ uniformly in $r>0$, and 
we conclude  with (\ref{eq:bound:evluation:largescales}).
\end{proof}

\section{The construction of the shift}
\label{sec:shift}

We fix a character $h \in \fxi \circ \CH$, and we finally construct a sequence $\srn_\delta\in\SMsinf$ for $\delta>0$ such that (\ref{eq:shiftednoisehomo}) and (\ref{eq:shiftedtreeexpectationA}) hold. Let us first motivate the construction below. The convergence in (\ref{eq:shiftednoisehomo}) requires us to choose $\srn_\delta$ in such a way that for any $\ft\in\FL_-$ one has 
\begin{align}\label{eq:shifttozero}
\xi_\ft+(\srn_\delta)_\ft\to 0
\end{align} 
in $\SMsz$. This could simply be accomplished by setting $(\srn_\delta)_\ft=-\xi_\ft^\delta$, where $\xi^\delta$ is a $\delta$-regularisation of $\xi$. However, with this choice there is no hope of satisfying (\ref{eq:shiftedtreeexpectationA}) as well. At this point we make the observation that introducing a perturbation of $-\xi^\delta_\ft$, which lives on scales much smaller than $\delta$, may not destroy the convergence $\eqref{eq:shifttozero}$. On the other hand, such small-scale perturbations of $-\xi^\delta_\ft$ generate resonances in expressions of the type (\ref{eq:shiftedtreeexpectationA}), and the fact that a tree $\tau$ has negative homogeneity implies that perturbations weak enough not to destroy (\ref{eq:shifttozero}) might at the same time give non-vanishing contributions to (\ref{eq:shiftedtreeexpectationA}).

Let us briefly compare this idea to the strategy used in \cite{Friz2016} to show a support theorem for the 2D multiplicative heat equation with purely spatial white noise, known as the 2D-PAM equation. Although the set-up in their paper differs slightly from ours (they use the theory of paracontrolled distributions rather than regularity structures and hard cutoffs of the noise in Fourier space rather than regularisations via convolution) the spirit of the two approaches are similar.
At this stage the authors of \cite{Friz2016} use \emph{deterministic} perturbations of $-\xi^\delta$ at a fixed frequency in order to generate the required resonances. Deterministic perturbations do not fall in our setting, since we assume our noises to be stationary and centred (one could of course use randomly shifted oscillations at a fixed frequency, but this does not seem to generalise well). There are two major reasons why we prefer to use perturbations which are random instead of deterministic.

The first reason concerns the type of expression one gets when calculating expected values of the form (\ref{eq:shiftedtreeexpectationA}). By considering random stationary shifts, we ensure that these expressions are constant (as opposed to space-time dependent). Moreover, by choosing the shift to be non-Gaussian we have a freedom to control the cumulants built between the original noise and the shift. This will be crucial in order to control the expected value of all $\tau \in \FT_-$ (see Definition~\ref{def:FT}) simultaneously.

The second reason concerns the bound of variances of the shifted model, once the expectation can be controlled. This argument was carried out in Proposition~\ref{prop:shiftednoise}. The proof of this proposition uses crucially the results from \cite{ChandraHairer2016}, which in turn requires the shift to be stationary and centred. 

Both of these points were carried out in \cite{Friz2016} by hand, and the success of this strategy seems to rely heavily on the fact that the corresponding regularity structure is relatively simple (in particular the set $\FT_-$ contains only a single tree). 

\subsection{Enlarging the regularity structure}\label{sec:ext:reg:str}

Following the discussion above, we will choose the shift $\srn_\delta = -\xi^\delta + k^\delta$ for some random perturbation $k^\delta$ living on scales much smaller than $\delta$. The random smooth function $k^\delta$ in turn will be written as a sum over functions $k^\delta_{(\Xi,\tau)}$, where $(\Xi,\tau)$ runs over all pairs of noise types $\Xi \in \FL_-$ and trees $\tau \in \FT_-$ with the property that $\Xi \in \ft(L(\tau))$. 
In order to keep the notation clean, we will introduce an extended set of noise types as follows.
For any type $\Xi\in\FL_-$ we let $\tilde\Xi$ be a new symbol such that $\tilde\Xi\notin\FL_-$.
We then define for any type $\Xi\in\FL_-$ the set (recall that $\FL_- \ssq \CJ$ if one identifies elements of $\FL_-$ with elements of $\CT_-$, see Definition~\ref{def:CH}, so that $\tau \in \FT_-$ implies $|L(\tau)| \ge 2$)
\begin{equ}[idx:sFLm]
\sFLm[\Xi]:=\{\Xi,\tilde\Xi\}\cup\{(\Xi,\tau):
	\tau\in\FT_-
	\text{ such that } \exists u \in L(\tau) \text{ with } \ft(u)=\Xi\}\;.
\end{equ}

Consider then the construction given in \cite[Sec.~5]{2DYangMills}, in particular the
class of natural transformations considered in Remark~5.18, the direct sum 
decompositions of Section~5.3, and the construction of regularity structures 
and associated spaces in Sections~5.5--5.8. Recall that the purpose of this construction
is the following. Take a set of types $\FL$ and a rule $R$ as above, as well as a
``space assignment'' $V$, namely a collection of finite-dimensional vector
spaces $V_\Xi$, one for every type $\Xi \in \FL$. Then, \cite[Sec.~5.6]{2DYangMills}
describes a way of using this data to build regularity structures
$\CT$, $\CT^\ex$, spaces $\CT_+$, $\CT_-$, etc which is analogous to the construction
of \cite{BrunedHairerZambotti2016}, but with a copy of $V_\Xi$ ``attached'' to every edge of type $\Xi$,
so that a tree $\tau$ now isn't a basis vector of $\CT$, but defines
a subspace $\CT[\tau]$ that is isomorphic to a suitable symmetrisation of 
$\bigotimes_{e \in E} V_{\ft(e)}$, where $E$ denotes the edge set of $\tau$
and $\ft(e)$ is the type of an edge $e$.
(Symmetrisation is needed for example for $\tau = \Xi^2$ which is isomorphic
to the symmetric tensor product $V_\Xi \otimes_s V_\Xi$.)
The ``classical'' construction of these spaces is then obtained as the special case
of the space assignment $\R$ which simply assigns $\R$ to every type.

The construction is functorial in the sense that one constructs ``abstract'' counterparts
$\mathbf{\CT}$, $\mathbf{\CT}_+$, etc of the spaces $\CT$, $\CT_+$, etc as well as 
of the various linear maps $\Delta$, $\Delta_+$,
etc between them as objects and morphisms of a monoidal category of ``symmetric structures''. 
Every space assignment $V$ then yields a functor $\mathbf{F}_V$ mapping
the abstract objects to their ``concrete'' counterparts. Furthermore, given two
vector space assignments $V$ and $W$, as well as a collection of 
linear maps $A_\Xi \colon V_\Xi \to W_\Xi$, \cite[Rem.~5.18]{2DYangMills} 
yields a natural transformation from $\mathbf{F}_V$ to $\mathbf{F}_W$.
In other words, for any ``abstract'' space $\boldsymbol{\CA}$, $A$ determines 
a linear map $A\colon \mathbf{F}_V(\boldsymbol{\CA}) \to \mathbf{F}_W(\boldsymbol{\CA})$
(note the symbol overload here)
intertwining $\mathbf{F}_V(f)$ and $\mathbf{F}_W(f)$ for any morphism 
$f \colon \boldsymbol{\CA}_1 \to \boldsymbol{\CA}_2$.

Note now that \eqref{idx:sFLm} yields a vector space assignment $V$ by setting
$V_\Xi = \R^{\sFLm[\Xi]}$ for $\Xi \in \FL_-$ and $V_\Xi = \R$ for $\Xi \in \FL_+$.
We henceforth use the convention that $\sCT = \mathbf{F}_V(\boldsymbol{\CT})$
and similarly for $\sCTm$, $\sCTex$, etc. \label{idx:extended:reg:str}
For every $\Xi \in \FL_-$, we have a natural embedding
$\iota_\Xi \colon \R \to V_\Xi$ mapping $1$ to $\Xi$, so that the natural transformation 
of \cite[Rem.~5.18]{2DYangMills} mentioned above yields natural embeddings
$\iota \colon \CT \to \sCT$, etc, which we henceforth simply write as $\CT \subset \sCT$, etc.
We also write $\iota_\Xi^* \colon V_\Xi \to \R$ for the ``adjoint'' obtained by mapping $\Xi$ to $1$ and 
all elements of $\sFLm[\Xi] \setminus \{\Xi\}$ to $0$. Similarly, this yields 
projections $\iota^*\colon \sCT \to \CT$, etc.
In particular, $\iota^*$ yields an embedding $\CG_- \hookrightarrow \sCGm$
obtained by mapping any $\CG_- \ni \ell \colon \CTm \to \R$ to $\ell \circ \iota^* \in \sCGm$.

One important remark is given by the ``direct sum decomposition'' verified in \cite[Sec.~5.3]{2DYangMills}. 
When combined with the construction of \cite[Sec.~5.5--5.6]{2DYangMills}, it yields a canonical 
identification of $\sCT$ (and $\sCTm$, $\sCTex$, etc) with the regularity structure
built from the set of symbols $\sFLm = \bigsqcup_{\Xi \in \FL} \sFLm[\Xi]$ with the rule $\shift R$
obtained from $R$ by allowing to replace any given instance of $\Xi$ by an arbitrary element of $\sFLm[\Xi]$.

Our construction comes with a natural ``summation map''
$\SS_\Xi\colon \R \to V_\Xi$ given by
\begin{equ}[e:defShiftOperator]
\SS_\Xi 1 = \sum  \sFLm[\Xi]\;,
\end{equ}
which yields linear maps $\SS \colon \CT \to \sCT$, etc that also commute with all the operations built in
\cite[Sec.~5.5--5.6]{2DYangMills}, so for example
\begin{equ}[lem:shift:operator:co:product]
(\SS \otimes \SS) \Delta = \Delta \SS\;,\qquad
(\SS \otimes \SS) \cpm = \cpm \SS\;,
\end{equ}
etc. Similarly, we have its ``adjoint'' $\SS^* \colon \sCT \to \CT$ 
defined from the linear maps $\SS_\Xi^* \colon V_\Xi \to \R$ 
mapping every element of $\sFLm[\Xi]$ to $1$. It will be convenient to write 
$\SS[\tau]$ for the collection of those canonical basis vectors $\bar \tau \in \sCT$ such that 
$\SS^* \bar\tau = \tau$.

We denote by \label{idx:sSMinfty}$\sSMinfty := \SMsinf(\sFLm)$ the set of smooth noises as in
Definitions~\ref{def:smooth:noise} and~\ref{def:shifted:noise}, and $\sSMsz := \SMsz(\sFLm)$ for its closure under the 
norm (\ref{eq:noise:norm}).
Since $\Vec(\sFLm) \simeq \bigoplus_{\Xi \in \FL} V_\Xi$, we can define
$\SS^*:\sSMinfty\to\SM_\infty$ by $\SS^* \eta = \eta \circ \SS$, and similarly for $\iota^*$.

The following lemma connects the construction of this section to the discussion of the last section.

\begin{lemma}\label{lem:shiftoperator}
Let $\xi \in \SM_\infty$ be a smooth noise, let $\eta\in\sSMinfty$ be a smooth noise 
extending $\xi$ in the sense that $\iota^* \eta = \xi$, and let 
$\srn \in \SM_\infty$ be defined by $\srn:=\SS^*\eta-\xi$.
Then one has for any $\tau\in\CT$
\begin{align} \label{eq:eta:zeta}
(T_{\srn}\hat{\PPi}^{\xi})\tau
=\PPi^{\xi + \srn} M^{g^\xi}\tau 
=\PPi^{\eta} \SS M^{{g}^\xi}\tau.
\end{align}
\end{lemma}
\begin{proof}
This follows from Theorem~\ref{thm:translationoperator}.
\end{proof}

We will show that there exists a double sequence $\eta^{\eps,\delta} \in \sSMinfty$, $\eps, \delta>0$, 
of smooth, random noises with the property that $\eta^{\eps,\delta}$ extends $\xi^\eps$,
one has $\SS^* \eta^{\eps,\delta} \to 0$  in $\sSMsz$ in the limit $\eps \to 0$ and $\delta \to 0$, 
and one has $\lim_{\delta \to 0} \lim_{\eps\to 0} \Upsilon^{\eta^{\eps,\delta}}M^{g^\eps} \SS \tau = h(\tau)$ for any $\tau \in \FT_-$. Setting $\srn:= \SS^*\eta-\xi$ then concludes the proof of Proposition~\ref{prop:shiftednoise}.

We now identify those trees $\sigma \in \SS[\tau]$ that have the property that their expected value depends linearly on the shift. They will give the dominating contribution to $\Upsilon^{\eta^{\eps,\delta}}M^{g^\eps} \SS \tau$.

\begin{definition}\label{def:oSS}
For any $\tau\in\FT_-$ we define $\oSS[\tau]$ as the set of $\sigma\in\SS[\tau]$ such that there exists a noise type edge $u\in L(\sigma)$ such that 
\begin{itemize}
\item one has $\ft(u) = (\Xi,\tilde\tau)$ for some $\Xi \in \FL_-$ and $\tilde\tau \in \FT_-$ with $\tilde\tau \sim\tau$, and
\item for any noise type edge $v\in L(\sigma)\backslash\{ u \}$ one has $\ft(v)\in\FL_-$.
\end{itemize}
Recall Definition~\ref{def:sim} for the definition of the equivalence relation $\sim$ used here.
We also set $\uSS[\tau]:=\SS[\tau]\backslash (\oSS[\tau] \sqcup \{\tau\})$. With this notation, we define
\[
\oSS\tau := \sum_{ \sigma \in \oSS[\tau] } \sigma
\quad{\text{ and }}\quad
\uSS\tau := \sum_{ \sigma \in \uSS[\tau] } \sigma.
\]
\end{definition}
Note that one has the identity
$
\SS=\oSS+\uSS+\Id.
$

\begin{example}
We visualise this construction on the example of the tree $\treeExampleKPZaa$ from the generalised KPZ equation.
Here the set  $\SS[\treeExampleKPZaa]$ is given by
\begin{multline*}
\treeExampleKPZaa, \treeExampleKPZaaEa, \treeExampleKPZaaEaa, \treeExampleKPZaaEb, \treeExampleKPZaaEc, \treeExampleKPZaaEd, \treeExampleKPZaaEe, \treeExampleKPZaaEf, \treeExampleKPZaaEg, \treeExampleKPZaaEh, \\
\treeExampleKPZaaEi, \treeExampleKPZaaEj, \treeExampleKPZaaEk, \treeExampleKPZaaEk, \treeExampleKPZaaEl, \treeExampleKPZaaEm
\treeExampleKPZaaEn
\end{multline*}
Here a leave drawn as square $\leafsE$ is a placeholder for an element in $\sFLm[\leafs] \backslash \{\leafs\}$, and similar for $\leafsa$ and $\leafsaE$. Hence every tree drawn above (except the first one) is really a placeholder for a finite family of trees.
The set $\oSS[\treeExampleKPZaa]$ is then given by
\[
 \treeExampleKPZaaEEa, \treeExampleKPZaaEEaa, \treeExampleKPZaaEEb, \treeExampleKPZaaEEc,
\]
where now $\leafsEE$ only runs over extended noise types $\snt \in \sFLm[\leafs]$ of the form $\snt = (\leafs,\tilde \tau)$, where $\tilde\tau \in \FT_-$ is a tree with exactly 2 noises of type $\leafs$ and 2 noises of type $\leafsa$ (and no other noises); and similar for $\leafsaEE$ and $\leafsa$
\end{example}

\subsection{Construction of the shift as Wiener chaos}
\label{sec:shift:WC}

We will choose the perturbation $\Extn_{(\Xi,\tau)}$ in a homogeneous Wiener chaos of fixed order, so that $\Extn_{(\Xi,\tau)}$ is determined by specifying a kernel $K_{(\Xi,\tau)}^{\eps,\delta}$. 
(To clarify the idea behind the construction below, consider Example~\ref{ex:shift:dominating:intro} in the introduction).

The kernels will be constructed by fixing a smooth, compactly supported function and rescaling it to scales much smaller than $\delta$ at some homogeneity $\shift\fs(\Xi)$, see (\ref{eq:K}), (\ref{eq:K0}). It will be crucial that we choose this homogeneity $\shift\fs(\Xi)$ carefully in such a way that shifted trees $\sigma$ as in Example~\ref{ex:shift:dominating:intro} (i.e.\ where exactly one noise $\Xi$ of some tree $\tau \in \FT_-$ is replaced by $(\Xi,\tau)$) have just slightly negative homogeneity. For this, we fix $\bar\kappa>0$ small enough and we 
define a homogeneity assignment $\shift\fs: \sFL \to \R$ in the following way:

\begin{definition}\label{def:en:homo}
Set $\shift\fs(\Xi):=\shift\fs(\tilde\Xi):=\fs(\Xi)$ for any noise type $\Xi \in \FL_-$, and $\shift\fs(\ft):=\fs(\ft)$ for any kernel-type $\ft \in \FL_+$. For any noise type of the form $(\Xi,\tau)\in\sFLm$, set
\[
\shift\fs(\Xi,\tau):=\fs(\Xi)-\fancynorm\tau_\fs-\bar\kappa.
\]
\end{definition}

We now have two homogeneity assignments $\fs$ and $\shift\fs$ with $\shift\fs\ge \fs - \bar\kappa$ on $\sFLm$.
For any $\eps>0$ and $\delta>0$ we are going to define a random smooth noise $\Extn \in \sSMinfty$ satisfying the following.
\begin{itemize}
\item For any noise type $\Xi\in\FL_-$ one has that $\Extn_\Xi=\xi^\eps_\Xi$ and $\Extn_{\tilde\Xi}=-\xi^\delta_\Xi$.
\item For any noise type $\snt \in\sFLm\backslash\FL_-$ the noise $\Extn_\snt$ is independent of $\eps$.
\item For any noise type of the form $(\Xi,\tau)$ the noise $\Extn_{(\Xi,\tau)}$ is a random centred stationary smooth function that takes values in the $\cm(\Xi,\tau)$-th homogeneous Wiener chaos with respect to $\xi$, where \label{idx:cmXi}$\cm(\Xi,\tau):=[L(\tau),\ft] \backslash\{ \Xi \}$. (Note that $[L(\tau),\ft] \backslash\{ \Xi \}$ denotes the multiset where \emph{exactly one} instance of $\Xi$ is removed from $[L(\tau),\ft]$.)
\end{itemize}

We also write $\cm(\Xi)=\cm(\tilde\Xi):=\{\Xi\}$ for any $\Xi\in\FL_-$.
We now define for any $(\Xi,\tau)\in\sFLm$ 
a smooth kernel $K^{\delta}_{(\Xi,\tau)} \in \CYsimpp{m(\tau) }$, where $m(\tau):=\#L(\tau)-1$, depending only on
$\delta>0$ (compare (\ref{eq:CYsimp}) for the notation used here).
We define $K^{\delta}_{(\Xi,\tau)}$ by rescaling a fixed kernel 
$\Phi_{(\Xi,\tau)} \in \CYsimpp{m(\tau)}$,
independent of $\delta>0$, which will be determined in Lemma~\ref{lem:dominatingpart} below.
In order to avoid case distinctions, we also define for any noise type $\Xi\in\FL_-$ the kernels
\[
\Phi_\Xi := \rho, \qquad \Phi_{\tilde\Xi} := -\rho,
\]
so that $\Phi_\Xi, \Phi_{\tilde\Xi} \in \CYsimpp{1}$.
Recall that $\rho$ is a compactly supported smooth cut-off function integrating to one.
Before we choose the kernels  $\Phi_{(\Xi,\tau)}$, we describe how we rescale them in order to obtain the kernels $K_{(\Xi,\tau)} ^{\delta}$.
Let us first define for any $n \ge 1$,
any scale $\lambda>0$ and any homogeneity $\alpha\in\R$ the rescaling operator $\rescale(\lambda, \alpha) : \CYsimp \to \CYsimp$ by
\begin{align}\label{eq:rescalingoperator}
\rescale(\lambda,\alpha) (K_0 \otimes \ldots \otimes K_n) :=
\lambda^{\alpha-|\fs|} (K_0(\lambda^{-\fs} \cdot) \otimes \ldots \otimes K_n( \lambda^{-\fs} \cdot)).
\end{align}
Note that $\CU K$ transforms as
\[
(\CU \rescale(\lambda,\alpha) K) (x_1, \ldots, x_n)
= 
\lambda^{\alpha} (\CU K) (\lambda^{-\fs} x_1, \ldots ,\lambda^{-\fs} x_n)
\]
for any $K \in \CYsimp$. 

%

The correct homogeneity to rescale a kernel $\Phi_{\sxt}$ so that the random variable $\Extn_\sxt$ is of order $1$ for the homogeneity $\sfs\sxt$ is given by
\begin{align}\label{eq:alpha}
\alpha_{(\Xi,\tau)}:=\shift\fs(\Xi,\tau)- m(\tau) \shalf
\end{align}
for any $(\Xi,\tau)\in\sFLm$. (Recall that $m(\tau) =\#L(\tau)-1$.) This follows from the fact that the covariance of $\Extn_\sxt$ is given by $|\E[\Extn_\sxt(z) \Extn_\sxt(\bar z)]| = |\int dx K_\sxt^\delta(z,x) K_\sxt^\delta(\bar z,x)| \lesssim |z-\bar z|^{2\alpha_\sxt + m(\tau)|\fs|}$.
We will later on choose for any $\delta>0$ and any tree $\tau\in\FT_-$ a scale $\lambda^\delta_{\tau}\in (0,1)$ and a real constant $a^\delta_{\tau}\in\R$. 
Let $\sparaA := \R^{\FT_-}$ and denote by $\sparaL$ the set of scales $\lambda \in (0,1)^{\FT_- \sqcup\{ \star, \dagger \}}$ such that $\lambda_{\tau}$ depends only on the equivalence class $[\tau]_\sim$ of $\tau$. 
(This property will be useful in the proof of Lemma~\ref{lem:dominatingpart} below.)  
For fixed scales $\lambda_\tau \in \sparaLb$ and constants $a_\tau \in \sparaL$ we now make the following definition.

\begin{definition}
For any $(a,\lambda) \in \sparaA \times \sparaL$ and any $\sxt \in \sFLm$ with $\fancynorm\tau_\fs<0$, we define the kernel
\begin{align}\label{eq:K}
K_{(\Xi,\tau)}^{\para}:=
a_{\tau}\rescale
(
	{\lambda_{\tau}}, \alpha_{(\Xi,\tau)}
)
\Phi_{(\Xi,\tau)}.
\end{align}
For $\fancynorm\tau_\fs=0$ we use a slightly different definition
\begin{align}\label{eq:K0}
K_{(\Xi,\tau)}^{\para}:=
a_{\tau}
\frac{1}{N^\lambda _\tau}
\sum_{k=0}^{N^\lambda_\tau-1} 2^{-\bar\kappa k}
\rescale
(
	{2^{-k}\lambda_{\tau}},\alpha_{(\Xi,\tau)}
) \Phi_{(\Xi,\tau)},
\end{align}
where $N^\lambda _\tau$ is the smallest integer larger then $(\lambda_\tau)^{-1}$.
\end{definition}
We also set
\begin{align}\label{eq:ham}
K^{\para}_{\Xi}:=\rescale(\eps,-|\fs|)\rho
\quad\text{ and }\quad
K^{\para}_{\tilde\Xi}:= \rescale(\delta,-|\fs|)\rho,
\end{align}
where we write $\delta := \lambda_\star$ and $\eps:=\lambda_\dagger$. 
\begin{remark}
We include $\eps$ and $\delta$ into the data $\lambda$ in order to avoid case distinctions in some expressions below. Sometimes it will be useful to make $\eps$ explicit. In these cases we write $\eta^{\eps,\para}$ and $K^{\eps,\para}$ with $a \in \sparaA$ and $\lambda \in \sparaLa:= (0,1)^{ \FT_- \sqcup\{ \star \}}$.
\end{remark}
\begin{example}\label{ex:log}
To understand (\ref{eq:K0}), consider first a tree $\tau \in \FT_-$ with $\fancynorm{\tau}_\fs<0$ and assume for simplicity that $\tau$ does not contain any divergent proper subtree. Consider two trees $\taua,\taub \in \SS[\tau]$, where in $\taua$ (resp. $\taub$) exactly one noise type $\Xi$ (resp.\ two noise types $\Xi$, $\tilde\Xi$) are replaced by $(\Xi,\tau)$ (resp.\ $(\Xi,\tau)$, $(\tilde\Xi,\tau)$), so that in the notation of Definition~\ref{def:oSS} one has $\taua \in \oSS[\tau]$ and $\taub \in \uSS[\tau]$. It then follows from a simple scaling argument and (\ref{eq:K}) that one has 
\begin{align}\label{eq:K:K0:example}
|\Upsilon^{ \eta^{\eps,a,\lambda} } \taua| \simeq a_\tau \lambda_\tau^{-\bar \kappa}
\qquad
\text{and}
\qquad
|\Upsilon^{ \eta^{\eps,a,\lambda} } \taub| \lesssim a_\tau^2.
\end{align}
The second bound follows from the fact that $|\taub|_{\sfs}>0$. We will choose $a_\tau$ such that $\Upsilon^{ \eta^{\eps,a,\lambda} } \taua$ is of order $1$ as $\lambda_\tau \to 0$, hence $a_\tau \simeq \lambda_\tau^{\bar\kappa}$. Thus, one has $\Upsilon^{ \eta^{\eps,a,\lambda} } \taub \to 0$.

This argument used crucially that $|\taub|_{\sfs}>0$, which fails in case $\fancynorm{\tau}_\fs=0$ where one has $\fancynorm{\taua}_{\sfs}= -\bar\kappa$ and $\fancynorm{\taub}_{\sfs}= -2\bar\kappa$. It follows that if we simply defined $K_{(\Xi,\tau)}^{a,\lambda}$ via \eqref{eq:K} in this case, we would get that $|\Upsilon^{ \eta^{\eps,a,\lambda} } \taub|$ is order $1$. In order to continue, we ``spread out'' the kernel $K_{(\Xi,\tau)}^{a,\lambda}$ in frequency space via \eqref{eq:K0}. One can readily check that the first relation in (\ref{eq:K:K0:example}) still holds, while for the second relation essentially only the resonant terms contribute (compare the proof of Lemma~\ref{lem:blowuprenormconstantzerohomo}, where this is made precise), and we obtain the bound
\begin{align*}
|\Upsilon^{ \eta^{\eps,a,\lambda} } \taub| 
&\lesssim 
\frac{a_\tau^2 }{(N_\tau^\lambda)^2} \sum_{k=0}^{N_\tau^\lambda-1} 
2^{-2\bar\kappa k} (2^{2\bar\kappa k}\lambda_\tau^2)\lesssim  (N_\tau^\lambda)^{-1},
\end{align*}
which converges to $0$ as $N_\tau^\lambda \to \infty$.
\end{example}

We are now given a family of kernels $K_\snt^{\para} \in \CYsimpp{m(\snt)}$ and a multiset $\cm(\snt)$ for any noise type $\snt\in\sFLm$, where $m(\snt) := \#\cm(\snt)$.
(Think of $\cm(\snt) = \cm(\Xi,\tau)$ as being defined as discussed after Definition~\ref{def:en:homo}.)
With this notation we now make the following definition, for which 
we recall Definition~\ref{def:smooth:noise}.

\begin{definition}
Let $\FKernel^{\para} \in \CYN$ be defined by setting $(\FKernel^{\para})_\cm^\snt := K_\snt^\para \I_{\cm = \cm(\snt)}$ for any multiset $\snt$ with values in $\FLm$ and any $\snt \in \sFLm$.
We then define the smooth noises $\eta^{\para} \in \sSMinfty$ by setting,
for any noise type $\snt \in \sFLm$, 
\begin{align}\label{eq:eta}
\eta^{\para}_{\snt} := \FKtN ( \FKernel^{\para}).
\end{align}
\end{definition}
From (\ref{eq:smooth:noise}) it follows that $\eta^{\para}_\snt = \sintf_{\cm(\snt)}( \CU K^{\para}_\snt)$.
Recall \eqref{eq:simple:kernel} that the operator $\CU$ is given by
\begin{align}
\CU K (x_1, \ldots , x_n) = \int_{\bar\domain} dy K_0(y) K_1 (x_1 - y) \ldots K_n (x_n - y).
\end{align}



The following is then a simple consequence of this definition.
\begin{lemma}
For any $\parad \in \sparaL$ 
and any $\Xi\in\FL_-$ one has $\eta_\Xi^{ \para }=\xi^\eps_\Xi$ and $\eta_{\tilde\Xi}^{ \para }=-\xi^\delta_\Xi$.
\end{lemma}
%

In order to determine our shift, we are left to choose for any $\sxt \in \sFLm$ a compactly supported kernel $\Phi_{(\Xi,\tau)}$,  and for any $\delta>\eps> 0$ a choice of parameters $(\para) \in \sparaA \times \sparaL$ with $\lambda_\dagger = \eps$ and $\lambda_\star = \delta$.

The following lemma determines a choice of smooth kernels $\Phi_{(\Xi,\ft)}$.

\begin{lemma}\label{lem:dominatingpart}
Let $\oSS$ be the operator from Definition~\ref{def:oSS} and let $g^\eps$ be the BPHZ character for the noise $\xi^\eps$ as in Section~\ref{sec:BPHZ:theorem}, which we view as an element of $\sCGm$ as in Section~\ref{sec:ext:reg:str}.
Then, there exists a choice of kernels $\Phi_{(\Xi,\tau)}\in \CYzsimpp{m(\tau)}$ for any $(\Xi,\tau) \in \sFLm$ such that the following holds.
For any tree $\tau \in \FT_-$ and any $C>0$ one has the identity 
\begin{align*}
\lim_{\eps \to 0} \Upsilon^{\eta^{\eps,\para}} M^{g^\eps}\oSS\tau
= a_\tau (\lambda_\tau)^{-\bar\kappa}+o((\lambda_\tau)^{-\bar\kappa}).
\end{align*}
where the $o((\lambda_\tau^\delta)^{-\bar\kappa})$ constant is such that 
\[
(\lambda_\tau^\delta)^{\bar\kappa}o((\lambda_\tau^\delta)^{-\bar\kappa})
\to 0
\]
as $\lambda_\tau \to 0$ uniformly over $\lambda \in \sparaLb$ and $a \in \sparaA$ with $|a|_\infty<C$.
\end{lemma}
\begin{proof}
Fix for the entire proof an equivalence class $\Theta\in\FT_-/_\sim$. Let $\tau\in\Theta$ and
$\sigma = (S^\fn_\fe,\ft) \in\oSS[\tau]$, and let $w\in L(\sigma)$ be the unique noise type edge such that $\ft(w)=(\Xi,\tilde\tau)$ for some $\Xi\in\FL_-$ and some $\tilde\tau\in\Theta$. 
It follows that
\begin{align*}
\Upsilon^{\eta^{\eps,\para}}M^{g^\eps}\sigma  
&= \Upsilon^{\eta^{\eps, \para}}\sigma \\
&= \int_{\bar\domain^{L(\tau)}} dx \, 
\big( \CK_K \tau \big) ((x_u)_{u\in L(\tau)}) 
\E^c[
	((\xi^{\eps}_{\ft(u)} (x_u))
		_{u\in L(\tau)\backslash\{w\}},
	\eta
		_{(\Xi,\tilde\tau)}^{\eps,\para}(x_w)) ].
\end{align*}
(Both equalities are consequences of the fact that $\eta^{\eps,\para}_{\ft(w)}$ is in a homogeneous Wiener chaos of order $m(\tau) = \#L(\tau)-1$.)
In the limit $\eps\to 0$ we obtain
\begin{align*}
\int_{\bar\domain^{L(\tau)}} dx \, 
\big( \CK_K \tau \big)((x_u)_{u\in L(\tau)}) 
\E^c[(\xi_{\ft(u)}(x_u))_{u\in L(\tau)\backslash\{w\}},
I_{\cm(\Xi,\tau)}(\CU K_(\Xi,\tilde\tau)^{a,\lambda}(x_w,\cdot)) ].
\end{align*}
For $\fancynorm\tau<0$ this expression is equal to
\begin{align}\label{eq:yup}
a_{\tilde\tau} \zeta_{\tau,\Xi}
\int_{\bar\domain^{L(\tau)}} dx \, 
\big(
	\CK_K \tau
\big)(x) 
\rescale(\lambda_{\tau} , \alpha_{(\Xi,\tau)}) \CU\Phi_{(\Xi,\tilde\tau)}(x_w ; x \circ \phi^{-1}),
\end{align}
where we used that $\lambda_\tau = \lambda_{\tilde\tau}$ and $\alpha_{(\Xi,\tau)} = \alpha_{(\Xi,\tilde\tau)}$, 
where $\zeta_{\tau,\Xi}\in\N$ denotes a symmetry factor, and where  $\phi:L(\tau)\backslash\{w\}\to \td(\cm(\Xi,\tau))$ denotes an arbitrary bijection with the property that $\ft(u)=\phi_1(u)$ for any $u \in L(\tau)\backslash\{w\}$. 
It follows from the definition of $\cm(\Xi,\tau)$ that such a bijection exists and from the symmetry properties of $\Phi_{(\Xi,\tilde\tau)}$ that the integral is independent of this choice. 
(Here we assume without loss of generality that $\Phi_{(\Xi,\tau)}$ is symmetric under all permutations 
of $[m(\tau)]$ which leave the noise-type $\ft$ invariant, where $\ft:[m(\tau)] \to \FLm$ is the unique order preserving map such that $[[m(\tau)], \ft] = \cm(\Xi,\tau)$.)
Recall now the definition of $\alpha_{(\Xi,\tau)}$ from (\ref{eq:alpha}), and note that after a change of integration $x \to \lambda_\tau^\fs x$ we obtain the expression
\begin{equ}
a_{\tilde\tau} \zeta_{\tau,\Xi}(\lambda_\tau)^{-\bar\kappa}
\int_{\bar\domain^{L(\tau)}} dx \, 
\big( \CK_{K^{(R)}} \tau \big)(x) 
\CU\Phi_{(\Xi,\tilde\tau)}(x_w;x \circ \phi^{-1})\;,
\end{equ}
where $R=(\lambda_\tau)^{-1}$ and the assignment $K^{(R)}$ is given by $K^{(R)}_\ft (x) = R^{|\fs|-|\ft|_\fs} K_\ft(R^{-\fs} x)$ for any $\ft \in \FL_+$.

As $\lambda_\tau\to 0$ the integral in the last expression converges to
\begin{align}\label{eq:dominatingintegral}
\int_{\bar \domain^{L(\tau)}} dx\,
\big( \CK_{\hat K}\tau \big)(x)   
\CU\Phi_{(\Xi,\tilde\tau)}(x_w;x \circ \phi^{-1}).
\end{align}
(The integrand is absolutely integrable. The fact that this integral is finite on small scales is easy to see, the bound on large scales follows from the assumption that $\CU\Phi_{(\Xi,\tilde\tau)}$ is compactly supported and Lemma~\ref{lem:super:regularity:implies:large:scale:bound}. 
One could also see this directly from a simple power counting argument, or equivalently from \cite[Thm.~4.3]{Hairer2017}.) 

For $\fancynorm\tau=0$ equation \eqref{eq:dominatingintegral} follows almost identically. Indeed, in this case (\ref{eq:yup}) should be replaced by 
\begin{align}
a_{\tilde\tau} \zeta_{\tau,\Xi}
\frac{1}{N_\tau^\lambda}
\sum_{k=0}^{N_\tau^\lambda-1}
2^{-\bar \kappa k}
\int_{\bar\domain^{L(\tau)}} dx \, 
\big(
	\CK_K \tau
\big)(x) 
\rescale(2^{-k}\lambda_{\tau} , \alpha_{(\Xi,\tau)}) \CU\Phi_{(\Xi,\tilde\tau)}(x_w ; x \circ \phi^{-1}),
\end{align}
which as above gives $a_{\tilde\tau} \zeta_{\tau,\Xi} (\lambda_\tau)^{-\bar\kappa}$ times an integral expression which converges to 
\begin{align}
\frac{1}{N_\tau^\lambda}
\sum_{k=0}^{N_\tau^\lambda-1}
2^{-\bar \kappa k}
2^{\bar\kappa k}
\int_{\bar \domain^{L(\tau)}} dx\,
\big( \CK_{\hat K}\tau \big)(x)   
\CU\Phi_{(\Xi,\tilde\tau)}(x_w;x \circ \phi^{-1}),
\end{align}
so that we recover (\ref{eq:dominatingintegral}).

It remains to argue that there exists a choice of $\Phi_{(\Xi,\tau)} \in \CYzsimpp{m(\tau)}$ for any $(\Xi,\tau) \in \sFLm$
such that for any $\tau, \tilde\tau \in \FT_-$ with $\tau \sim \tilde\tau$ the expression in (\ref{eq:dominatingintegral}) is equal to $\zeta_{\tau,\Xi}^{-1}\delta_{\tau,\tilde\tau}$. 
For this we recall that $\CJ\cap \linspace{\FT_-}  = \{ 0 \}$ by Definition~\ref{def:FT}.
Moreover, the space of functions of the form $\CU\Phi \in \bar\CC_c^\infty(\bar\domain^\cm)$ for $\Phi \in \CYzsimpp{m(\tau)}$ which is symmetric under permutations of $[m]$ which preserve the noise-type in the same sense as above are dense in $\bar\CC_c^\infty( \bar\domain^\cm/\groupD)$.
The claim now follows from the definition of the ideal $\CJ$ in Definition~\ref{def:CH}. 
\end{proof}

Finally, we want to bound the norm $\|\eta^{\para}\|_{\sfss}$ for homogeneity assignments $\sfss:\sFLm \to \R_-$, compare \eqref{eq:noise:norm}.
As long as $\sfss\le\shift\fs$, with $\sfs$ as in Definition~\ref{def:en:homo}, we obtain a bound uniformly in $\lambda \in \sparaL$. If this condition is violated a uniform bound of this form is in general not true. However, it is still possible to derive a bound on this quantity in terms of the scales $\lambda_\tau$ in the following way.
\begin{lemma}\label{lem:cumulantexpansionB}
Let $\sfss : \sFLm \to \R_-$ be a homogeneity assignment.
For any scale $\lambda \in \sparaL$ we define the quantity
\begin{align*}
\ul\lambda (\sfss):=
&\min\{ \lambda_\snt : \snt \in \sFLm \, ,  \,
\sfss(\snt)
>
\shift\fs(\snt) \}\land 1.
\end{align*}
For any natural number $N\in\N$ and any $\ul\lambda>0$ there exists a constant $C_{N}(\ul\lambda)>0$ such that the following holds.
For any $A>0$ one has the bound
\begin{align}\label{eq:eta:eps:delta:bound}
\|\eta^{\para}\|_{\sfss}
\le
C_{N}(\ul\lambda(\sfss))
\end{align}
uniformly $\lambda \in \sparaL$ and $a_\tau \in \R$ with the property that $|a|_\infty < A$.

In particular, the convergence (\ref{eq:shiftednoisehomo}) holds provided that $a_{\tau}^{(\delta)}\to 0$ as $\delta\to 0$ for any $\tau\in\FT_-$.
\end{lemma}
\begin{proof}
We have to bound $\| K_\snt^{\para} \|_{\sfss(\snt) - m(\snt) \shalf}$ for any $\snt \in \sFLm$, and $\varnorm{\FKernel^{\para}}$, compare (\ref{eq:CYNnorm}). 
Fix $\snt = (\Xi,\tau) \in \sFLm$ and assume first that $\homofancy{\tau} < 0$. We only treat the slightly more difficult case $m(\tau) \ge 2$ in detail. 
We show bounds uniform in $\bar\alpha$ as in (\ref{eq:CYsimp:norm}), so that $\bar\alpha_i \in \R_-$, $i=0, \ldots, m(\snt)$ is such that $\bar \alpha_0 > -|\fs|-1$, $\bar\alpha_i > -|\fs|$ and $\sum_{i\ge 0} \bar \alpha_i = \sfss(\snt) - m(\snt) \shalf - |\fs|$. 
For the purpose of this proof we assume for notational simplicity that $\Phi_{(\Xi,\tau)}$ is a simple tensor product (in general, it is a linear combination of such terms, but since the number of summands does not change under rescaling one can repeat the argument given here for each summand individually).
Write $\Phi_{(\Xi,\tau)} := \Phi_0 \otimes \ldots \otimes \Phi_n$, so that 
\[
K_\snt^{\para} = a_\tau \lambda_\tau^{\alpha_{(\Xi,\tau)} - |\fs| } \Phi_0(\lambda ^{-\fs} \cdot) \otimes \ldots \otimes  \Phi_n(\lambda ^{-\fs} \cdot).
\]
Since $\int (\Phi_\snt^{\para})_0 = 0$ by definition one has 
$\| (K_\snt^{\para})_0 \|_{\bar\alpha_0} \lesssim \lambda^{-\bar\alpha_0}$ and for $i \ge 1$ one has $\| (K_\snt^{\para})_0 \|_{\bar\alpha_i} \lesssim \lambda^{-\bar\alpha_i}$, both uniformly over all $\bar\alpha$ as above. It follows that
\[
\| K_\snt^{\para} \|_{\sfss(\snt) - m(\tau) \shalf}
\lesssim
\lambda_\tau^{-\sfss(\snt) + m(\tau) \shalf + |\fs| + \alpha_{(\Xi,\tau)} - |\fs|}
=
\lambda_\tau^{\sfs(\Xi,\tau) - \sfss(\Xi,\tau)}.
\]
Since $\sfs(\Xi,\tau) - \sfss(\Xi,\tau) <0$ implies $\lambda_\tau  > \ul\lambda(\sfss)$, the required bound follows. 

In case that $\homofancy\tau = 0$ one proceeds in the same way, using the fact that $\sum_{k \ge 0} 2^{-\kappa k}$ is finite.

Finally, bounding the ``variances'' (\ref{eq:varnorm}) is a simple exercise using the fact $(\Phi_\snt^{\para})_0$ integrates to zero and $2\sfs \ge -|\fs| - 2\kappa - 2\bar\kappa > -|\fs|-1$.
\end{proof}

\subsection{A recursive strategy for choosing \texorpdfstring{$\lambda_\ft^{\eps,\delta}$}{lambda t e}}\label{sec:monotonestatements}

Our shift $\srn_\delta$ is defined up to specifying a sequence of constants $a^{\delta} \in \sparaA$ and a sequence of scales $\lambda^{ \delta } \in \sparaLa$ with $\lambda^{ \delta }_\star = \delta$ for any $\delta>0$. 

The constants $a_\tau^\delta$ will be chosen in the subsequent section as solutions to a fixed-point problem which we will show has a solution provided that the scales $\lambda^\delta_\tau$ are chosen in a good way. We will choose the scales $\lambda_{\tau}^\delta$ only depending on the homogeneity $\fancynorm{\tau}_\fs$ and the number of leaves of $\tau$. (In particular the scale only depends on the equivalence class $[\tau]_\sim$ of $\tau$, so that $\lambda^\delta \in \sparaLb$). To this end we write $\ci(\tau):= \big( \fancynorm{\tau}_\fs,\#L(\tau) \big)$, and we define the set\label{idx:CI}
\[
\CI:= \big \{
		\ci(\tau) : \tau\in\FT_-
\big \},
\]
On $\CI$ we define the total order
\[
(\gamma,l)\le (\beta,r),
\] 
if and only if either $\gamma<\beta$, or $\gamma= \beta$ and $l\ge r$. (Note the reversed direction of the second inequality!)
\begin{example}
Consider as an example the KPZ equation, where one has 
$$\TT_-=\{ \<0>,\, \<2>, \,\<11>, \,\<21>,\,\<2x>,\, \<40>, \,\<211> \}.$$ 
By definition of $\CJ$ one has $\FT_- = \{  \<2>, \<21>, \<40>, \<211> \}$, and $\CI=\{(-1,2), (-\half,3), (0,4)\}$.
\end{example}
It will also be convenient to introduce the set $\CI^\star:=\CI\cup\{\star\}$, where we extend the total order $\le$ to $\CI^\star$ by setting $\ci\le\star$ for any $\ci\in\CI$. As above we always set $\lambda_\star^\delta:=\delta$ for any $\delta>0$.


The arguments showing the existence of a solution to the fixed point problem we will be looking at (c.f.\ (\ref{eq:fpp})) will in general only hold if one chooses the scales $\lambda^\delta_{\ci(\tau)}$ such that $\lambda_{(\gamma, l)}^\delta$ is ``small enough'' compared to $\lambda_{(\beta,r)}^\delta$ 
whenever $(\gamma,l)\le (\beta,r)$.
In order to make the arguments below more systematic, we introduce the set
\begin{align}\label{eq:Lambda:indexset}
\sparaLc :=\{\lambda\in (0,1)^{\CI^\star} : \ci< \cj \text{ implies }\lambda_\ci<\lambda_\cj \}.
\end{align}
We will view $\sparaLc$ as a subset of $\sparaLb$ by setting $\lambda_\tau := \lambda_{\ci(\tau)}$ for any $\tau \in \FT_-$ and any $\lambda \in \sparaLc$.
\begin{remark}
It follows that in our notation we have $\lambda_\tau^\delta=\lambda_{(\Xi,\tau)}^\delta=\lambda_{\ci(\tau)}^\delta$ for any $\tau\in\FT_-$ and any $\Xi\in\FL_-$ such that $(\Xi,\tau)\in\sFLm$.
\end{remark}
We also introduce a bit of notation for general finite totally ordered sets $(I,\le)$. For notational convenience we formulate out statements for $(I,\le) = ([M],\le)$, where $M \in \N$ is the given by $M=\#I$.
\begin{definition}\label{def:attainable}
Let $([M],\le)$ be a finite totally ordered set and let $S=S(\lambda)$ be any statement depending on $\lambda\in(0,1)^M$. We define recursively in the number of elements $M$ the notion of an attainable statement. If $M=1$, then we call $S$ an \emph{attainable} statement if there exists $\bar\lambda_1>0$ such that $S(\lambda_1)$ holds for any $\lambda_1 \in(0,\bar\lambda_1)$. 
For $M \ge 2$ and any fixed $\lambda_{M}>0$ we denote by $(S|\lambda_M)(\lambda_{i_1},\ldots,\lambda_{{M-1}})$ the statement depending on $\lambda_{1},\ldots,\lambda_{{M-1}}$ defined by
\[
(S|\lambda_{M})(\lambda_{1},\ldots,\lambda_{{M-1}})\Longleftrightarrow S(\lambda_{1},\ldots,\lambda_{M}).
\]
We then call the statement $S$ \emph{attainable} if there exists $\bar \lambda_M>0$ such that for any $\lambda_{M}<\bar\lambda_M$ the statement $(S|\lambda_{M})$ is attainable.
\end{definition}
We will often use the following lemma, which is a direct consequence of the definition of attainable statements.
\begin{lemma}\label{lem:attainablestatements}
	Let $I$ be a finite, totally ordered set and let $R,S$ be attainable statements on $I$. Then the conjunction $R\land S$ is attainable.
\end{lemma}

The strategy of the following sections will be as follows. We will show various lemmas whose statements are attainable statements for the family $\lambda\in (0,1)^{\CI^\star}$. These (finitely many) statements in conjunction imply that (\ref{eq:fpp}) can be solved, so that the existence of a solution is an attainable statement. This in particular implies that there exists a choice of scales $\lambda^\delta \in \sparaLc$ for any $\delta>0$ small enough such that the statement holds true, and this concludes the proof.


\subsection{A fixed point argument}\label{sec:fpa}

Our goal is to find a family $(a^\delta,\lambda^\delta) \in \sparaA \times \sparaLc$ for $\delta>0$, converging to $0$ as $\delta \to 0 $, such that 
\[
\lim_{\delta\to 0}\lim_{\eps\to 0}
\Upsilon ^{\eta^{\eps}(a^\delta,\lambda^\delta)} M^{ g^\eps}\SS\tau  = h(\tau)
\]
for any $\tau\in\FT_-$. Here and below we write $\eta^\eps(a,\lambda):= \eta^{\eps,a,\lambda}$ in order to make some expressions more readable.
For $\tau \in \TT_-$ we introduce the function $F_\tau : \sA \times \sparaLc \to \R$ by setting
\begin{align}\label{eq:Ftau}
F_\tau
(a, \lambda):=
\lim_{\eps\to 0} \Upsilon^{\eta^{\eps}(a,\lambda)} M^{g^\eps}\SS\tau.
\end{align}
Note that restricted to $\tau \in \FT_-$ we obtain a map
\[
F : \sA \times \sparaLc \to \sA.
\]
We also define the functions $\oF$ and $\uF$ in the same way with $\SS$ replaced by $\oSS$ and $\uSS$ respectively, see Definition~\ref{def:oSS}. Since $\SS=\oSS+\uSS+\Id$ and $\Upsilon^{\eta^{\eps}(a,\lambda)} M^{ g^\eps }\tau=0$ for any $\tau\in\TT_-$ one has $F=\oF+\uF$. 

\begin{remark}
For fixed $(a,\lambda) \in \spara$ the noise $\eta_{\snt}^\eps(a,\lambda)$ is independent of $\eps$ unless $\snt \in \FLm$. On the other hand, the expression $M^{g^\eps}\SS\tau$ coincides with the BPHZ renormalisation, if the homogeneity of the noise types $\snt \in \sFLm \setminus \FLm$ are viewed as zero (or more precisely as $-\kappa$ for some $\kappa$ small enough, compare  Lemma~\ref{lem:shift:upsilon:bound} below for a precise statement). In this sense the right-hand side of (\ref{eq:Ftau}) is just the expectation of $\PPi\tau(0)$,  where $\PPi$ denotes the BPHZ renormalised (in the sense of the previous sentence) canonical lift of $\eta^\eps(a,\lambda)$. It follows in particular that the right-hand side of (\ref{eq:Ftau}) is indeed convergent.

Note however that this expression does not vanish. This is not a contradiction to the characterisation \cite[Eq.~6.25]{BrunedHairerZambotti2016} of the BPHZ character, since with respect to the homogeneity constructed in the previous paragraph, $\SS$ does not leave homogeneity invariant. In fact, any tree $\sigma \in \SS[\tau] \backslash\{\tau\}$ is of positive homogeneity in this sense. The identity \cite[Eq.~6.25]{BrunedHairerZambotti2016} on the other hand is only guaranteed to hold for trees of negative homogeneity.
\end{remark}

Our intuition behind this definition is that $\uF_\tau$ should be small compared to $\oF_\tau$, in the sense that a statement of the form $\uF_\tau(a,\lambda) \ll \oF_\tau(a,\lambda)$ is an attainable statement. It turns out that this is not quite true, since in general there will be sub-divergencies of $\tau$ that cause $\uF_\tau$ to become dominant. However, assuming that we have good bounds on these sub-divergencies, this statement becomes attainable. More precisely, we have the following result. 
\begin{lemma}\label{lem:Fthenondominatingpart}
For any $\tau \in \FT_-$ there exists a smooth function $G_\tau:\R^{\FT_-}\times \R^{\TT_-^{\prec\tau}} \times \sparaLc \to \R$ such that
\begin{align}\label{eq:GTFT}
G_\tau(a, (F_{\tilde\tau}(a,\lambda))_{\tilde\tau\in\TT_-^{\prec\tau}}, \lambda)=\uF_\tau(a,\lambda),
\end{align}
for any $(a,\lambda) \in \sA \times \sparaLc$, and such that for any fixed $\rho > 0$ and $\beta>0$ the following bound is attainable: One has
\begin{align}\label{eq:GTbound}
|G_\tau(a,b, \lambda)|\le \beta \lambda_{\ci(\tau)}^{-\bar\kappa} 
\end{align}
uniformly over all $(a,b) \in \sA \times \sATT{\tau}$ such that $\max_{\tilde\tau\in\FT_-}|a_{\tilde\tau}| \lor  \max_{\tilde\tau\in\TT_-^{\prec\tau}}|b_{\tilde\tau}|\le \rho$.
\end{lemma}

Before we prove Lemma~\ref{lem:Fthenondominatingpart}, we show how to use this in order to finish the proof of Proposition~\ref{prop:shiftednoise}. We first argue that one can strengthen the statement of this lemma.
\begin{lemma}\label{lem:Fthenondominatingpart:2}
For any $\tau \in \FT_-$ and $\rho>0$ there exists a continuous function $\tilde G _\tau:\sA \times \sAFT{\tau} \times \sparaLc \to \R$ such that 
\begin{align}\label{eq:GTFT:2}
	\tilde G_\tau(a, (F_{\tilde\tau}(a,\lambda)) _{\tilde\tau\in\FT_-^{\prec\tau}} , \lambda)
	= \uF_\tau(a,\lambda),
\end{align}
holds for any $(a,\lambda) \in \sA \times \sparaLc$ with $\sup_{\taua \in \FT_-}|a_\taua| \lor \sup_{\taua \in \FT_-^{\prec \tau}}|F_\taua (a,\lambda) | \le \rho$, and for any $\beta>0$ the bound
\begin{align}\label{eq:GTbound:2}
|\tilde G_\tau( a , b , \lambda )|
\le 
\beta  \lambda_{ \ci( \tau ) }  ^{ -\bar\kappa} 
\end{align}
uniformly over $(a,b) \in \sA \times \sAFT{\tau}$ such that $\max_{\tilde\tau\in\FT_-}|a_{\tilde\tau}| \lor  \max_{\tilde\tau\in\FT_-^{\prec\tau}}|b_{\tilde\tau}|\le \rho$ is attainable.
\end{lemma}
\begin{proof}
Let $A_\rho$ denote the set of $a \in \sA$ such that $\sup_{\taua \in \FT_-}|a_\taua| \lor \sup_{\taua \in \FT_-^{\prec \tau}}|F_\taua (a,\lambda) | \le \rho$. 
We first argue that there exists $R>0$ such that one has $|F_{\taua}(a,\lambda)| \le R$ for any $\taua \in \TT_-^{\prec\tau}$ and any $(a, \lambda) \in A_\rho \times \sparaLc$. Once this is shown, it is not hard to see that the function
\[
\tilde G_\tau( a, b, \lambda)
:=
G_\tau( a, (b_\taub)_{\taub \in \FT_-^{\prec\tau}} \sqcup (F_\taub(a, \lambda) \land R )_{\taub \in \TT_-^{\prec\tau} \backslash \FT_-^{\prec\tau}}, \lambda)
\]
has all the properties we were looking for.

We denote by $\sg=\sg(a,\lambda) \in \sCGm$ the BPHZ-character for the noise $\eta^{\eps}(a,\lambda)$ and we define $\sh = \sh(a,\lambda) \in \sCGm$ via the identity $\sh\circ \sg = g^\eps$, so that one has
\[
F_\taua(a,\lambda) = \lim_{\eps \to 0}
\Upsilon^{\eta^{\eps}} M^{\sg} M^{\sh} \SS\taua
\]
for any $\taua \in \TT_-$, where we suppress the dependence on $(a,\lambda)$ in the notation on the right-hand side.
By Lemma~\ref{lem:cumulantexpansionB} the noise $\eta^{\eps}(a,\lambda)$ is uniformly bounded with respect to $\| \cdot \|_{\fs - \bar\kappa}$ over $(a,\lambda) \in A_\rho \times \sparaLc$ and $\eps>0$, 
and it follows from \cite{ChandraHairer2016} that 
\[
|\Upsilon^{\eta^{\eps}} M^{\sg} \sigma|
\lesssim
1
\]
for any $\sigma \in \sCT$ uniformly over $(a,\lambda) \in A_\rho \times \sparaLc$ and $\eps>0$.
We recall at this point (Lemma~\ref{lem:shift:operator:co:product}) that $\SS$ commutes with the coproduct, so that it remains to show that $\sh\SS\taua$ is bounded for any $\taua \in \CT_-$ uniformly over $(a,\lambda) \in A_\rho \times \sparaLc$ and $\eps>0$. 
We denote by $h^{\eps,\delta} \in \CG_-$ the character from Proposition~\ref{prop:sequencetoconstant} for the shift $\zeta_\delta := \SS^*\eta^\eps(a,\lambda)-\xi^\eps$, 
and we claim that one has $\sh\SS = h^{\eps,\delta}$ on $\CT_-$.
Indeed, one has
\begin{align*}
\CM( \sg \SS \otimes \Upsilon^{\xi^\eps+\zeta_\delta}) \cpmi
=
\CM( \sg \otimes \Upsilon^{\eta^{\eps}}) \cpmi \SS
=
\one^\star \SS = \one^\star
\end{align*}
on $\CTm$, and since this relation characterizes the BPHZ character, one has $\sg\SS = g^{\eps,\delta}$. It remains to argue that
\[
\sh\SS \circ g^{\eps,\delta} = \sh\SS \circ \sg\SS = (\sh \circ \sg) \SS = g^\eps \SS = g^\eps
\]
on $\CTm$.
We can now argue inductively with respect to $\prec$ in the same way as in the proof of (\ref{eq:convergenceh}). 
The only difference is that in (\ref{eq:convergenceh}) we showed convergence based on the assumption that $F_{\tilde\tau} (a,\lambda)$ converges for $\tilde\tau \in \FT_-^{\prec\tau}$, now we show boundedness based on the assumption that these quantities are bounded.
\end{proof}

With this we can finish the proof of Proposition~\ref{prop:shiftednoise}. We recall at this point that we fixed $h \in \fxi \circ \CH$ at the beginning of Section~\ref{sec:renormalisationgroupargument}, see also (\ref{eq:shiftedtreeexpectationA}) in Proposition~\ref{prop:shiftednoise}.

\begin{proposition}\label{prop:choice_of_a}
For any $\delta>0$ the following is an attainable statement: there exists a family of constants $a\in\R^{\FT_-}$ such that $\sup_{\tau \in \FT_-} |a_\tau|\le \delta$ and such that
\begin{align}\label{eq:fpp}
		F_\tau(a,\lambda)=h(\tau)
\end{align}
for any $\tau\in\FT_-$. 

In particular, one can choose a sequence of scales $\lambda^\delta \in \sparaLc$ and a sequence of constants $a^\delta \in \sA$, both converging to $0$ as $\delta \to 0$, such that the statement of Proposition~\ref{prop:shiftednoise} holds for the corresponding shift $\srn_\delta = \SS^* \eta^{\eps}(a,\lambda)-\xi^\eps$.
\end{proposition}

\begin{proof}
The key step is to find a solution $a \in \sA$ with $\max_{\tau \in \FT_- } |a_\tau|<\delta$ to the system of equations
\begin{align}\label{eq:barFTplusulG}
\oF_\tau(a,\lambda)+
\tilde G_\tau  
\big(
	a, (h(\taua))_{\taua \in \FT_-^{\prec\tau}} , \lambda
\big)
=
h(\tau) \,, 
\qquad\quad{\tau \in \FT_-},
\end{align}
where $\tilde G_\tau$ is as in Lemma~\ref{lem:Fthenondominatingpart:2} for $\rho:= 2 \max_{\tau \in \FT_-} h(\tau) \land 2\delta$. We then argue inductively: Fixing $\tau \in \FT_-$ and assuming that $F_\taua(a, \lambda) = h(\taua)$ for any $\taua\in \FT_-^{\prec\tau}$, then the assumptions of (\ref{eq:GTFT:2}) are met for $\tau$, so that the left-hand side of (\ref{eq:barFTplusulG}) is equal to $F_\tau(a, \lambda)$.

We rephrase (\ref{eq:barFTplusulG}) slightly into a fixed-point problem. Define the function $g : \sA \times \sparaLc \to \sA$ by
\begin{align}\label{eq:fpp:g}
g_\tau( a , \lambda ) :=
- \lambda_{\ci(\tau)} ^{\bar\kappa} 
\Big (
	\oF_\tau(a,\lambda)+
	\tilde G_\tau \big(a, (h(\tilde\tau)) _{\tilde\tau \in \FT_-^{\prec\tau}} , \lambda
	\big)
	-h(\tau) 
\Big)
+a_\tau
\end{align}
for any $\tau\in\FT_-$.
Then $a \in \sA$  with $\max_{\tau \in \FT_- } |a_\tau|<\delta$ is a fixed point of $g(\cdot, \lambda)$ if and only if $(a,\lambda)$ is a solution to (\ref{eq:fpp}). 
It follows from Lemma~\ref{lem:dominatingpart} that $-\lambda_{\ci(\tau)}^{\bar\kappa}  \oF_\tau (a,\lambda) + a_\tau \to 0$ as $\lambda\to 0$ uniformly over $a \in \sA$ as above, and from Lemma~\ref{lem:Fthenondominatingpart:2} that for any $\beta>0$ the bound 
\[
\left|
\lambda_{\ci(\tau)}^{\bar\kappa} \tilde G_\tau 
	\big( a, (h(\tilde\tau)) _{\tilde\tau\in\FT_-^{\prec\tau}} , \lambda \big) 
\right|
\le \beta 
\]
for any $a \in \sA$ as above is an attainable statement.

It follows that the statement
\[
\max_{\tau\in\FT_-} |a_\tau|\le \delta
\,\,\text{ implies }\,\,
\max_{\tau\in\FT_-} |g_\tau(a,\lambda)| \le \delta
\]
is attainable, and by Schauder's fixed point theorem there exists a solution $a$ to (\ref{eq:fpp:g}) in the $\delta$-neighborhood of the origin.
\end{proof}

It remains to show Lemma~\ref{lem:Fthenondominatingpart}.
We fix $\skappa > 0$ with the property that $\bar\kappa> \skappa$ and such that $|\tau|_{\sfs}>-\shalf+\shift\kappa$ for any tree $\tau\in\shift\CT$ with $\#K(\tau)\ge 1$. 
For fixed $\cif\in\CI$ we introduce a homogeneity assignment $\sfsi$ on $\sFLm$ which treats noises regularised on scales larger than $\lambda_\cif$ as smooth. More precisely, we set $\sfsi(\Xi) := \sfs(\Xi)$ and $\sfsi(\tilde\Xi) := -\skappa$ for any $\Xi \in \FL_-$, and we define
\begin{align*}
\sfsi(\Xi,\tau):=
\bigg\{
\begin{split}
	\sfs(\Xi,\tau)	\qquad &\text{ if } \ci(\tau) \le \cif \\
	-\skappa 		\qquad &\text{ if } \ci(\tau)>\cif
\end{split}
\end{align*}
for any noise type $(\Xi,\tau) \in \sFLm$. Here we use the total order $\le$ on $\CI$ introduced in Section~\ref{sec:monotonestatements}.
\begin{lemma}\label{lem:cumulantbound}
Let $\ci\in\CI$ and let $\ci^\uparrow:=\min\{\cj\in\CI^\star : \cj>\ci \}$.
For any $A>0$ there exists a constant $C_{N}(\lambda_{\ci^\uparrow})>0$, such that for any $N\in\N$ the bound
\[
\|\eta^{\eps}(a,\lambda)\|_{\sfsi}
\le C_{N}( \lambda_{\ci^\uparrow})
\]
holds uniformly over all families $\lambda \in \sparaLc$ and $a \in \sA$ with $\max_{\tau \in \FT_-} |a_\tau|< A$.
\end{lemma}
\begin{proof}
This follows directly from Lemma~\ref{lem:cumulantexpansionB}.
\end{proof}
We denote by $\sCTi \ssq \sCTm$ the unital subalgebra generated by trees of negative $|\cdot|_{\sfsin}$-homogeneity. Note that $\skappa$ was chosen small enough so that for any tree $\tau \in \sCTm$ one has $\tau \in \sCTi$ if and only if for any noise type edge $ e \in L(\tau)$ on has either $\ft(e) \in \FL_-$ or $\ft(e) = (\Xi,\tau)$ with $\ci(\tau)\le \ci_0$. We denote by $\spi$ the multiplicative projection of $\sCTm$ onto $\sCTi$, and we define $\sgi := \sg \spi$ (we usually suppress the dependence of $(a,\lambda)$ in this notation). 

It follows that $\sgi$ restricted to $\sCTi$ is just the BPHZ-character for the homogeneity assignment $\sfsin$ and the evaluation $\eta^{\eps}$.
Applying the results of \cite{ChandraHairer2016} to the homogeneity assignment $\sfsin$ 
we obtain the following estimate.

\begin{lemma}\label{lem:shift:upsilon:bound}
For any $\ci\in\CI$ there exists a constant $C_{N}(\lambda_{\ci^\uparrow})>0$ such that for any $\tau\in \sCTmhat$ the bound
\[
|\Upsilon ^{\eta^{\eps}} M^{ \sgi }\tau |
\le 
C_{N} (\lambda_{\ci^\uparrow})
\]
holds uniformly over all families $\lambda \in \sparaLc$ and $a \in \sA$ with $\max_{\tau \in \FT_-} |a_\tau|< A$.

Moreover, one has
\[
\cg_i(a, \lambda)\tau := \lim_{\eps \to 0}\Upsilon ^{\eta^{\eps}} M^{ \sgi }\tau 
\]
exists and is a continuous function in $(a,\lambda)$.
\end{lemma}
\begin{proof}
Both statements follow from Lemma~\ref{lem:cumulantbound} and \cite[Thm.~2.31]{ChandraHairer2016}. 
\end{proof}
We are finally in the position to prove Lemma~\ref{lem:Fthenondominatingpart}.
\begin{proof}[of Lemma~\ref{lem:Fthenondominatingpart}]
We fix from now on a tree $\tau\in\FT_-$ and assume that the statement of Lemma~\ref{lem:Fthenondominatingpart} holds for any $\taua \in \FT_-^{\prec\tau}$. We set $\cj:=\ci(\tau)$ and we define a character $\skj \in \sCGm$ by
\begin{align}\label{eq:kj}
\skj \circ \sgj = g^{\eps}.
\end{align}
It then follows that for any $\taua \in \TT_-^{\prec\tau}$ one has
\begin{align*}
\Upsilon^{\eta^{\eps}} M^{ g^\eps } \SS \taua
&= (\skj \otimes \Upsilon^{\eta^{\eps}} M^{\sgj}) \cpmi \SS \taua \\
&= (\skj \SS\otimes 
\Upsilon^{\eta^{\eps}} M^{\sgj} \SS )
	( 
		\cpmi - \Id \otimes \one
	)
\taua+ \skj \SS \taua.
\end{align*}
It follows from this identity and the definition of the coproduct that there exists a fixed polynomial $P_{\tilde\tau}$ in $\TT_-^{\preceq\taua} \times \TT$ variables such that
\begin{align}\label{eq:skj}
\skj \SS\taua
=
\Upsilon^{\eta^{\eps}} M^{ g^\eps} \SS \taua
+
P_{\tilde\tau}\big( \skj \SS \taub \, , \, \Upsilon^{\eta^\eps} M^{\sgj} \SS \tauc
\,:\, \taub \in \TT_-^{\preceq\taua} \, , \, \tauc \in \TT \big).
\end{align}
We prove inductively in $\prec$ that for any $\taua \in \TT_-^{\prec\tau}$ there exists a continuous function 
\begin{align}\label{eq:tildek}
\cf(\cdot,\cdot,\cdot) \taua : \sA \times \R^{\TT_-^{\prec\tau}} \times \sparaLc \to \R 
\end{align}
such that
$
\cf(a, (F_{\hat\tau}(a, \lambda))_{\hat \tau \in \TT_-^{\prec\tau}} ,\lambda) \taua
$
is equal to (\ref{eq:skj}) in the limit $\eps \to 0$ for any $(a,\lambda) \in \sA \times \sparaLc$, and such that
for any $C>0$ the estimate
\[
|\cf{}(a,b,\lambda) \taua| \lesssim C_{N}(\lambda_{\cj^\uparrow})
\]
holds uniformly over all $(a,b) \in \sA \times \sATT{\tau}$ such that $\sup_{\taub \in \FT_-}|a_{\taub}| \lor \sup_{\taub \in \TT_-^{\prec\tau}}|b_{\taub}|<C$ and $\eps>0$ small enough (where ``small enough'' may depend on $(a,b)$). 
We can write
\begin{align*}
\lim_{\eps \to 0}
\skj \SS\taua
&=
F_{\taua}(a,\lambda)
+
P_{\tilde\tau} \big( \cf(a, (F_{\hat\tau}(a, \lambda))_{\hat \tau \in \TT_-^{\prec\tau}} ,\lambda) \taub 
\, , \, 
\cg_j (a,\lambda) \tauc
\,:\, \taub \in \TT_-^{\preceq\taua} \, , \, \tauc \in \TT \big)
\\
&=:
\cf(a, (F_{\hat\tau}(a, \lambda))_{\hat \tau \in \TT_-^{\prec\tau}} ,\lambda) \taua,
\end{align*}
which is bounded in the required way by Lemma~\ref{lem:shift:upsilon:bound} and the induction hypothesis.
%
In order to continue, we now make the claim that for any tree $\tilde\tau \in \sCTm$ one has $\skj \taua = 0$ if at least one of the following three properties is satisfied.
\begin{enumerate}
\item One has $\taua\in\CT_-$.
\item For any partition $\pi\in\parti{L(\taua)}$ one has
\begin{align}\label{eq:cumulant:vanish}
\prod_{P\in\pi}\E^c[(\eta^{\eps}_{t(e)}(x_e))_{e\in P}]= 0
\end{align}
for any $(x_e)_{e\in L(\taua)}\in(\bar\domain)^{L(\taua)}$, where $\ft:L(\taua)\to\sFLm$ denotes the type map of $\taua$.
\item One has $|\taua|_{\sfsj}>0$. 
\end{enumerate}
The first claim follows from the fact that $\CT_-$ is a Hopf subalgebra of $\sCTm$ and on $\CT_-$ the characters $g^\eps$ and $\sgj$ agree. For the second claim 
we denote by $\CI \ssq \sCTm$ the ideal in $\sCTm$ generated by all trees $\taua \in \sCTm$ with the property that (\ref{eq:cumulant:vanish}) holds. Then $\CI$ forms a Hopf ideal, and by definition of the BPHZ character it follows that both $g^\eps$ and $\sgj$ vanish on $\CI$. The claim now follows from the fact that the annihilator of any Hopf ideal forms a subgroup of $\sCGm$. 
The third claim follows similarly, noting that $\cpm$ preserves the $|\cdot|_{\sfsj}$-homogeneity, in the sense that $|\cpm\taua|_{\sfsj} = |\taua|_{\sfsj}$, 
where on the left-hand side we add up the homogeneities of products of trees. It follows that the ideal $\CI_\cj^+$ in $\sCTm$ generated by trees $\taua \in \sCTm$ such that $|\taua|_{\sfsj}>0$ is a Hopf ideal. Since moreover one has that $\sgj$ and $g^\eps$ vanish on $\CI_\ci^+$, the claim follows again from the fact that annihilators of Hopf ideals are subgroups.

As a corollary, we obtain the identity
\begin{align}\label{eq:cor1}
(\skj \otimes \Upsilon^{\eta^{\eps}} M^{\sgj}) ( \cpmi - \Id \otimes \one )
(\oSS+\Id)\tau=0.
\end{align}
Indeed, let $\taua\in \oSS [\tau]$. Then there exists a unique noise type edge $e\in L(\taua)$ such that $\ft(e)\notin\FL_-$, and for this edge $\eta^\eps_{\ft(e)}$ is an element of the $m(\tau)$-th homogeneous Wiener chaos. It follows that whenever $\taub\ne\one$ is a proper subtree of $\taua$, so that in particular $L(\taub) \subsetneq L(\taua)$, then either the first or the second point above are satisfied. 
The claim then follows, since on the one hand $\Upsilon^{\eta^{\eps}} M^{\sgj} X^k=0$ for any $k \in \N^{d}\backslash\{0\}$, and on the other hand $|\taua|_{\sfsj}<0$ implies that $\Upsilon^{\eta^{\eps}} M^{\sgj}\taua=0$. 
Here we use that $\sfsj$ was chosen in such a way that $|\taua|_{\sfsj} \le \bar\kappa < 0$ for any $\taua\in \oSS [\tau]$.

As a second corollary, we get that if $\tau$ is such that $\fancynorm{\tau}_\scale<0$, then
\begin{align}\label{eq:cor2}
\skj \uSS\tau=0.
\end{align}
To see this, let $\taua\in\uSS[\tau]$. It is clear that whenever there exists $e \in L(\taua)$ such that either $\ft(e)=\tilde\Xi$ for some $\Xi \in \FL_-$, or $\ft(e) = (\Xi,\taub)$ with $\ci(\taub) > \cj$, then one has $|\taua|_{\sfsj} >0$, and by the third point above it follows that $\skj\taua=0$. Thus, we can assume that $\ft(e) \in \FL_-$ or $\ft(e) = (\Xi,\taub)$ with $\ci(\taub) \le \cj$ for all $e \in L(\tau)$. 

Assume now that in addition there are two distinct $e,f\in L(\taua)$ such that $\ft(e),\ft(f)\notin\FL_-$, say $\ft(e)=(\Xib,\taub)$ and $\ft(f)=(\Xic,\tauc)$.
Then the assumption $\ci(\taub) \lor \ci(\tauc) \le \cj$ implies that $\fancynorm{\taub}_\fs\lor\fancynorm{\tauc}_\fs \le \fancynorm\tau_\fs$.  Upon choosing $\bar\kappa$ small enough, one has 
\begin{align}\label{eq:cor2sigma}
\begin{split}
|\taua|_{\sfsj}
&\ge |\tau|_{\fs}-(\fancynorm{\taub}_{\fs}+\bar\kappa)
-(\fancynorm{\tauc}_{\fs}+\bar\kappa) \\
&\ge
|\tau|_{\fs}-2(\fancynorm{\tau}_{\fs}+\bar\kappa) >0.
\end{split}
\end{align}

In the remaining case there exists a unique $e\in L(\taua)$ such that $\ft(e)\notin\FL_-$, say $\ft(e)=(\Xib,\taub)$. We distinguish the case $\fancynorm{\taub}_\fs =\fancynorm{\tau}_\fs$ and $\fancynorm{\taub}_\fs <\fancynorm{\tau}_\fs$ (by the discussion above we have $\ci(\taub) \le \cj$, so that the case $\fancynorm{\taub}_\fs>\fancynorm{\tau}_\fs$ is ruled out). In the first case, we have by definition of $\le$ and $\uSS$ that $\#L(\taub)>\#L(\tau)$, so that $\eta^{\eps}_{\ft(e)}$ takes values in a homogeneous Wiener chaos of order strictly greater than $m(\tau)$, so that the second point above applies. In the second case it follows similarly to before that 
\[
|\taua|_{\sfsj}\ge |\tau|_{\fs}-\fancynorm{\taub}_{\fs}-\bar\kappa>0,
\] 
so that the third point above applies, and this finishes the proof of the claim.

We now conclude that in case $\fancynorm\tau_\scale<0$, it follows from (\ref{eq:cor1}) and (\ref{eq:cor2}) that
\begin{align*}
\uF_\tau(a,\lambda) 
&= 
\lim_{\eps \to 0}
(\skj \otimes 
\Upsilon ^{\eta^{\eps}} M^{\sgj} )
(\cpmi - \Id \otimes \one) \uSS\tau
+ \skj \uSS\tau \\
&= 
\lim_{\eps \to 0}
(\skj \SS\otimes
\Upsilon^{\eta^{\eps}} M^{\sgj}\SS )
(\cpmi - \Id \otimes \one) \tau,
\end{align*}
and using the definition of $\cf$ in (\ref{eq:tildek}) together with Lemma~\ref{lem:shift:upsilon:bound} this can be re-written as
\[
\uF_\tau(a,\lambda) 
=
(\cf(a,(F_\taua)_{\tau \in \TT_-^{\prec\tau}}(a),\lambda) \otimes \cg(a,\lambda))
(\cpmi - \Id \otimes \one) \tau
\] 
so that
\[
G_\tau(a, b, \lambda)
:=
(\cf(a,b,\lambda) \otimes \cg(a,\lambda))
(\cpmi - \Id \otimes \one) \tau
\]
has the desired form (\ref{eq:GTFT}).

It remains to treat the case $\fancynorm\tau_\fs=0$, where the estimate (\ref{eq:cor2sigma}) fails in general. However, we will show that a slightly weaker statement than (\ref{eq:cor2}) still holds, namely there exists a constant $C_{N}(\lambda_{\cj^\uparrow})$ such that 
\begin{align}\label{eq:kboundzerohomo}
|\skj \uSS \tau|\le C_{N}(\lambda_{\cj^\uparrow}).
\end{align}
Proceeding identically to above, this suffices to finish the proof of Lemma~\ref{lem:Fthenondominatingpart}.

To show (\ref{eq:kboundzerohomo}) we recall that by Assumption~\ref{ass:zero-homo} 
one has for any $\taua \in \CV$ and any $\taub \in \SS[\taua]$ that $g^\eps\taua = \sgj\taub = 0$. 
It follows from this, (\ref{eq:kj}) and the fact that the unital algebra generated by $\bigcup_{\tau \in \CV} \SS[\tau]$ is a Hopf subalgebra of $\sCTm$ that one also has $\skj\taub=0$ for any $\taub \in \SS[\taua]$ and any $\taua \in \CV$. In particular 
\[
0
=
g^\eps \uSS\tau
=
(\skj \otimes \sgj)
	\cpmh \uSS\tau = 
\skj\uSS\tau
+\sgj\uSS\tau.
\]
The estimate (\ref{eq:kboundzerohomo}) now follows from (\ref{eq:BPHZboundzerohomo}) below, using the fact that one has the identity $\sgj \taub = \sg \taub$ for any tree $\taub \in \sCTm$ such that $\sgj \taub \ne 0$.
\end{proof}
\begin{remark}
The last step of the previous proof, relating $\skj \uSS \tau$ and $\sgj \uSS \tau$, does not need Assumption~\ref{ass:zero-homo}, although the argument is greatly simplified. The assumption is however needed in the proof of (\ref{eq:BPHZboundzerohomo}) below. 
\end{remark}

\begin{lemma}\label{lem:blowuprenormconstantzerohomo}
Let $\tau\in\TT_-$ satisfy $\fancynorm\tau_\fs=0$, and let $\taua\in \uSS[\tau]$ be such that $\fancynorm\taua_{\sfs}\le 0$. Then for any noise type edge $e \in L(\taua)$ one has $\ft(e)=\Xi$ or $\ft(e)=\tilde\Xi$ for some $\Xi \in \FL_-$, or $\ft(e) = (\Xi,\taub)$ with $\taub\in\FT_-$ such that $\fancynorm{\taub}_\fs=0$.

Moreover, setting $\cj:= \ci(\tau)$, for any $\rho>0$ the bound 
\begin{align}\label{eq:BPHZboundzerohomo}
|\sg\taua|\le C(\lambda_{\cj^\uparrow}),
\end{align}
uniformly over $a \in \sA$ with $\max_{\tau \in \FT_-}|a_\tau|<\rho$ is attainable.
\end{lemma}
\begin{proof}
The first statement follows directly from the definition.
For (\ref{eq:BPHZboundzerohomo}) we distinguish three cases. 

\smallskip\noindent\textbf{First case.}~There exists $e \in L(\taua)$ with $\lambda(\ft(e)) > \lambda_\cj$; that is, one has either $\ft(e) = \tilde \Xi \in \tilde\FL_-$ or $\ft(e)= (\Xi,\taub)$ with $\fancynorm{\taub}_\fs=0$ and $\# L(\taub)< \# L(\taua)$. 
In this case consider the homogeneity assignment $\sfsa$ given by 
\[
\sfsa(\ft) := \sfs(\ft) + \theta \I_{\ft = \ft(e)},
\]
for any $\ft \in \sFLm$ where $\theta:= - |\taua|_{\sfs} + \kappa$. Then, one has $| \taua |_{\sfsa}>0$. 
Let $\sga$ be the BPHZ character for this homogeneity assignment and the noise $\eta^{\eps}$. 
From Assumption~\ref{ass:zero-homo} it follows that
\begin{multline}\label{eq:zero:homo}
\sg(\taua) 
= 
(\sg \otimes \Upsilon^{\eta^\eps}) (\cpmi - \Id \otimes \one) \taua
=
(\sga \otimes  \Upsilon^{\eta^\eps}) \cpmi  \taua
=
\\
 \Upsilon^{\eta^\eps} M^{\sga}\taua  
\lesssim 
\lambda(\ft(e))^{\fancynorm{\taua}_{\sfsa}} 
\le 
C(\lambda_{\cj^\uparrow}).
\end{multline}
Here we used that by construction $\sga\taua =0$. We also used that whenever $\taub^\fn_\fe \ssq \taua$ is a proper subtree such that for some polynomial decoration $\tilde\fn$ one has $\fancynorm{\taub^{\tilde\fn}_\fe}_\sfs = 0$ and $\ft(e) \in \ft(L(\taub))$, then $|\taub^{\tilde\fn}_\fe|_{\sfss}>0$, so that $\sga\taub=0$. On the other hand, by Assumption~\ref{ass:technical} one also has $\sg \taub^{\tilde\fn}_\fe=0$.

\smallskip\noindent\textbf{Second case.}~There exists a unique noise type edge $e \in L(\taua)$ with $\fe(e) \notin \FL_-$. We only need to consider the case that $\ft(e)$ is of the form $(\Xi,\taub)$ for some $\taub \in \TT_-$ with $\fancynorm{\taub}_\fs =0$ and $\# L(\taub) > \# L(\taua)$; otherwise either the first case above applies, or $\taua \in \oSS[\tau]$. In this case however we recall that $\eta^\eps_{(\Xi,\taub)}$ takes values in the $(\#L(\taub)-1)$-th Wiener chaos, and since $\# L(\taub)-1 \ge \# L(\taua)$ and all other noise type edges (there are only $\#L(\taua)-1$ such edges) carry Gaussian noises, there is no non-vanishing cumulant, and one has $\sg(\taua)=0$.

\smallskip\noindent\textbf{Third case.}~In the final case there exist $r\ge 2$ distinct noise type edges $e_1, \ldots , e_r \in L(\taua)$ with $\ft(e_i)=(\bar\Xi_i,\taub_i)$, and one has $\fancynorm{\taub_i}_\fs =0$ and $\#L (\taub_i) \ge \# L(\taua)$.
At this point we recall the definition (\ref{eq:K0}) of $K_{(\bar\Xi_i,\taub_i)}$ from which it follows that these two kernels are sums over $N_i:=N_{\taub_i}$ kernels respectively, and we write
\[
K_{(\bar\Xi_i,\taub_i)} = \frac{a_{\taub_i}}{N_i}
\lambda_{\tau_i}^{- \bar\kappa } 
\sum_{m=0}^{N_i - 1}K^m_{i}
\]
with (recall that $\alpha_{(\bar\Xi_i,\taub_i)} = -\#L(\tau_i) \shalf -\bar\kappa$)
\[
K_i^m :=
\lambda^{-\bar\kappa m}
\rescale
\big(
	{2^{-m}\lambda_{\tau_i}}, \#L(\tau_i) \shalf
\big) \Phi_{ (\bar\Xi_i,\taub_i) }.
\]

In order to simplify the argument below, we assume that the noise types $(\bar\Xi_i,\taub_i)$ are all different. (If this is not the case, extend the regularity structure at this point by introducing sufficiently many distinct copies of the noise types $(\bar\Xi_i, \taub_i)$, and extend $\eta^\eps$ such that it acts identically on each copy of any given noise type. The argument below can then be applied to the extended regularity structure and the extended set of noise types.) 

Given $n = (n_1, \ldots n_r)$ with $n_i \in \{ 0, \ldots, N_i-1 \}$, we write $\eta^\eps_{n}$ for the noise defined by (\ref{eq:eta}) but with $ K_{(\bar\Xi_i,\taub_i)}$ replaced by $ K^{n_i}_i$ and we write $\sgmn$ for the BPHZ character for the noise $\eta^\eps_{n}$. It follows that with $a = a_{\taub_1} \cdots a_{\taub_n}$ one has
\begin{align}\label{eq:g:n:eps}
\sg(\taua) =
a\frac{\lambda_{\tau_1}^{-\bar \kappa} \ldots \lambda_{\tau_r}^{-\bar \kappa} }{N_1 \ldots N_r}
 \sum_{n_1=0}^{N_1 - 1} \cdots \sum_{n_r = 0}^{N_r - 1} 
\sgmn(\taua).
\end{align}
By Corollary~\ref{cor:log-void} there exists $\theta>0$ such that
\begin{align}\label{eq:g:n:eps:bound}
|\sgmn(\taua)| \lesssim 
\Big(
	\frac{ \min_i 2^{-n_i} \lambda_{\tau_i} }{ \max_i 2^{-n_i} \lambda_{\tau_i} }
\Big)^{\theta}
\end{align}
We can assume that $\lambda_1 \le \lambda_2 \le \cdots \le \lambda_r$ and hence also $N_1 \ge N_2 \ge \cdots \ge N_r$.
At this point we recall that the scales are ``well separated'', and in particular there is no loss of generality to assume that whenever $\lambda_{\tau_i} > \lambda_{\tau_j}$ one also has $2^{-N_i} \lambda_{\tau_i} > \lambda_{\tau_j}$. Then one has the bound
\begin{align}\label{eq:g:eps:taua:bound}
|\sg(\taua)| \lesssim
\frac{\lambda_1 ^{-r \bar \kappa}  }{N_1 \ldots N_r}
\sum_{n_1=0}^{N_1 - 1} \cdots \sum_{n_r = 0}^{N_r - 1} 
	2^{-\theta(\max_i n_i - \min_i n_i) }.
\end{align}
Up to a combinatorial factor we can restrict the sum to the regime $n_1 \ge \ldots \ge n_r$. Changing variables in the sum so that $k=(\max_i n_i) - (\min_i n_i) = n_1 - n_r$, we obtain the bound
\[
\sum_{n_1=0}^{N_1 - 1} \cdots \sum_{n_r = 0}^{N_r - 1} 
	2^{-\theta(\max_i n_i - \min_i n_i) }
\lesssim
\sum_k
\sum_{n_2 =0}^{N_2 - 1} \cdots \sum_{n_r = 0}^{N_r - 1} 
	2^{-\theta k}
\lesssim
N_2 \ldots N_r.
\]
With (\ref{eq:g:eps:taua:bound}) we obtain
\[
|\sg(\taua)| \lesssim N_1^{-1} \lambda_1^{-r \bar \kappa},
\]
and recalling that $N_1 \sim \lambda_1^{-1}$ it remains to choose $\bar\kappa$ small enough so that $\bar r \bar\kappa<1$, where $\bar r$ denotes the maximal number of noise type edges appearing in any tree $\tau \in \TT_-$.
\end{proof}

\appendix

\section{Technical proofs and notations}

\subsection{The coproduct via forests}
\label{sec:i-forests}

In some of the arguments we are going to perform a construction for which it will be important to first derive certain identities involving the co-product. The proofs of some of these statements are relatively straightforward but turn out to be a bit fiddly, and these arguments are going to get more clear using some notation that we introduce in this section.

\begin{definition}
For a rooted, typed tree $(T,\ft)$ way say that $\CF$ is a sub-forest of $T$ if $\CF$ is a subgraph of $\tau$ without isolated vertices. We call $\CF$ a subtree if $\CF$ is non-empty and connected. We write $T/\CF$ for the rooted, typed tree obtained by contracting any connected component of $T$ to a single vertex. 
\end{definition}

We write $\bar\CF$ for the set of the connected components of $\CF$. In a natural way one can view any connected component of $S \in \bar\CF$ again as a rooted, typed tree $(S,\ft)$ where the type $\ft$ is simply taken over from $T$. The set of vertices $V(T/\CF)$ can now be naturally identified with the set $(V(T)\backslash V(\CF))\sqcup \{ u_A : A \in \bar \CF \}$, where $u_A \in V(T / \CF)$ denotes the vertex obtained by contracting the subtree $A$ of $T$. It follows that there exists a map $\varphi_T^\CF : V(T) \to V(\CF/T)$ defined by 
\begin{align}\label{eq:iforests:vertex:map}
\varphi_T^\CF (u)
:=
\begin{cases}
	u_A		&\qquad \text{ if } u \in V(A) \text{ with } A \in \bar\CF \\
	u		&\qquad \text{ if } u \in V(T) \backslash V(\CF).
\end{cases}
\end{align}
Given a tree $\tau \in {\CT^\ex}$ then $\tau$ is of the form $\tau=T^{\fn,\fo}_\fe$ for some rooted, typed tree $T$, and we say that $\CF$ is a sub-forest of $\tau$ if $\CF$ is a sub-forest of $T$. 

Given a tree $\tau=T^{\fn,\fo}_\fe \in {\CT^\ex}$ we write $\div(\tau)$ for the set of sub-forests $\CF$ of $\tau$ with the property that for any $S \in \bar\CF$ one has $|S^0_\fe|_\fs<0$. We write $\partial_\CF E(\tau) \subseteq E(\tau)$ for the set of $e \in E(\tau)$ with the property that $e \notin E(\CF)$ but $e^\downarrow \in N(\CF)$. For a map $\bar \ce : E(\tau) \to \Z^d \oplus \Z(\FL)$ we write $\pi \bar\ce : N(\tau) \to \Z^d \oplus \Z(\FL)$ for the map defined by 
\[
\pi \bar \ce (u) := \sum_{e \in E(\tau), e^\downarrow = u} \bar\ce(e).
\]
Finally, if $\fm : N(\tau) \to \Z^d \oplus \Z(\FL)$ the we define $\fm/\CF: N(\tau/\CF) \to \Z^d \oplus \Z(\FL)$ by setting $(\fm/\CF)(u):= \fm(u)$ of $u \in N(\tau)\backslash N(\CF)$ and
\[
(\fm/\CF)(u) := \sum_{v \in N(S)} \fm(v)
\]
if $u \in N(\tau/\CF)$ was generated by contracting the subtree $S \in \bar\CF$.
With this notation we have the following formula for the coproduct $\cpm : \CT^\ex \to \CT_- \otimes \CT^\ex$ from \cite[Def.~3.3,Def.~3.18]{BrunedHairerZambotti2016}:
\begin{align}\label{eq:coproduct:ex}
\cpm \tau = 
\sum_{\CF \in \div(\tau)}
\sum_{\fn_\CF, \ce_\CF}
\frac{1}{\ce_\CF !}
\binom{\fn}{\fn_\CF}
\prod_{S \in \CF} S^{\fn_\CF + \pi\ce_\CF}_\fe
\otimes
(T/\CF)^{\fn-\fn_\CF,  [\fo]_\CF}
_{\fe + \fe_\CF},
\end{align}
where $[\fo]_\CF:=(\fo +\fn_\CF + \pi(\ce_\CF - \fe|_\CF + \ft|_{E(\CF)})/\CF$. Here we use the same convention as in \cite[Def.~3.3]{BrunedHairerZambotti2016}: Given a typed tree $\tau$ and a subforest $\CF$ of $\tau$, then the notations $\fn_\CF$ and $\ce_\CF$ always denote decorations $\fn_\CF : N(\tau) \to \N^d$ and $\ce_\CF : E(\tau) \to \N^d$ with the property that 
\begin{align}\label{eq:iforest:decoration:convention}
\supp \fn_\CF \subseteq N(\CF)
\qquad\text{ and }\qquad
\supp \fe_\CF \subseteq \partial_\CF E(\tau).
\end{align}
We furthermore use the convention that a sum over $\fn_\CF$ and $\fe_\CF$ ranges over all decorations satisfying (\ref{eq:iforest:decoration:convention}).

\subsection{Some technical proofs}
\label{sec:technical:proofs}

\begin{proof}[of Lemma~\ref{lem:CT:pl}]
Note first that for an undecorated tree $S$ one has that $S^\fn_\fe \in \wl\CI$ holds either for all choices of decorations $\fn,\fe$ or for no choice of decoration, and for the purpose of this proof we write $S \in\wl{\CI}^\star$ in the first case.

In order to see that (\ref{eq:coproduct:pl}) holds, we only need to show that for any forest $\CF \in \div(\tau)$ which is not of the form $\forestlegs\CG$ for some $\CG \in \div(\pi\tau)$ one has
\begin{align}\label{eq:forest:vanishes}
\prod_{S \in \CF} S^{\fn_\CF + \pi\ce_{\CF}}_\fe
\otimes
(T/\CF)^{\fn-\fn_{\CF},  [\fo]_{\CF}}
_{\fe + \fe_{\CF}}
\in
\wl\CI \otimes \wCThat + \wCT \otimes \wl{\hat\CI}.
\end{align}
If there exists $S \in \bar\CF$ such that $E(S) = L_\Legtype(S)$ (i.e.\ such that $S$ consists only of the root with a finite number of legs attached to it),  then (\ref{eq:forest:vanishes}) follows at once. Otherwise, we write $\pi \CF$ for the sub forest of $\CF$ given by removing all legs, i.e.\ $\pi\CF$ is the subgraph of $\pi\tau$ induced by the edge set
\[
E(\pi\CF) := E(\CF) \backslash L_\Legtype(\tau).
\]
Then one has $\pi\CF \in \div(\pi\tau)$, and since $\CF \ne \forestlegs{\pi\CF}$ by assumption, there exists $S \in \CF$ with the property such that $L_\Legtype(S) \ne L_\Legtype(\forestlegs{\pi S})$. Assume first that $L_\Legtype(S)$ contains a leg $e$ with the property that $e \notin L_\Legtype(\forestlegs{\pi S})$. Then, since $\tau$ is properly legged, the leg $e$ has a unique partner $\tilde e \in L_\Legtype(T)$ in $T$ and one has $\tilde e^\downarrow \notin N(S)$ (since otherwise one would have $e^\downarrow, \tilde e^\downarrow\in N(S)$ and hence $e,\tilde e\in L_\Legtype(\forestlegs{\pi S})$ by definition of $\forestlegs{ \pi S }$). Thus, $e$ does not have a partner in $S$ and hence $S \in \wl{\CI}^\star$. 
Otherwise, $L_\Legtype(S) \subsetneq L_\Legtype(\forestlegs{\pi S})$. Then, there exists legs $e, \tilde e\in L_\Legtype( \forestlegs{ \pi S})$ such that $\{\ft(e),\ft(\tilde e)\}\in\CW$, and such that at least one of these two legs is not an element of $L_\Legtype(S)$. If $e \in L_\Legtype(S)$ but $\tilde e \notin L_\Legtype(S)$ (or the other way round), then one has $S \in \wl{\CI}^\star$. If $e,\tilde e \notin L_\Legtype(S)$ then one has $e^\downarrow = \tilde e^\downarrow$ in $T/\CF$ (since the vertices $e^\downarrow$ and $\tilde e^\downarrow$ in $T$ belong to the same connected component of $\CF$) and thus $T/\CF \in \wl{\hat\CI}^\star$.

We now show that $\plCT$ is a Hopf subalgebra, the claim that $\plCThat$ is a co-module follows very similarly. We need to show that $\cpmh \plCT \ssq \plCT \otimes \plCT$. For this it is sufficient to show that $\cpmh \wlP \tau = (\wlP \otimes \wlP )\cpmh \tau\in \plCT \otimes \plCT$ for any properly legged tree $\tau = T^\fn_\fe \in \wCT$, which in turn follows once we show that any sub-forest $\CF=\forestlegs{\CG}$ for some $\CG\in \div(\pi\tau)$ has the property that the trees $S \in \bar \CF$ and the tree $T/\CF$ are all properly legged.

It follows from the definition of the coproduct that if $\tau$ is properly legged, then point~\ref{item:properly:legged:legtypeunique}.\ in Definition~\ref{def:propery:legged} carries over to any subtree $S\in\CF$ and also to $T/\CF$. Moreover, point~\ref{item:properly:legged:typei}.\ also immediately carries over to subtrees $S \in \CF$. To see that point~\ref{item:properly:legged:typei}.\ is inherited also by $T/\CF$, we note that
by assumption on the regularity structure it follows that whenever $ S$ is a subtree of $\pi \tau$ with $\fancynorm{ S^0_\fe}_\fs<0$, then for any $u \in L(\tau)$ one has $u \in L(S)$ if and only of $u^\downarrow \in N(S)$. Applying this to fact to the trees $S \in \CG$, it follows that one has $\CL(\tau/\CF) = \CL(\tau) \backslash N(\CF)$ and moreover, for any $u \in \CL(\tau/\CF)$ one has that the sets $E(u,\tau/\CF)$ and $E(u,\tau)$, given respectively as the sets containing all edges $e \in E(\tau/\CF)$ and $e \in E(\tau)$ with  $e^\downarrow = u$, coincide, which together imply that~\ref{item:properly:legged:typei}.\ holds with $\itype = \itype|_{\CL(\tau/\CF)}$. 

The fact that points~\ref{item:properly:legged:leaf_coupling}.\ and~\ref{item:properly:legged:leg_coupling}.\ of Definition~\ref{def:propery:legged} hold for any $S \in \bar\CF$ follow from the definition of $\CF = \forestlegs{\CG}$. The fact that~\ref{item:properly:legged:leaf_coupling}.\ holds for $T/\CF$ follows from the fact that $E(u,\tau) = E(u, \tau/\CF)$ for any $u \in \CL(\tau)$. A very similar argument shows that~\ref{item:properly:legged:leg_coupling} holds also for $T/\CF$.
Finally, note that~\ref{item:properly:legged:leaf_hat_coupling}.\ holds trivially for the left component of $\Delta_-^\ex \tau$ since this component does not contain coloured vertices by definition.
The fact that~\ref{item:properly:legged:leaf_hat_coupling}.\ holds for $T/\CF$ can be argued very similarly to~\ref{item:properly:legged:leaf_coupling}.
\end{proof}

\begin{proof}[of Lemma~\ref{lem:CTadze:hopf:ideal}]
Using Lemma~\ref{lem:tensor:kernel}, it suffices to show that $\cpmh \ker \plQz \ssq \ker (\plQz \otimes \plQz)$, which follows once we show that $(\plQz \otimes \plQz)\cpmh = (\plQz \otimes \plQz) \cpmh \plQz$. This in turn is a consequence of
\begin{align}\label{eq:plz:plz:Qz}
(\plZ \otimes \plZ) \cpmh \tau= (\plZ \otimes \plZ) \cpmh \plQz \tau
\end{align}
for any $\tau \in \plCT$. Since both sides are linear and multiplicative, it suffices to show this for trees $\tau \in \plCT$.

In the case that the derivative decoration $\fe$ of $\tau$ does not vanish identically on legs, identity (\ref{eq:plz:plz:Qz}) follows directly from the fact that the coproduct $\cpmh$ never decreases the decoration $\fe$, so that either the right or the left component of $\cpmh\tau$ contains at least one leg with non-vanishing derivative decoration. Hence, both sides in (\ref{eq:plz:plz:Qz}) vanish.
 
In the remaining case one has $\fe|_{L_\LT(\tau)} = 0$ and thus $\plQz \tau = \plQ\tau$. 
Recall that one has $\cpm = (\Id \otimes \pmi) \cpmwi$ on $\plCT$. We then note that one has $\pi \tau = \pi \plQ \tau$, so that from (\ref{eq:coproduct:pl}) we infer
\[
\cpmwi \plQ \tau = 
\sum_{\CF \in \div(\pi \tau)}
\sum_{\fn_{\forestlegs \CF}, \ce_{\forestlegs\CF} }
\frac{1}{\ce_{\forestlegs\CF} !}
\binom{\fn}{\fn_{\forestlegs\CF}}
\prod_{S \in \forestlegs\CF} S^{\fn _{\forestlegs\CF}+ \pi\ce_{\forestlegs\CF}}_\fe
\otimes
(T/\forestlegs\CF)^{\fn-\fn_{\forestlegs\CF},  [\fo]_{\forestlegs\CF}}
_{\fe + \fe_{\forestlegs\CF}},
\]
where $\forestlegs{\CF} = \forestlegs{\CF}[\plQ\tau]$. If we compare the second sum in this identity to the corresponding sum in (\ref{eq:coproduct:pl}), we see that they only differ by the range of the decoration $\fe_{\forestlegs{\CF}}$, since the sum above puts derivatives also on superfluous legs. If we write $L_\Legtype^{*}(\tau) \ssq L_\Legtype(\tau)$ for the set of superfluous legs of $\tau$, it follows that one has the idenity
\begin{align*}
\cpmwi \plQ \tau = 
\sum_{\CF \in \div(\pi \tau)}
\sum_{  \substack  {
	\fn_{\forestlegs \CF}, \ce_{\forestlegs\CF}  \\
	\fe_{\forestlegs \CF}|_{L^*_\Legtype(\tau)}=0
} }
\frac{1}{\ce_{\forestlegs\CF} !}
\binom{\fn}{\fn_{\forestlegs\CF}}
\prod_{S \in \forestlegs\CF} 
\plQ S^{\fn_{\forestlegs\CF} + \pi\ce_{\forestlegs\CF}}_\fe
\otimes
\plQ (T/\forestlegs\CF )^{\fn-\fn_{\forestlegs\CF},  [\fo]_{\forestlegs\CF}}
_{\fe + \fe_{\forestlegs\CF}},
\end{align*}
where this time $\forestlegs{\CF} = \forestlegs{\CF}[\tau]$. Since any term on the right-hand side of (\ref{eq:coproduct:pl}) with $\fe_{\forestlegs \CF}$ non vanishing on $L_\Legtype^*(\tau)$ yields an element of $\plCT \otimes \ker \plZ$, the claim follows.

\end{proof}

\begin{lemma}\label{lem:ideal:generators}
Let $\CA$ be the symmetric algebra of a finite-dimensional vector space $\CB$, 
and let $\Phi \ssq \CB^*$ be a linear subspace of the dual space $\CB^*$ of $\CB$. 
For any $f \in \CB^*$, denote by $f_* \in {\CA^*}$ the unique character of $\CA$ extending $f$. 
Finally, for any $C \ssq \CA$, write $\CJ(C)$ for the ideal in $\CA$ generated by $C$.

Then, one has
$
\CI := \CJ\big(\bigcap_{\varphi \in \Phi} \ker \varphi\big) = \bigcap_{\varphi \in \Phi} \ker\varphi_* =: \bar \CI
$.
\end{lemma}
\begin{proof}
The inclusion $\CI \ssq \bar\CI$ is trivial so we only need to show that $\bar\CI \ssq \CI$.

Define  $\Phi^\perp \ssq \CB$ by 
$
\Phi^\perp := \bigcap_{\vphi \in \Phi} \ker\vphi
$,
Let $X$ be a basis of $\Phi^\perp$, let $Y$ be a basis of a complement of $\Phi^\perp$ in $\CB$, 
and observe that $X \cup Y$ generates $\CA$ freely as a commutative, unital algebra. So, for any $a \in \CA$ there exist $r \ge 0$, $c_i \in \R$ and $b_i \colon X\cup Y \to \N$ such that
\begin{equ}[eq:ideal:generators:a]
a = \sum_{i \le r} c_i \prod_{x \in X} x^{b_i(x)} \prod_{y \in Y} y^{b_i(y)}\;.
\end{equ}
We will always assume that this sum is minimal in the sense that $c_i \ne 0$ and $i \ne j \Rightarrow b_i \ne b_j$, which makes the representation unique, modulo a permutation of the index set $\{1, \ldots, r\}$. 

The ideal $\CI$ consists precisely of those elements $a \in \CA$ such that in the representation (\ref{eq:ideal:generators:a}) one has $b_i \restr X \ne 0$ for all $i$. 
For $a \in \bar\CI$, we can therefore write $a = a_0 + a_1$ with $a_1 \in \CI$ and
$a_0$ belonging to the subalgebra $\CA_Y \subset \CA$ generated by $Y$.
Assuming by contradiction that $a_0 \neq 0$, one can find a character $\vphi$ of $\CA_Y$
such that $\phi(a_0) \neq 0$. 
If we extend $\vphi$ to all of $\CA$ by setting $\vphi(x)=0$ for $x \in X$, then 
$\vphi \in (\Phi^\perp)^\perp = \Phi$, so that $\vphi(a_0) = 0$ and therefore $\vphi(a) \neq 0$
 in contradiction with the assumption that $a \in \bar\CI$.
\end{proof}

\subsection{Feynman diagrams}\label{sec:Feynman-diagrmas}

We state and sketch the proof of (a slight generalisation) of \cite[Thm.~3.1, Thm.~4.3]{Hairer2017}. 
Let for this $\CL_\star$ be a non-empty set of types and let $\CL:=\CL_\star \sqcup\{ \delta \}$. In analogy to \cite[Def.~2.1]{Hairer2017} we make the following definition.
\begin{definition}
A \emph{Feynman diagram} is a finite directed graph $\Gamma=(V,E)$ endowed with the following additional data
\begin{itemize}
\item An ordered set of distinct vertices $\bar V=\{ [1], \ldots , [k]\} \ssq V$ such that each $[i]$ has exactly one outgoing edge called ``leg'' and no incoming edge and such the each connected component of $\Gamma$ contains at least one leg. 
We write $V_\star := V \setminus \bar V$ and $E_\star \ssq E$ for the set of internal edges, 
i.e.\ edges which are not legs. 
For each $[i] \in \bar V$ we denote by $i_\star$ the vertex such that $([i],i_\star) \in E$.
\item For every connected component $\tilde \Gamma$ of $\Gamma$ we choose a distinguished vertex $v_\star(\tilde\Gamma)$. For any $u \in V$ we write $u_\star$ for the distinguished vertex $v_\star(\tilde\Gamma)$ of the connected component $\tilde\Gamma$ which contains $u$. 
\item Decorations $\ft: E \to \CL$ such that $\ft(e) = \delta$ if and only if $e$ is a leg,
 $\fe : E \to \N^d$ and $\fn : V_\star \to \N^d$.
\end{itemize}
We write $\Gamma^\fn_\fe$ whenever we want to make the decorations explicit.
\end{definition}

This definition differs slightly from \cite[Def.~2.1]{Hairer2017} since we include a polynomial decoration $\fn$. 

As in \cite[Def.~2.7]{Hairer2017} we define a \emph{vacuum diagram} as a Feynman diagram $\Gamma$ such that each connected component contains exactly one leg. We write $\CD$ for the linear space generated by all Feynman diagrams, and we write $\hat\CD_-$ for the algebra of all vacuum diagrams such that each connected component contains at least one internal edge. As in \cite{Hairer2017} we factor out a subspace (resp. an ideal) on which the valuation (which we will define below) vanishes. We define $\partial\CD$ (resp. $\partial \hat\CD_-$) as the smallest subspace of $\CD$ (resp. the smallest ideal in $\hat\CD_-$) which contains the expressions \cite[Eq.~2.16, 2.17, 2.18]{Hairer2017} for any connected Feynman diagram, and we set
\[
\CH := \CD / \partial\CD
\qquad\text{ and }\qquad
\hat\CH_- := \hat\CD_- / \partial\hat\CD_-.
\]


\subsubsection*{Degree assignments}

In \cite{Hairer2017} it was assumed that we are given a degree assignment $\deg : \CL \to \R_-$ and for any $C>0$ bounds were derived uniformly in kernel assignments $K$ such that $\|K_\ft\|_{\deg\ft}<C$ for any $\ft \in \CL$. We will generalise this setting slightly to allow some of the kernels to ''exchange'' homogeneity. This is possible due to the fact that the bounds we are interested in only depend on the product $\prod_{e \in E(\Gamma)}\|K_{\ft(e)}\|_{\deg\ft(e)}$.
\begin{example}
As a typical example, consider two edges $e,f \in E(\Gamma)$ and two smooth, compactly supported functions $\vphi,\psi$, and assume that 
\[
K_{\ft(e)} (x) := \lambda^{\alpha} \vphi(\lambda^{-\fs} x)
\qquad\text{ and }\qquad
K_{\ft(f)}(x) := \lambda^{\beta}  \psi(\lambda^{-\fs} x)
\]
for some $\alpha,\beta<0$ and $\lambda \in (0,1)$. While $\deg\ft(e)=\alpha$ and $\deg\ft(f)=\beta$ seems the most natural choice, if we are interested in bounds uniformly in $\lambda>0$ we could make any choice of the form $\deg\ft(e)=\alpha-\theta$ and $\deg\ft(f)=\beta+\theta$ for some $\theta\in\R$ (as long as $\alpha-\theta<0$ and $\beta+\theta<0$). 
\end{example}

Tweaking the degrees in this way may alter the renormalisation structure (i.e.\ the degree of a sub diagram may cross a non-positive integer), so that one could try and find a degree assignment which minimises the number of sub-diagrams that need to be renormalised. Unfortunately, in the situations we are interested in, this turns out to be impossible: Given any ``tweaked'' choice of degrees, there are always sub graphs which appear divergent, but are actually completely fine. 

To overcome this issue we consider the following construction.
Assume that we are given a partition $\CL = \bigsqcup_{\fdl\in\FDL} \CL_\fdl$ of the set of types. 
We then call a type $\ft \in \CL$ \emph{strong} if $\{\ft\}=\FL_\fdl$ for some $\fdl \in \FDL$, and \emph{weak} otherwise, and we write $\CLs$ and $\CLw$ for the subsets of strong and weak types, respectively. We always assume that $\delta$ is a strong type. 

\begin{definition}
A Feynman diagram $\Gamma$ is called \emph{admissible} if any weak type $\ft$ appears at most once in $\Gamma$, and for any $\fdl \in \FDL$ with $\ft(E(\Gamma)) \cap \CL_\fdl \ne \emptyset$ one has $\CL_\fdl \ssq \ft(E(\Gamma))$ and there exists a vertex $u$ with the property that all $e \in E(\Gamma)$ with type $\ft(e) \in \CL_\fdl$ are connected to $u$. 
\end{definition}
The last part of the previous definition rules out the possibility that for some $\fdl \in \FDL$ there are two non-overlapping subgraphs which contain the edges $e \in E(\Gamma)$ with $\ft(e) \in \CL_\fdl$  (in which case we may be able to avoid either sub-divergence, but possibly not both at the same time).
\begin{example}
A typical example to which we apply this setting is given by a Feynman diagram where the ``strong'' edges are the kernels-type edges of an underlying tree $\tau \in \CT$ and the ``weak edges'' represent kernels of noises living in fixed homogeneous Wiener chaoses. For instance, one could look at the following Feynman diagram:
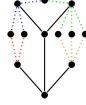
\[
\treeExampleWeakEdges
\]
Here, we draw bold lines for strong edges, dotted lines for weak edges, and we colour weak edges according to the partition $\bigsqcup_{\fdl \in \FDL} \CL_\fdl$.
\end{example}
We assume we are given a degree assignment $\fddeg: \FDL \to \R_-$, such that $\fddeg(\delta) = -|\fs|$. Here and below we write $\fddeg(\ft) := \fddeg(\fdl)$ if $\ft \in \CLs$ is a strong type such that $\CL_\fdl=\{ \ft \}$. Finally, we assume we are given a homogeneity assignment $\fddegmin:\CL \to \R_-$ such that $\fddegmin \ft := \fddeg \ft$ for any strong type $\ft \in \CLs$. We then write
\[
\FDdeg := \{ \deg : \CL \to \R_- : 
\deg \ge \fddegmin 
\text{ and for all } \fdl \in \FDL \text{ one has } \sum_{\ft \in \CL_\fdl} \deg\ft = \fddeg \,\fdl 
\},
\]
and for any Feynman diagram $\Gamma^\fn_\fe$ we define the quantity
\[
\fddeg \Gamma
:=
\sup_{\deg \in \Deg}
\sum_{ e \in E(\Gamma) } (\deg \ft(e) - |\fe(e)|_\fs)
+
\sum_{ u \in V(\Gamma) } |\fn(u)|_\fs
+
|\fs|(\# V(\Gamma) - 1).
\]
(Note that the expression inside the $\sup$ does not depend on $\deg\in\Deg$ for admissible Feynman diagrams.) 
\begin{example}
Consider two non-admissible, overlapping (but not nested) sub-diagrams $\tilde\Gamma_1$, $\tilde\Gamma_2$ with $\fddeg\tilde\Gamma_i>0$. Consider furthermore a spanning tree $\T$ such that $\tilde\Gamma_1$ collapses for $\T$ (i.e.\ for some interior node $\mu$ of $\T$ one has that $V(\tilde\Gamma_1)$ is given by the set of $u \in V(\Gamma)$ such that $u \ge \mu$ with respect to the tree order). By definition we can find a degree assignment $\deg\in\Deg$ such that $\tilde\Gamma_1$ is of positive degree, and since $\tilde\Gamma_2$ is overlapping with $\tilde\Gamma_1$ it does not collapse. Consequently, neither of these two sub-diagrams need to be renormalised. An identical argument works in the case that $\tilde\Gamma_2$ is collapsing. However, there might not exist a fixed degree assignment $\deg \in \Deg$ such that both statements are true at the same time.
\end{example}

Finally, we fix another degree assignment $\deg_\infty : \CL \to [-\infty,0]$ with $\deg_\infty(\delta):= - \infty$.

\subsubsection*{Kernel assignments and valuations}

We write $\CC_1$ for the space of smooth functions $\phi \in \CC_c^\infty( \bar \domain )$ supported in the $\fs$-unit ball of radius $1$. We then write $\CK^-_\infty$ for the set kernel assignments $(K_\ft)_{\ft \in \CL}$ such that $K_\ft \in \CC_1$ for any $\ft \in \CL$, and $\CK^+_\infty$ for the set of kernel assignments $(R_\ft)_{\ft \in \CL}$ such that $R_\ft \in \CC_c^\infty(\bar\domain)$ for any $\ft \in \CL$ and $R_\ft:=0$ for any weak type $\ft \in \CLw$. We also let $l_i := \fe([i], i_\star)$ for $i=1,\ldots,k$, and with this notation we define an valuation $\Pi^{K,R}$ on $\CH$ by setting
\begin{multline*}
(\Pi^{K,R} \Gamma)(\vphi)
:=
\int_{\bar \domain^{V_\star}}
	\prod_{e \in E_\star} D^{ \fe( e ) } (K+R)_{ \ft ( e ) } (x_{e_+} - x_{e_-} )
\\
	\prod_{u \in V_\star} ( x_u - x_{u_\star} )^{ \fn ( u ) }
	(D^{l_1}_1 \ldots D^{l_k}_k \vphi) (x _{ v _1 } , \ldots , x _{ v _k } )dx,
\end{multline*}
for any $(K,R) \in \CK^-_\infty \times \CK_\infty^+$.

Recall that we want to allow types $\ft,\tilde\ft \in \CL_\fdl$ to ``exchange homogeneity'', and as consequence there is no natural norm on $\CKminf$ which we can use. Instead we are forced to work with tensor products, which is very similar to Definition~\ref{def:CYsimp}. For any $\fdl \in \FDL$ we define the space 
\[
\fdCKminf := \bigotimes_{\ft \in \CL_\fdl} \CC_1
\]
together with the norm
\begin{align}\label{eq:CKtmz:norm}
\CKfdlnorm{ K } := \sup_{\deg \in \FDdeg} \prod_{\ft \in \CL_\fdl} \| K_\ft \|_{\deg\ft},
\end{align}
where $\|\cdot\|_{\deg\ft}$ is as in (\ref{eq:norm:blowup:1}), (\ref{eq:norm:blowup:2}),
and we define $\fdCKmz$ as the closure of $\fdCKminf$ under this norm. We also write
\[
\CKtminf := \bigoplus_{\fdl \in \FDL} \fdCKminf
\qquad\text{ and}\qquad
\CKtmz := \bigoplus_{\fdl \in \FDL} \fdCKmz.
\]

We next note that, for admissible Feynman diagrams $\Gamma$, one can define $\Pi^{K,R}\Gamma$ for any $(K,R) \in \CKtminf \times \CKpinf$ in a canonical way by imposing this to be linear on each component $\fdCKminf$. To be more precise, for fixed $K \in \CKminf$ define $\tilde K \in \CKtminf$ by setting
\begin{align}\label{eq:CKtminf}
\tilde K_{\fdl} := \bigotimes_{\ft \in \CL_\fdl} K_\ft
\end{align}
for any $\fdl \in \FDL$. Then, if $\Gamma$ is admissible, the quantity
\[
\Pi^{\tilde K,R} \Gamma := 
\Pi^{K,R} \Gamma 
\]
is well defined (recall that $R_\ft=0$ for any weak type $\ft \in \CLw$) and can be linearly extended to $\tilde K \in \CKtminf$. 

We finally define $\CKpz$ analogously to \cite[Sec.~4]{Hairer2017}  as the closure of $\CKpinf$ under the norm $\|R\|_{\infty,\deg_\infty}$ given by the smallest constant such that
\[
|D^k R_\ft(x)| \le \|R\|_{\infty,\deg_\infty	} (1+|x|)^{\deg_\infty \ft}
\]
for all $\ft \in \CL$, $x \in \bar\domain$ and $|k|_\fs < r$.


\subsubsection*{Renormalisation}

We denote by $\CH_-$ the algebras of Feynman vacuum diagrams defined as in \cite[below Rem.~2.9]{Hairer2017}. Recall that 
$\CH_-$ can be identified with the  factor algebra  $\hat\CH_-/\CJ_+$, where $\CJ_+ \ssq \hat\CH_-$ denotes the ideal generated by connected vacuum diagrams of positive $\fddeg$-degree. 
We introduce a co-product $\Delta_- : \CX \to \CH_- \otimes \CX$ for $\CX \in \{\CH, \hat \CH_- , \CH_-\}$  in analogue to \cite[Eq.~2.19, 2.24]{Hairer2017} by setting
\begin{align}\label{eq:coproduct:fd}
\Delta_-  \Gamma^\fn_\fe
=
\sum_{\tilde\Gamma \ssq \Gamma} \sum_{\tilde \fe, \tilde \fn}
	\frac{(-1)^{|\out \tilde \fe|}}{\tilde \fe !}
	\binom{\fn}{\tilde\fn}
	\tilde\Gamma^{\tilde\fn+\pi\tilde\fe}_{\fe} \otimes (\Gamma/\tilde\Gamma)^{\fn-\tilde\fn}_{
	[\tilde \fe]+ \fe}
\end{align}
where we use the convention that the first sum runs over full subgraphs\footnote{Recall \cite[pg.~7]{Hairer2017} that by definition a \emph{subgraph} does not contain legs or isolated vertices} $\tilde\Gamma$ of $\Gamma$ with the property that any connected component of $\tilde\Gamma^0_\fe$ is of negative degree, and the second sum runs over all decorations $\tilde\fe: \partial_{\tilde\Gamma} E(\Gamma) \to \N^d$ and $\tilde\fn : V(\Gamma) \to \N^d$ such that $\supp \tilde \fn \ssq V(\tilde \Gamma)$. Here, we write $\partial_{\tilde\Gamma} E(\Gamma)$ for the set of half-edges $(e,v)$ with $e \in E(\Gamma) \backslash E(\tilde\Gamma)$ and $v \in e \cap V(\tilde\Gamma)$, and we write $[\tilde\fe](e):=\sum_{u \in e}\tilde\fe(e,u)$. We call a subgraph $\bar\Gamma$ of $\Gamma$ \emph{full} if it has the property that $E(\bar \Gamma)$ is given by the set of all $e=(u,v) \in E(\Gamma)$ such that $\{u,v\}\subseteq V(\bar\Gamma)$.
This definition agrees (apart form the fact the we restrict to full subgraphs) with \cite[Eq.~2.24]{Hairer2017} in case $\CX \in \{ \hat\CH_-, \CH_- \}$ and is a slight generalisation of \cite[Eq.~2.19]{Hairer2017} in case $\CX = \CH$ since we include polynomial decorations.

We moreover write $\tilde\Delta_-$ for the coproducts acting between the same spaces, which are defined similarly to (\ref{eq:coproduct:fd}), but where $\bar\Gamma$ ranges also over subgraphs which are not necessarily full.
We define the twisted antipode $\hat\CA : \CH_- \to \hat\CH_-$ as in \cite[Eq.~2.28]{Hairer2017} as the unique multiplicative map satisfying $\CM(\hat\CA \otimes \Id)\Delta_- \Gamma=0$ for any $\Gamma \in \hat\CH_-$ such that $\fddeg\Gamma < 0$, 
and we write $\tilde\CA : \CH_- \to \hat\CH_-$ for the operator that satisfies the same identity with $\Delta_-$ replaced by $\tilde \Delta_-$. 
It follows with arguments identical to those carried out in \cite{Hairer2017} that the spaces $\CH_-$ and $\hat\CH_-$ equipped with the full coproduct $\Delta_-$ form a Hopf algebra and a co-module, respectively.
 When we refer to $\CH_-$ as a Hopf algebra, and in particular when we refer to the group product in the character group of $\CH_-$, it is always the full coproduct $\Delta_-$ that we have in mind.

Finally, given a smooth kernel assignment $K \in \CKminf$, we write $g(K)$ and $g^\full(K)$ for the respective BPHZ characters, defined as characters on the Hopf algebra $\CH_-$ via the identities
\[
g(K) := \Pi^{K} \hat\CA
\quad \text{ and }\quad
\tilde g(K) := \Pi^{K} \tilde \CA.
\]
It then follows from \cite[Prop.~3.11]{Hairer2017} that one has
\begin{align}\label{eq:coproduct:full}
\hat \Pi^K := (g(K) \otimes \Pi^K) \Delta_- = (\tilde g (K) \otimes \Pi^K ) \tilde \Delta_-
\end{align}
on $\CH$. 
We first show a simple lemma that extends (\ref{eq:coproduct:full}) to the situation where one has non-vanishing large-scale kernel assignments.
\begin{lemma}\label{lem:renorm:large:scale:full}
Assume that $R$ is a smooth, compactly supported large-scale kernel assignment as in \cite[Sec.~4]{Hairer2017} 
and assume that $\Gamma \in \hat\CH_-$ is a connected vacuum diagram with the following property. Whenever $\tilde\Gamma \subseteq \Gamma$ is a connected subgraph such that $\fddeg \tilde\Gamma <0$, then for any $e=(u,v) \in E(\Gamma)\backslash E(\tilde \Gamma)$ with $\{u,v\} \ssq V(\tilde\Gamma)$ one has $R_{\ft(e)}=0$. Then one has
\[
\hat \Pi^{K,R} \Gamma:= (g(K) \otimes \Pi^{K,R}) \Delta_- \Gamma= (\tilde g (K) \otimes \Pi^{K,R} ) \tilde \Delta_-\Gamma.
\]
\end{lemma}
\begin{proof}
We only sketch how to adapt the proof of \cite[Prop.~3.11]{Hairer2017} to our situation. In the notation of \cite[Prop.~3.11]{Hairer2017}, the difference to our case is that in \cite[Eqn. 3.16]{Hairer2017} the evaluations $\Pi_-$ and $\Pi$ are build on different kernel assignments (since the former ignores the large scale kernel assignment). Without the extra assumption made in the statement of our lemma,  \cite[Eqn. 3.16]{Hairer2017} is not independent of $\CB$ whenever there is an element $\gamma^{\cl} \in \CB$ which is a root of the forest $\CF^p \cup \CF_{\diamond}^{full} \cup \CB$. 
By definition, $\gamma^\cl$ is the closure of some root $\gamma$ if $\CF^p$.
Thanks to the additional assumption made in our lemma, the two evaluations $\Pi^K$ and $\Pi^{K,R}$ act identically when applied to an edge $e \in E(\gamma^\cl) \backslash E(\gamma)$, and the proof can be finished as in  \cite[Prop.~3.11]{Hairer2017}.
\end{proof}
Finally, note that as before, for admissible Feynman diagrams $\Gamma$, the quantity $\hat\Pi^{K,R}\Gamma$ is well-defined for $K \in \CKtminf$ by linear extension.
\begin{remark}
The character $g(K)$ is in general not well defined for $K \in \CKtminf$. This is because divergent subgraphs $\tilde\Gamma$ of an admissible Feynman diagram $\Gamma$ need not be admissible. However, the map $K \mapsto(g(K) \otimes \Pi^{K,R}) \Delta_- \Gamma$, for $K \in \CKminf$, has the multi-linearity property described above, and can therefore be extended uniquely to $\CKtminf$ by linearity.
\end{remark}

\subsubsection*{A slight generalisation of \cite{Hairer2017}}

We now state a generalisation of the results of \cite{Hairer2017}. 
For this we recall that for any Feynman diagram $\Gamma$ with legs $1_\star, \ldots , k_\star$, we call a partition $\CP$ of $V_\star$ \emph{tight} if $\#\CP\ge 2$ and all legs are contained in the same element, i.e. there exists $P_\star \in \CP$ such that for all $1 \le i \le k$ one has $i_\star  \in P_\star$. 
For any such partition we introduce the notation $E_\CP$ for the set of edges $e \in E_\star$ such that $e$ is not subset of any $P \in \CP$, and the notation $V_\CP$ for the set of vertices $u \in V_\star$ such that $\{u,u_\star\}$ is not subset of any $P \in \CP$. We then define
\[
\deg_\infty \CP := \sum_{e \in E_\CP} \deg_\infty \ft(e) - |\fe(e)|_\fs + \sum_{u \in V_\CP} |\fn(u)| + |\fs|(\#\CP -1 ).
\]
This differs slightly from \cite[pg.~43]{Hairer2017} since we include polynomial decorations.

\begin{theorem}\label{thm:BPHZ:FD}
Let $\Gamma$ be an admissible Feynman diagram such that $\deg_\infty \CP<0$ for any tight partition of $V$. Then, for fixed $K \in \CKtmz$ the map $R \mapsto \Pi^{K,R}\Gamma$ extends continuously to the space $\CK_0^+$ and the map $(K,R) \mapsto \hat\Pi^{K,R}\Gamma$ extends continuously to the space $\CKtmz \times \CK_0^+$.
\end{theorem}
\begin{proof}
We only sketch the difference to \cite{Hairer2017}. Let us first discuss the bound on small scales, i.e.\ the continuous extension of $K \mapsto \hat \Pi^{K,0} \Gamma$ to $\CKtmz$. There are two differences to the case treated in \cite[Sec.~3]{Hairer2017}. One is that we allow $\Gamma$ to have polynomial decoration $\fn$ and the other is the presence of weak types. 

It is straightforward to convince oneself that the proof given in \cite{Hairer2017} works without any changed for non-vanishing the polynomial decoration. 
To see that weak edges cause no problem, we recall a few pieces of notation. We write $\FF_\Gamma^-$ for the set of all forests $\CF$ of $\Gamma$. Recall \cite[Sec.~3.1]{Hairer2017} that a \emph{forest} $\CF$ is a family of divergent subgraphs of $\Gamma$ such that any two elements of $\CF$ are non-overlapping (i.e.\ either node-disjoint or nested). Recall further that a \emph{forest interval} $\M$ is a subset of $\FF_\Gamma^-$ with the property that there exists $\ol \M, \ul\M \in \FF_\Gamma^-$ such that $\M$ contains exactly those forests $\CF\in\FF_\Gamma^-$ such that $\ul\M \ssq \CF \ssq \ol\M$. The bound in \cite{Hairer2017} is then obtained by fixing a Hepp sector $\T$ \cite[Def.~3.5]{Hairer2017}, which allows to partition the set of forests into a family of forests intervals indexed by safe-forests for $\T$. 
The main step of the proof \cite[Eqs~3.9, 3.10, Lem.~3.7, Lem.~3.8]{Hairer2017} is then performed for each of these forest intervals separately. 
The only difference to the present setup is that we have a set $\FDdeg$ of possible degree assignments to choose from. By our definitions, a subgraph $\tilde\Gamma$ is divergent if and only if $\deg \tilde\Gamma<0$ for any $\deg \in \FDdeg$, so that the definition of the set of forests $\FF_\Gamma^-$ does not depend on a choice $\deg \in \FDdeg$. 
Neither do the notions of forest interval and save forest. 
We can then use exactly the same proof as in \cite{Hairer2017}. 
The only difference is that we first fix a Hepp sector $\T$ and a safe forest $\CF_s =\ul\M$. Only afterwards do we choose $\deg \in \FDdeg$ in such a way to make sure that for any subgraph $\tilde\Gamma$ of $\Gamma$ which is unsafe for $\CF_s$ one has $\deg \tilde\Gamma  = \FDdeg\tilde\Gamma$. This is always possible, since by definition, for any weak type $\fdl$, all edges $e \in E(\Gamma)$ of type $\ft(e) \in \CL_\fdl$ are connected to the same vertex. If we denote by $E_\fdl \ssq E(\Gamma)$ the set of these edges and $\nu_e \in \topcirc\T$ the node of the spanning tree at which $e$ collapses, then the fact that all $e \in E_\fdl$ have a vertex in common implies that $\{ \nu_e : e \in E_\fdl \}$ is a totally ordered set with respect to the tree order. It follows that one can recursively choose $\deg \in \FDdeg$ to optimise the degree of subdiagrams containing edges $e \in E_\fdl$.


We finally discuss the bound on large scales. In \cite[Sec.~4]{Hairer2017} the analogue statement was shown again without polynomial decorations and weak types. It is again easy to convince oneself that polynomial decorations pose no problems. Furthermore, the proof of \cite[Thm.~4.3]{Hairer2017} uses the bound on small scales (i.e.\ the continuous extension of $\hat \Pi^{K,0}$ to $K \in \CKtmz$) as a black box, otherwise only the a-priori bounds on the large-scale kernel assignments are used. It remains to point out that since $R_\ft=0$ for weak types $\ft \in \CLw$ it suffices to consider in \cite[pg.~45]{Hairer2017} subsets $\tilde E$ of the set of edges $E$ such that each $e \in \tilde E$ has a strong type, and for such sets $\CU(\Gamma,\tilde E)$ constructed in \cite[pg.~45]{Hairer2017} is again an admissible Feynman diagram.	
\end{proof}

An important application of the previous theorem is the proof of Theorem~\ref{thm:evaluation:trees:largescale}, see the end of Section~\ref{sec:application:to:trees} below. 


\subsubsection*{A scale dependent bound}

We also show that one can infer a scale dependent bound from \cite{Hairer2017} (even though this is not explicitly stated in this paper).
Given a Feynman diagram $\Gamma \in \CH$ we write $\fdforests \ssq \FF_\Gamma^-$ for the set of all forests $\CF \in \FF_\Gamma^-$ of $\Gamma$ such that $\Gamma \notin \CF$.
To any forest interval $\M$ we associate a linear combination of Feyman diagrams $\hat\CR_\M \Gamma$ as in \cite[Eq.~3.7]{Hairer2017}.  
\begin{remark}
A reader who is not familiar with \cite{Hairer2017} should think of $\hat\CR_\M\Gamma$ as a sum over all possible ways of ``pulling out and contracting'' the divergent sub-diagrams in $\ol\M$, with the restriction that any element of $\ul\M$ is always pulled out, and adjusting the sign according to the number of sub-diagrams which are pulled out. One specify property of $\hat\CR_\M$ is that we view the vertex set of every Feynman diagrams that we sum over in $\hat\CR_\M\Gamma$ as equal to $V(\Gamma)$. This can be obtained by ``reattaching'' one vertex of the pulled-out sub-diagram to the vertex which has been created by contracting it. (This is not canonical, but depends on a choice of distinct vertex in the pulled-out diagram. The ambiguity can be removed by fixing an arbitrary total order on $V(\Gamma)$.)
Note that this last property forces us to abstain from viewing $\hat\CR_\M$ as an operator acting on the algebra $\CH$ (the operator viewed in this way is denoted by $\CR_\M$ in \cite{Hairer2017}).
\end{remark}
We then write $\hat \CR:= \sum_{\M \in \CP} \hat\CR_\M$ where $\CP$ is some partition of $\fdforests$ into forest intervals (the definition of $\hat\CR$ is independent of this choice), and we write $\CWW^K$ for the map defined below \cite[Eq.~3.7]{Hairer2017}, so that
\[
	\CWW^K  \Gamma \in \CC_c^\infty( \ws^{V(\Gamma)})
\]
for any $K \in \CK_\infty^-$. The function $\CWW^K  \Gamma$ should be thought of as introducing for every edge $e \in E(\Gamma)$ a factor $D^{\fe(e)} K_{\ft(e)}$ evaluated between its endpoints. Note that by definition of $\hat\CR \Gamma$, one has $\CWW^K  \hat\CR \Gamma \in \CC_c^\infty( \ws^{V(\Gamma)})$.
 It follows as in \cite[Lem.~3.4]{Hairer2017} that one has
\begin{align}\label{eq:CW:hatCR}
g^\full(K)\Gamma = \int_{\ws^{V(\Gamma)}}
dx
\delta(x_{v_\star})
(\CWW^K \hat\CR\Gamma) (x)
\end{align}
if $\deg\Gamma \le 0$, while in case that $\deg\Gamma>0$ the right-hand side of (\ref{eq:CW:hatCR}) is equal to $\hat\Pi^K_\BPHZ \Gamma$.
We define the ``scale'' $\fdm(x)$ of $x \in \ws\backslash\{0\}$ as the largest integer smaller than $-\ld(|x|_\fs)$ (so that $|x|_\fs$ is of order $2^{-\fdm(x)}$), and we set
\begin{align}\label{eq:character:scale}
\esn 
:=
\int_{\ws^{V(\Gamma)}} dx \, \delta(x_{v_\star}) \CWW^K \hat\CR\Gamma(x) 
\,
\I\{\min_{u,v \in V(\Gamma)} \fdm(|x_u - x_v|)= n\}.
\end{align}
The indicator function ensures that we only integrate over point configurations $x$ such that the maximal distance $|x_u - x_v|$ for $u,v \in V(\Gamma)$ is of order $2^{-n}$, and the sum
\begin{align*}
\sum_{n \ge 0} \esn
\end{align*}
is equal to the right-hand side of (\ref{eq:CW:hatCR}).
The following result follows as in \cite{Hairer2017}.
\begin{theorem}\label{thm:BPHZ:scaledependent}
Let $\Gamma$ be an admissible Feynman diagram. Then one has the bound
\begin{align}\label{eq:BPHZ:scaledependent}
|\esn| \lesssim 
\Big(
	\prod_{e \in E(\Gamma)} \|K_{\ft(e)}\|_{\deg\ft(e)}
\Big)
2^{-n \, \fddeg\Gamma}
\end{align}
for any $\deg \in \Deg$. (Note that the right-hand side does not depend on $\deg \in \Deg$.)
Here, the implicit constant only depends on $\Gamma$ and $\fddeg$, but is uniform in $n$ and $K \in \CKtminf$.
\end{theorem}
\begin{proof}
Given a decorated spanning tree $(\T,\n)$ for $V(\Gamma)$ with $\n : \topcirc \T \to \N$, we denote by $D_{(\T,\n)} \ssq \ws^{V(\Gamma)}$ the Hepp sector associated to $(\T,\n)$ defined via \cite[Eq.~2.10]{Hairer2017}, and for $n\in \N$ we write $D_{\T,n}:= \bigcup_{\n: \n(\rho_\T)=n} D_{(\T,\n)}$.
Given furthermore a forest interval $\M$ of $\bar\FF_\Gamma^-$, we write $\esan$ for the constant given by 
\begin{align}\label{eq:character:Hepp}
\esan
:=
\int_{D_{\T,n}} |dx \, \delta(x_{v_\star}) \CWW^K \hat\CR_\M\Gamma(x)|,
\end{align}
so that it follows from the definitions that
\[
|\esn|
\lesssim
\sum_{\T}\sum_{\M \in \CP_\T}\sum_{m=n-n_0}^{n+n_0} \esam,
\]
where $n_0 \in \N$ depends only on the choice of $C$ in \cite[Eq.~2.10]{Hairer2017}. Here $\CP_\T$ is the partition of $\bar\FF_\Gamma^-$ defined as in \cite[p.~29]{Hairer2017}.
It suffices to show the bound (\ref{eq:BPHZ:scaledependent}) for $\esam$ for any spanning tree $\T$, any $m \in \N$ and any $\M \in \bar\FF_\Gamma^-$ separately. 

We now choose a degree assignment $\deg \in \Deg$ with the property for that any sub graph $\tilde\Gamma$ of $\Gamma$ which collapses for $\T$ one has $\roundup{\deg{\tilde\Gamma}}=\roundup{\fddeg{\tilde\Gamma}}$. 
Identically to \cite[Eq.~3.9]{Hairer2017} we obtain the bound
\[
\esam
\lesssim
\sum_{i \in I} \sum_{\fn:\fn(\rho_\T)=m} \prod_{v \in \topcirc \T} 2^{-\eta_i(v) \fn_v}
\]
where $\eta_i(v)$ is defined as in \cite[Eq.~3.17]{Hairer2017} for the degree assignment $\deg$. We can show \cite[Eq.~3.10]{Hairer2017} for any $v \in \topcirc\T \backslash\{ \rho_\T\}$ exactly as in \cite[p.~34]{Hairer2017}, and it follows that
\[
\esam
\lesssim
\sum_{i \in I}
\prod_{v \in \topcirc\T} 
2^{- \eta_i(v) m}
\lesssim
2^{- \fddeg \Gamma m},
\]
where we used that $\sum_{v \in \topcirc\T} \eta_i(v) = \fddeg\Gamma$ for any $i \in I$.
\end{proof}

One application is the following corollary which shows the absence of logarithmic divergencies in certain situations. 
Before we state the next definition, we introduce a piece of notation. Given a Feynman diagram $\Gamma_\fe^\fn$ and a subgraph $\tilde\Gamma \ssq \Gamma$, we want to identify all polynomial decorations $\tilde\fn : V(\tilde\Gamma) \to \N^d$ such that $\tilde\Gamma^{\tilde\fn}_\fe$ appears on the right-hand side of the coproduct $\Delta_-$ applied to $\Gamma^\fn_\fe$.
For this we write $\CN(\tilde\Gamma)$ for the set of decorations $\tilde\fn : V(\tilde\Gamma) \to \N^d$ 
such that for any $u \in V(\tilde\Gamma)$ 
with the property that there does not exist $e \in E(\Gamma) \setminus E(\tilde\Gamma)$ with $u \in e$ 
one has $\tilde\fn(u) \le \fn(u)$.
\begin{definition}\label{def:log-avoiding}
Let $\Gamma = \Gamma^\fn_\fe$ be an admissible Feynman diagram such that $\fddeg\Gamma^\fn_\fe=0$ and let $\fdl \in \FDL$. 
We say that a kernel assignment $K \in \CKtminf$ is \emph{log-avoiding} for $\Gamma$ and $\fdl$, if
$\CL_\fdl \ssq \ft(E(\Gamma))$ and for any proper subgraph $\tilde\Gamma \ssq \Gamma$ with $\CL_\fdl \cap \ft(E(\tilde\Gamma)) \ne \emptyset$ and any polynomial decoration $\tilde\fn \in \CN(\tilde\Gamma)$ such that $\fddeg \tilde\Gamma^{\tilde\fn}_\fe=0$ one has $g(K)\tilde\Gamma^{\tilde\fn}_\fe=0$.
\end{definition}
Given $\fdl \in \FDL$ and $\theta>0$ we also introduce the seminorm
\[
\logavnorm{ K }{\fdl}{\theta} := 
\CKfdlnormm{ K_\fdl }{\fddeg\fdl + \theta}
\CKfdlnormm{ K_\fdl }{\fddeg\fdl - \theta}
\]
for $K \in \CKtmz$.
We then have the following statement.

\begin{corollary}\label{cor:no-log-divergence}
In the setting above, let $\Gamma$ be an admissible Feynman diagram, let $\fdly \in \FDL$ 
and let $K \in \CKtminf$ be a kernel assignment which is log-avoiding for $\Gamma$ and $\fdly$.
Then for all $\theta>0$ small enough one has
\begin{align}\label{eq:log-avoiding:bound}
|g(K)\Gamma| \lesssim \CKtmnorm{K} + \logavnorm{K}{\fdly}{\theta}
\end{align}
uniformly over all $K \in \CKtminf$.

Moreover, if $\fdlz \in \FDL$ is such that $K$ is also log-avoiding for $\Gamma$ and $\fdlz$, then for all $\theta>0$ small enough one has
\begin{align}\label{eq:log-avoiding:bound:2}
|g(K)\Gamma| \lesssim 
{\CKfdlnormm{ K _\fdly }{\fddeg \fdly + \theta}} {\CKfdlnormm{ K_{ \fdlz }} { \fddeg \fdlz - \theta}}
\end{align}
uniformly over $K \in \CKtminf$ such that $\CKtmnorm{K}\lesssim 1$.
\end{corollary}
\begin{proof} 
First note that (\ref{eq:log-avoiding:bound}) is a consequence of (\ref{eq:log-avoiding:bound:2}) with $\fdly = \fdlz$.
We consider two degree assignments $\fddegp$ and $\fddegm$ on $\FDL$ such that $\fddegp \fdl = \fddeg \fdl$ except for $\fdl = \fdly$ and $\fddegm \fdl = \fddeg \fdl$ except for $\fdl = \fdlz$, which are defined by $\fddegp \fdly := \fddeg \fdly+\theta$ and $\fddegm \fdlz := \fddeg \fdlz-\theta$.
If we denote by $g_+(K)$ and $g_-(K)$ the BPHZ characters for $K$ and the degree assignments $\fddegp$ and $\fddegm$, respectively, then it is not hard to see that for $\theta>0$ small enough one has 
$$
g_-(K)\Gamma^\fn_\fe= g(K)\Gamma^\fn_\fe = \Pi^{K}M^{g_+(K)}\Gamma^\fn_\fe.
$$ 
\def\tGtfnfe{\tilde\Gamma ^{\tilde\fn} _\fe}
Indeed, the first identity follows from the fact that $\roundup{\fddeg\tilde\Gamma} =\roundup {\fddegm\tilde\Gamma}$ for any Feynman diagram $\tilde\Gamma$ and for any $\theta>0$ small enough. The second identity is a bit more subtle, since in general the ``divergence structure'' of $\Gamma$ is not the same for the homogeneity assignments $\fddeg$ and $\fddegp$. 
But writing $\Delta_-^{\fddeg}$ and $\Delta_-^{\fddegp}$ for the coproducts obtained from the respective homogeneity assignments, one has $\Delta_-^{\fddegp}= (\pfddegp \otimes \Id)\Delta_-^{\fddeg}$ 
with $\pfddegp$ the projection onto the algebra generated by diagrams of non-positive $\fddegp$-degree, so that it suffices to show that $g_+(K) \pfddegp \tilde\Gamma^{\tilde\fn}_\fe = g(K) \tilde\Gamma ^{\tilde\fn} _\fe$ for any sub diagram $\tilde\Gamma$ of $\Gamma$ such that $\fddeg \tilde\Gamma ^{\tilde\fn} _\fe \le 0$. 
If we assume inductively that this is true for all proper sub diagrams of $\tilde\Gamma$, 
then we observe that if $\theta>0$ is small enough then $\fddegp \tilde\Gamma ^{\tilde\fn} _\fe > 0$ implies $\fddeg \tilde\Gamma^{\tilde\fn}_\fe = 0$, and hence also $g(K)\tilde \Gamma^ {\tilde\fn} _\fe=0$ (by assumption), and otherwise
\begin{align*}
g(K)\tGtfnfe 
&= -(g(K) \otimes \Id ) (\Delta _-^{\fddeg} - \Id \otimes \one) \tGtfnfe
=
-(g_+(K) \otimes \Id ) (\Delta _-^{\fddegp} - \Id \otimes \one) \tGtfnfe
\\
&=
g_+(K) \tGtfnfe,
\end{align*}
which proves the claim.

Let now $\lambda>0$ be such that 
${\CKfdlnormm{ K _\fdly }{\fddeg \fdly + \theta}}  = \lambda^{-\theta}$, so that one also has 
${\CKfdlnormm{ K_{ \fdlz }} { \fddeg \fdlz - \theta}} \le C(\theta) \lambda^{\theta}$, 
where $C(\theta)$ denote the right-hand side of (\ref{eq:log-avoiding:bound:2}).
Let moreover $m \in \N$ such that $2^{-m-1} \le \lambda < 2^m$. We estimate the sum over large scales
using (\ref{eq:BPHZ:scaledependent}) for the degree assignment $\fddegm$ so that 
\[
\sum_{n \le m} |\esn|
\lesssim
2^{m \theta} \lambda^{\theta} \simeq C(\theta)
\]
and the sum over small scales
using (\ref{eq:BPHZ:scaledependent}) for the degree assignment $\fddegp$ so that 
\[
\sum_{n \ge m} |\esn|
\lesssim
2^{-m \theta} \lambda^{-\theta} \simeq 1,
\]
where both estimates hold uniformly over $K \in \CKtminf$ such that $\CKtmnorm{K} \le C$. 
\end{proof}

\subsubsection{Application to trees}\label{sec:application:to:trees}

We want to apply the result of Section~\ref{sec:Feynman-diagrmas} to models obtained from smooth noises $\eta \in \SM_\infty$, see Section~\ref{sec:non-gaussian-noises}. In Sections~\ref{sec:constraints} and \ref{sec:shift} we consider two different enlargements of the regularity structure (by including legs and by enlarging the set of noise types, respectively). We want to use the construction carried out in this section in both cases, so we simply formulate our results on the regularity structure which is enlarged in both ways. We write $\aFLm$ for the enlarged set of noise types (not including leg types) and $\Legtype$ for the set of leg types, and we write $\wFLm := \aFLm \sqcup \Legtype$. We then use the notation $\wT$ etc. as in Section~\ref{sec:ext:reg:strct} with $\FL_-$ replaced by $\aFLm$. (Note that this is not really a generalisation, since all assumptions which Section~\ref{sec:constraints} puts on $\FLm$ and $\CT$ are satisfied for the enlargement $\aFLm$ and $\aCT$ as well.)

Moreover, in order not to overcomplicate the presentation of the current section, we assume that for any $\ant \in \aFLm$ we are given a multi-set $\cm_\ant$ with values in $\FLm$, and we write $\CYN(\aFLm, \cm)$ for the set of kernels $\FK \in \CYN(\aFLm)$ such that for any $\ant \in \aFLm$ one has $\FK^\ant_{\cm}=0$ unless $\cm = \cm_\ant$. For $\FK \in \CYN(\aFLm, \cm)$ we simply write $\FK^\ant := \FK^\ant_{\cm_\ant}$. 
We set $m_\ant := \# \cm_\ant$. 

Define the set of labels 
\[
	\CL_\star := \FL_+ \sqcup \tilde\CL := \FL_+ \sqcup \{ (\ant,k) : \I_{m_\ant = 1} \le k \le  m_\ant ,\, \ant \in \aFLm \},
\]
and we write $\CH$ for the linear space of Feynman diagrams as above.
We also define the partition $\FDL$ of $\CL$ given by 
\begin{align}\label{eq:fdl:trees}
\FDL := \big\{ \{\fl\} : \fl \in \FL_+ \big\}
\sqcup
\big\{ 
	\fdl_\ant : \ant \in \aFLm, 
\big\},
\end{align}
where $\fdl_\ant := 	\{(\ant,k) : 0 \le k \le m_\ant\}$ if $m_\ant > 1$ and $\fdl_\ant := \{ (\snt,1) \}$ if $m_\ant = 1$.

Trees $\tau \in \wT$ contain a finite number of legs $e \in L_\Legtype(\tau)$ and a finite number of noise type edges $f \in L(\tau)$. The formulation of this section will be cleaner by focusing on trees $\tau \in \wT$ with the property that any leg type $\legtype$ and any noise type $\ant$ appears at most once in $\tau$, so that $[L_\Legtype(\tau),\ft]$ and $[L(\tau),\ft]$ are proper sets.  We write $\wTu \ssq \wT$ for the subspace generated by such trees. 
We will also work the Hopf subalgebra $\wCTu \ssq \wCT$ generated by such trees.
We also fix an arbitrary total order $\le$ on $\Legtype$ and  we note that $\le$ induces an order $\le$ on $L_\Legtype(\tau)$ for any tree $\tau \in \wTu$. 

Given $\FK \in \CYN$ (see Definition~\ref{def:CYN}) we always write $\eta := \FKtN \FK$ for ease of notation.
We first construct continuous linear operators
\[
\ttf : \wTu \to \CH
\qquad\text{ and }\qquad
\nka : \CYN(\aFLm, \cm) \to \CKtminf(\tilde\CL)
\]
with the property that, for any tree $\tau \in \wTu$, $\ttf\tau$ takes values in the span of Feynman diagrams $\Gamma$ with exactly $k(\tau) := \# L_\Legtype(\tau)$ legs, and such 
\begin{align}\label{eq:FDT:1}
(\Pi^{\nkaFK, R} \ttf \tau)(\psi_\tau) = \evaln{\eta}{\psi}{R} \tau
\end{align}
for any $\tau \in \wTu$, $\FK \in \CYN(\aFLm, \cm)$ and $\psi \in \Psi$. Here, we write $\psi_\tau = \psi_{( k(\tau), \ft_\tau )} \in \CC_c^\infty( \bar \domain ^{k(\tau)})$, where $\ft_\tau : k(\tau) \to \ft(L_\Legtype(\tau))$ denotes the unique order preserving map.



Fix a tree $\tau \in \wTu$ and denote the legs of $\tau$ by $e_1, \ldots , e_{k(\tau)}$ in increasing order. Let $\Lambda := \{ (u,k) : u \in L(\tau) ,\, 1 \le k \le m(\ft(u)) \}$. 
\begin{definition}\label{def:CP}
We denote by $\pairings$ the set of pairings $\pairing$ of $\Lambda$ with the property that for any $\{ (u,k), (v,l) \} \in \pairing$ one has $u \ne v$ and $\cm_{\ft(u)}[k] = \cm_{\ft(v)}[l]$. For $u \in L(\tau)$ and $\pairing \in \pairings$ we write $\pairing[u] \ssq \pairing$ for all elements of the form $\{(u,k), (v,l) \} \in \pairing$ for some $v \in L(\tau)$ and $k,l \in \N$.
\end{definition}

 For any pairing $\pairing \in \pairings$ we now construct a connected, direct graph $\Gamma_\tau^\pairing = (V^\pairing_\tau, E_\tau^\pairing)$. Let $L_{> 1}(\tau)$ denote the set of noise-type edges $u \in L(\tau)$ such that $m(\ft(u)) > 1$ and $\Leo (\tau) := L(\tau) \setminus \Lgo$.
We first set
\[
V^P_\tau := N(\tau) \sqcup \pairing \sqcup \{ u^* : u \in L_{> 1}(\tau) \} \sqcup \{ [1] , \ldots, [k(\tau)] \}.
\]
We define for any $u \in \Lgo(\tau)$ the set $E_u :=\{ (q, u^*) : q \in \pairing[u] \}\sqcup\{ (u^*, u^\downarrow )\}$, while for $u \in \Leo(\tau)$ we simply set $E_u :=\{ (q, u^\downarrow) : q \in \pairing[u] \}$. To avoid case distinctions, we also set $u^* := u^\downarrow$ for any $u \in \Leo(\tau)$. (As always, $u^\downarrow \in N(\tau)$ denotes the unique node to which the noise type edge $u$ is connected.)
We then set
\[
E^\pairing_\tau := K(\tau) \sqcup \bigsqcup_{u \in L(\tau)} E_u \sqcup \{ ([i] , e_i^\downarrow) : 1 \le i \le k(\tau) \}.
\]
We also fix an edge decoration $\fe : E_\tau^P \to \N^d$ and a node decoration $\fn : V_\tau^P \to \N^d$ by extending the corresponding decorations coming from $\tau$ by setting them to zero everywhere else. We choose as ``special'' vertex the root $u_\star(\Gamma_\tau^P) := \rho_\tau \in N(\tau)$.

\begin{example}
To illustrate this construction, we take the following tree as an example, where nodes $u \in L(\tau)$ are coloured, and legs are drawn as short thick grey edges
\[
\treePairingExample. 
\]
We then have $k(\tau) = 4$, and we assume that $m(\leafs)=m(\leafsa)=1$ and $m(\leafsc) = m(\leafsd) = 3$. In particular we have $L_{>1}(\tau) = \{ \leafsc^*, \leafsd^* \}$, and  for the current example we will write ${\nodec} := \leafsc^*$ and $\noded := \leafsd^*$. Finally, lets fix a pairing 
\[
P := 
\big\{
\{ (\leafs,1),  (\leafsc,1) \},\;
\{ (\leafsa,1), (\leafsd,1) \},\;
\{ (\leafsc,2), (\leafsd,2) \},\;
\{ (\leafsc,3), (\leafsd,3) \}  
\big\}.
\]
The resulting diagram $\Gamma_\tau^P$ can then visualised as
\[
\treePairingExampleA.
\]
\end{example}

It remains to specify a type map $\ft: (E_\tau^P)_\star \to \CL$ to obtain an element of $\CH$. On $K(\tau)$ we define $\ft$ to be equal to the type map of $\tau$. On edges $(q, u^*) \in E_u$ for $u \in L(\tau)$ and $q \in P[u]$ we set $\ft(q, u^*):= (\ft(u),k)$, where $k \le m(\ft(u))$ is the unique integer such that $(u,k) \in q$. Finally, we define $\ft(u^*, u^\downarrow) := (\ft(u),0)$ for any $u \in \Lgo(\tau)$. 

We then set
\begin{align}\label{eq:CW}
\ttf \tau := \sum_{\pairing \in \pairings} \Gamma^\pairing_\tau,
\end{align}
and we extend $\ttf$ to a linear operator $\ttf : \wTu \to \CH$.


Next we define a kernel assignment $\nkaFK \in \CKtminf(\tilde\CL)$ for any $\FK \in \CYN(\aFLm, \cm)$. For this we set for any $\snt \in \sFLm$ with $m(\snt) >1$
\[
(\nkaFK)_{\fdl_\ant} := \FK^\ant,
\]
where we identify $\{ 0, \ldots, m(\ant) \}$ with $\{ (\ant, 0) , \ldots , (\ant,m_\ant) \}$, so that the space (\ref{eq:CKtminf}) with $\fdl_\ant$ is naturally isomorphic to (\ref{eq:CYsimp}) with $n=m_\ant$. 
For $\snt \in \sFLm$ with $m(\snt)=1$ we set $(\nkaFK)_{\fdl_\ant} := \FKernel^\snt_0 \star \FKernel^\snt_1$.
We extend any element $K \in \CKtminf(\tilde\CL)$ to a kernel assignment $K \in \CKtminf(\CL)$ by defining $K_\ft$ to agree with the truncated integration kernel (see Section~\ref{sec:kernels}).
Finally, we fix a degree assignment $\deg_\infty : \CL \to [-\infty,0]$ 
such that $\deg_\infty \ft := -\infty$ for any $\ft \in \tilde\CL$. 
It then follows directly from the definition that one has the following identity.
\begin{lemma}\label{lem:evaluation:fd:trees}
One has
\[
(\Pi^{\nkaFK, R} \ttf \tau)(\psi_\tau) = \evaln{\eta}{\psi}{R} \tau
\]
for any $\tau \in \wTu$, any $\FK \in \CYN(\aFLm, \cm)$, any $R \in \CK_\infty^+$ and $\psi \in \Psi$. Here $\psi_\tau$ is as in (\ref{eq:FDT:1}). Here we write as above $\eta := \FKtN \FK$.
\end{lemma}
\begin{proof}
This follows almost directly from the definition. See the proof of Lemma~\ref{lem:identity:evaluations:trees:feynman:diagrams} below for a very similar statement.
\end{proof}

Our next goal is to show that a similar identity holds for the BPHZ renormalised evaluations. (Actually, this will only be true modulo an order one change in renormalisation, see below). As in Section~\ref{sec:non-gaussian-noises} we fix a homogeneity $\sfs: \aFLm \to \R_-$ with $\sfs(\ant) \ge -\shalf - \kappa$ (for some $\kappa>0$ small enough), and we define a degree assignment $\fddeg:\FDL \to \R_-$ by setting
\begin{equation}\label{eq:deg:trees:FD}
\begin{gathered}
\fddeg\ft := -|\fs|+|\ft|_\fs
\,,\qquad
\fddeg\fdl_\ant := \sfs(\ant) - m_\ant \shalf - |\fs| \I_{m_\snt>1}
\end{gathered}
\end{equation}
for any $\ft \in \FL_+$ and $\ant \in \aFLm$. We also set $\fddegmin (\snt,0) := - |\fs|-1+\kappa$ and $\fddegmin(\snt,k) := -|\fs|$ for any $k \ge 1$.
The degree assignments $\fddeg$ and $\fddegmin$ give us a natural norm on $\fdCKminf$ as in \eqref{eq:CKtmz:norm}, and we respect to this norm and (\ref{eq:CYNnorm}) the map $\nka : \SM_\infty(\aFLm, \cm) \to \fdCKminf(\tilde\CL)$ constructed becomes a bounded linear map. 
We extend $\sfs$ to $\wFLm$ by setting $\sfs(\legtype) := 0$ for any leg type $\legtype \in \Legtype$.
Then, a quick computation shows the following.
\begin{lemma}\label{lem:deg:sfs}
For any tree $\tau \in \wTu$ any pairing $\pairing \in \pairings$ one has $\fddeg \Gamma_\tau^\pairing = |\tau|_\sfs$ and $\Gamma_\tau^\pairing$ is an admissible Feynman diagram.
\end{lemma}

We still fix a tree $\tau \in \wTu$ and a pairing $P \in \pairings$.

\begin{definition}
We call a subtree $\sigma \ssq \tau$ \emph{closed} for $\pairing$ if for any $\{ (u,k), (v,l) \} \in \pairing$ one has either $\{ u,v\} \ssq L(\sigma)$ or $\{ u,v \}\cap L(\sigma) = \emptyset$.
\end{definition}

Let $\sigma \ssq \tau$ be a closed subtree. Then we denote by $\Gamma^P_\tau(\sigma) \ssq \Gamma$ the full connected subgraph of $\Gamma^P_\tau$ which is induced by the vertex set 
$$
V(\Gamma^P_\tau(\sigma)):= N(\sigma) \sqcup \{u^*: u \in \Lgo(\sigma) \} \sqcup\{q: q \in \pairing[u], u \in L(\sigma) \}.
$$ 
We show next that divergent subgraphs $\tilde\Gamma$ of $\Gamma^P_\tau$ correspond (almost) to closed divergent subtrees $\sigma$ of $\tau$. 

\begin{lemma}\label{lem:closed:subtrees}
Let $P \in \pairings$ and let $\sigma \ssq \tau$ be a closed subtree of $\tau$. Then one has $|\sigma_\fe^0|_\sfss = \fddeg (\Gamma^P_\tau(\sigma))^0_\fe$. 
Conversely, if $\tilde\Gamma \ssq \Gamma^P_\tau$ is a connected full subgraph such that $\fddeg \tilde\Gamma^0_\fe \le 0$, then either $\tilde\Gamma = \Gamma^P_\tau(\sigma)$ for a closed subtree $\sigma \ssq \tau$, or there does not exists an edge $e \in E(\tilde\Gamma)$ with $\ft(e) \in \FL_+$.
\end{lemma}
\begin{proof}
The first statement follows from Lemma~\ref{lem:deg:sfs}. For the second statement, let $E_\trees$ denote the set of edges $e \in \tilde E:= E(\tilde\Gamma)$ with $\ft(e) \in \FL_+$, and assume that $E_\trees \ne \emptyset$. 
Let moreover $E_\psi$ (resp. $E_\noise$) denote the set of edges $e \in E(\tilde\Gamma)$ with $\ft(e) = (\ant,0)$ (resp. $\ft(e) = (\ant,k)$, $k \ge 1$) for some $\ant \in \aFLm$. 
Let also $\tilde V := V(\tilde\Gamma)$. 
The set $\tilde V$ induces a subforest $\sigma \ssq \tau$ with 
$N(\sigma) := \tilde V \cap N(\tau)$, $K(\sigma) := \{ e \in K(\tau) : e \ssq N(\sigma) \}$, and 
$L(\sigma) := \{ u \in L(\tau) : u^\downarrow \in \tilde V \}$. We also set $L_\Legtype(\sigma) := \emptyset$. 

\smallskip\noindent\textbf{Special case.}~Assume that $E_\psi$ contains all edge $e \in E$ with $\ft(e) = (\ant,0)$ for some $\ant \in \aFLm$ and $e \cap \tilde V \ne \emptyset$.

Let $N \ge 1$ denote the number of connected components of $\sigma$. Let furthermore $H \ssq E_\noise$ denote the set of edges $e \in E_\noise$ which are ``hanging'' in the following sense. Since $e \in E_\noise$ one has $e = (q, u^*)$ for some $u \in L(\tau)$ and $q \in P[u]$, say $q = \{ (u,k) , (v,l) \}$. 
Consequently $(q, v^*) \in E(\Gamma^P_\tau)$, and we let $e \in H$ if $(q, v^*) \notin E_\noise$. 
Let finally $Q \ssq L(\sigma)$ denote the set of noise type edges $u \in L(\sigma)$ such that $u^\downarrow \in \tilde V$ but there exists $q \in P[u]$ such that $(q,u^*) \notin E_\noise$. The proof of the first step is finished if we can show that $N=1$ and $H=Q=\emptyset$. We have the bound
\begin{align*}
\fddeg \tilde\Gamma_\fe^0
&\ge
\sum_{e \in E_\trees} \deg(e)
+
\sum_{u \in L(\sigma)} \sfs( \ft( u ) )
+
\shalf \# H
+ 
|\sfs|( \# N(\sigma) - 1)  \\
&=
|\sigma|_\sfs
+
\shalf \# H
+
|\fs|(N-1),
\end{align*}
where $|\sigma|_\sfs$ denotes the sum of the homogeneities of each connected component $\hat\sigma$ of $\sigma$. Since $|\hat\sigma|_\sfs \ge -\shalf-\kappa$ for any tree $\hat\sigma \in \wT$ one has $|\sigma|_\sfs \ge -\shalf N - \kappa N$, so that, provided $\kappa$ and $\kappadelta$ are small enough, one has has $N=1$, since the left hand side is non-positive by assumption. It follows that $\sigma$ is a tree, and since $\sigma$ contains at least one kernel-type edge by assumption, one has $|\sigma|_\sfs > -\shalf$. Hence one has $\# H = 0$.

Let now denote by $a \ge 0$ the number of edges of the form $(q,u^*)$ for some $u \in Q$ and $q \in P[u]$ such that $(q,u^*) \notin E_\noise$. Then we have the bound
\begin{align*}
\fddeg \tilde\Gamma_\fe^0
&\ge
\sum_{e \in E_\trees} \deg(e)
+
\sum_{u \in L(\sigma) \setminus Q} \sfs( \ft( u ) )
+
\shalf a
+ 
|\sfs|( \# N(\sigma) - 1)  \\
&\ge
|\sigma|_\sfs
+
\shalf a,
\end{align*}
so that with the same argument as above one has $a=0$ and hence $Q = \emptyset$.

\smallskip\noindent\textbf{General case.}~Define $\hat\Gamma \ssq \Gamma^P_\tau$ as the subgraph induced by the edge set
\[
E(\hat\Gamma) 
:= E(\tilde\Gamma) \cup 
\{ e \in E(\Gamma^P_\tau) 
: \ft(e) = (\ant,0) \text{ for some }\ant \in \aFLm 
\text{ and } e \cap \tilde V \ne \emptyset  \}.
\]
Then $\hat\Gamma$ is a connected subgraph of $\Gamma$ and one has $\fddeg \hat\Gamma \le \fddeg \tilde\Gamma$. Hence $\hat\Gamma$ satisfies the conditions of the special case above, so that in particular $\hat\Gamma = \Gamma^P_\tau(\sigma)$ for some closed subtree $\sigma \ssq \tau$. Let now $e \in E(\hat\Gamma) \setminus E(\tilde\Gamma)$. Then necessarily $\ft(e) = (\ant,0)$ for some $\ant \in \aFLm$ and from the definition we infer $\deg \fdl_\ant < - |\fs|$. The first part of the proof shows that $e \ssq V(\tilde\Gamma)$, hence $\fddeg \tilde\Gamma > \fddeg \hat \Gamma + |\fs|>0$, in contradiction to the assumption. Hence we must have $\hat\Gamma = \tilde\Gamma$, and this concludes the proof.
%
%
\end{proof}

We denote by $g(\nkaFK)$ BPHZ characters on $\CH_-$ and by $g^{\FKtN\FK} \in \CGm$ the BPHZ character on $\CTm$. We introduce furthermore a character $g_\trees(L\FK)$ on $\CH_-$ which corresponds to $g^\eta$.
For this we introduce the canonical projection $\p_- : \hat\CH_- \to \CH_-$. 
We write $\fdi: \CH_- \to \hat\CH_-$ for the embedding which is a right inverse of $\p_-$ such that the range of $\fdi$ is given by the subalgebra of $\hat\CH_-$ generated by Feynman diagrams of non-positive homogeneity.
Furthermore, we write $\CH^\trees_-$ for the unital subalgebra of $\CH_-$ generated by connected vacuum Feynman diagrams $\Gamma$ such that there exists an edge $e \in E(\Gamma)$ with $\ft(e) \in \FL_+$,
and we denote by $\CJ_\noise$ and $\CH_-^\noise$ (resp. $\hat\CJ_\noise$ and $\hat\CH_-^\noise$) the ideal and unital subalgebra of $\CH_-$ (resp. $\hat\CH_-$) generated by Feynman vacuum diagrams $\Gamma$ with the property $\ft(e) \in \tilde\CL$ for any edge $e \in E(\Gamma)$. 
Finally, we define $\p_\trees : \CH_- \to \CH_-^\trees$ as the multiplicative projection which is the identity on $\CH_-^\trees$ and annihilates $\CJ_\noise$.

With this notation we define $g_\trees(\nkaFK)$ as the unique character on $\CH_-$ which satisfies the two relations
\begin{align}\label{eq:character:g:trees}
(g_\trees(\nkaFK) \otimes \Pi^{\nkaFK}) \cpmh \fdi & = 0	\qquad \text{ on } \CH_\trees \\
\label{eq:character:g:trees:noise}
g_\trees(\nkaFK) & = 0								\qquad \text{ on } \CJ_\noise.
\end{align}
Let finally $h(\nkaFK)$ be the character defined by
\[
h(\nkaFK) \circ g_{\trees}(\nkaFK) = g(\nkaFK).
\]
We would like to show that $h$ is bounded in the character group uniformly over $\FK \in \CYN(\sFLm,\cm)$ such that $\|\FK \|_\sfs \le C$. This is not quite true, however it is true if $h$ is restricted to the Hopf subalgebra $\CHms \ssq \CH_-$ generated by all connected vacuum diagrams of the form $\Gamma^P_\tau$ for some $\tau \in \wTu$ and some $P \in \pairings$. 
\begin{lemma}\label{lem:gamma:rep:var}
Let $\Gamma \in \CH_-^\noise \cap \CHms$ be a connected vacuum diagram. Then either $\Gamma$ contains exectly one edge $e$, and one has $\ft(e)=(\snt,1)$ for any $\snt \in \sFLm$ with $m_\snt=1$, or $\Gamma$ ``represents a covariance'' in the sense that $\Gamma$ is of the form
\begin{align}\label{eq:gamma:1}
\Gamma = \treeExampleVariance \,.
\end{align}
Here we coloured edges $e$ with weak types $\ft(e)$ belong to the same element of the partition $\CL_\fdl$ in the same colour.
\end{lemma}
\begin{proof}
Let $\tau \in \wTu$ and fix a pairing $P \in \pairings$. Assume that we are given a family $L=(L_i)_{i\le n}$ of disjoint subsets $L_i \ssq L(\tau)$, $i=1, \ldots , n$ such that each $L_i$ is ``closed'' under $P$ in the sense that whenever $\{ (u,k) , (v,l) \} \in P$ for any $u \in L_i$ and $k,l \in \N$, then one has $v \in L_i$ as well. 
Each set $L_i$ defines a connected subgraph $\Gamma_i \ssq \Gamma^P_\tau$, $i=1, \ldots, n$,  induced by the vertex set $V(\Gamma_i) := \{ u^* , u^\downarrow, q : u \in L_i, q \in P[u] \}$.
Then we define the ``contraction'' $\Gamma^P_\tau|L \in \CH_-^\trees$ by contracting each subdiagrams $L_i$ to a single vertex.

We claim that the unital algebra $\CHmsa \ssq \CH_-$ generated by connected vacuum diagrams of the form $\Gamma^P_\tau|L$ where $L$ is as a (possible empty) family as above and (\ref{eq:gamma:1}) forms a Hopf algebra. This immediately concludes the proof.

Fix $\tau \in \wTu$ and $P \in \pairings$ and let $\CF$ be a forest of $\Gamma^P_\tau$, i.e.\ $\CF$ is a collection of node-disjoint subgraphs $\tilde\Gamma \ssq \Gamma^P_\tau$ such that $\fddeg \tilde\Gamma_\fe^0 \le 0$. We first show that for any $\tilde\Gamma \in \CF$ and any polynomial decoration $\fn : V(\tilde\Gamma) \to \N^d$ one has $\tilde\Gamma_\fe^\fn \in \CHmsa$. 
First it follows with the same arguments as in the proof of Lemma~\ref{lem:closed:subtrees} that $\tilde\Gamma$ is admissible. If $\tilde\Gamma \in \CH^\trees_-$ then $\tilde\Gamma = \Gamma^P_\tau(\sigma)$ for some subtree $\sigma \ssq \tau$ by Lemma~\ref{lem:closed:subtrees}, and the latter is element of $\CHmsa$ by definition. Otherwise there exists $\tilde\FDL \ssq \FDL$ such that $\ft(E(\tilde\Gamma)) = \bigsqcup_{\fdl \in \tilde\FDL} \CL_\fdl$, and it follows that
\[
\fddeg \tilde\Gamma \ge -n(\shalf+\kappa) + (n-1)|\fs|,
\]
where $n := \# \tilde\FDL$. This can only be negative if $n\le 2$. If $n=1$, say $\tilde\FDL = \{ \fdl_\snt \}$, then $\fddeg\tilde\Gamma = \sfs(\snt) + m_\snt \shalf$, so that $m_\snt=1$ and hence $\Gamma$ contains a single leg. Otherwise $n=2$ and hence $\Gamma$ is of the form \eqref{eq:gamma:1}. 

Now let $\hat\Gamma$ be the Feynman diagram generated by contracting each graph $\tilde\Gamma \in \CF$ to a single vertex. Since the operation of contracting node disjoint subgraphs is commutative, we can first contract those elements $\tilde\Gamma$ of $\CF$ for which $\tilde\Gamma \in \CH_-^\trees$. 
By Lemma~\ref{lem:closed:subtrees} for any such $\tilde\Gamma$ there exists a closed subtree $\sigma(\tilde\Gamma) \ssq \tau$ such that $\tilde\Gamma = \Gamma^P_\tau(\sigma(\tilde\Gamma))$. 
The Feynman diagram resulting from this contractions is then again of the form $\Gamma^{\hat P}_{\hat \tau}$, where $\hat\tau \in \wTu$ is the tree obtained by contracting each $\sigma(\tilde\Gamma)$ to a single vertex and $\hat P$ is the pairing induced by $P$. 
To proceed we can hence assume that each connected component $\tilde\Gamma$ of $\CF$ is an element of $\CH_-^\noise$. Then each $\tilde\Gamma$ induces a subset $L(\tilde\Gamma) \ssq L(\tau)$ by setting $L(\tilde\Gamma) := \{ u \in L(\tau) : u^\downarrow \in V(\tilde\Gamma) \}$ and thus $\hat\Gamma = \Gamma^P_\tau|L$ and this concludes the proof.
\end{proof}

The next lemma shows that when restricted to $\CHms$ the character $h$ is uniformly bounded.
\begin{lemma}\label{lem:bound:character}
For any $C>0$ the character $h(\nkaFK)$ restricted to the Hopf subalgebra $\CHms$ is bounded uniformly over all noises $\FK \in \CYN(\sFLm, \cm)$ with $\|\FK\|_{\sfs} \le C$
\end{lemma}
\begin{proof}
By definition of the BPHZ character $g(\nkaFK)$ one has
\begin{align*}
0 = \big(g(\nkaFK) \otimes \Pi^{\nkaFK} \big) \fdcp \fdi
&=
\big(h \otimes g_\trees(\nkaFK) \otimes \Pi^{\nkaFK}\big) (\fdcph \otimes \Id) \fdcp \fdi \\
&=
\big(h \otimes (g_\trees(\nkaFK) \otimes \Pi^{\nkaFK}) \fdcp\big) \fdcp \fdi \;.
\end{align*}
Since $(g_\trees(\nkaFK) \otimes \Pi^{\nkaFK}) \fdcp = \Pi^{\nkaFK}$ on $\hat\CH_-^\noise$, it follows that $h = g$ on $\CH_-^\noise$. 
The fact that $|g(\nkaFK) \Gamma| \lesssim 1$ for any $\Gamma \in \CH_-^\noise \cap \CHms$ is straightforward from the definitions. 

 We now show inductively in the number of edges of connected Feynman diagrams $\Gamma$, and in the quantity $\sum_{u \in N(\Gamma)} |\fn(u)|_\fs$, that one has $h(\nkaFK)\Gamma = 0$ for any $\Gamma \in \CH_-^\trees$. Indeed, one has
\begin{align}
h(\nkaFK)\Gamma &= -\big( h(\nkaFK) \otimes (g_\trees(\nkaFK) \otimes \Pi^{\nkaFK}) \fdcp ) (\fdcp - \Id \otimes \one) \fdi \Gamma
\\
&\label{eq:h:CK}
= -\big( h(\nkaFK) \p_\noise \otimes (g_\trees(\nkaFK) \otimes \Pi^{\nkaFK}) \fdcp ) (\fdcp - \Id \otimes \one) \fdi\Gamma,
\end{align}
where we used the induction hypothesis in order to get the projection $\p_\noise$ onto $\CH_-^\noise$ in the last line. 

One has
\begin{align}\label{eq:fd:1}
(\p_\noise \otimes \Id) (\fdcp - \Id \otimes \one) \fdi \Gamma
\ssq
\CH_-^\noise \otimes \hat\CH_-^\trees.
\end{align}
(Note that the second component contains an edge $e$ of type $\ft(e) \in \FL_+$, otherwise such an edge would be in the left component and thus the term would be killed by the projection.) 
Since $(g_\trees(\nkaFK) \otimes \Pi^{\nkaFK}) \cpmh \fdi$ vanishes on $\CH_-^\trees$ by definition, it remains to show that the right component of \eqref{eq:fd:1} is of non-positive degree. 
But this follows for $\kappa>0$ small enough, since 
from the fact that any connected diagram $\Gamma \in \CH_-^\trees$ satisfies $\fddeg \Gamma \ge - n \kappa$ where $n > 0$ denotes the number of nodes $w \in V(\Gamma)$ of the form $w=u^*$ for some $u \in L(\tau)$. (We omit the details of this argument which are very similar to the one carried out in the proof of Lemma~\ref{lem:closed:subtrees}.) 
\end{proof}

Finally, we have the following relation between the renormalised valuations on $\CH$ and $\wTu$.

\begin{proposition}\label{prop:trees:FD}
One has the identity
\begin{align}\label{eq:characters}
g_\trees(\nkaFK) \ttf \qNL = g^{\eta}
\end{align}
on $\wCTu$. Here $\qNL : \wCTu \to \CTm$ denote the projection which kills trees $\tau \in \wCTu$ such that $L_\Legtype(\tau) \ne \emptyset$, and we write as above $\eta := \FKtN \FK$.  Moreover, one has
\begin{align}\label{eq:renorm:valuations}
(\Pi^{\nkaFK,R} M^{g_\trees(\nkaFK)}\ttf \tau)(\psi_\tau)  =  \evaln{\eta}{\psi}{R} M^{g^{\eta}} \tau 
\end{align}
for any $\tau \in \wTu$, any $\FK \in \CYN(\aFLm, \cm)$, any $R \in \CK_\infty^+$ and $\psi \in \Psi$. Here $\psi_\tau$ is as in (\ref{eq:FDT:1}).
\end{proposition}
\begin{proof}
We show (\ref{eq:characters}). The character $g^{\eta}$ when viewed as a character of the Hopf algebra $\wCTu$ is determined by the relations
\begin{align*}
(g^\eta \otimes \Upsilon^\eta) \cpmi &= 0  	&&\quad \text{ on }\aCTm
\\
g^\eta &= 0									&&\quad \text{ on } \wCJlegs
\end{align*}
where $\wCJlegs \ssq \wCTu$ denote the ideal generated by trees $\tau \in \wCTu$ that contain legs. The second identity holds for the character $g_\trees(\nkaFK) \ttf \qNL$ by definition of $\qNL$. To see the first one, note that
\[
(g_\trees(\nkaFK) \ttf \qNL \otimes \Upsilon^\eta ) \cpmi
=
(g_\trees(\nkaFK) \otimes \Pi^{\nkaFK}) (\ttf \qNL \otimes \ttf) \cpmi
\]
on $\aCTm$. 
From Lemma~\ref{lem:closed:subtrees} and the definition of the respective coproducts we infer that
\[
(\ttf \qNL\otimes \ttf )\cpmi = (\p_\trees \otimes \Id)\fdcp \fdi \ttf
\]
on $\aCTm$, which together with (\ref{eq:character:g:trees}) concludes the proof.

The proof for (\ref{eq:renorm:valuations}) is very similar. One has
\[
\evaln{\eta}{\psi}{R} M^{g^\eta} \tau 
=
(g_\trees( \nkaFK) \ttf \qNL \otimes \Pi^{\nkaFK}\ttf) \cpm
\]
on $\wTu$, where we used (\ref{eq:FDT:1}) and (\ref{eq:characters}). Using the identity
\[
(\ttf \qNL\otimes \ttf )\cpm = (\p_\trees \otimes \Id)\fdcp \ttf
\]
on $\wTu$ concludes the proof.
%
%
\end{proof}

As an important application of this construction we proof Theorem~\ref{thm:evaluation:trees:largescale}

\begin{proof}[of Theorem~\ref{thm:evaluation:trees:largescale}]
We fix a tree $\tau \in \wT$ and assume without loss of generality that leg types and noise types are unique in $\tau$, so that $\tau \in \wTu$.

By Lemma~\ref{lem:evaluation:fd:trees} the continuous extension of $R \mapsto \evaln{\eta}{\psi}{R}$ to $\CK_0^+$ is a consequence of the continuous extension of the evaluation $R \mapsto \Pi^{\ttf \eta, R}$ which in turn is the content of the first part of Theorem~\ref{thm:BPHZ:FD}. 

The continuous extension of the map $(\eta,R) \mapsto \hat{\Upsilon}^{\eta,\legfn}_R \tau $ to the space $\SMstarz \times \CK_0^+$ is a consequence of (\ref{eq:renorm:valuations}), Lemma~\ref{lem:bound:character} and the second part of Theorem~\ref{thm:BPHZ:FD}.
\end{proof}

Another consequence is the following corollary for which we assume as in Section~\ref{sec:non-gaussian-noises} that we are given a set of types $\sFLm$ such that $\FLm \ssq \sFLm$ and a homogeneity assignment $\sfs:\sFLm \to \R_-$.

\begin{corollary}\label{cor:log-void}
Assume that Assumption~\ref{ass:technical} holds. Let $N \in \N$ and let $\FKernel=(\FKernel^\snt_\cm) \in \CYN$, and assume that 
\begin{itemize}
\item one has $\FKernel^\snt_\cm \in \CYzsimpp{\#\cm}$ for any $\snt \in \sFLm \setminus \FLm$ and any mulitset $\cm$ (see Definition~\ref{def:CYzsimp} for the definition of this space), and
\item one has $\FKernel^\Xi_\cm =0$ for any $\Xi \in \FLm$ and multiset $\cm$ with $\#\cm>1$.
\end{itemize}
Fix $\tau \in \aCT$  such that $\fancynorm{\tau}_\sfs=0$ and let $\sFLmz:= (\sFLm \setminus \FLm) \cap \ft(L(\tau))$.
Let finally $\eta := \FKtN(\FKernel) \in \SMinf$ and denote  by $g^\eta \in \sCGm$ the BPHZ character for $\eta$. Then for any $\snt, \snta \in \sFLmz$ there exists $\theta>0$ such that one has
\begin{align}\label{eq:log-av}
|g^{\eta} \tau| \lesssim 
\supp_{\cm,\tilde\cm}
	\| \FK^\snt_\cm \|_{ \beta^\snt_\cm + \theta}
	\| \FK^\snta_\cm \|_{ \beta^\snta_{\tilde\cm} - \theta}
\end{align}
uniformly over all $\FKernel$ as above such that $\|\FKernel\|_\sfs \le C$. Here the  supremum runs over all multisets $\cm,\tilde\cm$ with values in $\FLm$ and such that $\#\cm \lor \#\tilde\cm \le N$. (See Definition~\ref{def:CYN} for the definition of $\beta^\snt_\cm$.)
\end{corollary}
\begin{proof}
Fix a tree $\tau \in \aCT$. We can assume that any noise type $\snt \in \ft(L(\tau))$ is unique, so that $\ft(L(\tau))$ is a proper set. 
Note that $\FK$ can be written as a finite sum $\FK = \sum_{\cm} [\FK]_\cm$, where the sum runs over all families $\cm = (\cm_\snt)$, $\snt \in \sFLm$, of multisets $\cm_\snt$ with values in $\FLm$ and such that $\#\cm_\snt \le N$, and where $[\FK]_\cm \in \CYN(\sFLm, \cm)$ (see Section~\ref{sec:application:to:trees}). Since one has $g^\eta\tau = \sum_\cm g^{[\eta]_\cm}\tau$, where $[\eta]_\cm := \FKtN ([\FK]_\cm)$, it suffices to assume that $\FK \in \CYN(\sFLm, \cm)$.
By (\ref{eq:characters}) it then suffices to bound $g_\trees(\nkaFK) \ttf \tau$, and by definition of $\ttf$ in (\ref{eq:CW}) is suffices to fix a pairing $P\in\pairings$ and bound $g_\trees(\nkaFK) \Gamma_\tau^P$.

We want to use the second statement of Corollary~\ref{cor:no-log-divergence}, which is formulated in terms of $g(\nkaFK)$ rather than $g_\trees(\nkaFK)$. We recall that one has
\[
g_\trees(\nkaFK) = h(\nkaFK)^{-1} \circ g(\nkaFK),
\]
and by Lemma~\ref{lem:bound:character} the character $h(\nkaFK)$ is uniformly bounded when restricted to $\CHms$. It follows from the proof of Lemma~\ref{lem:bound:character} that $h(\nkaFK)$ vanishes on $\CH^\trees_- \cap \CHms$. Moreover, $h(\nkaFK)\tilde\Gamma = 0$ for any subgraph $\tilde\Gamma \in \CHms \cap \CH^\noise_-$ which contains a type $\fdl_\snt \in \FDL$ with $\snt \in \sFLm \setminus \FLm$. 
To see this, recall from Lemma~\ref{lem:gamma:rep:var} that $\tilde\Gamma$ represents a variance as in (\ref{eq:gamma:1}), so that $\tilde\Gamma$ has no proper subdivergences, and $g(\tilde\Gamma) = 0 = g_\trees(\tilde\Gamma)$, where the first equality follows from $\FKernel^\snt_{\cm_\snt} \in \CYzsimpp{\#\cm_\snt}$ and the second equality follows from (\ref{eq:character:g:trees:noise}). 
It follows that the only subgraphs $\tilde\Gamma$ of $\Gamma^P_\tau$ on which $h(\nkaFK)\tilde\Gamma^{\tilde\fn}_\fe$ does not vanish have the property that every edge $e \in E(\tilde\Gamma)$ is of type $\ft(e) = (0,\Xi)$ for some $\Xi \in \FLm$.
We denote by $\CE \ssq E(\Gamma^P_\tau)$ the set of edges $e \in E(\Gamma^P_\tau)$ with the property,
such that $\{\ft(e)\} = (\Xi,0)$ with $\Xi \in \FLm$,
and we write $A(\tau,P)$ for the collection of diagrams of the form $\Gamma^\tau_P|\tilde\CE$ which are obtained from $\Gamma^\tau_P$ by fixing $\tilde\CE \ssq \CE$ and contracting each edge in $\tilde\CE$ to one vertex. Then, this paragraph implies in particular that one has 
\begin{align}\label{eq:AlgA}
(h(\nkaFK) \otimes \Id) \Delta_- \linspace{A(\tau,P)} \ssq \linspace{A(\tau,P)}.
\end{align}

It now remains to show that for any fixed $\Gamma \in A(\tau,P)$ one has that $g(\nkaFK) \Gamma $ is bounded by the right-hand side of (\ref{eq:log-av}).
For this we note first that $\CKtmnorm{\nkaFK} \lesssim \| \FKernel \|_\sfs \lesssim 1$. Fix now $\snt, \snta \in \sFLmz$ and let $\fdly := \fdl_\snt$ and $\fdlz := \fdl_{\snta}$, see \eqref{eq:fdl:trees}. Then, one has $\fddeg \fdly = \beta^\snt_\cm - |\fs|$ and $\fddeg \fdlz = \beta^\snta_\cm - |\fs|$ and by definition (\ref{eq:CKtmz:norm}) and (\ref{eq:CYsimp:norm:2}) one has
\[
\CKfdlnormm{ (\nkaFK)_{\fdly} } {\fddeg \fdly + \theta} 
= 
\| \FK^\snt_{\cm_\snt} \|_{\beta^\snt_{\cm_\snt} + \theta}
\quad\text{ and }\quad
\CKfdlnormm{ (\nkaFK)_{\fdlz} } {\fddeg \fdlz - \theta} 
= 
\| \FK^\snta_{\cm_\snta} \|_{\beta^\snta_{\cm_\snta} - \theta}.
\]

It remains to show that $\nkaFK$ is log-avoiding for $\Gamma$ and $\fdl_\snt \in \CL$ for any $\snt \in \sFLmz$. 
First note that $\fddeg \Gamma = \fddeg \Gamma^P_\tau = \fancynorm\tau_\sfs = 0$. (In the first equality we used that $\fddeg \ft(e) = -|\fs|$ for any $e \in \CE$, the second equality holds by construction and the third by assumption.)
Fix a type $\snt \in \sFLmz$ and let 
$\tilde\Gamma \ssq \Gamma$ be a subgraph with $\ft(E(\tilde\Gamma)) \cap \CL_{\fdl_\snt} \ne \emptyset$ and let $\tilde\fn \in \CN(\tilde\Gamma)$ 
be a polynomial decoration such that $\fddeg \tilde\Gamma^{\tilde\fn}_\fe = 0$. We have to show that $g(\nkaFK) \tilde\Gamma^{\tilde\fn}_\fe = 0$.
If $\tilde\Gamma \in \CHms \cap \CH_-^\noise$, then $g(\tilde\Gamma^{\tilde\fn}_\fe)=0$ with the same argument as above. 
Otherwise, note that $\tilde\Gamma$ can be written as $\hat\Gamma|(\tilde\CE \cap E(\hat\Gamma))$ for some subdiagram $\hat\Gamma$ of $\Gamma^P_\tau$ and $\tilde\CE$ as above, where $E(\hat\Gamma):=E(\tilde\Gamma) \sqcup\{ e\in \tilde\CE : e \ssq V(\tilde\Gamma) \}$.
By Lemma~\ref{lem:closed:subtrees} there exists a closed subtree $\sigma \ssq \tau$ such that $\hat\Gamma = \Gamma^P_\tau(\sigma)$, and it follows that $\tilde\Gamma^{\tilde\fn}_\fe \in A(\sigma^{\tilde\fn}_\fe,P)$. 
We can assume that $\sigma$ is connected to its complement in $\tau$ with at least two nodes (otherwise one has $g(\nkaFK)\Gamma=0$ and there is nothing to show).
 By (\ref{eq:AlgA}) it suffices to show that $g_\trees(\nkaFK)$ vanishes on $A(\sigma^{\tilde\fn}_\fe, P)$.
By (\ref{eq:characters}) one has $g_\trees(\nkaFK) \hat\Gamma^{\tilde\fn}_\fe = g^\eta \sigma^{\tilde\fn}_\fe = 0$, where the last equality follows from Assumption~\ref{ass:technical}, which shows the required identity for $\tilde\CE = \emptyset$. 
In general, we obtain the identity via a limit argument. Let $\hat\CE := \tilde\CE \cap E(\hat\Gamma)$ and let $\FK^\eps$ be defined by setting, for any $\Xi \in \FLm$ such that $(\Xi,0) \in \ft(\tilde\CE)$,
\[
(\FK^\eps)_{\cm_\Xi}^\Xi := \rho^\eps,
\]
where $\rho^\eps \to \delta_0$ as $\eps \to 0$ (note that $\#\cm_\Xi =1$ by assumption). For any $\snt \in \sFLm $ which is not of this type we set $(\FK^\eps)^{\snt}_{\cm_\snt} := \FK^{\snt}_{\cm_\snt}$.
Then one has 
\[
g_\trees(\nkaFK^\eps) \hat\Gamma^{\tilde\fn}_\fe = g^{\eta^\eps} \sigma^{\tilde\fn}_\fe = 0
\]
for any $\eps>0$, where $\eta^\eps := \FKtN (\FK^\eps)$, and on the other hand one has
\[
g_\trees(\nkaFK^\eps) \hat\Gamma^{\tilde\fn}_\fe \to 
g_\trees(\nkaFK)	\hat\Gamma^{\tilde\fn}_\fe|\hat\CE \,, \qquad\text{ as }\eps \to 0,
\]
which concludes the proof.
\end{proof}

\section{Proof of Proposition~\ref{prop:eps:beta:bound}}
\label{sec:proof:of:eps:beta:bound}

We will work with Proposition~\ref{prop:eps:beta:bound}, although this proposition is not formulated in the most natural way and not well adapted to the proof we will give below. We will now state a more general (but essentially equivalent) formulation. For this we start with the following definition.
\begin{definition}
Given a system $\LTsys \in \LTsyss$, a compact set $K \ssq \bar\domain$, and $\delta>0$ we define the set $\SN(\LTsys,K,\delta) \ssq \SN$ as the set of $\phi \in \SN(\LTsys,\delta)$ such that additionally one has $\supp \phi_\cw \ssq K$ for any $\cw \in \CW$.
On the space $\SN(\LTsys,K,\delta)$ we introduce the norm $\|\cdot\|_{\LTsys}$ given as the smallest constant such that
\begin{align} \label{eq:bound:SN}
|D^k \phi_\cw (x)| \le \|\phi\|_{\LTsys} \,|x|_\fs ^{\LTsysdeg\cw - |k|_\fs}
\end{align}
for any $\cw \in \CW$, $k \in \N^d$ with $|k|_\fs < r$, and $x \in \bar \domain$. 
\end{definition}

With this notation, we will show below the following proposition.

\begin{proposition}\label{prop:eps:beta:bound:2} 
Let $\tau \in \adCT$, let $\LTsys \in \LTsyss(\tau)$, let $K \ssq \bar\domain$ be a compact set, let $\delta>0$, and let $\LTsysdeg$ be the degree assignment defined in (\ref{eq:definition:deg:FJ}). 
Then one has for any $C>0$ the bound
\begin{align}\label{eq:bound:eps:beta:1}
\big|
\big(
	\LchR	\otimes
	\eval{\phi} {\tilde R}
\big)
	\cpmwi
	\tau
\big|
 \lesssim 
1
\end{align}
uniformly over $\phi \in \SN(\LTsys, K, \delta)$ and $R, \tilde R \in \CK^+_0$ with $\|\phi\|_{\LTsys} \lor \|R\|_{\CK^+} \lor \|\tilde R\|_{\CK^+} < C$.
\end{proposition}

The proof of this proposition is the content for the next three sections, see in particular Section~\ref{sec:proof:eps:beta:bound} below.
We end this section by showing that Proposition~\ref{prop:eps:beta:bound} follows from Proposition~\ref{prop:eps:beta:bound:2}.

\begin{proof}[of Proposition~\ref{prop:eps:beta:bound}]
We apply Proposition~\ref{prop:eps:beta:bound:2} for $\boxnorm\cdot_\fs^{(\beta)} := \boxnorm\cdot_\fs - \beta$, where $\beta>0$ is small enough so that one still has $\boxnorm\Xi_\fs^{(\beta)} > \homos\Xi$ for any $\Xi \in \FL_-$. We denote by $\deg^{\LTsys,\beta} : \CW \to \R_- \cup\{0\}$ the degree assignment defined as in (\ref{eq:definition:deg:I}) and (\ref{eq:definition:deg:FJ}) but with $\boxnorm\cdot_\fs$ replaced by $\boxnorm\cdot_\fs^{(\beta)}$, and we write $\| \phi \|_{\LTsys,\beta}$ for the norm defined as in (\ref{eq:bound:SN}) with $\deg^\LTsys$ replaced by $\deg^{\LTsys,\beta}$.
We choose a compact set $K \ssq \bar\domain$ that supports the functions $\phi_\cw^\eps$ for any $\cw \in \CW$ and $\eps>0$, so that one has $\phi^\eps \in \SN(\LTsys,K,\delta)$ for any $\eps \in (0,1]$.

Let now $\tilde \cw \in \CW$ be such that $\tilde\cw \in \ft(L_\Legtype(\tau))$ and $\tilde\cw \in \LTa \in \LTsys$, and define for $0<\eps\le 1$ the tuple $\tilde\phi^\eps \in \SN(\LTsys,K,\delta)$ by setting
$\tilde \phi^\eps_{\tilde\cw} := \eps^{-\frac{\beta}{2}} \phi^\eps_{\tilde\cw}$ and $\tilde \phi^\eps_{\cw} := \phi^\eps_{\cw}$ for any $\cw \in \CW \backslash \{ \tilde \cw \}$. It follows that $\|\tilde \phi^\eps \|_{\LTsys,\beta} \lesssim 1$. Since moreover $\deg_\infty$ was chosen in such a way that $\|\hat K - K\|_{\CK^+}$ is finite, it follows from (\ref{eq:bound:eps:beta:1}) that one has 
\begin{align}\label{eq:bound:eps:beta:tilde}
\big|
\big(
\LcheR
	\otimes
	\eva{\tilde\phi^\eps}
\big )
\cpmwi
	\tau
\big| 
\lesssim 1,
\end{align}
for $ R \in \{ 0 , \hat K - K \}$. 

It remains to show that the left-hand side of (\ref{eq:bound:eps:beta:tilde}) is equal to $\eps^{-\frac{\beta}{2}}$ times the left-hand side of (\ref{eq:bound:eps:beta}). For this let $\CF$ be a subforest of $\tau$ and choose decorations $\fn_\CF$ and $\ce_\CF$ as in (\ref{eq:iforest:decoration:convention}). We then distinguish two cases. In the first case, one has $\tilde\lt \in \ft(L_\Legtype( T/\CF ))$, and it follows that $\eva{\tilde\phi^\eps} (T/\CF)^{\fn- \fn_\CF}_{\fe + \fe_\CF} = \eps^{-\frac{\beta}{2}} \eva{\phi^\eps} (T/\CF)^{\fn- \fn_\CF}_{\fe + \fe_\CF}$. In the second case, there exists $S \in \bar\CF$ such that $\tilde\lt \in \ft(L_\Legtype(S))$, and in this case it follows that 
$
\LcheR S^{\fn_\CF}_{\fe_\CF}
=
\eps^{-\frac{\beta}{2}}
\LcheR S^{\fn_\CF}_{\fe_\CF}.
$
\end{proof}

\subsection{Feynman diagrams}

We are going to show Proposition~\ref{prop:eps:beta:bound:2} by applying the results of \cite{Hairer2017}. To this end we recall the notation of Section~\ref{sec:Feynman-diagrmas} about Feynman diagrams, which we are going to apply to the type set $\CL:= \ltfd \sqcup \FL_+$, where we define $\ltfd$ as the set of all $(\lt, \bar\lt) \in \Legtype\times \Legtype$ with $\lt < \bar \lt$. Fix  a system $\LTsys\in\LTsyss$. We then define a degree assignment $\LTsysdeg$ on $\CL$ by setting $\LTsysdeg(\lt,\bar \lt) :=2\LTsysdeg(\lt)$ for $(\lt,\bar\lt) \in \ltfd$ and $\LTsysdeg \ft:= |\ft|_\fs - |\fs|$ for any kernel type $\ft \in \FL_+$. 

Given  an element $\phi \in \SN$ and a large-scale kernel assignment $R \in \CK^+_0$ we define
\begin{align}\label{eq:feynman:diagram:kernel:assi}
\CK_{\ft} :=
\begin{cases}
	K_\ft 									&\quad\text{ if } \ft \in \FL_+ \\
	\phi_{\lt}	&\quad\text{ if } \ft=(\lt, \bar\lt) \in \ltfd
\end{cases}
\qquad \text{ and } \qquad
\CR_{\ft} :=
\begin{cases}
	R_\ft		&\quad\text{ if } \ft \in \FL_+ \\
	0						&\quad\text{ if } \ft \in \ltfd,
\end{cases}
\end{align}

In the notation of Proposition~\ref{prop:eps:beta:bound:2}, let $\CK=\CK(\phi)$ be defined as in (\ref{eq:feynman:diagram:kernel:assi}) from some tuple $\phi \in \SN(\LTsys, K, \delta)$, and let $\CR$ and $\tilde \CR$ be the large scale kernel assignments defined as in (\ref{eq:feynman:diagram:kernel:assi}) from $R$ and $\tilde R$. Then we have the following result, which is an immediate Corollary of \cite[Thm.~4.3]{Hairer2017}. 
\begin{theorem}
Assume that $\Gamma \in \hat\CH_-$ is a connected vacuum diagram that has the property described in Lemma~\ref{lem:renorm:large:scale:full} and let $\CK=\CK(\phi)$ be as above. Then for any $C>0$ one has the bound
\begin{align}\label{eq:Feynman:diagram:bound}
|( g^\full (\CK) \otimes \Pi^{\CK,\tilde \CR} ) \Delta_-^\full \Gamma | \lesssim 1
\end{align}
uniformly over $\phi \in \SN(\LTsys,K,\delta)$ and $\tilde \CR \in \CK_\infty^+$ with $\|\phi\|_{\LTsys} \lor \|\tilde \CR\|_{\CK^+} \le C$.
\end{theorem}
\begin{proof}
Comparing this formulation to \cite[Thm.~4.3]{Hairer2017}, we only need to note that one has $\|\CK(\phi)\|_{	\CK^- } \lesssim \| \phi \|_{\LTsys}$ uniformly over all $\phi \in \SN(\LTsys, K, \delta)$, where the norm $\| \cdot \|_{\CK^-}$ is defined as the smallest constant such that
\[
|D^k \CK_\ft(x)| 
\le 
\| \CK \|_{\CK^-} 
|x| ^{\deg^\FJ \ft - |k|_\fs}
\]
for any $\ft \in \CL$ and $k \in \N^{d+1}$ with $|k|_\fs < r$ (c.f.\ \cite[Eq.~2.2]{Hairer2017}). 
\end{proof}

\subsection{Embedding the tree algebra into the Feynman diagram algebra}
We now construct for any properly legged tree $\tau \in \plCThat$ a Feynman vacuum diagram $\Gamma(\tau) := (V_\Gamma(\tau),E_\Gamma(\tau))$ together with the necessary decorations $\fl : E_\Gamma(\tau) \to \CL$ and $\fn : V_\Gamma(\tau) \to \N^d$. To this end, we first introduce the notation that for $e \in E_\Gamma(\tau)$ we write $e^+,e^- \in V_\Gamma(\tau)$ for the two vertices such that $e$ is an edge from $e^-$ to $e^+$\footnote{We do not identify an edge $e$ with the pair $(e^-,e^+)$, since we will have to consider multiple edges between the same pair of vertices.}. The total order $\le$ on $\Legtype$ induces a total order $\le$ on $L_\Legtype(\tau)$, and we define $\CE_\Legtype(\tau)$ as the set of all ordered pairs $(e, \bar e)$ with $e\le \bar e$ (recall that $\bar e$ denotes the partner of $e$). We interpret any $(e,\bar e) \in \CE_\Legtype(\tau)$ as an edge by setting $(e, \bar e)^- := e^\downarrow$ and $(e, \bar e)^+ := \bar e^\downarrow$, and with this notation we set
\[
V_\Gamma(\tau) := N(\tau)
\quad \text{ and } \quad
E_\Gamma(\tau) := \CE_\Legtype(\tau) \sqcup K(\tau).
\]
The decoration $\fn$ is then taken over from $\tau$, and the decoration $\fl : E_\Gamma(\tau) \to \CL$ is defined by setting $\fl(e):=\ft(e)^{(\fe(e))}$ for any $e\in K(\tau)$ and $\fl(e, \bar e):= (\ft(e),\ft(\bar e)) ^{(\fe(e)+\fe(\bar e))}$ for any $(e,\bar e) \in \CE_\Legtype(\tau)$. We finally specify that the distinguished vertex in $V_\Gamma(\tau)$ is given by $v_\star := \rho(\tau) \in V_\Gamma(\tau)$. This specifies a Feynman vacuum diagram, and we summarise this in the following lemma.

\begin{lemma}
For any properly legged tree $\tau \in \plCThat$ the vacuum diagram $\Gamma(\tau)=(\Gamma(\tau),\fn,\fl,v_\star)$ is a connected vacuum diagram and element of  the algebra $\hat\CH_-$. 
\end{lemma}

In plain words, we can view $\Gamma(\tau)$ as the Feynman diagram obtained from $\tau$ by killing the noise types edge $e \in L(\tau)$ an marrying each leg of $\tau$ with its respective partner. 
We also set
\begin{align}\label{eq:CV}
\CV\tau := (-1)^{m(\tau)} \Gamma(\tau)
\end{align}
where we define for any properly legged tree $\tau \in \plCThat$ the quantity
\[
m(\tau) := \sum_{ (e,\bar e) \in \CE_\Legtype(\tau) } \fe(e).
\]
If we extend $\CV$  multiplicatively to a map on $\plCThat$, we obtain an algebra monomorphism
\[
\CV : \plCThat \to \hat\CH_-.
\]
We now have the following relation between the evaluations $\Pi$ on $\hat\CH_-$ and $\bar \Upsilon$ on $\plCThat$, respectively.
\begin{lemma}\label{lem:identity:evaluations:trees:feynman:diagrams}
For any $\phi \in \SN$ and any large-scale kernel assignment $R \in \CK^+_\infty$, one has the identity
\begin{align}\label{eq:identity:evaluations:trees:feynman}
	\Pi^{\CK,\CR} \CV = \eval{\phi}{R},
\end{align}
on $\plCThat$, where $\CK$ and $\CR$ are constructed from $\phi$ and $R$ as in (\ref{eq:feynman:diagram:kernel:assi}).
\end{lemma}
\begin{proof}
Let $\tau \in \plCThat$ be a tree. We have to compare the definition of $\bar\Upsilon ^{\one,\phi} _R \tau$ in (\ref{eq:bar:Upsilon:large:scale}) with $\psi$ given by (\ref{eq:Upsilon:bar:phi:psi}) to the definition of $\Pi^{\CK,\CR}\CV \tau$ in \cite[Eqs~2.15, 4.3]{Hairer2017}. We re-write the integrand in \cite[Eqs~2.15, 4.3]{Hairer2017} as 
\begin{multline}
(-1)^{m(\tau)}
\delta_0(x_\rho)
\prod_{ e \in K(\tau) } 
	D^{\fe(e)}{(K+R)}_{\ft(e)}( x_{e^+} - x_{e^-} )
\\
\prod_{ \substack { e,\bar e \in \CE_\Legtype(\tau) } } 
	D^{\fe(e) + \fe(\bar e)} \phi_{ \ft(e) }  (x_{e^\downarrow} - x_{\bar e^\downarrow})
\prod_{ u \in N(\tau) } x_u^{\fn(u)}.
\end{multline}
Moreover, we have the identity
\begin{multline}
\prod_{ \substack { e,\bar e \in \CE_\Legtype(\tau) } } 
	D^{\fe(e) + \fe(\bar e)} \phi_{ \ft(e) }  (x_{v^\downarrow} - x_{u^\downarrow})
\\
=
(-1)^{m(\tau)}
\int_{L_\Legtype(\tau)} dx
\prod_{e \in L_\Legtype(\tau)}
	D^{ \fe( e ) } \delta_0(x_{u^\downarrow} - x_u )
\hat\phi_{L_\LT(\tau)} (x_{L_\Legtype(\tau)})
\end{multline}
where $\hat\phi_{L_\LT(\tau)}$ is as in (\ref{eq:Upsilon:bar:phi:psi}). Comparing this with (\ref{eq:Upsilon:bar:phi:psi}) the lemma follows at once.
\end{proof}

In a the next step we would like to understand the relation between the coproducts $\cpm$ and $\cpm$ on $\plCThat$ and $\hat\CH_-$, respectively. This in general quite messy, as for general trees $\tau \in \plCThat$ there is no obvious relation between the homogeneity $|\taua|_\fs$ and the degree $\LTdeg \Gamma(\taua)$ for subtrees $\taua$ of $\tau$. However, the situation is much nicer for trees of the form $\wli\tau$ for some $\tau \in \adCT$ with the property that $\LTsys  \in \LTsyss(\tau)$.

\begin{lemma}\label{lem:CV:subgraphs}
Let $\tau=T^\fn_\fe \in \adCT$, let $\LTsys \in \LTsyss(\tau)$, and let $(\Gamma, \fn, \fl) := \Gamma (\tau)$. Then, the set of full, connected subgraphs $\tilde \Gamma$ of  $\Gamma$ with $\LTsysdeg \tilde\Gamma^0_\fe<0$ coincides with the set of subgraphs $\tilde \Gamma$ of $\Gamma$ that satisfy one of the following two criteria.
\begin{enumerate}
\item \label{item:CV:subgraphs:trees}
There exists a subtree $\tilde \tau=\tilde T^0_\fe$ of $\tau$ with $|\tilde \tau|_\fs<0$ such that $\tilde \Gamma$ is the full subgraph of $\Gamma$ induced by $N(\tilde \tau)$.
\item \label{item:CV:subgraphs:legs}
The graph $\tilde \Gamma$ contains a single edge $e$ with the property that $\fl(e)=\cw^{(k)}$ for some $\cw \in \ltfd$ and $k \in \N^d$.
\end{enumerate}
In the first case one has $\boxnorm{\tilde T^0_\fe}_\fs = \LTsysdeg \CV \tilde \Gamma^0_\fe$ and $\tilde \Gamma = \Gamma (\forestlegs{\pi\tilde\tau})$, where $\pi$ is as above the projection that removes legs and $\forestlegs{\pi \tilde\tau}$ is as in (\ref{eq:forests:legs}). Finally,
 there exists $\bar\CM \ssq \CM$ such that $\CL(\tilde \tau) = \bigsqcup \bar\CM$, where $\CM$ is as in (\ref{eq:II:tau}) for $\LTsys$.
\end{lemma}
\begin{proof}
Let $\tilde\Gamma$ be a connected, full subgraph of $\Gamma(\tau)$ such that $\LTsysdeg \tilde\Gamma_\fe^0<0$, and let $\hat \tau$ be the subgraph of $\tau$ induced by the edge set $E(\hat \tau):= K(\tau)\cap E(\tilde\Gamma)$. We first argue that either $\hat \tau$ is a subtree of $\tau$, or point~\ref{item:CV:subgraphs:legs}. above applies. 
For this we denote by $\hat\tau_1, \ldots , \hat \tau_m$ for some $m \ge 1$ the connected components of $\hat\tau$, so that $\hat\tau_i$ is a subtree of $\tau$ for any $i \le m$. We obtain another tree $\tilde \tau_i$ from $\tilde\tau_i$ by adding all noise type edges $e \in L(\tau)$ incident to $\hat\tau_i$, so that $\tilde \tau_i$ is the subtree of $\tau$ induced by the edge set $E(\tilde \tau_i) := E( \hat \tau _i ) \sqcup \{ e \in L(\tau) : e^\downarrow \in N(\hat \tau_i) \}$. It now follows from a counting argument identical to  (\ref{eq:definition:deg:I:explain}) that $\LTsysdeg \tilde \Gamma^0_\fe$ is given by
\begin{multline}\label{eq:CV:subgraphs:1}
\sum_{i=1}^m
\sum_{e \in K(\hat\tau_i)} (|\ft(e)|_\fs - |\fe(e)|_\fs) + (m-1)|\fs|
\\+
\sum_{
	\substack{
		(e,\bar e) \in \CE_\Legtype(\tau) \cap E(\hat\Gamma)
	}
}
\LTsysdeg(\ft(e)) + \LTsysdeg(\ft(\bar e))
+
\sum_{u \in V(\tilde\Gamma)} |\fn(u)|_\fs
\\
\ge
\sum_{i=1}^m \boxnorm{(\tilde\tau_i)^0_\fe}_\fs + (m-1)|\fs|,
\end{multline}
Now, by our assumption on the regularity structure one has $\boxnorm{(\tilde\tau_i)^0_\fe}_\fs>-\shalf$ unless $\tilde\tau_i = \Xi$ for some $\Xi \in \FL_-$ with $\homofancys\Xi = -\shalf$. It follows that this expression can only be negative for $m=1$ or for $m=2$, and in the second case one has necessarily that $\hat\tau_i$ is the trivial tree for $i=1,2$, so that point~\ref{item:CV:subgraphs:legs} above applies.

Assume for the rest of the proof that $m=1$ and hence $\tilde\tau$ is a subtree of $\tau$. It then follows that $\tilde\Gamma$ is the full subgraph of $\Gamma$ induced by set $N(\tilde \tau)$ of nodes of $\tilde \tau$. Moreover, from (\ref{eq:CV:subgraphs:1}) we infer that $\boxnorm{\tilde \tau^0_\fe}_\fs<0$. We are left to show that $\CL(\tilde\tau)$ can be written as a disjoint union of some $\bar\CM \ssq \CM$. Assume this does not hold. We distinguish two cases. In the first case there exists $u \in \CL(\tilde \tau)$ such that one has $u \notin \bigsqcup \CM$. Let $e \in L_\Legtype(\tilde\tau)$ be the noise type edge with $e^\downarrow = u$, and observe that one gets similarly to (\ref{eq:CV:subgraphs:1}) the estimate
\[
\LTsysdeg \tilde\Gamma^0_\fe \ge \boxnorm{\tilde \tau^0_\fe} _\fs - \boxnorm{\ft(e)}_\fs>0.
\]
The last inequality follows again from the assumptions made on the regularity structure.
In the second case there exists $M \in \CM$ such that $M \cap \CL(\tilde\tau) \ne \emptyset$, but $M \nsubseteq \CL(\tilde \tau)$. 
Then we set $\alpha := \max\{ \boxnorm{\ft(e)}_\fs : e \in L(\tilde \tau), e^\downarrow \in M \cap \CL(\tilde\tau)\}$ and we have similar to (\ref{eq:CV:subgraphs:1}) the estimate
\begin{align}\label{eq:inequality:deg:FJ}
\LTsysdeg \tilde \Gamma^0_\fe
&
\ge 
\boxnorm{\tilde\tau^0_\fe}_\fs - 
	\sum_{e \in L(\tilde\tau) , e^\downarrow \in M \cap \CL(\tilde\tau)} 
		\frac{\# M -  \# (M \cap \CL(\tilde\tau))}{\#M -1} \boxnorm{\ft(e)}_\fs
\\
&
\ge
\boxnorm{\tilde\tau_\fe^0}_\fs -  
		\frac{\# M -  \# (M \cap \CL(\tilde\tau))}{\#M -1}
			 \# (M \cap \CL(\tilde\tau)) \alpha.
\end{align}
Since $1\le \# (M \cap \CL(\tilde\tau)) \le \#M-1$ we can bound this expression by
\[
\boxnorm{\tilde\tau^0_\fe}_\fs -  \alpha \ge 0.
\]
The last inequality follows again from the assumption on the regularity structure and the fact that there exists a noise type edge $e \in L(\tilde\tau)$ such that $\boxnorm{\ft(e)}_\fs = \alpha$.
\end{proof}

We are now in a position to show an identity between the coproduct on the respective spaces.
For this we introduce the canonical projection $\p_- : \hat\CH_- \to \CH_-$, and we define the projection $\p_\trees : \CH_- \to \CH_-$ as the multiplicative projection onto the subalgebra  of $\CH_-$ generated by connected vacuum diagrams $(\Gamma^\fn_\fe,\fl) \in \CH$ with the property that there exists an edge $e \in E(\Gamma)$ such that $\fl(e) \in \FL_+$. With this notation, we have the following lemma.

\begin{lemma}\label{lem:CV:coproduct}One has the identity
\begin{align}\label{eq:CV:coproduct}
( \p_- \CV \otimes \CV)
	\Delta_-^\ex \wli
=
(\p_\trees \otimes \Id)
	\Delta_- \CV \wli
\end{align}
on $\adCT$.
\end{lemma}
\begin{proof}
Since the expression on both sides are multiplicative and linear, it suffices to show this identity for trees, and we fix for the entire proof a tree $\tau = T^\fn_\fe \in \adCT$. We start with the expression given by applying the right-hand side of (\ref{eq:CV:coproduct}) to $\tau$ and transform it into the left-hand side.

By definition one has
\[
\Delta_- \CV \wli \tau 
=
(-1)^{m(\tau)}
\sum_{\tilde\Gamma} \sum_{\tilde \fe, \tilde \fn}
	\frac{(-1)^{|\out \tilde \fe|}}{\tilde \fe !}
	\binom{\fn}{\tilde\fn}
	\tilde\Gamma^{\tilde\fn+\pi\tilde\fe}_{\fe} \otimes (\Gamma/\hat\Gamma)^{\fn-\tilde\fn}_{
	[\tilde \fe]+ \fe}
\]
where we use the convention that the first sum runs over subgraphs $\tilde\Gamma$ of $\Gamma^\fn_\fe := \CV \tau$ with the property that any connected component of $\tilde\Gamma^0_\fe$ is divergent, and the second sum runs over all decorations $\tilde\fe: \partial_{\tilde\Gamma} E(\Gamma) \to \N^d$ and $\tilde\fn : V(\Gamma) \to \N^d$ such that $\supp \tilde \fn \ssq V(\tilde \Gamma)$. Here, we write $\partial_{\tilde\Gamma} E(\Gamma)$ for the set of half-edges $(e,v)$ with $e \in E(\Gamma) \backslash E(\tilde\Gamma)$ and $v \in e \cap V(\tilde\Gamma)$, and we write $[\tilde\fe](e):=\sum_{u \in e}\tilde\fe(e,u)$.

After applying $\p_\trees\otimes \Id$ to this identity, we restrict the first sum to those subgraphs $\tilde\Gamma$ with the property that each connected component of $\tilde\Gamma$ is of the first type in Lemma~\ref{lem:CV:subgraphs}. In this case we can write this graph in the form $\tilde \Gamma = \prod_{S \in \forestlegs{\bar\CF}} \Gamma(S)$ for some forest $\CF \in \div^\star(\pi\tau)$, where we write $\div^\star(\pi\tau) \subseteq \div(\pi\tau)$ for the set of forests $\CF \in \div(\pi\tau)$ with the property that each tree $S \in \bar\CF$ satisfies the first condition of Lemma~\ref{lem:CV:subgraphs}. We can now write 
\begin{align*}
(\p_\trees &\otimes \Id)\Delta_- \CV \wli\tau 
\\
&=
(-1)^{m(\tau)}
\sum_{\CF \in \div^\star{(\pi\tau)}} 
	\sum_{\fe_{\forestlegs\CF},  \fn_{\forestlegs\CF}}
	\frac{ (-1) ^ { \langle\fe_\CF\rangle } }{\fe_{\forestlegs \CF} !}
	\binom{\fn}{\fn_{\forestlegs \CF}}
	\Gamma(\forestlegs{\CF}) ^{\fn_\CF + \pi \fe_{\forestlegs \CF}} _{ \fe } 
	\otimes 
	(\Gamma( T/\forestlegs\CF)) ^{\fn-\fn_{\forestlegs \CF}}
		_{\fe_\CF + \fe}
\\
&=
\sum_{\CF \in \div^\star{(\pi\tau)}} 
	\sum_{\fe_{\forestlegs\CF},  \fn_{\forestlegs\CF}}
	\frac{ 1 }{\fe_{\forestlegs \CF} !}
	\binom{\fn}{\fn_{\forestlegs \CF}}
	\CV(\forestlegs{\CF}) ^{\fn_{\forestlegs \CF} + \pi \fe_{\forestlegs \CF}} _{ \fe } 
	\otimes 
	(\CV( T/\forestlegs\CF)) ^{\fn-\fn_{\forestlegs \CF}}
		_{\fe_{\forestlegs \CF} + \fe},
\end{align*}
where $\langle \fe_\CF \rangle:=| \sum_ { ( u , v ) \in \CE_\Legtype ( \tau ) } \fe_\CF(u)|$.
The sums here run over all decorations $\fn_{\forestlegs \CF}$ and $\ce_{\forestlegs\CF}$ satisfying the condition that $\deg \Gamma ( S)^{\fn_{\forestlegs \CF} + \pi \fe_{\forestlegs \CF}} _{ \fe } <0$ for any $S \in \forestlegs{\bar\CF}$, which, due to Lemma~\ref{lem:CV:subgraphs}, is equivalent to $|S^{\fn_{\forestlegs \CF} + \pi \fe_{\forestlegs \CF}} _{ \fe } |_\fs<0$. Moreover, again with Lemma~\ref{lem:CV:subgraphs}, it follows that this condition is violated for any $\CF \in \div (\tau) \backslash \div^\star(\tau)$ for any choice of decoration, so that we can re-write this expression further into
\begin{equ}
(\p_\trees \otimes \Id)\Delta_- \CV \wli \tau 
=
(\p_- \CV \otimes \CV)
\sum_{\CF \in \div{\tau}} 
	\sum_{\fe_{\forestlegs\CF},  \fn_{\forestlegs\CF}}
	\frac{ 1 }{\fe_{\forestlegs \CF} !}
	\binom{\fn}{\fn_{\forestlegs \CF}}
	\forestlegs{\CF} ^{\fn_{\forestlegs \CF} + \pi \fe_{\forestlegs \CF}} _{ \fe } 
	\otimes 
	( T/\forestlegs\CF) ^{\fn-\fn_{\forestlegs \CF}}
		_{\fe_{\forestlegs \CF} + \fe}.
\end{equ}
Comparing this with the definition of the coproduct $\Delta_-^\ex$ in (\ref{eq:coproduct:pl}), and noting that the extended $\fo$-decoration is irrelevant due to the definition of the operator $\CV$, concludes the proof.
\end{proof}

We now construct a character $h(\CK,\CR)$ on $\CH_-$ in an analogous way to (\ref{eq:h:CI:phi:R}).
Given a set $\LTa \ssq \LT$ which is closed under conjugation, we write $\CH_-^\LTa \ssq \CH_- $ and $\hat\CH_-^\LTa \ssq \hat\CH_-$ for the linear sub-space of $\CH_-$ and $\hat\CH_-$ respectively, spanned by all connected Feynman diagrams $\Gamma = (\CV, \CE)$ with the property that for any $e \in \CE$ one has either $\ft(e) \in \FL_+$ or $\ft(e)=(\lt,\bar\lt) \in \LT/_-$ with $\lt \in \LTa$, and we write $P_\LTa : \CH_- \to \CH_-^\LTa$ for the canonical projection.

With this notation we define a character $h(\CK,\CR)$ on $\CH_-$ by setting
\begin{align}\label{eq:eps:beta:bound:h}
h(\CK,\CR)\Gamma: = 
	- \sum_{\LTa \in \LTsys} 
		\Pi^{\CK,\CR} P_\LTa \Gamma
\end{align}
for any connected vacuum Feynman diagram $\Gamma$, and extending this linearly and multiplicatively. We leave the set $\LTsys$ implicit in this notation, since it is fixed for the entire proof anyway.

Before we state the next Lemma, let us give an equivalent definition of the characters $h(\FJ,\phi,R)$ and $h(\CK,\CR)$ defined in (\ref{eq:h:CI:phi:R}) and (\ref{eq:eps:beta:bound:h}). 

First note that we introduced linear projections $P_\LTa : \plCT \to \plCT$ and $P_\LTa : \CH_- \to \CH_-$. We generalise this notation to systems $\LTsysa \in \LTsyss$ in the following way. We write $\plCT[\LTsysa] \ssq \plCT$ (resp. $\CH_-^\LTsysa \ssq \CH_-$) for the linear subspaces spanned by all products of trees $\prod_{\LTa \in \LTsysa}\tau_\LTa$ (resp. vacuum diagrams $\prod_{\LTa \in \LTsysa}\Gamma_\LTsysa$) with the property that $\ft(L_\LT(\tau_\LTa))=\LTa$ (resp. $\Gamma_\LTa \in \CH_-^\LTa$) for any $\LTa \in \LTsysa$. We then write $P_\LTsysa$ for the linear projection onto $\plCT[\LTsysa]$ and $\CH_-^\LTsysa$, respectively. We overload the notation $P_\LTsysa$ here since these projections are closely related, compare (\ref{eq:P:CV:commute}) below. 
With this notation, we have the following identities:
\begin{align}
\label{eq:h:trees:rewrite}
\LchR &:= \sum_{\LTsysa \ssq \LTsys} (-1)^{\# \LTsysa} 
	\barg \Upsilon^{\one, \phi}_R P_\LTsysa 
\\
\label{eq:h:feynman:rewrite}
h(\CK, \CR) &:= \sum_{\LTsysa \ssq \LTsys} (-1)^{\# \LTsysa} 
	\Pi^{\CK,\CR} P_\LTsysa
\end{align}
on $\plCT$ and $\CH_-$, respectively. 

\begin{lemma} 
Let $\CK$, $\CR$ and $\tilde \CR$ be constructed from $\phi$, $R$ and $\tilde R$ as in (\ref{eq:feynman:diagram:kernel:assi}). 
Then one has the identity
\begin{align}\label{eq:bound:eps:beta:2}
( \LchR
	\otimes
	\eval{\phi}{\tilde R} )
	\cpmwi
=
(h(\CK, \CR)
	\otimes
	\Pi^{\CK,\tilde \CR})
	\Delta_-^\full
	\CV
	\wli
\end{align}
on $\adCT$.
\end{lemma}
\begin{proof}
We first claim that one has the identity
\begin{align}\label{eq:P:CV:commute}
P_\LTsysa \CV = \CV P_\LTsysa 
\end{align}
on $\adCT$ for any $\LTsysa \ssq \LTsys$. By definition of $P_\LTsysa$ it is clear that it is enough to show this identity for $\LTsysa=\{\LTa\}$ for any $\LTa \in \LTsys$. Let now $\tau \in \adCT$ be a tree, and observe that one has $\tau \in \rng P_\LTa$ if and only if $\ft(L_\Legtype(\tau)) = \LTa$. This implies in particular that $\CV \tau \in \CH_-^\LTa$, so that
\[
P_\LTa \CV \tau = \CV P_\LTa \tau = \CV  \tau.
\]  
Conversely, assume that $\ft(L_\LT(\tau))\ne \LTa$. Then this implies in particular that $\CV \tau \notin \CH_-^\LTa$ by construction, so that both sides of the claimed identity vanish. 

Now, using the expression (\ref{eq:h:trees:rewrite}) for $h(\LTsys,\phi,R)$ we can re-write the left-hand side of (\ref{eq:bound:eps:beta:2}) into
\[
\sum_{\LTsysa \subseteq \LTsys}
		(-1)^{\#\LTsysa}
\Big(
\eval{\phi}{R} P_\LTsysa 
\otimes
\eval{\phi}{\tilde R}
\Big)
\Delta_-^\ex \wli.
\]
Using Lemma~\ref{lem:identity:evaluations:trees:feynman:diagrams} and Lemma~\ref{lem:CV:coproduct}, we can re-write this into
\begin{multline}
\sum_{\LTsysa \subseteq \LTsys}
		(-1)^{\#\LTsysa}
\Big(
	\Pi ^{\CK,\CR} P_\LTsysa \p_-\CV 
	\otimes 
	\Pi ^{\CK,\tilde\CR} \CV 
\Big)
	\Delta_-^\ex \wli 
\\
=
\sum_{\LTsysa \subseteq \LTsys}
		(-1)^{\#\LTsysa}
\Big(
	\Pi ^{\CK,\CR} P_\LTsysa \p_\trees
	\otimes 
	\Pi ^{\CK,\tilde\CR} 
\Big)
	\Delta_-^\full \CV \wli.
\end{multline}
We now note that the projection $\p_\trees$ on the right-hand side is irrelevant, since the only divergent connected subgraphs of $\CV\iota\tau$ that get killed by $\p_\trees$ are of type 2 in Lemma~\ref{lem:CV:subgraphs} and thus get killed by $\Pi^{\CK,\CR}$ anyway. Using (\ref{eq:h:feynman:rewrite}) we see that this expression is equal to the right-hand side of (\ref{eq:bound:eps:beta:2}) as required.
\end{proof}

\subsection{Proof of Proposition~\ref{prop:eps:beta:bound:2}}
\label{sec:proof:eps:beta:bound}

For $\eps>0$ let $(\bar K^\eps_\ft)_{\ft \in \FL_+}$ be a kernel assignment such that $\bar K^\eps_\ft \in \CC_c^\infty(\bar\domain \backslash \{0\})$ for any $\ft \in \FL_+$ and $\eps>0$, and auch that $\bar K^\eps_\ft$ is equal to $\hat K^\eps_\ft$ in some neighbourhood of the origin, but compactly supported in a ball of radius $\frac{\delta}{M}$ around the origin, where $\delta$ is as in Proposition~\ref{prop:eps:beta:bound} and $M$ is the maximal number of edges appearing in some tree $\tau \in \TT_-$. Let also $\bar\CK$ be defined as in (\ref{eq:feynman:diagram:kernel:assi}) with $K$ replaced by $\bar K$.
We first have the following Lemma.
\begin{lemma}\label{lem:g:full:h}
Under the assumptions of Proposition~\ref{prop:eps:beta:bound}, one has 
\[
( g^\full (\bar \CK) \otimes \Pi^{\bar \CK, \CR} ) \Delta_-^\full \CV \wli  
= 
( h(\bar \CK,0) \otimes \Pi^{\bar \CK, \CR} ) \Delta_-^\full \CV \wli  
\]
on $\adCT$.
\end{lemma}
\begin{proof}
Let $\Gamma : = \CV \wli \tau$. It is sufficient to show that 
\begin{align}\label{eq:g:full:h:2}
g^\full(\bar \CK)\tilde\Gamma_{\fe}^{\tilde\fn}
=
h(\bar \CK, 0) \tilde\Gamma_\fe^{\tilde\fn}
\end{align}
for any connected, full subgraph $\tilde\Gamma$ of $\Gamma$ and any node decoration $\tilde \fn$ with the property that $\LTsysdeg \tilde\Gamma^{\tilde\fn}_\fe <0$.
Assume first that $\tilde \Gamma$ satisfies point~\ref{item:CV:subgraphs:legs}. of Lemma~\ref{lem:CV:subgraphs} and let $e$ be the unique edge of $\tilde \Gamma$. Then $h(\bar \CK,0)$ vanishes by definition, and one has 
$$
g^\full(\bar\CK)\tilde\Gamma 
	= -\Pi^{\bar \CK}\tilde \Gamma = -\int \phi_{\fl(e)}(x)dx = 0,
$$ 
where the last equality follows from the definition of $\SN(\LTsys,\delta)$.
Otherwise, one has that  $\tilde\Gamma$ satisfies~\ref{item:CV:subgraphs:trees}. in Lemma~\ref{lem:CV:subgraphs}, and we denote by $\tilde \tau$ be the subtree of $\tau$ such that $\tilde\Gamma$ is induced as a full subgraph of $\Gamma$ by $N(\tilde \tau)$. 
Then one has $\CL(\tilde \tau) = \bigsqcup \bar\CM$ for some $\bar\CM \ssq \CM$ with $\# \CM  \ge 1$. In case $\# \bar\CM = 1$ one has that all full subgraphs $\hat\Gamma$ of $\tilde\Gamma$ of negative degree are of type~\ref{item:CV:subgraphs:legs} in Lemma~\ref{lem:CV:subgraphs}, so that (\ref{eq:g:full:h:2}) follows from
\begin{align*}
g^\full(\bar\CK)\tilde \Gamma ^{\tilde\fn} _\fe
&= -(g^\full(\bar\CK) \otimes \Pi^{\bar\CK})
	\sum_{\hat \Gamma \subsetneq \tilde\Gamma} \sum_{\hat \fe, \hat \fn}
	\frac{(-1)^{|\out \hat \fe|}}{\hat \fe !}
	\binom{\fn}{\hat\fn}
	\hat\Gamma^{\hat\fn+\pi\hat\fe}_{\fe} \otimes \Gamma^{\fn-\hat\fn}_{\hat \fe}/(\hat\Gamma,\pi\hat\fe)\\
&= - \Pi^{\bar\CK} \tilde\Gamma^{\tilde\fn}_\fe
=h(\bar \CK,0) \tilde \Gamma ^{\tilde\fn} _\fe.
\end{align*} 
In case $\# \bar\CM  \ge 2$ one has $h(\bar \CK,0)\tilde\Gamma_\fe^{\tilde\fn}=0$ by definition. On the other hand, there exists distinct $M,N \in \bar\CM $ with $M\ne N$, and we can choose elements $u \in M$ and $v \in \N$. There exists a unique edge $e \in E(\tilde\Gamma)$ connecting $u$ and $v$. By definition of $\CM$ (c.f.\ (\ref{eq:II:tau})) one has $\fl(e) \notin \bigsqcup \LTsys/_-$, and by definition of $\SN(\LTsys,\delta)$ one has $\phi_{ \fl(e)}=0$ in a $\delta$-neighbourhood of the origin. Combined  with the support properties of the kernel assignment $\bar K$, we infer that one has $\Pi^{\bar \CK}\tilde \Gamma^{\tilde\fn}_{\fe} =0$. The same reasoning applies to any other Feynman diagram containing the edge of type $\fl(e)$. It follows thus from the definition of the coproduct that one has $g^\full(\bar\CK)\hat \Gamma^{\hat\fn}_{\hat\fe}=0$.
\end{proof}

With this lemma, comparing (\ref{eq:Feynman:diagram:bound}) and the right-hand side of (\ref{eq:bound:eps:beta:2}), we are left to compare the characters 
$$
h(\CK,\CR) = h(\bar \CK, \bar \CR) 
\qquad\text{ and }\qquad
h(\bar\CK,0),
$$
where we define $\bar \CR := \CR - \bar \CK + \CK$. We first claim that for any $\tau \in \adCT$ one has
\[
(f \circ h(\bar \CK,0)) (\CV \tau) = f(\CV \tau) + h(\bar \CK,0)(\CV \tau)
\]
for any character $f$ in the character group of $\CH_-$. This can been seen in a way very similar to the last step of the proof of Lemma~\ref{lem:g:full:h}, since whenever $\tilde\Gamma$ is a non-empty, proper subgraph of $\Gamma = \CV \tau$ of negative homogeneity then there exists an edge $e \in E(\Gamma / \tilde \Gamma)$ with $\fl(e) \notin \bigsqcup \LTsys/_-$, so that $h(\bar K,0)(\Gamma^\fn_\fe/(\tilde\Gamma, \tilde \fe))$ vanishes for any such subdiagram.

It remains to show that the expression
\[
f(\bar \CK, \bar \CR)\Gamma
:=
h(\bar \CK, \bar \CR) \Gamma
-
h(\bar \CK, 0) \Gamma
\]
is bounded by a constant uniformly over $\phi \in \SN(\LTsys, K, \delta)$ and $\bar \CR \in \CK^+_0$ such that $\|\phi\|_{\LTsys} \lor \|\bar \CR \|_{\CK^+} \le C$ for any $\Gamma \in \CH_-$. By definition, it is sufficient to show this for connected Feynman diagrams $\Gamma \in \CH^I_-$ for any $I \in \LTsys$. In this case one has
\[
f(\bar \CK, \bar \CR)\Gamma
=
(\Pi^{\bar \CK,\bar \CR}-\Pi^{\bar\CK,0})\Gamma,
\]
and since $\Gamma$ does not contain any sub-divergencies in this case, this expression is bounded in the required way as a consequence of \cite[Sec.~4]{Hairer2017} and Lemma~\ref{lem:super:regularity:implies:large:scale:bound}.

\section{Applications}

\subsection{The \texorpdfstring{$\Phi^4_{3}$}{Phi 4 3} equation}\label{sec:Phi43}

We show that our support theorem applies to the solution to $\Phi^4_3$ started at any deterministic initial condition 
$u_0 \in \CC^\eta(\T^3)$ with $\eta > -\frac{2}{3}$, which then concludes the proof of Theorem~\ref{thm:Phi4}.
While it is known that $u$ is a Markov process  which can be started from a deterministic initial condition and is a continuous function in time (see \cite[Sec.~9.4]{Hairer2014}), 
none of these statements follow immediately from \cite{BrunedChandraChevyrecHairer2017}. 
In case of $\Phi^4_3$, the process $\CS^-(\xi) = \lim_{\eps \to 0} \CS^-_\eps(\xi)$ is the stationary solution to the stochastic heat equation on $\T^3$, so that $\CS^-_\eps(\xi)(0,\cdot)$ is (in law)  a smooth approximation 
of the Gaussian free field. 
In order to see that one can start the equation at a deterministic initial condition, one has to use the fact that the critical regularity for the initial condition 
is $-\frac{2}{3}$, see \cite[Eqn.~9.13]{Hairer2014}, 
and hence lower than the regularity of the Gaussian free field. 
One can now choose the initial condition for the remainder $\eps$-dependent of the form $v^{(0)}-\CS^-_\eps(\xi)(0,\cdot)$, use the fact that this converges in probability in $\CC^{-\frac{2}{3} + \kappa}(\T^3)$ for any $\kappa < \frac{1}{6}$, and argue with the fact that the solution constructed in \cite[Thm.~2.13]{BrunedChandraChevyrecHairer2017} is almost surely 
continuous as a functional of the initial condition. The last statement follows from the second bullet in \cite[Thm.~2.13]{BrunedChandraChevyrecHairer2017} 
with $\CC^{\mathrm{ireg}} := \CC^{\eta}(\T^3)$ and $\eta \in (-\frac{3}{2}, -\half)$. 

While this procedure provides a robust interpretation of what we mean by a solution to $\Phi^4_3$ 
starting from a deterministic initial condition $v^{(0)} \in \CC^{\mathrm{ireg}}$, the process defined in this way
fails to be a continuous function of the model. 
Note that while $\CS^-_\eps(\xi)$ is a continuous function of the model with values 
in $\CC^{-\half-\kappa}_\fs(\spacetime)$, evaluating at a fixed time is not well defined on this space,
so that $\CS^-_\eps(\xi)(0,\cdot)$ fails to be a continuous function of the model.

To overcome this difficulty we work with a slightly stronger topology on the space of models, compare \cite[Prop.~9.8]{Hairer2014}, generated by the system of pseudo-metrics
\[
\fancynorm{Z, \tilde Z}_{\gamma,T} :=
\boxnorm{Z, \tilde Z}_{\gamma,[-T,T]\times \T^3} + 
\| K \star \PPi^Z \Xi - K \star \PPi^{\tilde Z} \Xi \|_{ \CC( [-T,T] , \CC^{\mathrm{ireg}})  }
\]
for any $T>0$.
Here $\boxnorm{Z, \tilde Z}_{\gamma,[-T,T]\times \T^3}$ denotes the usual metric on the model space as in \cite[Eqn.~2.17]{Hairer2014}.
With respect to this topology it is clear that $\CS^-_\eps(\xi)(0,\cdot) = (K \star  \PPi^{\hat Z^\eps} \Xi) (0,\cdot)$ is a continuous function of the model with values in $\CC^{\mathrm{ireg}}$. The fact that the BPHZ renormalised model converges in this stronger topology follows from \cite[Prop.~9.5]{Hairer2014}. 
To show that our support theorem holds for $\Phi^4_3$ 
it remains to argue that the proof of the support theorem for random models
 also applies in this stronger topology. 
For this we first note that once Proposition~\ref{prop:constantsInSupport} is proved, 
the arguments carried out in Section~\ref{sec:supp:models} 
only use the fact that the shift operator and the renormalisation group act continuously 
on the space of models, which is still true in this stronger topology.
As in $d=2$, Assumptions~\ref{ass:CJHopfIdeal} and~\ref{ass:CHBPHZcharacters} are trivial in this case, 
so that Section~\ref{sec:constraints} is not needed. 
\begin{remark}
We outline the proof that $\CJ$ is the ideal generated by $\Xi$. Recall that $\CJ$ is generated by linear combinations of trees with same number of leaves. From this we already infer that the only possible generated of $\CJ$ other than $\leafs$ must be a linear combination of
$\treePhia$ and $\treePhib$.
We can rule out that such a linear combination is element of $\CJ$ by choosing a sequence of test functions $\psi_\eps(x_1, x_2, x_3, x_4) = \prod_{1 \le i < j \le 4} \psi_{i,j,\eps}(x_i-x_j)$ for smooth symmetric functions $\psi_{i,j}$, where we set
$\psi_{i,j,\eps} := \psi_{i,j}^{(\eps)}$ if $\{i,j\} \in \{ \{1,2\},\{3,4\} \}$ and $\psi_{i,j,\eps} = \psi_{i,j}$ otherwise. The divergence structure of the two trees in question then implies the asymptotic behaviour
\[
\langle \CK_{\hat K} \treePhia, \psi_\eps \rangle \simeq \eps^{-2}
\, \text{ and } \qquad
\langle \CK_{\hat K} \treePhib, \psi_\eps \rangle \simeq \eps^{-1}.
\]
\end{remark}

Section~\ref{sec:shift} is formulated entirely at the level of the space of noises $\SM_0$ and 
never refers to the topology on the model space.
The remaining caveat is Section~\ref{sec:renormalisationgroupargument}. 
The topology on the model space enters explicitly in the final step of the proof of Proposition~\ref{prop:sequencetoconstant} via the identity
\[
\lim_{\delta\to 0}\lim_{\eps\to 0}\renorm{g^{\eps,\delta}}\Zcan(\xi^\eps + \srn_\delta)=\Zcan(0)\;,
\]
so we need to show that this convergence holds also with respect to the stronger topology defined above. Using \cite[Prop.~9.5]{Hairer2014}, which shows that $K \star \xi^\eps \to K\star \xi$ in $\CC( [-T,T] , \CC^{\mathrm{ireg}})$ almost surely, we need to provide an additional argument showing that
\[
K \star \zeta_\delta \to 0 \qquad \text{ in }\CC( [-T,T] , \CC^{\mathrm{ireg}})
\]
as $\delta \to 0$ in probability.
This can be shown with an argument very similar to the proof of \cite[Eqn.~9.15]{Hairer2014}.
Indeed, setting $\CX := \CC^{\frac{\tilde\kappa}{2}} ([-T,T] , \CC^{\eta+\tilde\kappa}(\T^3))$, where $\eta \in (-\frac{3}{2}, -\half)$ is as above and $\tilde\kappa>0$ is small enough such that $\eta+2\tilde\kappa<-\half$, it suffices to bound $K \star \zeta_\delta$ uniformly in $\CX$. 
Write $K = \sum_{n \ge 0} K_n$, where $K_n$ is supported in an annulus of order $2^{-n}$ as in 	\cite[Ass.~5.1]{Hairer2014}. 
By Kolmogorov's continuity criterion and the fact that, since $\zeta_\delta$ belongs to a Wiener chaos of fixed order and therefore
enjoys equivalence of moments, it suffices to show that for some $r>0$ one has 
\begin{align}\label{eq:CCTbound}
\E \Big( 
	\int \psi^\lambda(x) (K_n *\zeta_\delta(x,t) - K_n * \zeta_\delta(x,0)) dx 
\Big)^2
\lesssim
2^{-r n} |t|^{\tilde \kappa + r} \lambda^{2\eta + 2 \tilde \kappa + r}.
\end{align}
This expression is of the form \cite[Eqn.~9.17]{Hairer2014} with white noise $\xi$ replaced by $\zeta_\delta$. For the proof we can now proceed along the same lines as in \cite{Hairer2014}, noting that by definition $\zeta_\delta$ is linear combination (with uniformly bounded coefficients) of random stationary smooth functions $\eta^\delta_{(\Xi,\tau)}$ with the property that
\[
\rho^\delta_{(\Xi,\tau)}(x,t) := \E \eta^\delta_{(\Xi,\tau)}(x,t) \eta^\delta_{(\Xi,\tau)}(0,0)
\]
satisfies the scaling relation
\[
\rho^\delta_{(\Xi,\tau)} = 
(\lambda^\delta_{(\Xi,\tau)})^{-2 \fancynorm{\tau}_\fs - 2 \bar\kappa}
(\rho^1_{(\Xi,\tau)})^{(\lambda^\delta_{(\Xi,\tau)})}.
\]
The proof is now straightforward in case that $\fancynorm{\tau}_\fs<0$, where the right hand side can be estimated by an approximate $\delta_0$. In case $\fancynorm{\tau}_\fs = 0$ the fact that these covariances integrate to zero comes to rescue in the same way as in the proof of \eqref{eq:varnorm}.

\subsection{The \texorpdfstring{$\Phi^4_{4-\kappa}$}{Phi 4 4-e} equation}\label{sec:Phi44}

The $\Phi^4_{4- \kappa}$ with $\kappa$ irrational satisfies all our assumptions, except that the noise is not white. Recall that the space-time scaling is given by $\fs=(2,1,1,1,1)$ with $|\fs|=6$. We assume that $\xi = P \star \tilde \xi$, where $\tilde\xi$ is space-time white noise on $\T^4 \times \R$, the symbol $\star $ denotes spatial convolution, and $P \in \CC_c^\infty(\R^4 \backslash\{0\})$ is some integration kernel on $\R^4$ which is homogeneous on small scales $P(\lambda x) = \lambda^{-4+\kappa} P(x)$ for any $\lambda \in (0,1)$ and $x \in \R^4$ with $|x|\le \half$ (say). Here we assume that $\kappa>0$ is irrational (in order to avoid log-divergencies, which could destroy Assumption~\ref{ass:technical}). 
There exists a unique homogeneous kernel $\hat P:\T^4 \backslash\{0\} \to \R$ such that $\hat P = P$ in a neighbourhood of the origin. We denote by $\hat K$ the heat-kernel and we assume that $\hat P$ is such that $\hat K \star  \hat P$ can be decomposed as in Section~\ref{sec:kernels}. (This is certainly possible for $\hat P(x) = |x|^{-4+\kappa}$, which is a natural choice.)

We fix a set of two kernel types $\FL_+:=\{\ft, \fta\}$ representing heat kernel $\hat K$ and the convolution $\hat K\star  \hat P$ respectively, with $\fancynorm{\ft_K}:=2$ and $\fancynorm{\ft_P}:=2+\kappa$, and we fix a single noise type $\FL_-:=\{\Xi\}$ representing white noise with $\fancynorm\Xi = -3$ and $|\Xi|=-3-\bar\kappa$ for some $\bar\kappa>0$ small enough. A rule $R$ is given by the completion of $\bar R$, defined by setting
\[
\bar R(\ft) := \{ \emptyset, [\ftb_1], [\ftb_1,\ftb_2], [\ftb_1,\ftb_2,\ftb_3] : \ftb_i \in \{\ft,\fta\} \}
\quad \text{ and }\quad
\bar R(\fta) := \{ \emptyset, [\Xi] \}.
\] 
Provided that $\bar\kappa<\kappa$, the rule $\bar R$ (and hence $R$) is subcritical (see \cite[Def.~5.14]{BrunedHairerZambotti2016}). 
We fix a truncation $K_\ft$ and $K_\fta$ of $\hat K$ and $\hat K \star  \hat P$ as in Section~\ref{sec:kernels}, we denote by $\hat Z^\eps$ the BPHZ-renormalised canonical lift of the regularised white noise $\xi^\eps$, and we write $\hat Z := \lim_{\eps \to 0} \hat Z^\eps$. The existence of this limit follows from \cite{ChandraHairer2016} (see also \cite{CMW2019}). Moreover, the solution to $\Phi^4_{4-\kappa}$ equation is path-wise continuous in $\hat Z$. 

We now argue why Assumptions~\ref{ass:main:thm} to~\ref{ass:last} hold, which finalises the proof that Theorem~\ref{thm:main} can be applied to $\hat Z$.
Assumptions~\ref{ass:main:thm} to \ref{ass:noises:derivatives:products} are shown in 
\cite[Sec.~2.8.2]{BrunedChandraChevyrecHairer2017}. The heat kernel is homogeneous, so that Assumption~\ref{ass:kernelhomo} is satisfied.
Finally, for irrational $\kappa$ there are no trees of integer degree (and in particular no tree of zero degree), hence Assumptions \ref{ass:zero-homo} and \ref{ass:CVz} hold.

\subsection{Proof of Theorem~\ref{theo:gKPZ}}
\label{sec:gKPZproof}

Recall that we are interested in characterising the support of the solutions,
in the sense of \cite[Thm~1.2]{BrunedHairer2019}, to 
\begin{equ}\label{eq:kpz:main2}
\partial_t u^i
=
\partial_x^2 u^i
+
\Gamma^{i}_{j,k}(u) \partial_x u^j \partial_x u^k
+
h^i ( u )
+
\sigma_\mu^i ( u ) \xi^\mu\;.
\end{equ}
As before, $i,j,k=1,\ldots,n$, $\mu = 1,\ldots,m$, and Einstein's convention is used.
We also denote as in \eqref{e:connectionGamma} by $\nabla$ the connection on $\R^n$
given by $\Gamma$.

We first show that the generalised KPZ equation \eqref{eq:kpz:main2} satisfies Assumptions~\ref{ass:main:reg} to~\ref{ass:last}. Assumption~\ref{ass:main:reg} was shown in \cite{BrunedHairer2019}, Assumptions~\ref{ass:noises:derivatives:products} and~\ref{ass:kernelhomo} are clear. To see Assumption~\ref{ass:technical} we choose $\groupD := \groupD_2$ and we note that
\[
\CV = \big\{ \treeKPZVa, \treeKPZVb, \treeKPZVc , \treeKPZVd, \treeKPZVe, \treeKPZVf, \treeKPZVg \big\}.
\]
Here, thin black lines denote the heat kernel, while thick grey lines denote its spatial derivative.
We write $\leafs$ for an instance of white noise and polynomial label $\fn(\leafs) = (0,0)$, $\leafsx$ for a node with white noise and polynomial label $\fn(\leafsx) = (0,1)$, and $\leafsb$ for a node without noise and polynomial label $\fn(\leafsb) = (0,1)$. In the notation we drop the type decoration from the noises for simplicity, so that any tree in $\CV$ should be thought of as a finite collection of trees.
It is easy to see that the kernels $\CKhat\tau$ for $\tau \in \CV$ are all anti-symmetric under the transformation $(t,x) \mapsto (t,-x)$, and since the covariance of a shifted noise is symmetric under this transformation, one has $\E\PPi^\eta\tau (0)=0$, as required. Assumption~\ref{ass:last} is then trivially satisfied since in this example one
has $\CV_0 = \CV$.

Write now $\CS_\geo \subset \Vec{\TT_-}$ for the linear subspace of dimension $15$ generated by the ``geometric'' 
counterterms, as defined in 
\cite[Def.~3.2]{BrunedHairer2019} and characterised in \cite[Prop.~6.11 \&\ Rem.~6.17]{BrunedHairer2019}. 
We also write $\Upsilon_{\Gamma,\sigma} \colon \CS_\geo \to \CC^\infty(\R^n,\R^n)$ for the evaluation map
defined in \cite[Eq.~2.6]{BrunedHairer2019} (but note also the remark just before Eq.~6.2 in that article).
We can interpret $(\CS_\geo,+)$ as a subgroup of the renormalisation group $\CG_-$ and its action
on the space of right hand sides for \eqref{eq:kpz:main2} is given by $\tau \mapsto \Upsilon_{\Gamma,\sigma}\tau$.
As in \cite[Rem.~2.9]{BrunedHairer2019}, it will be convenient to introduce on $\CT_-$ (and therefore also on $\CS_\geo$) an
inner product by specifying that any two trees are orthogonal and their norm squared is given
by their symmetry factor.
We will use the suggestive notation of \cite{BrunedHairer2019} for elements of $\CS_\geo$, 
so that for example
\begin{equ}
 \Upsilon_{\Gamma,\sigma}\Nabla_{\<blue>}\<blue> = \sum_\mu \nabla_{\sigma_\mu}\sigma_\mu\;.
\end{equ}
We are now in a position to apply Theorem~\ref{thm:main:sv}.
First, we have the following result.
\begin{lemma}\label{lem:CH:Sgeo}
Let $\CH$ be the subspace of $\Vec{\TT_-}$ defined in Theorem~\ref{thm:main:sv}. Then one has $\CH \ssq \CS_\geo$.
\end{lemma}
\begin{remark}
The renormalisation group for the generalised KPZ equation is naturally isomorphic to $(\Vec{\TT_-}^*,+)$, so that it is more convenient to work with the linear space $\Vec{\TT_-}$ instead of the full algebra $\CT_-$. The scalar product introduced above provides an isomorphism $\Vec{\TT_-} \simeq \Vec{\TT_-}^*$ via Riesz identification. 
In the statement of the Lemma we made the slight abuse of notation and identify $\CH$ with a subset of $\Vec{\TT_-}$ given by the Riesz identification of the set $\CH$ viewed as a subspace of $\Vec{\TT_-}^*$. 
\end{remark}
\begin{proof}
Let $\sigma^i_l$ and $\Gamma^{i}_{j,k}$ be smooth functions on $\R^m$ for $i,j,k \le m$ and $l \le n$. Let furthermore $\vphi: \R^m \to \R^m$ be a diffeomorphism, and define $\vphi\cdot\Gamma$ and $\vphi\cdot T$ for any tensor $T$ by the usual transformation rules for Christoffel symbols and tensors under the diffeomorphism $\vphi$, see \cite[Eq.~1.6]{BrunedHairer2019}.
By \cite[Thm~1.2]{BrunedHairer2019} there exists a sequence of elements $g^\eps \in \CS_\geo$ such that $g^\eps-g^\eps_{\BPHZ}$ converges to a finite limit $f_\circ \in \Vec{\TT_-}$ as $\eps \to 0$. We let $\hat Z^\eps := \CR^{g^\eps} \Zcan(\xi^\eps)$ and $\hat Z := \lim_{\eps \to 0} \hat Z^\eps$.

Fix $h \in f_\circ + f^\xi + \CH$, where $f^\xi$ is the character defined in Assumption \ref{ass:CHBPHZcharacters} (which holds by Proposition~\ref{prop:ass_implies_ass}). It suffices to show that $h \in \CS_\geo$ (note that this proves furthermore that $f_\circ + f^\xi \in \CS_\geo$), 
which by \cite[Def.~3.2]{BrunedHairer2019} is equivalent to the property that 
$\Upsilon_{\vphi\cdot\Gamma, \vphi\cdot\sigma} h = \vphi\cdot\Upsilon_{\Gamma,\sigma}h$ 
for any $\Gamma$, $\sigma$ and $\vphi$ as above. 
Fix an initial condition $v_0 \in \CC^{\half^-}(\T)$ and let $v$ and $w$ denote the images of the model $\CZ(h)$ under the solution maps for $(\Gamma,\sigma)$ and $(\vphi\cdot\Gamma, \vphi\cdot\sigma)$, respectively, with initial condition $v(0) = v_0$ and $w(0) = w_0 := \vphi(v_0)$, so that
\begin{align}
\partial_t v &= \d_x^2 v + \Gamma(v) \partial_x v \partial_x v + (\Upsilon_{\Gamma,\sigma} h)(v)\;, \\
\label{eq:w}
\partial_t w &= \d_x^2 w + (\vphi\cdot\Gamma)(w) \partial_x w \partial_x w + 
(\Upsilon_{\vphi\cdot\Gamma, \vphi \cdot \sigma} h)(w)\;.
\end{align}
(Here we omit the indices for simplicity.) Note that $\vphi(v)$ satisfies an equation 
analogous to (\ref{eq:w}) but with counterterm given by $\vphi \cdot \Upsilon_{\Gamma,\sigma}h$, 
so that, by a simple special case of \cite[Thm~3.5]{BrunedHairer2019}, 
the proof is complete if we can show that $w = \vphi(v)$.
By Proposition~\ref{prop:sequencetoconstant} there exists a sequence of 
smooth random functions $\zeta_\delta$ such that $Z^{\eps,\delta} := T_{\zeta_\delta} \hat Z^\eps \to \CZ(h)$ as $\eps \to 0$ and $\delta \to 0$. Similar to above, denote by $v^{\eps,\delta}$ and $w^{\eps,\delta}$ the images of the (random) model $Z^{\eps,\delta}$ under the solution map for the data $(\Gamma,\sigma)$ and $(\vphi\cdot\Gamma, \vphi\cdot\sigma)$, so that $v^{\eps,\delta} \to v$ and $w^{\eps,\delta} \to w$ in probability. But since $g^\eps \in \CS_\geo$, one has $w^{\eps,\delta} = \vphi(v^{\eps,\delta})$, and this concludes the proof.
\end{proof}

Write $\Moll \subset \CC_0^\infty(\R^2)$ for the collection of test functions 
that are supported in the unit ball and integrate to $1$.
It then follows from \cite[Thm~1.2]{BrunedHairer2019} that there exists 
$\tau_\star \in \CS_\geo$
as well as maps $\Moll \ni \rho \mapsto \tau_\rho \in \CS_\geo$ and 
$\Moll \ni \rho \mapsto C_\rho \in \R$ such that, for every 
mollifier $\rho \in \Moll$ and for $\xi_\eps^\mu = \rho_\eps \star \xi^\mu$, one 
has $u = \lim_{\eps \to 0} u_\eps$ with
\begin{equs}\label{eq:kpz:eps}
\partial_t u_\eps
&=
\partial_x^2 u_\eps
+
\Gamma(u_\eps)(\partial_x u_\eps, \partial_x u_\eps)
+
h ( u_\eps ) +
\sigma_\mu ( u_\eps ) \xi_\eps^\mu\\
&\quad + (\Upsilon_{\Gamma,\sigma} \tau_\rho)(u_\eps) + (\Upsilon_{\Gamma,\sigma} \tau_\star)(u_\eps) \log \eps
+ \frac{C_\rho} {\eps} (\Upsilon_{\Gamma,\sigma}\Nabla_{\<blue>}\<blue>)(u_\eps)\;.
\end{equs}
Combining this with Theorem~\ref{thm:mainThm} and Remark~\ref{rmk:a-priori} we conclude that there exists $\bar \tau \in \CS_\geo$
such that the support $S_u$ of the law of $u$ is given by the closure in $\CC^\alpha$ of all solutions to
\begin{equs}
\partial_t u
&=
\partial_x^2 u
+
\Gamma(u_\eps)(\partial_x u, \partial_x u)
+
h ( u ) +
\sigma_\mu ( u ) \eta^\mu\\
&\quad + (\Upsilon_{\Gamma,\sigma} \bar \tau)(u_\eps) + K_\star (\Upsilon_{\Gamma,\sigma} \tau_\star)(u_\eps)
+ K_0 (\Upsilon_{\Gamma,\sigma}\Nabla_{\<blue>}\<blue>)(u_\eps)\;,
\end{equs}
for arbitrary smooth controls $\eta^\mu$ and arbitrary constants $K_\star$ and $K_0$.
Note that $\Upsilon_{\Gamma,\sigma}\Nabla_{\<blue>}\<blue>$ is nothing but the vector field
$V$ in Theorem~\ref{theo:gKPZ} while $\Upsilon_{\Gamma,\sigma}\tau_\star = V_\star$.
We also write $\hat \tau \in \CS_\geo$ for the element such that 
$\Upsilon_{\Gamma,\sigma}\hat \tau = \hat V$ with $\hat V$ as in Theorem~\ref{theo:gKPZ}, so
that
\begin{equ}
\hat \tau = \Nabla_{\Nabla_{\<blue>}\<dblue>}\Nabla_{\<blue>}\<dblue>\;.
\end{equ}
We also introduce the following notation. Given two collections 
$\CA, \bar \CA \subset \CC^\infty(\R^n, \R^n)$ and $H \in \CC^\infty(\R^n, \R^n)$, we write
$\CU(H,\CA,\bar \CA)$ for the closure in $\CC^\alpha$ of all solutions to 
\begin{equs}
\partial_t u
&=
\partial_x^2 u
+
\Gamma(u_\eps)(\partial_x u, \partial_x u)
+
H ( u ) +
\sum_{A \in \CA} \eta_A A  + \sum_{B \in \bar \CA} K_B B\;,
\end{equs}
where the $\eta_A$ are arbitrary smooth functions and the $K_B$ are arbitrary
real constants. An important remark is that one has the identity
\begin{equ}[e:trivialIdentity]
\CU(H,\CA,\bar \CA) = \CU(H + \bar H, \CA, \bar \CA \cup \CB)
\end{equ}
for any $\bar H \in \Vec{(\CA \cup \bar \CA)}$ and any $\CB \subset \Vec{\CA}$.

To complete the proof of Theorem~\ref{theo:gKPZ}, it then remains to show that,
for any $\bar \tau \in \CS_\geo$ there exists $\hat c$ such that 
\begin{equ}[e:wantedCU]
\CU(h + \bar \tau,\{\sigma_\mu\},\{\tau_\star, \Nabla_{\<blue>}\<blue>\}) = 
\CU(h + \hat c \hat \tau ,\{\sigma_\mu\},\{\tau_\star, \Nabla_{\<blue>}\<blue>\})\;,
\end{equ}
thus reducing the dimensionality of the unknown quantity from $15$ to $1$. 
Here, we implicitly identify elements of $\CS_\geo$ with elements of $\CC^\infty(\R^n,\R^n)$
via $\Upsilon_{\Gamma,\sigma}$ to shorten notations.  
To show \eqref{e:wantedCU}, we will make extensive use of the following result.

\begin{lemma}\label{lem:oscillating}
Let $\hat H \colon \R_+ \times \T \times \R^n \to \R^n$, $A, B \colon \R^n \to \R^n$, 
and $\zeta, \hat \zeta, \eta, \hat \eta  \colon \R_+ \times \T \to \R$
be smooth functions and let $C \in \R$. Then there 
exist $C_\eps \in \R$ and smooth functions 
$\eta_\eps, \hat \eta_\eps \colon \R_+ \times \T \to \R^m$ such that the solution
to 
\begin{equs}[e:oscillating]
\partial_t u_\eps
&=
\partial_x^2 u_\eps
+
\Gamma(u_\eps)(\partial_x u_\eps, \partial_x u_\eps)
+
\hat H(t,x,u_\eps) \\
&\quad + A(u_\eps) \eta_\eps + B(u_\eps) \hat \eta_\eps
+ {C_\eps} \bigl(\nabla_{A}A\bigr)(u_\eps)\;,
\end{equs}
converges in $\CC^\alpha$ as $\eps \to 0$ to the solution $u$ to 
\begin{equs}
\partial_t u
&=
\partial_x^2 u
+
\Gamma(u)(\partial_x u, \partial_x u)
+
\hat H (t,x,u) \\
&\quad + A(u) \eta + B(u) \hat \eta
+ \bigl(\nabla_A B\bigr)(u) \zeta + \bigl(\nabla_B A\bigr)(u) \hat \zeta
+ {C} \bigl(\nabla_{A}A\bigr)(u)\;.
\end{equs}
\end{lemma}

\begin{proof}
We  consider the singular SPDE given by
\begin{equs}
\partial_t u
&=
\partial_x^2 u
+
\Gamma(u)(\partial_x u, \partial_x u)
+
\hat H (t,x,u) \\
&\quad + A(u) \eta + B(u) \hat \eta
+ \bigl(\nabla_A B\bigr)(u) \zeta + \bigl(\nabla_B A\bigr)(u) \hat \zeta
+ {C} \bigl(\nabla_{A}A\bigr)(u)\\
&\quad + A(u) \tilde\xi_\eps + B(u) \bigl(\zeta \cdot \xi_\eps + \hat \zeta \cdot \hat \xi_\eps\bigr)\;,
\end{equs}
driven by the three ``noises'' $\xi_\eps \in \CC^{\kappa-1}$, $\hat \xi_\eps \in \CC^{\kappa-1}$ and $\tilde \xi_\eps \in \CC^{-1-3\kappa}$. If we choose $\kappa$ sufficiently small, the only symbols of 
negative degree appearing in the corresponding regularity structure (besides those representing the
noises themselves and the one representing the product of $\xi_\eps$ with the spatial coordinate) are
\begin{equ}
\<Xi2> ,\; \<I1Xitwo>,\; \<Xi2mixed>,\; \<Xi2mixedt>,\; \<Xi2mixed2>,\; \<Xi2mixed2t> ,\; \<I1Xitwomixed>,\; \<I1Xitwomixedt>\;,
\end{equ}
where we denote the symbol representing $\tilde \xi_\eps$ by $\<Xit>$, the one
representing $\xi_\eps$ by $\<Xi>$ and the one representing $\hat \xi_\eps$ by $\<Xih>$.
Thin lines represent the heat kernel and thick lines its spatial derivative as usual.

One then proceeds as follows: choose first a symmetric function $\tilde\rho \in \CC_0^\infty$ such that
$\int  (\tilde\rho\star \tilde\rho)(z)P(z)\,dz = 1$ for $P$ the heat kernel on the whole space and $z = (t,x) \in \R^2$
and set
$\tilde \xi_\eps = \eps^{-1-2\kappa} \tilde\rho_\eps \star \xi$, where $\xi$ is space-time white noise
and $\tilde\rho_\eps(t,x) = \tilde\rho(t/\eps^2,x/\eps)$.
One then fixes two asymmetric $\CC_0^\infty$ functions $\rho$ and $\hat \rho$ such that
 the following identities hold:
\begin{equs}[e:choicerho]
\int P(z) (\tilde \rho \star \rho)(z)\,dz &= 1\;,\quad &
\int P(-z) (\tilde \rho \star \rho)(z)\,dz &= 0\;, \\
\int P(z) (\tilde \rho \star \hat \rho)(z)\,dz &= 0\;,\quad &
\int P(-z) (\tilde \rho \star \hat \rho)(z)\,dz &= 1\;.
\end{equs}
With this choice, we then set
\begin{equ}
\xi_\eps = \eps^{2\kappa-1} \rho_\eps \star \xi\;,\quad 
\hat\xi_\eps = \eps^{2\kappa-1} \hat\rho_\eps \star \xi\;.
\end{equ}
Since all of these noises weakly converge to $0$, it is immediate from \cite{ChandraHairer2016}
(but in this case this is also a simple exercise along the lines of the examples treated in  \cite{Hairer2014})
that the BPHZ model associated to this choice converges to the canonical lift of $0$.

Furthermore, as a consequence of \eqref{e:choicerho}, the scaling of the noise, and the identity
\begin{equ}
\d_x P \star \d_x \bar P = \frac{1}{2}(P + \bar P)\;,
\end{equ}
where $\bar P(z) = P(-z)$, the BPHZ character $g^\eps$ for our choice of 
``noise'' is given by
\begin{equs}
g^\eps(\<Xi2>) &=  g^\eps(\<I1Xitwo>) = -\eps^{-4\kappa}\;,&\quad g^\eps(\<Xi2mixed>) &= g^\eps(\<Xi2mixed2t>) = -1\;, \\
g^\eps(\<Xi2mixedt>) &= g^\eps(\<Xi2mixed2>) = 0\;,&\quad
g^\eps(\<I1Xitwomixed>) &= g^\eps(\<I1Xitwomixedt>) = -\f12\;.
\end{equs}
It then suffices to apply the results of \cite{Hairer2014,BrunedChandraChevyrecHairer2017} to conclude that the BPHZ renormalised equation solves
\eqref{e:oscillating} with the choice $\hat \eta_\eps = \hat \eta + \zeta \cdot \xi_\eps + \hat \zeta \cdot \hat \xi_\eps$
and $\eta_\eps = \eta + \tilde \xi_\eps$, so that the claim follows.
\end{proof}

\begin{corollary}\label{cor:addFields}
One has the identity
\begin{equ}
\CU(H,\CA,\bar \CA) = \CU(H, \CA \cup \{\nabla_AB,\nabla_BA\}, \bar \CA)\;,
\end{equ}
for any $A, B \in \CA$ such that $\nabla_AA \in \bar A$.
\end{corollary}

Define now a sequence of collections
of vector fields $\CA_k$ by setting (for $k \ge 1$)
\begin{equ}
\CA_1 = \{\sigma_\mu\,|\, \mu=1,\ldots,m\}\;,\quad
\CA_{k+1} = \{\nabla_{A} B, \nabla_B A\,|\, A\in \CA_1,\, B \in \CA_k\}\;.
\end{equ}
It now follows for the same reason as in Lemma~\ref{lem:oscillating} that for any two 
of the noises appearing in \eqref{eq:kpz:main2} (denote them by $\<Xi>$ and $\<Xit>$, say)
one has
\begin{equ}
2\<I1Xitwomixed> - \<Xi2mixed> - \<Xi2mixed2>  \in \CJ\;,
\end{equ}
while the kernels associated to $\<Xi2mixed2>$ and $\<Xi2mixed>$ are linearly independent. This 
shows in particular that 
\begin{equ}
\Big\{\<Xi2mixed> + \f12 \<I1Xitwomixed>,\<Xi2mixed2> + \f12 \<I1Xitwomixed>\Big\} = \{\Nabla_{\<smXit>}{\<smXi>},\Nabla_{\<smXi>}{\<smXit>}\} \in \CH\;,
\end{equ}
so that first applying
Theorem~\ref{thm:main:sv} (combined with Definition~\ref{def:CH}) 
and then Corollary~\ref{cor:addFields} implies that, for any vector field $H$ and
any finite collection of vector fields $\CB$, one has
\begin{equs}[e:corollarySupport]
\CU(H,\{\sigma_\mu\},\CB \cup \{\Nabla_{\<blue>}\<blue>\})
&=
\CU(H,\{\sigma_\mu\},\CB \cup \{\Nabla_{\<blue>}\<blue>\} \cup \CA_2) \\
&=
\CU(H,\{\sigma_\mu\} \cup \CA_2 \cup \CA_4, \CB \cup \CA_2)\;.
\end{equs}
(We could have added any of the $\CA_k$'s to the right hand side, but only $\CA_2$ and $\CA_4$
matter for the sequel.) 
Setting 
\begin{equ}
\CS_\geo^\star = \{\tau \in \CS_\geo\,:\, \Upsilon_{\Gamma,\sigma}\tau \in \Vec{(\CA_2 \cup \CA_4)}\}
\end{equ}
and combining the description of $\CS_\geo$ given in \cite[Eq.~1.8]{BrunedHairer2019} with the definition of the $\CA_i$ we
see that one has the decomposition
\begin{equ}
\CS_\geo = \CS_\geo^\star \oplus \Vec\big\{\Nabla_{\Nabla_{\<blue>}\<blue>}\Nabla_{\<dblue>}\<dblue>, 
\Nabla_{\Nabla_{\<blue>}\<dblue>}\Nabla_{\<blue>}\<dblue>,
\Nabla_{\Nabla_{\<blue>}\<dblue>}\Nabla_{\<dblue>}\<blue>\big\}\;.
\end{equ}
Similarly, it follows from \cite[Eq.~3.22]{BrunedHairer2019} that there exists a constant $c$ such that
\begin{equ}
\tau_\star - c\big(2 \Nabla_{\Nabla_{\<blue>}\<dblue>}\Nabla_{\<blue>}\<dblue>-\Nabla_{\Nabla_{\<blue>}\<dblue>}\Nabla_{\<dblue>}\<blue>\big)
\in \CS_\geo^\star\;.
\end{equ}
Setting $\tilde \tau = \Nabla_{\Nabla_{\<blue>}\<blue>}\Nabla_{\<dblue>}\<dblue>$ and combining this
with \eqref{e:trivialIdentity} and \eqref{e:corollarySupport}, we conclude that there exist constants 
$\hat c_0$ and $ \tilde c$ such that the support of $u$ is given by
\begin{equ}
\CU(h + \bar \tau,\{\sigma_\mu\},\{\tau_\star, \Nabla_{\<blue>}\<blue>\}) = 
\CU(h + \hat c_0 \hat \tau + \tilde c \tilde \tau,\{\sigma_\mu\}\cup \CA_2 \cup \CA_4,\{\tau_\star\} \cup \CA_2)\;.
\end{equ}
In order to eliminate $\tilde \tau$ we note that, as a consequence of \cite[Eq.~6.19]{BrunedHairer2019}, setting
\begin{equ}
\sigma_1 = \<2I1Xi4c2> \;,\qquad    \sigma_2 = \<2I1Xi4c1> \;,
\end{equ}
and writing $\pi \colon \CS_\geo \to \Vec\{\sigma_1,\sigma_2\}$ for the orthogonal projection,
we have
\begin{equ}
\pi \CS_\geo^\star = 0\;,\qquad \pi \tilde \tau = \sigma_1
\;,\qquad \pi \hat \tau = \pi \tau_\star = \sigma_2\;.
\end{equ}
Since furthermore $\sigma_1 \not \in \CJ$ by Definition~\ref{def:CH}, there exists $\tilde \sigma \in \CH$
with $\scal{\tilde \sigma, \sigma_1} \neq 0$ and therefore, by Lemma~\ref{lem:CH:Sgeo}, $\scal{\tilde \sigma, \tilde \tau} \neq 0$.
We conclude that there exists $\hat c$ such that the support of $u$ is given by 
\begin{equs}
\CU(h + \bar \tau&,\{\sigma_\mu\},\{\tau_\star, \Nabla_{\<blue>}\<blue>\}) =
\CU(h + \bar \tau,\{\sigma_\mu\},\{\tau_\star, \Nabla_{\<blue>}\<blue>, \tilde \sigma\}) \\
&= 
\CU(h + \hat c \hat \tau,\{\sigma_\mu\}\cup \CA_2 \cup \CA_4,\{\tau_\star, \tilde \sigma\} \cup \CA_2) 
=
\CU(h + \hat c \hat \tau,\{\sigma_\mu\},\{\tau_\star,  \Nabla_{\<blue>}\<blue>\} )\;,
\end{equs}
thus concluding the proof of Theorem~\ref{theo:gKPZ}. Here, the last identity follows from the fact that 
the preceding sequence of identities holds for any choice of $h$.

\section{Symbolic index}
\label{sec:index}

Here, we collect some of the most used symbols of the article, together
with their meaning and the page where they were first introduced.

\begin{center}
\renewcommand{\arraystretch}{1.1}
\begin{longtable}{lll}
\toprule
Symb. & Meaning & P.\\
\midrule
\endfirsthead
\toprule
Symb. & Meaning & P.\\
\midrule
\endhead
\midrule
\endfoot
\midrule
\endlastfoot
$|\cdot|_\fs$			& Homogeneity used to construct the regularity structure		& \pageref{idx:homo}\\
$\fancynorm{\cdot}_\fs$	& ``True'' homogeneity of the noise								& \pageref{idx:fancyhomo}\\
$\sim$		& Equivalence relation on $\TT_-$					& \pageref{def:sim}\\
$\preceq$	& Total order on $\TT_-$							& \pageref{idx:prec}\\
$[I,\vphi]$	& Multi-set, $[I,\varphi]_a = \#\{ i \in I : \varphi(i) = a \}$	& \pageref{idx:multiset[,]}\\
$\phantom{}^{\sfs}\fc$	& Cumulant homogeneity consistent with $\sfs$		& \pageref{idx:sfsfc}\\
$\alpha_{(\Xi,\tau)}$	& $\shift\fs(\Xi,\tau)- m(\tau) \shalf$					& \pageref{eq:alpha}\\
$\bar\CC^\infty_c$ & Functions invariant under translation of all arguments & \pageref{def:barCinfty}\\
$\td(\cm)$	& Set canonically associated to a multiset $\cm$ & \pageref{idx:multisetd}\\
$\domain$	& Domain of definition of the noise			& \pageref{idx:domain}\\
$\bar\domain$	& Whole space extension of $\domain$	& \pageref{idx:bardomain}\\
$\cpmh$		& Coproduct					& \pageref{idx:coproduct}\\
$\LTdeg$	& Degree assignment on $\Legtype$ 					& \pageref{eq:definition:deg:I}\\
$\LTsysdeg$	& Degree assignment on $\Legtype$ 					& \pageref{eq:definition:deg:FJ}\\
$\div(\tau)$  & Set of divergent subforests of $\tau$ 				& \pageref{idx:divtau}\\
$f^\eta$	& Character depending continuously on $\eta$					& \pageref{ass:CHBPHZcharacters}\\
$\zeta_\delta$	& Shift of the noise				& \pageref{idx:shift}\\
$\CG_-$ 	& Renormalisation group, character group of $\CTm$  	& \pageref{idx:RG} \\
$\wCG$		& Renormalisation group, character group of $\wCT$				& \pageref{idx:wCG}\\
$\groupD$	& Spatial symmetries of the equation	 & \pageref{symmetry} \\
$g^\eta$	& BPHZ character			& \pageref{idx:BPHZcharacter}\\
$\hat g^\eta$	& ``Tweaked'' BPHZ character $(f^\eta)^{-1} \circ g^\eta$	& \pageref{ass:CHBPHZcharacters}\\
$\wlchar{g^{\eta,\phi}_R}$	& Character on $\plCT$ and $\symCT$				& \pageref{eq:g:wl:eta:psi:R}\\
$G_\Legtype$	& Group of	permutations of $\Legtype$ consistent with $\imap$ and $\bar\cdot$	& \pageref{def:Glegs}\\
$\CH$		& Annihilator of $\CJ$					& \pageref{def:CH}\\
$h^\phi_{\LTsys,R}$		& Character on $\plCT$					& \pageref{eq:h:CI:phi:R}\\
$I_\cm(h)$	& Stochastic integration				& \pageref{idx:sint}\\
$\imap$		& Type map on $\Legtype$					& \pageref{idx:imap}\\
$\wli$		& Canonical embedding $\wCT \hookrightarrow \wCThat$ and $\plCT \hookrightarrow \plCThat$ &\pageref{idx:wli}\\
$\symfi$	& Canonical embedding $\symCT \hookrightarrow \symCThat$  &\pageref{idx:symfi}\\
$\iota$		& Admissible embedding $\CT_- \to \plCT$					& \pageref{def:admissible:embedding}\\
$\symiota$	& Admissible embedding $\CT_- \to \symCT$					& \pageref{lem:iotasym}\\
$\tilde\CJ$	& Ideal in $\CT_-$, kernel of $\evaA{}$	& \pageref{def:CJcon}\\
$\CJ$		& Ideal	in $\CT_-$ generated by $\tilde\CJ$ and trees with odd number of noises	& \pageref{def:CH}\\
$\sintf_\cm(h)$ & Stationary process in $\cm$th Wiener chaos with kernel $h$ &\pageref{idx:sint}\\
$\CKsimp^n$	& Space of smooth simple kernels in $n$ variables		& \pageref{idx:CKsimpn}\\
$K(\tau)$	& Kernel-type edges			& \pageref{idx:K(T)}\\
$\CK_G\tau$	& ``Kernel'' $\bar\domain^{L(\tau)} \to \R$ associated to $\tau$	& \pageref{ind:CK_G}\\
$\CK_\infty^+$	& Compactly supported large-scale kernel assignments					& \pageref{idx:CK+infty}\\
$\CK_0^+$	& Large-scale kernel assignments 					& \pageref{idx:CK+z}\\
$L(\tau)$	& Edges of noise type			& \pageref{idx:L(T)}\\
$L_\Legtype(\tau)$	& Edges of leg type					& \pageref{idx:LLegtype}\\
$\CL(\tau)$	& Nodes touching edges of noise type 					& \pageref{idx:CLtau}\\
$\hat\CL(\tau)$	& Nodes with non-vanishing extended decoration					& \pageref{idx:hatCL}\\
$\spacetime$	& Space-time domain			& \pageref{idx:lambda}\\
$\Lambda$	& Set of possible scales			& \pageref{eq:Lambda:indexset}\\
$\Lambda(\tau)$	& System of disjoint, non-empty subsets of $\CL(\tau)$			& \pageref{idx:Lambda(tau)}\\
$\Legtype$	& Set of leg types					& \pageref{idx:legtypes}\\
$\aFLm$		& Enlarged set of noise types 		& \pageref{idx:sFLm}\\
$[M]$ 		& Integers from $1$ to $M$ 	& \pageref{sec:notation} \\
$\CM_\infty$	& Space of smooth admissible models			& \pageref{idx:CMs}\\
$\tilde \cm$	& Map $\tilde\cm:[\#\cm] \to A$ associated to multiset $\cm$			& \pageref{idx:multisettilde}\\
$\CM_0$		& Space of admissible models			& \pageref{idx:CMz}\\
$M^g$		& Matrix acting on $\CT$				& \pageref{idx:Mg} \\
$\SMinf$	& Space of smooth noises				& \pageref{idx:SMinf}\\
$\SMz$		& Space	of singular noises				& \pageref{idx:SMz}\\
$\SMsinf$	& Space of shifted smooth noises		& \pageref{idx:SMsinf}\\
$\SMsz$	& Space of shifted singular noises			& \pageref{idx:SMzinf}\\
$\sSMinfty$	& Space of smooth noises for $\sFLm$				& \pageref{idx:sSMinfty}\\
$\cm(\Xi,\tau)$	& $[L(\tau),\ft]\setminus \{ \Xi \}$					& \pageref{idx:cmXi}\\
$\LTN(\tau)$	& Nodes of $\tau$ touching legs					& \pageref{idx:NLegtype}\\
$\SN$		& Families of smooth functions indexed by leg types					& \pageref{def:FN}\\
$\SN(\LTsys)$		& Elements of $\SN$ satisfying a constraint depending on $\LTsys$			& \pageref{def:SN:FP}\\
$\plZ$		& Projection on $\plCT$ killing trees with non vanishing $\fe$ on legs	& \pageref{idx:plZ}\\
$\plQ$		& Projection on $\plCT$ removing superfluous legs									& \pageref{idx:plQ}\\
$\plQz$		& Projection on $\plCT$, $\plQz=\plQ \plZ$					& \pageref{idx:plQz}\\
$\LTp$		& Projection on $\plCT$ onto trees $\tau$ with $\ft(L_\LT(\tau))=\LTa$				& \pageref{idx:LTp}\\
$\tilde\Psi$ 		& Set of families of test functions indexed by multisets		& \pageref{def:tildePsi}\\
$\pi$			& Projection that removes legs				& \pageref{idx:pi}\\
$\rho$		& Smooth mollifier						& \pageref{idx:rho} \\
$\CR^g$		& Action of $\CGm$ onto $\CM_0$			& \pageref{idx:CRg}\\
$\fs$		& Scaling on $\domain$ 					& \pageref{idx:fs}\\
$\sfs$		& Homogeneity assignment on $\sFLm$		& \pageref{def:en:homo} \\
$\SS$		& Shift operator $\SS:\CT \to \sCT$		& \pageref{e:defShiftOperator}\\
$\SS^\uparrow$		& ``Dominating'' part of the shift operator $\SS^\uparrow:\CT \to \sCT$		& \pageref{def:oSS}\\
$\SS^\downarrow$	& ``Non-dominating'' part of the shift operator $\SS^\downarrow:\CT \to \sCT$& \pageref{def:oSS}\\
$\CS(\lambda,\alpha)$	& Rescaling operator					& \pageref{eq:rescalingoperator}\\
$\CTex$		& Extended regularity structure			& \pageref{idx:extended:reg:str}\\
$\CT$		& Reduced regularity structure			& \pageref{idx:CT}\\
$\CTmex$	& Extended Hopf algebra		& \pageref{idx:CTmex}\\
$\CTm$		& Reduced Hopf algebra		& \pageref{idx:CTm}\\
$\CTmhatex$	& Algebra of extended trees	& \pageref{idx:CTmhatex}\\
$\CTmhat$	& Algebra of reduced trees	& \pageref{idx:CTmhat}\\
$\wTex$		& Extended regularity structure	with legs		& \pageref{idx:wTex}\\
$\wT$		& Reduced regularity structure with legs			& \pageref{idx:wT}\\
$\wCTex$	& Extended Hopf algebra with legs		& \pageref{idx:wCTex}\\
$\wCT$		& Reduced Hopf algebra with legs		& \pageref{idx:wCT}\\
$\wCThat$	& Algebra of extended trees	with legs 	& \pageref{idx:wCThat}\\
$\wlCT$		& Auxiliary Hopf algebra $\wCT/\CI$		& \pageref{def:wCT}\\
$\wlCThat$	& Auxiliary algebra $\wCThat/\hat\CI$	& \pageref{def:wCT}\\
$\plCT$		& Hopf algebra of properly legged trees		& \pageref{def:properly:legged:tree:algebra}\\
$\plCThat$	& Algebra of properly legged trees			& \pageref{def:properly:legged:tree:algebra}\\
$\adCT$		& Algebra of admissible trees				& \pageref{idx:adCT}\\
$\symCT$	& Symmetrised Hopf algebra of properly legged trees		& \pageref{def:symCT}\\
$\symCThat$	& Symmetrised algebra of properly legged trees			& \pageref{def:symCT}\\
$\pT$		& Hopf algebra isomorphic to $\CTm$						& \pageref{lem:CTadze:hopf:ideal}\\
$\sCTex$	& Enlarged regularity structure					& \pageref{idx:extended:reg:str}\\
$\TT$		& Set of trees				& \pageref{idx:TT}\\
$\TT_-$		& Set of trees of negative homogeneity			& \pageref{idx:TTmex}\\
$\wTT$		& Set of trees of negative homogeneity with legs					& \pageref{idx:wTT}\\
$\FT_-$		& Subset of $\TT_-$			& \pageref{def:FT}	\\
$T_h$		& Shift operator			& \pageref{thm:translationoperator}\\
$\CV$		& Set of trees appearing in Assumption~\ref{ass:technical}	& \pageref{idx:CV}\\
$\CV_0$	& Set of $\tau \in \TT_-$ with $\fancynorm{\tau}_\fs = 0$ and $\# L(\tau) = 2$ & \pageref{idx:CV0} \\
$\Psi$		& Set of functions indexed by typed sets 					& \pageref{def:Psi}\\
$\Upsilon^\eta\tau$	& $\E(\PPi^\eta \tau)(0)$								& \pageref{idx:Upsilon}\\
$\evalnA{\eta}{\psi}{R}\tau$	& Evaluation using large-scale kernel assignment $R$	& \pageref{idx:tildeUpsilon}\\
$\evaln{\eta}{\psi}{R}\tau$		& Like $\evalnA{\eta}{\psi}{R}\tau$, acting on trees with legs	& \pageref{eq:bar:Upsilon:large:scale}\\
$\hat\Upsilon^{\eta,\psi}_{R}\tau$		& Renormalised evaluation acting on trees with legs	& \pageref{eq:evaluation:largescale:renorm}\\
$\hopfiso$	& Hopf isomorphism $\CTm \to \pT$					& \pageref{lem:hopfiso}\\
$\CZ(\PPi)$	& Model constructed from $\PPi$			& \pageref{idx:CZ}\\
$\Zcan(f)$		& Canonical lift of $f$ to an admissible model	& \pageref{idx:canlift}\\
$\mfc(g)$	& $\CR^g Z(0)$				& \pageref{idx:mfc}\\
$\Omega_\infty$	& Space of smooth deterministic noises		& \pageref{idx:Omegainfty}\\
$\Omega$	& Space of rough deterministic noises			& \pageref{idx:Omega}
\end{longtable}
\end{center}

\section{Overview of Assumptions}
\label{app:assumptions}

\begin{center}
\renewcommand{\arraystretch}{1.1}
\begin{longtable}{lll}
\toprule
\# & Summary & P. \\
\midrule
\endfirsthead
\toprule
\# & Summary & P. \\
\midrule
\endhead
\midrule
\endfoot
\midrule
\endlastfoot
\ref{ass:main:thm}	& Ensures that the general theory of~\cite{BrunedChandraChevyrecHairer2017} 	applies	& \pageref{ass:main:thm}\\
\ref{ass:main:reg}	& Assumption necessary for the BPHZ theorem \cite{ChandraHairer2016}  		& \pageref{ass:main:reg}\\
\ref{ass:noises:derivatives:products}	& Rules out derivatives hitting noises, as well as direct products of noises & \pageref{ass:noises:derivatives:products}\\
\ref{ass:kernelhomo}	& The integration kernels are homogeneous		& \pageref{ass:kernelhomo}\\
\ref{ass:zero-homo}	& BPHZ character vanishes on zero-degree subtrees of zero-degree trees & \pageref{ass:zero-homo}\\
\ref{ass:CVz}	& BPHZ character vanishes on zero-degree trees with only two leaves	& \pageref{ass:CVz}\\
\ref{ass:CJHopfIdeal}	& The ideal $\CJ$ is a Hopf ideal	& \pageref{ass:CJHopfIdeal}\\
\ref{ass:CHBPHZcharacters}	& The BPHZ character is ``almost'' an element of $\CH$ 	& \pageref{ass:CHBPHZcharacters}\\
\end{longtable}
\end{center}
Assumptions \ref{ass:main:thm}, \ref{ass:main:reg}	and \ref{ass:noises:derivatives:products} are needed for the results from \cite{BrunedChandraChevyrecHairer2017} and \cite{ChandraHairer2016}  to apply. Assumption \ref{ass:kernelhomo} on the scale-invariance of kernels is crucial for our argument in Lemma~\ref{lem:dominatingpart} which gives \emph{lower} bounds on the blow-up of the certain renormalisation constants. Assumption \ref{ass:zero-homo} is needed for a technical argument in Lemma~\ref{lem:blowuprenormconstantzerohomo}. 
Finally, we show in Section~\ref{sec:constraints} that Assumption~\ref{ass:CVz} implies 
Assumptions~\ref{ass:CJHopfIdeal} and~\ref{ass:CHBPHZcharacters}. We believe that
the latter two assumptions are satisfied for all naturally occurring classes 
of SPDEs.

\begin{Backmatter}

\paragraph{Acknowledgments}

We are grateful to all three referees for their very careful reading
of the original manuscript, which lead to several improvements in the exposition.

\paragraph{Funding statement}

MH gratefully acknowledges financial support from the
 Leverhulme Trust through a leadership award,
from the European Research Council through a consolidator grant, project 615897,
and from the Royal Society through a research professorship.

\paragraph{Competing interests}

None

\bibliography{bibtex}
\bibliographystyle{Martin}

\end{Backmatter}

\end{document}